\documentclass[10pt,a4paper,twoside]{book}

   \usepackage[cp1250]{inputenc}
   \usepackage{url}
   \usepackage{fancyheadings}
   \usepackage{amssymb}
   \usepackage{amsmath}
   \usepackage{mathrsfs}
   \usepackage{stmaryrd}
   \usepackage{theorem}
   \usepackage{ifthen}
   \usepackage{color}
   \usepackage[all]{xypic}
   \xyoption{2cell}
   \UseAllTwocells

   {
     \theorembodyfont{\itshape}

     \newtheorem{theorem}{Theorem}[chapter]
     \newtheorem{lemma}[theorem]{Lemma}
     \newtheorem{proposition}[theorem]{Proposition}
     \newtheorem{corollary}[theorem]{Corollary}
     
     \newtheorem{izrek}{Izrek}
     \newtheorem{lema}[izrek]{Lema}
     \newtheorem{trditev}[izrek]{Trditev}
     \newtheorem{posledica}[izrek]{Posledica}
   }

   {
     \theorembodyfont{\rmfamily}
     
     \newtheorem{definition}[theorem]{Definition}

     \newtheorem{definicija}[izrek]{Definicija}
     
   }

   \newenvironment{proof}{
     \goodbreak\par
     \textit{Proof.}%
   }{%
     \nopagebreak
     \hfill{\vrule width 1ex height 1ex depth 0ex}
     \medskip
     \goodbreak
   }

	\newcommand{\setimpldiagram}{\mathop{=\hspace{-0.5em}=\hspace{-0.2em}\raisebox{0.05ex}{\text{\scriptsize$)$}}}}
	\newcommand{\setimpl}[1][]{\mathop{=\hspace{-0.5em}=\hspace{-0.5em}\raisebox{0.22ex}{\text{\scriptsize$)$}}\hspace{-0.15em}_{#1}}}
   \newcommand{\sizedescriptor}[2]
	{
		\ifthenelse{\equal{#1}{0}}{}{
		\ifthenelse{\equal{#1}{1}}{\big}{
		\ifthenelse{\equal{#1}{2}}{\Big}{
		\ifthenelse{\equal{#1}{3}}{\bigg}{
		\ifthenelse{\equal{#1}{4}}{\Bigg}{
		#2}}}}}
	}
	\newcommand{\longer}{-\hspace{-1ex}}
   \newcommand{\nc}[1][\land]{~{#1}~}
   \newcommand{\nl}[1][\land]{\nc[#1]\\&\nc[#1]}

   \newcommand{\ie}[1][~]{i.e.{#1}}
   
   \newcommand{\eg}[1][~]{e.g.{#1}}
   \newcommand{\cf}[1][~]{cf.{#1}}
   \newcommand{\etc}[1][~]{etc.{#1}}
   \newcommand{\Sier}{Sierpiński~}
   \newcommand{\Esc}[1][~]{Escardó{#1}}
   \newcommand{\Martin}{Martín}
   \newcommand{\MEsc}[1][~]{\Martin\ \Esc[#1]}

   \newcommand{\intermission}{\bigskip\medskip}
   \newcommand{\df}[1]{\emph{#1}}
   \newcommand{\ism}{\cong}
   \newcommand{\equ}{\sim}
   \newcommand{\dfeq}{:=}
   
   \newcommand{\sepeq}{\ =\ }
   \newcommand{\apart}{\mathrel{\#}}
   \newcommand{\id}[1][]{\textrm{Id}_{#1}}
   \newcommand{\insarg}{\text{---}}
   \newcommand{\cs}[1][\RR]{\text{cs}_{#1}}
   \newcommand{\verylongrightarrow}{\longer\longer\longer\longer\longer\longer\longer\longer\longrightarrow}
   \newcommand{\impl}{\Rightarrow}

   \newcommand{\all}[3]{\forall\, #1 \,{\in}\, #2\,.\left(#3\right)}
   \newcommand{\some}[3]{\exists\, #1 \,{\in}\, #2\,.\left(#3\right)}

   \newcommand{\xall}[3]{\forall\, #1 \,{\in}\, #2\,.\,#3}
   \newcommand{\xsome}[3]{\exists\, #1 \,{\in}\, #2\,.\,#3}
   \newcommand{\xexactlyone}[3]{\exists!\, #1 \,{\in}\, #2\,.\,#3}

   \newcommand{\st}[3][auto]{\sizedescriptor{#1}{\left}\{#2\;\sizedescriptor{#1}{\middle}|\;#3\sizedescriptor{#1}{\right}\}}
   \newcommand{\mlst}[5][\Big]{\begin{align*} #2 #1\{ #3 \ #1|& \ #4 #1\}#5 \end{align*}}
   \newcommand{\vsubset}{\Mapstochar\cap}
   \newcommand{\tp}{\mathcal{O}}
   \newcommand{\ctp}{\mathcal{Z}}
   \newcommand{\C}{\mathscr{C}}

   \newcommand{\NN}{\mathbb{N}}
   \newcommand{\ZZ}{\mathbb{Z}}
   \newcommand{\QQ}{\mathbb{Q}}
   \newcommand{\RR}{\mathbb{R}}

   \newcommand{\II}{\mathbb{I}}
   \newcommand{\oirs}{\mathbb{S}}
   \newcommand{\lrs}{\mathbb{D}}
   \newcommand{\intoo}[3][\RR]{{#1}_{(#2, #3)}}
   \newcommand{\intcc}[3][\RR]{{#1}_{[#2, #3]}}
   \newcommand{\intoc}[3][\RR]{{#1}_{(#2, #3]}}
   \newcommand{\intco}[3][\RR]{{#1}_{[#2, #3)}}

   \newcommand{\trivtwo}{\mathbf{T}}
   \newcommand{\sier}{\mathbf{S}}
   \newcommand{\twsier}{\tilde{\sier}}

   \newcommand{\kat}[1]{\mathbf{\underline{#1}}}
   \newcommand{\op}{^{\mathrm{op}}}
   \newcommand{\mrph}[1]{\text{Morph}\left(#1\right)}
   \newcommand{\dom}{\text{dom}}
   \newcommand{\cod}{\text{cod}}
   \newcommand{\one}{\mathbf{1}}
   \newcommand{\unit}{*}
   \newcommand{\two}{\mathbf{2}}
   \newcommand{\trm}[1][]{\mathord{!}_{#1}}
   \newcommand{\limit}[1]{\underleftarrow{\lim}{#1}}
   \newcommand{\colimit}[1]{\underrightarrow{\lim}{#1}}
   \newcommand{\ppair}[2]{\ar@<0.3ex>[r]^{#1} \ar@<-0.3ex>[r]_{#2}}

	\newcommand{\pbcorner}[1][dr]{\save*!/#1-1.2pc/#1:(-1,1)@^{|-}\restore}
	\newcommand{\pocorner}[1][dr]{\save*!/#1+1.2pc/#1:(1,-1)@^{|-}\restore}
	\newdir{ >}{{}*!/-8pt/@{>}}
	\newdir{ (}{{}*!/-5pt/@^{(}}
	\newdir{|>}{!/5pt/@{|}*:(1,-.2)@{>}*:(1,+.2)@_{>}}

   \newcommand{\cat}{\kat{C}}
   \newcommand{\topos}{\kat{\mathcal{E}}}
   \newcommand{\Set}{\kat{Set}}
   \newcommand{\Top}{\kat{Top}}

	\newcommand{\clsc}[1][]{CLASS{#1}}
	\newcommand{\russ}[1][]{RUSS{#1}}
	\newcommand{\bint}[1][]{INT${}^{+}${#1}}
	\newcommand{\grtp}[1][]{GROS{#1}}

\newcommand{\upward}[1]{\mathord{\uparrow} #1}

\newcommand{\postulate}[1]{\begin{quote}\textit{#1}\end{quote}}
\newcommand{\nnst}[1][\soc]{{#1}_{\lnot\lnot}}
\newcommand{\tst}{\mathrm{T}}
\newcommand{\opn}{\Sigma}
\newcommand{\cld}{\mathrm{Z}}
\newcommand{\soc}{\Omega}
\newcommand{\tetp}{\mathcal{T}}
\newcommand{\optp}{\mathcal{O}}
\newcommand{\cltp}{\mathcal{Z}}
\newcommand{\pst}{\mathcal{P}}

\newcommand{\fin}{\mathcal{F}}
\newcommand{\inhfin}{\fin^{+}}
\newcommand{\ov}{\mathrm{Overt}}
\newcommand{\inhov}{\ov^+}

\newcommand{\cp}{\mathrm{Comp}}
\newcommand{\itemimpl}[2]{$({#1}\Rightarrow{#2})$\vspace{0.2em}\\}
\newcommand{\lowercuts}[1][\opn]{\mathcal{L}_{#1}}
\newcommand{\uppercuts}[1][\opn]{\mathcal{U}_{#1}}
\newcommand{\Dedcuts}[1][\opn]{\mathcal{D}_{#1}}
\newcommand{\ball}[3][]{B_{#1}\left(#2, #3\right)}
\newcommand{\bball}[3][]{\overline{B}_{#1}\left(#2, #3\right)}

\newcommand{\rstr}[1]{\left.{#1}\right|}
\newcommand{\cmpl}[1]{\widehat{#1}}
\newcommand{\ed}{$\epsilon$-$\delta$}
\newcommand{\mtr}[1]{\mathbf{#1}}
\newcommand{\cmtr}[1]{\cmpl{\mathbf{#1}}}
\newcommand{\proven}[1]{\underline{#1}\vspace{0.2em}\\}
\newcommand{\Cauchy}{\textrm{Cauchy}}
\newcommand{\Cf}{\textrm{CF}}
\newcommand{\ac}{\text{\textsf{AC}}}
\newcommand{\AC}[2][]{\ac_{#1}(#2)}
\newcommand{\ACopn}{\AC[\opn]{\NN, \NN}}
\newcommand{\Kq}[1]{\textrm{Kq}(#1)}
\newcommand{\Kqm}[1][d]{\widetilde{#1}}
\newcommand{\Ccmpl}{\textrm{C}}
\newcommand{\lvl}{\mathbb{L}}
\newcommand{\Rcmpl}{\textrm{R}}
\newcommand{\loc}{\mathscr{L}}
\newcommand{\mtp}{\mathcal{M}}
\newcommand{\Ros}{\opn^0_1}
\newcommand{\ACRos}{\AC{\NN, \two^\NN \twoheadrightarrow \Ros}}
\newcommand{\Baire}{\mathbb{B}}
\newcommand{\Cantor}{\mathbb{C}}
\newcommand{\opcN}{\NN^\bullet}
\newcommand{\cms}[1][~]{\textbf{CMS}{#1}}
\newcommand{\ctb}[1][~]{\textbf{CTB}{#1}}
\newcommand{\zfc}[1][~]{\textbf{ZFC}{#1}}
\newcommand{\wso}[1][~]{\textbf{WSO}{#1}}
\newcommand{\lpo}[1][~]{\textbf{LPO}{#1}}
\newcommand{\lem}[1][~]{\textbf{LEM}{#1}}
\newcommand{\wlem}[1][~]{\textbf{WLEM}{#1}}
\newcommand{\principle}[2]{\begin{quote} ({#1}) \quad \textit{#2} \end{quote}}
\newcommand{\finseq}[1]{{#1}^*}
\newcommand{\fbs}{\finseq{\two}}

\newcommand{\prefix}[2]{\overline{#1}(#2)}
\newcommand{\lex}{\sqsubseteq}
\newcommand{\cnct}{{{:}{:}}}
\newcommand{\concat}[3]{#2 {\restriction_#1} #3}
\newcommand{\scom}[1]{{{#1}^C}}
\newcommand{\shift}[1]{S_{#1}}
\newcommand{\succm}{\text{succ}}
\newcommand{\predm}{\text{pred}}

\newcommand{\Ury}{\mathbb{U}}
\newcommand{\Hilbert}{\mathcal{H}}
\newcommand{\diam}{\text{diam}}
\newcommand{\lUt}{\text{lgt}}
\newcommand{\Utuple}[4][auto]{\left({#3}_{#2}, {#4}_{#2}\right)_{#2 \in \NN_{< \ifthenelse{\equal{#1}{auto}}{\lUt(#3)}{#1}}}}
\newcommand{\eUt}{()}
\newcommand{\aUt}{\text{age}}

\newcommand{\parto}{\mathrel{\rightharpoonup}}
\newcommand{\Tot}{\mathsf{Tot}}
\newcommand{\diag}[1][]{\Delta_{#1}}
\newcommand{\basis}{\mathscr{B}}

\newcommand{\site}{\mathfrak{T}}
\newcommand{\y}[1][]{\mathsf{y}{#1}}
\newcommand{\nat}{\text{Nat}}
\newcommand{\ev}{\epsilon}
\newcommand{\Sh}[1][\site]{\mathsf{Sh}(#1)}
\newcommand{\obj}[1]{\text{Obj}(#1)}
\newcommand{\ms}[1]{\cod^{-1}(#1)}
\newcommand{\sieve}[1][S]{\mathcal{#1}}
\newcommand{\im}{\text{im}}
\newcommand{\interior}{\text{Int}}

\makeatletter
\def\url@leostyle{%
  \@ifundefined{selectfont}{\def\UrlFont{\sf}}{\def\UrlFont{\small\ttfamily}}}
\makeatother
\newcommand{\auth}{Davorin~Lešnik}
\newcommand{\advisor}{izr.~prof.~dr.~Andrej~Bauer}
\newcommand{\tit}{Synthetic Topology and Constructive Metric Spaces}
\newcommand{\sltit}{Sintetična topologija in konstruktivni metrični prostori}
\newcommand{\fac}{University of Ljubljana \par Faculty of Mathematics and Physics \par Department of Mathematics}
\newcommand{\slfac}{Univerza v Ljubljani \par Fakulteta za matematiko in fiziko \par Oddelek za matematiko}

\author{\auth}
\title{\tit}

\renewcommand{\titlepage}[4]{
   \thispagestyle{empty}
   \begin{center}
      \vspace{\stretch{1}}
      
         \Large \textsc{#1}
      
      \vspace{\stretch{5}}
      
         \LARGE \auth

      \vspace{\stretch{1}}

         \huge {\bfseries #2}

      \vspace{\stretch{1}}

         \Large #3
     
      \vspace{\stretch{8}}
      
         \Large \textsc{#4:} \advisor
      
      \vspace{\stretch{5}}
      
         \Large Ljubljana, 2010
      
      \vspace{\stretch{2}}
   \end{center}
}

\begin{document}

   \pagestyle{empty}

   
   \titlepage{\fac}{\tit}{Doctoral~thesis}{Advisor}
   \cleardoublepage
   
   \titlepage{\slfac}{\sltit}{Doktorska~disertacija}{Mentor}
   \cleardoublepage


   \chapter*{Abstract}

      {
      The thesis presents the subject of synthetic topology, especially with relation to metric spaces. A model of synthetic topology is a categorical model in which objects possess an intrinsic topology in a suitable sense, and all morphisms are continuous with regard to it. We redefine synthetic topology in order to incorporate closed sets, and several generalizations are made. Real numbers are reconstructed (to suit the new background) as open Dedekind cuts. An extensive theory is developed when metric and intrinsic topology match. In the end the results are examined in four specific models.}

      \bigskip\bigskip\bigskip

      \noindent
      \textbf{Math.~Subj.~Class.~(MSC 2010):}
      \texttt{03F60},	
      \texttt{18C50}		

      \medskip

      \noindent
      \textbf{Keywords:}\\
      synthetic, topology, metric spaces, real numbers, constructive

   \chapter*{Povzetek}

      {
      Disertacija predstavi področje sintetične topologije, zlasti njeno povezavo z metričnimi prostori. Model sintetične topologije je kategorični model, v katerem objekti posedujejo intrinzično topologijo v ustreznem smislu, glede na katero so vsi morfizmi zvezni. Definicije in izreki sintetične topologije so posplošeni, med drugim tako, da vključujejo zaprte množice. Realna števila so (zaradi sintetičnega ozadja) rekonstruirana kot odprti Dedekindovi rezi. Obširna teorija je razvita, kdaj se metrična in intrinzična topologija ujemata. Na koncu proučimo dobljene rezultate v štirih izbranih modelih.
      }

      \bigskip\bigskip\bigskip

      \noindent
      \textbf{Math.~Subj.~Class.~(MSC 2010):}
      \texttt{03F60},	
      \texttt{18C50}		

      \medskip

      \noindent
      \textbf{Ključne besede:}\\
      sintetičen, topologija, metrični prostori, realna števila, konstruktiven

   \pagestyle{fancyplain}

   {
   \renewcommand{\markboth}[2]{}
   \tableofcontents
   }



   \chapter*{Introduction}\addcontentsline{toc}{chapter}{\numberline{}Introduction}

	\section{Acknowledgements}
	
		I would like to thank my parents for all the support over the years. This gratitude extends to my other relatives and all my friends, too numerous to mention.
		
		I am grateful to the government of the Republic of Slovenia for funding my doctoral study, and to the staff at the Institute of mathematics, physics and mechanics for the employment and all the cooperation over the years.
		
		I thank my roommates Sara, Samo and Vesna (not to mention the two cats Odi and Edi) for being closest to me while I was writing this dissertation (as well as having to endure my crankiness during this time).
		
		Regarding the mathematical work for the thesis, I am grateful for insights and discussions with Paul Taylor, Matija Pretnar and \MEsc[.] Matija, my former schoolmate, was the one who introduced me to synthetic topology, and has been a good friend and collaborator for years. \Martin\ was the first to explicitly define synthetic topology, and did good work on the subject; I am thankful for his cooperation and for hosting me in Birmingham during the finishing stage of the thesis.
		
		Most of all I would like to thank my advisor Andrej Bauer who singlehandedly taught me category theory, constructive mathematics and computability theory, has been extremely friendly, supportive and helpful, has taken considerable amount of time for discussions with me, introduced me to professional mathematics and professional mathematicians, and for being so patient with me, a trait which was no doubt tested numerous times. I am especially grateful for his trust in me, and for the freedom he has given me that I needed to finish my work. In his own thesis he wrote that he hoped to one day be as good an advisor to his students as his had been to him; I daresay he succeeded.

	\section{Overview}
	
		In this introductory section we explain the title and give an overview of the thesis.
		
		We start by explaining the first word in the title. The `synthetic' approach to mathematics started some three decades ago with synthetic differential geometry~\cite{Kock_A_2006:_synthetic_differential_geometry, Lawvere_FW_1998:_outline_of_synthetic_differential_geometry, Shulman_M_2006:_synthetic_differential_geometry}. The differentials were originally thought of as ``infinitely small'' quantities (infinitesimals) until it was shown that no such real numbers (aside from the trivial case of $0$) exist. Consequently a differential (form) nowadays means a section of a cotangent bundle, a distinctly more complicated object. In spite of this, reasoning with differentials as infinitesimal quantities still yields correct theorems and formulae, physicists get numerically correct results, and even mathematicians prefer this heuristic when dealing with differential equations (\eg when solving them via separation of variables). Is there a way then to explain this, and develop mathematics of infinitely small quantities formally?
		
		One way to do it is to change the notion of real numbers; see \eg Robinson's non-standard analysis~\cite{Robinson_A_1996:_nonstandard_analysis}. This method works, but has an unfortunate consequence that even though differentials by themselves become simpler (compared to the classical approach), the entire analysis becomes more complicated since the real numbers are now more complicated objects.
		
		An ideal solution would be if the real numbers stayed what they are (or at least altered in a way small enough as to not make them more complicated), and meaningful differentials still existed. As mentioned, their non-existence can in that case be proven, but the crucial assumption is the \df{decidable equality} of the reals, \ie for all real numbers $x, y \in \RR$, either $x = y$ or $x \neq y$. This is a special instance of the more general \df{law of excluded middle} \lem which states that for all propositions (more precisely, their truth values) $p$ either $p$ or the negation $\lnot{p}$ holds (``every statement is either true or false''). The corollary is that if we want non-trivial infinitesimals to be real numbers, we need to relinquish the use of decidability of equality on the reals, and consequently the use of the full \lem as well, leading us into the realm of constructive mathematics.
		
		Models of a theory in which we have the reals with infinitesimals can be constructed, showing that this theory is consistent. These models are complicated though (\eg certain kind of sheaves over differential manifolds); what makes this approach useful is that we can forget about specific models, reasoning instead with a different kind of axioms and logical rules of inference, but within those working just like in the usual set theory (the new axioms essentially capture the existence and uniqueness of a derivative of every map from reals to reals).
		
		The point is that when studying a certain structure (smoothness in this case), the classical approach is to equip sets with this additional structure and identify maps which preserve it, \eg groups and group homomorphisms, vector spaces and linear maps, partial order and monotone maps, topological spaces and continuous maps, measure spaces and measurable maps, smooth manifolds and smooth maps etc. The fact that the structure is an addition onto a set, not a part of its intrinsic structure, brings with itself a lot of baggage --- it is not sufficient to construct a space, one must also equip it with smooth structure (consider for example what it takes to fully define the tangent bundle), and after constructing a map, one has to verify that it is smooth (or if not, ``smoothen'' it, \ie suitably approximate it with a smooth map). However, within a model of synthetic differential geometry every object (satisfying a certain \df{microlinearity} property) automatically has a unique smooth structure, and every map is automatically differentiable.
		
		Because the structure we study is ingrained in the sets and maps themselves, constructions and proofs simplify significantly.\footnote{For example, if $\Delta \subseteq \RR$ denotes the set of differentials, then the tangent bundle on a manifold $M$ is just the set of maps from $\Delta$ to $M$, and the projection is the evaluation at $0$.} Still, by itself this merely means that it is easier to prove a statement \emph{within} a synthetic model, but we don't know yet whether that implies anything for the classical mathematics, so part of the synthetic approach is to also construct a bridge between them over which synthetic proofs translate to classical ones. In practice this usually means that a classical category is embedded into a synthetic model in such a way that content of logical statements is (sufficiently) preserved, and thus a synthetic theorem yields a corresponding classical one. For example, it is very easy to prove synthetically the usual formulae for length of curves, areas between curves etc., and this theory ensures that these proofs are good enough.
		
		This ``synthetic approach'' proves fruitful in contexts other than differential geometry as well. Generally it consists of the following.
		\begin{itemize}
			\item\proven{Identify what axioms and logic are suitable for study of a particular structure.}
				This is why this is called a \df{synthetic} approach: we create an artificial framework which is specifically tailored to the studied field. The logic in these systems is typically constructive (see Chapter~\ref{Chapter: constructivism}), either because assuming excluded middle yields a trivial special case (as for example in synthetic topology) or because the chosen axioms falsify \lem outright (as in the case of synthetic differential geometry).
			\item\proven{Reason within these axioms and logic to develop the theory.}
				The classical definitions must be reinterpreted within the synthetic setting; in particular, a notion might have several definitions which are classically, but not constructively, equivalent, so it must be determined which one (if any) is suitable synthetically. After the definitions the development of the theory usually starts by reproving classical theorems synthetically. Naturality and simplicity of definitions and proofs is a good indication of how good a choice of axioms we made. If we did well, we get an immediate payoff in simpler proofs, but it is better if the theory gives us new insight to prove new theorems as well.
			\item\proven{Construct models of the theory.}
				Since G\"odel onward we know that we can expect to have many non-isomorphic models satisfying prescribed axioms, as long as they are of any reasonable complexity (with apologies to ultrafinitists, natural numbers are practically a must). While inability to have an (up to isomorphism) unique (and therefore the ``right one'') model was seen detrimental to some mathematicians, this diversity is appreciated in the synthetic approach. At the very basic level constructing models serves to prove consistency of our chosen axioms, but it is otherwise useful --- and often insightful --- to see what general theorems mean in a specific model. Furthermore, specific models are a source of (counter)examples as well.
			\item\proven{Connect classical theory with synthetic one.}
				Validity of theorems in some odd model tells us little about the validity of these theorems in mathematics, done by a working mathematician (usually the classical one). Thus a method of translating proofs and validity from synthetic to classical setting is required. Once done, the outcome is that we can prove full classical results with simpler synthetic proofs.
		\end{itemize}
		
		\intermission
		
		The purpose of this thesis is to present the reinterpretation of the field of \emph{topology} in a synthetic way, and demonstrate synthetic methods specifically in the case of metric spaces. The last point --- the transfer of synthetic theorems to classical ones --- has been done~\cite{Escardo_M_2004:_synthetic_topology_of_data_types_and_classical_spaces, Taylor_P_2006:_computably_based_locally_compact_spaces, Pretnar_M_2005:_sinteticna_topologija}; hints of it may also be found in Section~\ref{Section: gros_topos_model} in this thesis, though we won't go into detail. Instead we focus on the other three points. Here is the breakdown of the thesis, together with notes which results are new.
		
		\begin{itemize}
			\item
				Chapter~\ref{Chapter: constructivism} recalls the notions and methods of constructivism which serves as background of our theory. Specifically we consider (higher-order) intuitionistic logic, as well as its categorical models, the topoi. We recall how number sets (natural numbers, integers, rationals, reals) are defined and constructed within this background.
				
				The chapter is introductory in nature (for classical mathematicians to constructivism), and contains results which are (for constructivists) already well-known.
			\item
				Chapter~\ref{Chapter: synthetic_topology} contains basic results in general synthetic topology.
				
				Section~\ref{Section: to_synthetic_topology} recaps basic definitions and known results, together with motivation for them from the point of view of a classical topologist.
				
				Section~\ref{Section: redefinition_of_synthetic_topology} presents my suggestion for the redefinition of synthetic topology (in order to accommodate closed sets), and develops this theory. The introduction of test truth values, the derivation of open and closed subsets from them, the redefinitions, the generalization of compactness (and overtness) to maps, the notion of condenseness, the codominance axiom, continuity of maps with regard to synthetically closed subsets, and the analysis of topology when open truth values are stable and form a bounded lattice, are new.
				
				Section~\ref{Section: (co)limits_and_bases} examines the topology of limits and colimits, and introduces the notion of a basis for synthetic topology. Aside from the obvious result that the topology of a colimit is the limit of topologies, the rest is new.
				
				Throughout the thesis models of synthetic topology are taken to be topoi. Section~\ref{Section: predicative_version} suggests how to develop synthetic topology in a general category (but does not go in depth). It is new.
			\item
				Chapter~\ref{Chapter: real_numbers_metric_spaces} develops real numbers, metric spaces and their generalizations, suitable for synthetic topology. It serves as the preparation for the next chapter.
				
				Section~\ref{Section: real_numbers} is an in-depth construction of real numbers, satisfying the (for us) crucial condition that the strict order is open in synthetic topology. The ideas for the construction of real numbers are of course not new, but the addition of the topological part (at least the way we do it) is. Also, the introduction of the structure of a \df{streak}, and using it for the construction of the reals (and probably the fact that the whole detailed construction of the reals is actually written down in one place) is new. The set of real numbers is defined as the terminal streak, and we show that open Dedekind cuts are its model.
				
				Section~\ref{Section: metric_spaces} introduces metric, pseudometric and protometric spaces, as well as the standard related definitions (such as metric separability). As such, most of the content is already familiar to mathematicians. The definition of total boundedness in the choiceless environment is new. As far as I know, usage of protometric spaces (here they are just defined; they are actually used in Section~\ref{Section: metrization_via_embeddings}) is new as well.
				
				Section~\ref{Section: completions} deals with completions of the spaces, defined in the previous section. The standard constructions of completions are well known, but we adopt them to synthetic setting. Also, the actual definition of completeness (the terminal object in the coslice category of (pseudo)metric spaces and dense isometries) is, as far as I know, new.
			\item
				Chapter~\ref{Chapter: metrization_theorems} examines the connection between the synthetic and metric topology, giving sufficient (and necessary) conditions for metrization (= match between them) for complete separable (or at least complete totally bounded) metric spaces. It is entirely new.
				
				Section~\ref{Section: wso} defines \wso[,] a principle closely connected to metrization, and explores its consequences. Sections~\ref{Section: metrization_via_quotients} and~\ref{Section: metrization_via_embeddings} prove the transfer of metrization via sufficiently nice maps, the former via quotients, and the latter via embeddings. As a bonus, Section~\ref{Section: metrization_via_embeddings} presents a constructive version of the Urysohn space.
			\item
				Chapter~\ref{Chapter: models} examines the results, especially those in the previous chapter, in specific models of synthetic topology: the classical set theory in Section~\ref{Section: classical_mathematics_model}, Type I computability in Section~\ref{Section: russian_constructivism_model}, Type II computability in Section~\ref{Section: Brouwer's_intuitionism_model}, and gros topos in Section~\ref{Section: gros_topos_model}. It is new.
		\end{itemize}
		Of the new parts of the thesis, Sections~\ref{Section: wso},~\ref{Section: metrization_via_quotients} and~\ref{Section: russian_constructivism_model} are joint work with my advisor Andrej Bauer, and Section~\ref{Section: Brouwer's_intuitionism_model} is entirely his work. The rest is mine.
		
		\intermission
		
		The text is written very linearly, and should be read in order, as nearly every section requires the ones before. The few exceptions to this rule are the following.
		\begin{itemize}
			\item
				Chapter~\ref{Chapter: constructivism} can be skipped by readers already familiar with constructivism.
			\item
				Section~\ref{Section: predicative_version} is a detour from the main thread, and is not required later.
			\item
				Section~\ref{Section: gros_topos_model} is independent from Sections~\ref{Section: russian_constructivism_model} and~\ref{Section: Brouwer's_intuitionism_model}, and can be read before or after them.
		\end{itemize}
		
		The required foreknowledge for the thesis is solid background in general topology and category theory (though a good portion of the thesis can be understood without the latter). Familiarity with logic and constructive (reverse) mathematics is welcome.

	\section{Terminology and Notation}
	
		Most of the terminology and notation is standard. We present possible exceptions and borderline cases here.
		
		\smallskip
		
		\begin{itemize}
			\item
				For a category $\cat$ we denote the family of its objects by $\obj{\cat}$, and the family of its morphisms by $\mrph{\cat}$. The domain and the codomain of a morphism $f \in \mrph{\cat}$ are denoted by $\dom(f)$ and $\cod(f)$, respectively.
			\item
				The category of sets and maps is denoted by $\Set$. The category of (classical) topological spaces and continuous maps is denoted by $\Top$.
			\item
				Number sets are denoted by $\NN$ (natural numbers), $\ZZ$ (integers), $\QQ$ (rationals), and $\RR$ (reals). Zero is taken as a natural number (so $\NN = \{0, 1, 2, 3,\ldots\}$).
			\item
				Subsets of number sets, obtained by comparison with a certain number, are denoted by the suitable order sign and that number in the index. For example, $\NN_{< 42}$ denotes the set $\st{n \in \NN}{n < 42} = \{0, 1, \ldots, 41\}$ of all natural numbers smaller than $42$, and $\RR_{\geq 0}$ denotes the set $\st{x \in \RR}{x \geq 0}$ of non-negative real numbers.
			\item
				Intervals between two numbers are denoted by these two numbers in brackets and in the index. Round, or open, brackets $(\ )$ denote the absence of the boundary in the set, and square, or closed, brackets $[\ ]$ its presence; for example $\intco[\NN]{5}{10} = \st{n \in \NN}{5 \leq n < 10} = \{5, 6, 7, 8, 9\}$. We specifically denote the closed unit real interval by $\II \dfeq \intcc{0}{1} = \st{x \in \RR}{0 \leq x \leq 1}$.
			\item
				The difference of sets $A$ and $B$ is denoted and defined as $A \setminus B \dfeq \st{x \in A}{x \notin B}$. If a set $X$ is fixed, then a complement of its subset $A$ is denoted and defined as $A^C \dfeq X \setminus A$.
			\item
				The statement $A \between B$ denotes that the sets $A$ and $B$ intersect (in constructive terms, that their intersection is inhabited).
			\item
				The set of maps from $A$ to $B$ is written as the exponential $B^A$.
			\item
				The set of finite sequences of elements in $A$ is denoted by $\finseq{A}$.
			\item
				Concatenation of sequences $a$ and $b$ is $a \cnct b$.
			\item
				Given sets $A \subseteq X$, $B \subseteq Y$ and a map $f\colon X \to Y$ with the image $\im(f) \subseteq B$, the restriction of $f$ to $A$ and $B$ is denoted by $\rstr{f}_A^B$. When we restrict only the domain or only the codomain, we write $\rstr{f}_A$ and $\rstr{f}^B$, respectively.
			\item
				The empty set (categorically, the initial object) is denoted by $\emptyset$, a singleton (the terminal object) by $\one$ (and its sole element by $*$), and a two-element set (binary coproduct $\one + \one$) by $\two$. In logical context we denote its elements by $\two = \{\top, \bot\}$, and in the context of number sets by $\two = \{0, 1\}$.
			\item
				The onto maps are called surjective, and the one-to-one maps injective. This should not be confused with injectivity in the categorical sense which is used nowhere in the thesis.
			\item
				The quotient of a set $X$ by an equivalence relation $\equ$ is denoted by $X/_\equ$. Its elements --- the equivalence classes --- are denoted by $[x]$ where $x \in X$ (\ie if $q\colon X \to X/_\equ$ is the quotient map, then $[x] \dfeq q(x)$).
			\item
				Maps $f\colon X \to Y$ and $g\colon Y \to X$ with the property $f \circ g = \id[Y]$ are called \df{split} (by one another). In that case $f$ is called a \df{retraction}, and $g$ its \df{section}.
			\item
				The topology (the set of open subsets) of $X$ is denoted by $\optp(X)$, and the set of closed subsets by $\cltp(X)$.
			\item
				Classically various definitions of finiteness are equivalent, but constructively this is not so. \df{Finite} in this thesis means \df{Kuratowski finite}, \ie nullary or unary or (recursively) a binary union of finite. See also the (equivalent) definition of finiteness in Section~\ref{Section: intuitionism} after assuming that we have the set of natural numbers. That said, we adopt the following terminology for (co)products: by a binary (co)product we mean one indexed by $\two$, a finite (co)product means indexed by $\NN_{< n}$ for some $n \in \NN$, and a countable or countably infinite (co)product means it is indexed by $\NN$.
			\item
				A lattice is called \df{bounded} when it has the smallest and the largest element (which are the nullary join and nullary meet respectively, so in other words, a bounded lattice is a partial order with all finite meets and joins). Similarly, a bounded meet-semilattice is the one which has the largest element (and thus all finite meets), and a bounded join-semillatice is the one which has the smallest element (has all finite joins).
			\item
				For notation of variants of the axiom of choice see Definition~\ref{Definition: ac}.
		\end{itemize}
		
		\intermission
		
		The following symbols are introduced later in the text (we present them here for reference):
		\begin{itemize}
			\item
				$\soc$: the set of truth values (Section~\ref{Section: intuitionism}), or the subobject classifier in a topos (Section~\ref{Section: topoi}),
			\item
				$\trivtwo$: a two-elements set, equipped with trivial topology (Section~\ref{Section: to_synthetic_topology}),
			\item
				$\sier$: the \Sier topological space (Section~\ref{Section: to_synthetic_topology}),
			\item
				$\opn$: the \Sier object, or the set of open truth values (Sections~\ref{Section: to_synthetic_topology} and~\ref{Section: redefinition_of_synthetic_topology}),
			\item
				$\tst$: the set of test truth values (Section~\ref{Section: redefinition_of_synthetic_topology}),
			\item
				$\cld$: the set of closed truth values (Section~\ref{Section: redefinition_of_synthetic_topology}),
			\item
				$\oirs$: a (fixed) overt interpolating ring streak, used to construct the real numbers from (Section~\ref{Section: real_numbers}),
			\item
				$\lrs$: a (fixed) lattice ring streak which serves as the set of possible distances in a metric space (Sections~\ref{Section: metric_spaces} and~\ref{Section: metrization_via_embeddings}),
			\item
				$d_E$: the Euclidean metric (Section~\ref{Section: metric_spaces}),
			\item
				$d_D$: the discrete metric (Section~\ref{Section: metric_spaces}),
			\item
				$\Baire$: the Baire space, \ie the metric space $\NN^\NN$ with the comparison metric $d_C$ (Chapter~\ref{Chapter: metrization_theorems}),
			\item
				$\Cantor$: the Cantor space, \ie the metric space $\two^\NN$ with (the restriction of) the comparison metric (Chapter~\ref{Chapter: metrization_theorems}),
			\item
				$\opcN$: the ``one-point compactification'' of $\NN$, the metric space of decreasing binary sequences, equipped with (the restriction of) the comparison metric (Chapter~\ref{Chapter: metrization_theorems}),
			\item
				$\succm$, $\predm$: the successor and predecessor map on $\opcN$ (Chapter~\ref{Chapter: metrization_theorems}),
			\item
				$\site$: the site of topological spaces over which gros topos is made (Section~\ref{Section: gros_topos_model}).
		\end{itemize}
		
		\intermission
		
		A few more words about style.
		\begin{itemize}
			\item
				Most of the thesis (when the writer is the subject) is written in first person plural (as is usual in mathematical texts), but occasionally I use first person singular. The former is used more formally (and when I expect the reader to go along with and except what is written), whereas the latter is used when I want to express my personal style, opinion, preference or suggestion.
			\item
				A crossed-out relation always denotes just the negation of that relation, as opposed to some constructively stronger notion. For example, $x \neq y$ means $\lnot(x = y)$ while the stronger apartness relation is denoted by $x \apart y$.
			\item
				Given a map $a\colon N \to A$ where $N$ is a subset of natural numbers, we often write simply $a_k$ instead of $a(k)$ for the value of $a$ at $k \in N$.
			\item
				In mathematical practice definitions typically follow the pattern that a term is defined to be equivalent to something having some property, so mathematically the phrase ``if and only if'' or similar should be used in definitions. It is tedious to always write this, and often doesn't sound natural, so in practice mathematicians usually write simply ``if''. This means equivalence in definitions, but merely an implication everywhere else, and this discrepancy bothers me. Therefore I use the word ``when'' in definitions; it's a single word, and has the connotation of equivalence.
			\item
				A statement is often identified with its truth value. Consequently I use the equivalence $\iff$ and equality $=$ between statements/truth values interchangeably.
			\item
				I often speak about \emph{the} (co)limit (for example, \emph{the} terminal object) and other categorical constructions, though this is not strictly correct since uniqueness of these constructions is only up to isomorphism, and that is how this definite article should be understood (the same way as we say \emph{the} set of natural numbers, and so on); but see also the explanation in Section~\ref{Section: topoi} about these constructions being functors.
			\item
				I don't always go into excruciating detail about foundational issues, and am willing to write a statement which for example quantifies over all sets\footnote{Technically one can do this if using classes, but definitely not in an internal language in a topos.}, especially if this makes the subject conceptually clearer. It is assumed that the reader possesses enough mathematical experience to translate such a statement to an unobjectionable formal mathematical form. For example, Theorem~\ref{Theorem: characterization_of_classical_compactness} states that $X$ is compact if and only if every projection $X \times Y \to Y$ is closed. While this looks like quantification over all topological spaces $Y$, one implication should be understood in a schematic sense (we have one theorem for every $Y$), and from the proof it is clear that we only require a \emph{set} of topological spaces for the other implication.
			\item
				In the Elephant~\cite{Johnstone_PT_2002:_sketches_of_an_elephant_a_topos_theory_compendium}, the ``Bible of topos theory'', Peter Johnstone mentions (only half-jokingly, I think) that one of the most difficult issues in topos theory is to decide on the plural of `topos'. One possibility is the form `topoi', following the Greek etymology, and the other is the anglicized version `toposes'. Johnstone himself argues for the latter option, saying, among other things, that ``when going for walks in the country you don't carry hot refreshments in `thermoi'{}''. I confess I have yet to hear someone say `thermoi' myself, but then again I haven't heard anyone say `thermoses' either. Regardless, for both patterns of plural there are other words in English which follow one or the other, so I would say that neither is incorrect as such, and that this is more a matter of preference and style. Personally I prefer the form `topoi', for the simple reason that I find it more euphonic. (This is, of course, just the English problem; for example in my (Slovenian) language there is (mercifully) no such issue.)
		\end{itemize}


   \chapter{Constructive Setting}\label{Chapter: constructivism}

	In this chapter we review the notions of constructivism, intuitionistic logic and topoi which will serve as background theory in the thesis. This is meant to be an introduction to (and motivation for) constructivism for a classical mathematician; a constructive mathematician can skip the chapter entirely. We focus on themes which we actually require; for a deeper (and more formal) exposition on the subject see~\cite{Troelstra_AS_Dalen_D_1988:_constructivism_in_mathematics_volume_1, Troelstra_AS_Dalen_D_1988:_constructivism_in_mathematics_volume_2}.
	
	Constructive mathematics is, very vaguely put, the method of doing mathematics in such a way that to prove the existence of something, it is actually constructed rather than merely proving the impossibility of non-existence. A classical mathematician considers the latter to be good enough due to the law of excluded middle (abbrev.~\lem[]) which states that every proposition is either true or false; thus if it cannot be false, it must be true. However, this does not actually specify an instance of the object in question, and so might not be good enough in practice (for example, a physicist might find it useless to know that a solution of his differential equation exists if he does not also have a way to construct it).
	
	Constructivist movement originated from philosophical issues (most notably by Brouwer), but has since developed into a full-fledged mathematical theory. Generally a mathematical theory consists of specifying the language, logical inference rules, and axioms. Most mathematicians (at the time of this writing, at least) work within the classical set theory, with \zfc axioms~\cite{Jech_T_2002:_set_theory}. Constructively some of the axioms and inference rules are omitted (definitely \lem[,] though it need not be the only one), and possibly some others added. Which ones those are depends on the type of constructive mathematics one does; there are many different versions of constructivism~\cite{Bridges_DS_Richman_F_1987:_varieties_of_constructive_mathematics}. For our purposes intuitionistic logic will serve.
	
	Intuitionism is a version of constructive mathematics which has the same language and inference rules as classical mathematics, with the exception of law of excluded middle; this logic is called \df{intuitionistic}. The axioms are mostly the same, the main difference is the lack of the axiom of choice since it implies \lem[.]

	\section{Intuitionistic Logic}\label{Section: intuitionism}
	
		Let $\soc$ denote the set of truth values. Classically there are exactly two truth values, \df{true} $\top$ and \df{false} $\bot$, so $\soc = \two = \{\top, \bot\}$. Indeed, this is precisely the law of excluded middle: every truth value is either true or false.
		
		Intuitionistically we still have truth and falsehood, so $\two \subseteq \soc$, but we do not require this inclusion to be surjective (in full generality we do not require it not to be surjective either, so classical logic is a special case of intuitionistic one).
		
		Classically if a map is not surjective, there is a point outside its image which in our case would mean that intuitionistically there is a third truth value, different from $\top$ and $\bot$. But this is provably false: there is no such value. Rather, one should view all truth values in intuitionistic $\soc$ to be to some extent true (\ie equal to $\top$), and to some extent false (equal to $\bot$), but none of them completely different from both $\top$ and $\bot$.
		
		Models of such a $\soc$ are complete Heyting lattices. A \df{Heyting lattice} is a bounded distributive lattice (the largest element is $\top$, the smallest is $\bot$, the infimum is conjunction $\land$, the supremum disjunction $\lor$) with implication $\impl$, defined to be the right adjoint to conjuntion in the partial order $\leq$ of the Heyting lattice, viewed as a category. Because the only isomorphisms in a partial order category are identities, this condition determines the implication uniquely; spelling it out, for $p, q \in \soc$ the implication $p \impl q$ is the unique element of $\soc$ which satisfies the condition that $x \leq p \impl q$ precisely when $x \land p \leq q$. Dropping the categorical jargon for a moment, in a (shall we say) normal language this means that $p \impl q$ is the largest truth value for which the condition\footnote{This is of course the well-known \df{modus ponens}.} $p \land (p \impl q) \leq q$ holds. Consequently a complete distributive lattice is automatically a (complete) Heyting lattice: $\top$ and $\bot$ are the nullary infimum and supremum respectively, and the implication is
		$$p \impl q \sepeq \bigvee\st{x \in \soc}{p \land x \leq q}.$$
		
		In any Heyting lattice the negation of $p \in \soc$ can be defined as $p$ being false. However since clearly always $\bot \impl p$, the negation can be given as $\lnot{p} \dfeq p \impl \bot$. Consequently $p \land \lnot{p}$ is false (a contradiction) just as in classical logic, and indeed $\lnot{p}$ is the largest such truth value, but $\lnot{p}$ is still not a complement of $p$ in a classical sense since $p \lor \lnot{p}$ need not hold (again, this is just \lem for $p$). When it does, $p$ is called \df{decidable} (since $p \lor \lnot{p}$ means precisely that we can decide whether $p$ is true or false).
		
		In the same vein, the double negation of a truth value need not imply that truth value (if something is not false, it need not automatically be true); those $p \in \soc$ for which it does, \ie $\lnot\lnot{p} \impl p$, are called \df{$\lnot\lnot$-stable} (or just \df{stable}). The existence statements are often not stable; the impossibility of non-existence of an object with some property does not mean we can actually find such an object. The converse, $p \impl \lnot\lnot{p}$ still always holds though; if something is true, it cannot be false.
		
		Decidable truth values are also stable. The converse need not hold; it is equivalent to the \df{weak law of excluded middle} \wlem which states $\lnot\lnot{p} \lor \lnot{p}$ for all $p \in \soc$. However, all truth values are stable if and only if all of them are decidable; in that case $\soc$ is a (classically familiar) \df{Boolean lattice}.
		
		A short intermezzo to clarify a point. Constructing an intuitionistic (or any other, for that matter) model means constructing it in some ambient theory (``meta-theory''), typically set theory. Reasoning within it is called \df{external} since it is outside the constructed model. But this model has a logic on its own, and reasoning within it is called \df{internal}. Validity of (syntactically) the same propositions may differ in the two interpretations. For example, any complete Heyting lattice is the set of truth values of some intuitionistic model, viewed externally, but most Heyting lattices have at least three distinct elements. However, the point is that internally this is not the case for $\soc$ since the truth value of equality can be measured by all elements in $\soc$. This is even more pronounced in case $\soc$ is a Boolean lattice: internally it then has precisely two elements.
		
		Just as classically one can reason intuitionistically by natural deduction. It is an exercise in intuitionistic logic to show that the following holds for all $p, q, r \in \soc$ and $\phi \in \soc^X$. Where there is only an implication between the left and the right side, the converse does not hold in general (as opposed to classical logic).
		$$p \impl (q \impl r) \sepeq (p \land q) \impl r \qquad\qquad \lnot\lnot\lnot{p} \sepeq \lnot{p} \qquad\qquad \lnot\lnot(p \land q) \sepeq \lnot\lnot{p} \land \lnot\lnot{q}$$
		$$\lnot{p} \land \lnot{q} \sepeq \lnot(p \lor q) \qquad\qquad \lnot{p} \lor \lnot{q} \implies \lnot(p \land q)$$
		$$\xall{x}{X}{\lnot\phi(x)} \sepeq \lnot\xsome{x}{X}{\phi(x)} \qquad\qquad \xsome{x}{X}{\lnot\phi(x)} \implies \lnot\xall{x}{X}{\phi(x)}$$
		$$\lnot(p \impl q) \sepeq \lnot(\lnot p \lor q) \sepeq \lnot\lnot{p} \land \lnot{q}$$
		$$\lnot\lnot(p \impl q) \sepeq p \impl \lnot\lnot{q} \sepeq \lnot\lnot{p} \impl \lnot\lnot{q} \sepeq \lnot{q} \impl \lnot{p}$$
		Let $\nnst \dfeq \st{p \in \soc}{\lnot\lnot{p} = p}$ be the set of stable truth values, and let $\lnot{A} \dfeq \st{\lnot{p}}{p \in A}$ for $A \subseteq \soc$. We see that
		$$\two \subseteq \nnst = \lnot\soc = \lnot\lnot\soc \subseteq \soc.$$
		
		Suppose we have a subset $A \subseteq X$. In classical mathematics (where $\soc = \{\top, \bot\}$) we can define its \df{characteristic map} $\chi_A\colon X \to \soc$ by
		$$\chi_A(x) \dfeq \begin{cases} \top & \text{ if } x \in A,\\ \bot & \text{ if } x \notin A. \end{cases}$$
		But this just means that $x \in X$ maps to the truth value of its membership in $A$, so we can define the characteristic map constructively as well by $\chi_A(x) \dfeq (x \in A)$. In fact, any map $\chi\colon X \to \soc$ is a characteristic map of some subset of $X$, namely $\chi^{-1}(\top)$. This defines a bijective correspondence between the powerset $\pst(X)$ and the set of characteristic maps (or predicates, if you will) $\soc^X$.
		
		An interesting special case is obtained by taking $X = \one$ since $\soc^\one \ism \soc$. This means that truth values are in bijective correspondence with subsets of a singleton (a truth value being the measure to which extent the unique element $* \in \one$ belongs to a subset). Explicitly, the bijections from one to the other are
		\begin{align*}
			\soc &\to \pst(\one)  &  \pst(\one) &\to \soc\\
			p &\mapsto \st{* \in \one}{p},  &  A &\mapsto * \in A.
		\end{align*}
		We will often identify truth values and subsets of $\one$ via this bijection.
		
		The Heyting lattice structure of $\soc$ induces the Heyting lattice structure on $\soc^X$ for all $X$ simply by exponentiation. On the level of powersets, conjunction $\land$ induces the intersection $\cap$, disjunction $\lor$ induces the union $\cup$, and negation $\lnot$ induces the complementation $\insarg^C$. We will also require the implication $\impl$ on powersets; in analogy with curving the tip in the case of conjunction and disjunction to obtain the induced operations, we'll use the symbol $\setimpl$ for the implication on powersets.
		
		Technically every powerset $\pst(X)$ has its own intersection and union, so it would be more precise to write something like $\cap_X$, $\cup_X$. However, this is not the mathematical practice since the intersection and union are ``independent'' of the superset in which we calculate it, in the sense that if $X \subseteq Y$, then $\pst(X) \subseteq \pst(Y)$, and the intersection and union of $\pst(Y)$ restrict to these operations on $\pst(X)$. This is not so in the case of the implication, so we diligently add the superset in the index. Explicitly, for $A, B \subseteq X$ we define
		$$A \setimpl[X] B \dfeq \st{x \in X}{x \in A \implies x \in B}.$$
		
		Because negative statements are $\lnot\lnot$-stable, they behave much like classical statements; ``they have the same content'' classically and constructively. For example, the definition (and meaning) of a subset being empty is the same in both interpretations. Positive (more precisely, the existence) statements have different content, though. For example, being non-empty is not the same as actually possessing an element. Constructively if a (sub)set contains an element, it is called \df{inhabited}. For example, non-empty subsets of $\one$ are all those, represented by truth values $p \in \soc$ for which $\lnot\lnot{p}$ holds, but the only inhabited subset of $\one$ is $\one$ (represented by $\top$). The two cases match if and only if $\lem$ holds.
		
		The equality on a set $X$ is a map of the form $=\colon X \times X \to \soc$. When its image is contained in $\two$ (\ie for every $x, y \in X$ the statement $x = y$ is decidable), we say that $X$ \df{has decidable equality}. Similarly, when its image is contained in $\nnst$, we say that $X$ \df{has stable equality}. Clearly if a set has decidable/stable equality, so does its every subset. In a binary coproduct membership to an individual summand is by definition decidable, so if summands have decidable/stable equality, so does the binary coproduct. Obviously $\one$ has decidable equality (it is always true), and then so do $\emptyset$ and $\two$.
		
		\intermission
		
		As classically, so also constructively we call a set $\NN$ the \df{set of natural numbers} when it satisfies Peano axioms. It can be proven by induction that the relations $=$, $\leq$, $<$ on $\NN$ are decidable.
		
		We call a set $X$ \df{finite} when there exists a surjective map $\NN_{< n} \to X$ for some $n \in \NN$, \ie we can enumerate the elements of $X$ with the first few natural numbers. Note that the empty set $\emptyset$ is finite by this definition since we can take $n = 0$. In fact, any finite set is either empty or inhabited; consider any surjection $\NN_{< n} \to X$, and decide whether $n$ equals or is greater than $0$.
         
      If we fix a surjection $a\colon \NN_{< n} \to X$, we can write a finite set as $X = \{a_0, a_1, \ldots, a_{n-1}\}$ . However, in this list some elements can potentially repeat since we only require $a$ to be a surjection, not a bijection. Therefore, contrary to the classical intuition, for a finite set $X$ there need not exist $n \in \NN$ such that $X$ would have exactly $n$ elements (in the sense that there is a bijection between $X$ and $\NN_{< n}$). In fact, this happens precisely when $X$ has decidable equality (since in that case we can remove the repetitions of elements in the list).\footnote{Some authors reserve the word `finite' only for sets in bijection with $\NN_{< n}$ while what we call finite they term \df{finitely enumerated}. Our definition of finiteness is equivalent to Kuratowski finiteness.}
      
      Next, we call a set $X$ \df{infinite} when there exists an injective map $\NN \to X$. Clearly if a set is infinite, it is also inhabited, and if it contains an infinite subset, it is itself infinite. The dual notion, the existence of a surjection $\NN \to X$, means we have an enumeration of elements of $A$ by natural numbers. Thus we might be tempted to call it countability, but it turns out to be more convenient to include into definition not necessarily inhabited sets (in particular, we want the empty set $\emptyset$ to be countable). The general definition is: a set $X$ is \df{countable} when there exists a surjective map $\NN \to \one + X$ (equivalently, when $X$ is an image of a decidable subset of natural numbers). However, inhabitedness of $X$ is the only issue here; there is a surjection $\NN \to X$ if any only if $X$ is both inhabited and countable. In fact, given a surjection $s\colon \NN \to \one + X$ for such an $X$, there is a canonical way of producing a surjection $\NN \to X$: compose $s$ with the identity on $X$, plus map $* \in \one$ to the first element from $X$ which $s$ enumerates.
      
      Any finite set is countable; given a surjection $a\colon \NN_{< n} \to X$, we can produce the surjection $s\colon \NN \to \one + X$ by
      $$s_k \dfeq \begin{cases} a_k & \text{ if } k < n,\\ * & \text{ if } k \geq n. \end{cases}$$
      In the same vein, any decidable subset $D$ of a countable set $X$ is countable since we can adjust the surjection $f\colon \NN \to \one + X$ to obtain the surjection $g\colon \NN \to \one + D$, defined by
      $$g_n \dfeq \begin{cases} f_n & \text{ if } f_n \in D, \\ * & \text{ if } f_n \notin D. \end{cases}$$
      
      Classically, if we have both an injection and a surjection from one set to another, they must be in bijective correspondence. Constructively this does not hold in general, but one special case in which it does is rather useful.
      
      \begin{lemma}\label{Lemma: in_bijection_with_N}
         If an infinite countable set has decidable equality, it is in bijection with $\NN$ (the converse obviously also holds). More precisely, if $X$ is an infinite set with decidable equality, and a surjection $s\colon \NN \to \one + X$ is given, then there exists a canonical choice of a bijection $\NN \to X$.
      \end{lemma}
      \begin{proof}
         Define $b$ inductively as follows: $b(n)$ is the first element of $X$ in the list $s_0, s_1, s_2, \ldots$ which differs from all elements $b(k)$ for $k \in \NN_{< n}$. Decidable equality enables us to actually decide whether the elements differ (and of course we use decidability of summands $\one$ and $X$ to see whether an element is in $X$ in the first place). Infiniteness of $X$ ensures that we can always find a candidate for $b(n)$.
      \end{proof}
      
      We specifically emphasize the existence of a bijective correspondence between $\NN$ and $\NN \times \NN$. A possible bijection $\NN \times \NN \to \NN$ is $(n, m) \mapsto 2^n \cdot (2 m + 1) - 1$. By induction $\NN$ is in bijective correspondence with $\NN^n$ for any $n \in \NN_{\geq 1}$.
      
      \intermission
      
		Intuitionistically the axiom of choice $\ac$ is not assumed since it implies the excluded middle. To see this, note that $\ac$ implies that every surjection is split (indeed, $\ac$ is equivalent to this), \ie has a section (and the surjection is a retraction). For an arbitrary $p \in \soc$ define a surjection $f\colon \two \to \{\top, p\}$ by $f(\top) \dfeq \top$ and $f(\bot) \dfeq p$ (in short, $f(x) \dfeq p \lor x$). If $f$ has a section $s\colon \{\top, p\} \to \two$ which is necessarily injective, then $\{\top, p\}$ has decidable equality because $\two$ does. This means that $p$ is either equal to true or not. Since $p$ was arbitrary, we infer \lem[.]
		
		That said, weaker versions of the axiom of choice might still hold in an intuitionistic model (even if it is non-classical). We introduce the following notation.
		\begin{definition}\label{Definition: ac}
			\
			\begin{itemize}
				\item
					For a set $A$ and a map $g\colon X \to Y$ the axiom $\AC{A, g}$ means that every map $f\colon A \to Y$ lifts along $g$ to some map $\bar{f}\colon A \to X$ (meaning $g \circ \bar{f} = f$).
				\item
					For sets $X$, $Y$ the axiom $\AC{X, Y}$ means that every total relation between elements of $X$ and elements of $Y$ has a choice function; spelling it out,
					$$\xall{R}{\pst(X \times Y)}{\Big(\xall{x}{X}\xsome{y}{Y}{(x, y) \in R} \implies \xsome{f}{Y^X}\xall{x}{X}{(x, f(x)) \in R}\Big)}.$$
					Denoting by $p\colon X \times Y \to X$ the projection, this is equivalent to $\rstr{p}_R\colon R \to X$ being split for every total relation $R \subseteq X \times Y$ which is moreover equivalent to $\AC{X, \rstr{p}_R\colon R \to X}$ for every total relation $R \subseteq X \times Y$.
				\item
					For a set $X$ the axiom schema $\AC{X}$ means that $\AC{X, Y}$ holds for all sets $Y$.
				\item
					The full axiom of choice $\ac$ (more precisely, the axiom schema) means that $\AC{X}$ holds for all sets $X$.
			\end{itemize}
			We also use a restricted version of these axioms. For $\opn \subseteq \soc$ we define the axiom $\AC[\opn]{X, Y}$ to mean that every total relation $R \subseteq X \times Y$, classified by $\opn$ (meaning that the characteristic map of $R$ has the image in $\opn$), has a choice function.\footnote{Later $\opn$ will denote the set of open truth values, so this axiom will ensure choice functions for open relations.} In particular $\AC{X, Y} = \AC[\soc]{X, Y}$. The schemata $\AC[\opn]{X}$ and $\ac_\opn$ are derived as above.
		\end{definition}
		
		Clearly in any of these axioms we can replace a set with one in a bijective correspondence, and still obtain an equivalent axiom. More generally, the following proposition holds.
		\begin{proposition}
			Suppose $\AC[\opn]{X, Y}$ holds for some $\opn \subseteq \soc$.
			\begin{enumerate}
				\item
					If $f\colon Y \to Y'$ is a surjection, then $\AC[\opn]{X, Y'}$ also holds.
				\item
					If $f\colon X \to X'$ is a retraction (split by a section $s\colon X' \to X$), then $\AC[\opn]{X', Y}$ also holds.
			\end{enumerate}
		\end{proposition}
		\begin{proof}
			\begin{enumerate}
				\item
					If $R \subseteq X \times Y'$ is a total relation, classified by $\opn$, then $(\id[X] \times f)^{-1}(R)$ is also a total relation (since $f$ is surjective), classified by $\opn$ (its characteristic map is the characteristic map of $R$, precomposed by $\id[X] \times f$). If $g\colon X \to Y$ is a choice function for $(\id[X] \times f)^{-1}(R)$, then $f \circ g\colon X \to Y'$ is a choice function for $R$.
				\item
					Similarly as in the previous item, if $R \subseteq X \times Y'$ is a total relation, classified by $\opn$, then $(f \times \id[Y])^{-1}(R)$ is also a total relation, classified by $\opn$. From a choice function $g\colon X \to Y$ of $(f \times \id[Y])^{-1}(R)$ we obtain then a choice function for $R$ as $g \circ s\colon X' \to Y$.
			\end{enumerate}
		\end{proof}
		
		It can be proven by induction that $\AC{\NN_{< n}}$ holds for all $n \in \NN$; this is called \df{finite choice} (this does not mean that $\AC{X}$ holds for all finite sets $X$; indeed that would imply \lem[,] as above). Many constructivists assume \df{countable choice} $\AC{\NN}$, or at least \df{number-number choice} $\AC{\NN, \NN}$. However we want to be as general as possible, so whenever we assume any choice that does not automatically\footnote{To be precise, we accept choice principles which automatically hold in a \emph{topos}. Beside finite choice we thus also silently assume the \df{axiom of unique choice} (stating that for any $X$, $Y$, any relation $R \subseteq X \times Y$ which satisfies $\xall{x}{X}\xexactlyone{y}{Y}{(x, y) \in R}$ is a graph of some map), something which some constructivists don't.} hold, we state this explicitly.
		
		We also mention a stronger version of countable choice, called \df{dependent choice}. It states that we can obtain a sequence of elements by choosing one from every set in a sequence of inhabited sets, even when each set (and the choice of an element from it) depends on the previously chosen elements.
		
		\intermission
		
		An important role in this thesis will be played by \df{semidecidable} truth values, \ie countable joins of decidable ones.\footnote{Here we adopt the definition of semidecidability for our purposes. In general, semidecidable statements are those for which there exists an algorithm which returns `true' when the statement holds, and runs forever when it doesn't. In (for example) Type I computability the truth values of such statements are precisely countable joins of decidable truth values. It is this property that will be important for us, hence the change in the definition.} More precisely, we denote and define the set of semidecidable truth values as\footnote{In the context of number sets we denote $\two = \{0, 1\}$ instead of $\two = \{\top, \bot\}$, and then $\Ros = \st{p \in \soc}{\some{\alpha}{\two^\NN}{p \iff \xsome{n}{\NN}{\alpha_n = 0}}}$. Depending on style some switch $0$ with $1$ in this definition.}
		$$\Ros \dfeq \st{p \in \soc}{\some{\alpha}{\two^\NN}{p \iff \xsome{n}{\NN}{\alpha_n}}}.$$
		The notation $\Ros$ comes from logic and recursion theory, as part of the arithmetical hierarchy.
		
		A subset is \df{semidecidable} when it is classified by (a characteristic map with its image in) $\Ros$. It is obvious that countable and semidecidable subsets of $\one$ match. Because of finite choice this is true for all sets $\NN_{< n}$ as well. As for $\NN$, its any countable subset is also semidecidable; if $A \subseteq \NN$ and $s\colon \NN \to \one + A$ is a surjection, then for any $m \in \NN$ we have
		$$m \in A \iff \xsome{n}{\NN}{s_n = m}$$
		(this works since the equality on $\one + \NN$ is decidable). The converse requires some amount of countable choice, though. Let $\two^\NN \twoheadrightarrow \Ros$ denote the standard surjection $\alpha \mapsto \xsome{n}{\NN}{\alpha_n}$.
		\begin{proposition}\label{Proposition: semidecidable_and_countable}
			The following statements are equivalent.
			\begin{enumerate}
				\item
					Every semidecidable subset of $\NN$ is countable (so semidecidable and countable subsets of $\NN$ match).
				\item
					$\ACRos$ holds.
			\end{enumerate}
		\end{proposition}
		\begin{proof}
			\begin{itemize}
				\item\proven{$(1 \impl 2)$}
					Take any map $f\colon \NN \to \Ros$. This is then a characteristic map of a semidecidable, hence countable, subset $f^{-1}(\top) \subseteq \NN$, so there exists a surjection $s\colon \NN \to \one + f^{-1}(\top)$. Define the lifting $\bar{f}\colon \NN \to \two^\NN$ as
					$$\bar{f}(m)_n \dfeq (s_n = m).$$
				\item\proven{$(2 \impl 1)$}
					Let $A \subseteq \NN$ be semidecidable. By $\ACRos$ its characteristic map $\chi_A\colon \NN \to \Ros$ lifts to a map $\bar{\chi}_A\colon \NN \to \two^\NN$. Let $f\colon \NN \times \NN \to \two$ be its transposition, and $b = (b_0, b_1)\colon \NN \to \NN \times \NN$ a bijection. A surjection $s\colon \NN \to \one + A$ can be defined as
					$$s(n) \dfeq \begin{cases} b_0(n-1) & \text{ if } n \geq 1 \text{ and } f(b(n-1)),\\ * & \text{ otherwise.} \end{cases}$$
			\end{itemize}
		\end{proof}
		
		In general in constructive mathematics not all semidecidable truth values are decidable; the statement that they are is called the \df{lesser principle of omniscience} (abbr.~\lpo[]). Explicitly, \lpo is defined as
		$$\xall{\alpha}{\two^\NN}{\big(\xsome{n}{\NN}{\alpha_n} \lor \xall{n}{\NN}{\lnot\alpha_n}\big)}.$$
		It is an instance of excluded middle, equivalent to the statement that the set of binary sequences $\two^\NN$ has decidable equality.
		
		\intermission
		
		We continue the discussion about number sets. From natural numbers we can construct the integers $\ZZ$ as formal differences of natural numbers,
		$$\ZZ = \NN \times \NN/_\equ \qquad \text{where} \quad (a, b) \equ (c, d) \iff a + d = b + c,$$
		and rationals $\QQ$ as formal quotients of integers,
		$$\QQ = \ZZ \times (\ZZ \setminus \{0\})/_\equ \qquad \text{where} \quad (a, b) \equ (c, d) \iff a \cdot d = b \cdot c,$$
		as usual. Since $\NN \times \NN$ is countable, so are $\ZZ$ and $\QQ$, and clearly they are infinite with decidable equality, so they are in bijection with $\NN$ by Lemma~\ref{Lemma: in_bijection_with_N}.
		
		The real numbers are more complicated. Classically there are many possible definitions and constructions which are not constructively equivalent, so the question arises which one to take. The two most common are the \df{Cauchy reals} $\RR_c$ and the \df{Dedekind reals} $\RR_d$.
		
		The Cauchy reals are defined as the set of all Cauchy sequences of rational numbers, quotiented out by the usual equivalence relation of ``approaching the same point''. Constructively, a Cauchy sequence is more that just the sequence with the Cauchy property; we need to also supply the \df{modulus of convergence} which is a rule how quickly the terms of the sequence get close together. The Cauchy reals are popular among constructivists who assume countable choice $\AC{\NN}$ since then they behave ``the same'' as the classical reals do; in fact even the modulus of convergence can be extracted with countable choice from the Cauchy condition (though in practice constructivists often prescribe the modulus of convergence in advance, and take only those sequences which satisfy it). Without countable choice however the Cauchy reals need not even be Cauchy complete (\ie not all Cauchy sequences of reals necessarily have a limit in reals)~\cite{Lubarsky:2007:CCC:1224238.1224293}.
		
		In a choiceless environment the Dedekind reals $\RR_d$ prove better behaved. They are defined as the set of all Dedekind cuts of rationals. A pair $(L, U) \in \pst(\QQ) \times \pst(\QQ)$ is a \df{Dedekind cut} when $L$ is a \df{lower cut}, $U$ is an \df{upper cut} (which means $\all{q}{\QQ}{q \in L \iff \xsome{x}{\QQ_{> q}}{x \in L}}$ and $\all{r}{\QQ}{r \in U \iff \xsome{x}{\QQ_{< r}}{x \in U}}$, respectively), and $L$, $U$ are inhabited, disjoint and \df{located} which means $\all{q, r}{\QQ}{q < r \implies q \in L \lor r \in U}$. The idea is that $(L, U)$ represents the unique real number such that $L$ is the set of all smaller rationals, and $U$ is the set of all larger rationals.
		
		In general $\RR_c \subseteq \RR_d$; in the presence of $\AC{\NN}$ the two notions match, though.
		
		We return to (and go more in depth into) this subject in Section~\ref{Section: real_numbers}.

	\section{Topoi}\label{Section: topoi}
	
		There are many equivalent definitions of topoi, the most concise one probably being that they are categories with finite limits and powerobjects. We opt for just a slightly longer definition.
		\begin{definition}
			A \df{topos} is a category which has:
			\begin{itemize}
				\item
					all finite limits,
				\item
					all exponentials,\footnote{A category with finite limits and exponentials is called \df{cartesian closed}, so we can also say that a topos is a cartesian closed category with a subobject classifier.}
				\item
					a subobject classifier.
			\end{itemize}
		\end{definition}
		We'll assume familiarity with limits and exponentials\footnote{Just the intuition: the exponential $B^A$ is an object of the category which represents (all) morphisms from the object $A$ to the object $B$.}, but we'll say something about the subobject classifier. In a category with finite limits the object $\soc$, together with the so-called \df{truth map} $\top\colon \one \to \soc$, is a \df{subobject classifier} when for every monomorphism $f\colon A \to X$ in the category there exists a unique morphism (the \df{characteristic map} of $f$) $\chi\colon X \to \soc$ such that the following diagram is a pullback (\ie $f$ is a pullback of the truth map along $\chi$).
		$$\xymatrix@+1em{
			A \pbcorner \ar[r]^{\trm[A]} \ar[d]_f  &  \one \ar[d]^{\top}  \\
			X \ar[r]_\chi  &  \soc
		}$$
		This means that we have bijective correspondence between subobjects\footnote{Recall that categorically a subobject of $X$ is an equivalence class of monomorphisms with codomain $X$ where two such monos $f\colon A \to X$, $g\colon B \to X$ are equivalent when they factor through each other via isomorphism, \ie there exists an isomorphism $h\colon A \to B$ such that $f = g \circ h$.} of $X$ and maps from $X$ to $\soc$ (one direction is the definition of the subobject classifier --- hence it gets its name --- and the other is the existence of pullbacks), so we can view $\soc$ as the ``object of truth values'', and maps $X \to \soc$ as characteristic maps of subobjects of $X$, or predicates on $X$. Furthermore, the exponential $\soc^X$, representing predicates on $X$, is the categorical powerobject of $X$ (intuitively, it represents the set of subsets of $X$).
		
		The categorical constructions, as is usually the case, are determined up to isomorphism. However, in practice we require more than just their existence; we need a canonical way of constructing them, so that they become functors (for example, binary product in a category $\cat$ is a functor $\times\colon \cat \times \cat \to \cat$). In case our background theory is classical with global choice (the axiom of choice for classes), there is no problem, but constructively being functors is a stronger assumption (see also~\cite{Johnstone_PT_2002:_sketches_of_an_elephant_a_topos_theory_compendium}). For our purposes whenever we say that a category has (a certain type of) limits, colimits, exponentials etc., we mean it in the strong sense, that we have actual functors (and consequently, we can also speak about \emph{the} product, and so on).
		
		Beside the constructions from the definition, a topos also has all finite colimits. The rigorous proof can be found in~\cite{Mac_Lane_S_Moerdijk_I_1992:_sheaves_in_geometry_and_logic, Johnstone_PT_2002:_sketches_of_an_elephant_a_topos_theory_compendium}, but the intuition is simple. In any cartesian closed category $\cat$, for any object $A \in \obj{\cat}$ the exponentiation functor $A^\insarg\colon \cat \to \cat\op$ is left adjoint to its opposite functor $A^\insarg\colon \cat\op \to \cat$, and as such preserves colimits. In other words, viewing it as a contravariant functor, it maps colimits to limits. For $A = \soc$ this means that a powerset of a colimit is a limit of powersets, and if we can construct what ought to be the powerset of an object in question, we should be able to reconstruct the object in question up to isomorphism (it is the subset of singletons in the powerset).
		
		Since the domain of the truth map is the terminal object, it is a split mono, and hence a regular mono; the following is an equalizer diagram.
		$$\xymatrix@+1em{
			\one \ar[r]^\top  &  \soc \ppair{\top \circ \trm[\soc]}{\id[\soc]}  &  \soc
		}$$
		A pullback of a regular mono is a regular mono, so in a category with a subobject classifier (in particular in a topos) every mono is regular.\footnote{Even if in a category not all monos are regular, it might still have $\soc$ (and a truth map) which classifies all regular monos; this is then called a \df{regular subobject classifier}.} Moreover, it turns out that every epimorphism is regular as well, and that every morphism has a unique epi-mono factorization. Explicitly, for $f\colon X \to Y$, first form a pullback
		$$
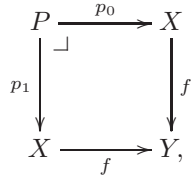
\xymatrix@+1em{
			P \pbcorner \ar[r]^{p_0} \ar[d]_{p_1}  &  X \ar[d]^f  \\
			X \ar[r]_f  &  Y,
		}$$
		then a pushout
		
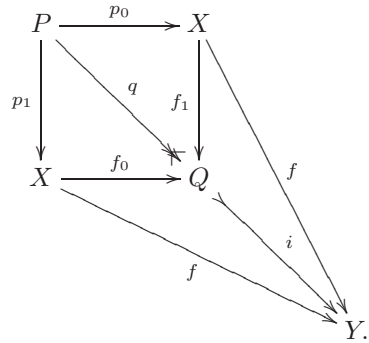
$$\xymatrix@+2em{
			P \ar[r]^{p_0} \ar[d]_{p_1} \ar@{->>}[rd]^q  &  X \ar[d]_{f_1} \ar[rdd]^f  &\\
			X \ar[r]^{f_0} \ar[rrd]_f  &  Q \pocorner \ar@{ >->}[rd]^i  &\\
			&&  Y.
		}$$
		The diagonal pushout map $q$ turns out to be epi, and the map $i$, induced by the universal property of a pushout, is mono; it represents the inclusion of the image of $f$ into the codomain.
		
		\intermission
		
		Topoi have many different viewpoints~\cite{Johnstone_PT_2002:_sketches_of_an_elephant_a_topos_theory_compendium}; for us they are important because they are precisely categorical models of (higher-order) intuitionistic logic which is to say that standard statements involving sets, relations, maps and logical connectives all have the interpretation in the internal language of a topos~\cite{Mac_Lane_S_Moerdijk_I_1992:_sheaves_in_geometry_and_logic, Johnstone_PT_2002:_sketches_of_an_elephant_a_topos_theory_compendium}, and they satisfy expected conditions. Thus working inside a topos and doing intuitionistic mathematics amounts to the same thing. Later models of synthetic topology will be topoi, but we'll avoid heavy-duty category theory by simply reasoning within intuitionistic set theory\footnote{\emph{Bounded} set theory, to be precise --- \ie no global $\in$, and all quantifiers are bounded.} (in particular, objects of the topos will be simply called sets, and so on).
		
		Here we recall only which morphisms in a topos represent logical connectives and (unparameterized) quantifiers.
		\begin{itemize}
			\item
				Truth is represented, not surprisingly, by the truth map $\top\colon \one \to \soc$; equivalently, it is the characteristic map of $\id[\one]$.
			\item
				Falsehood $\bot\colon \one \to \soc$ is the characteristic map of $\emptyset \to \one$.
			\item
				Conjunction $\land\colon \soc \times \soc \to \soc$ is the characteristic map of $(\top, \top)\colon \one \to \soc \times \soc$.
			\item
				Make a pushout diagram of the truth map with itself.
				$$
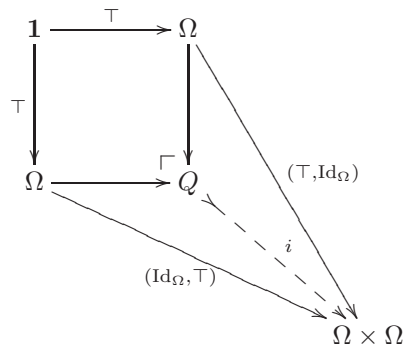
\xymatrix@+2em{
					\one \ar[r]^\top \ar[d]_\top  &  \soc \ar[d] \ar[rdd]^{(\top, \id[\soc])}  &\\
					\soc \ar[r] \ar[rrd]_{(\id[\soc], \top)}  &  Q \pocorner \ar@{ >-->}[rd]^i  &\\
					&&  \soc \times \soc
				}$$
				The inclined morphisms are to be understood in the sense $\one \times \soc \ism \soc \ism \soc \times \one$. The induced map $i$ turns out to be mono, and disjunction $\lor\colon \soc \times \soc \to \soc$ is its characteristic map.
			\item
				Let $p_0\colon \soc \times \soc \to \soc$ be the projection onto the first factor, and make the following equalizer.
				$$\xymatrix@+1em{
					E \ar[r]^e  &  \soc \times \soc \ppair{\land}{p_0}  &  \soc
				}$$
				Then the implication $\impl\colon \soc \times \soc \to \soc$ is the characteristic map of $e$.
			\item
				Negation $\lnot\colon \soc \to \soc$ can in principle be calculated from $\impl$ and $\bot$, but a simpler way to obtain it is as the characteristic map of the falsehood map $\bot\colon \one \to \soc$.
			\item
				Equivalence is, of course, just the equality $=\colon \soc \times \soc \to \soc$ (which is the characteristic map of the diagonal $\diag[\soc] = (\id[\soc], \id[\soc])\colon \soc \to \soc \times \soc$).
			\item
				Universal quantification on an object $X$ takes a predicate on $X$ and returns a truth value which measures the extent to which the predicate holds overall; thus it is a map $\forall_X\colon \soc^X \to \soc$. By definition\footnote{The more general definition allows parameters in which case a quantifier is a map $\soc^{X \times Y} \to \soc^Y$, defined similarly, but we won't need its explicit construction.} it is the right adjoint to the \df{weakening} $\soc \to \soc^X$, defined as the transpose of the projection $\soc \times X \to \soc$. It can be constructed as the characteristic map of the transpose $\one \to \soc^X$ of the composition $\one \times X \ism X \stackrel{\trm[X]}{\longrightarrow} \one \stackrel{\top}{\longrightarrow} \soc$.
			\item
				Existential quantification $\exists_X\colon \soc^X \to \soc$ is by definition the left adjoint to the weakening. To construct it, let $\in\colon \soc^X \times X \to \soc$ be the transpose of $\id[\soc^X]$, \ie the evaluation map (which is in this case the \df{membership predicate} --- imagine $\soc^X$ as the powerset of $X$). Let $P$ be the pullback of the truth map along it.
				$$\xymatrix@+1em{
					P \pbcorner \ar[r]^{\trm[P]} \ar[d]_p  &  \one \ar[d]^\top  \\
					\soc^X \times X \ar[r]_\in  &  \soc
				}$$
				Make the epi-mono factorization of $p$, composed with the projection $q\colon \soc^X \times X \to \soc^X$.
				$$\xymatrix@+2em{
					P \ar@{->>}[r] \ar@/_3ex/[rr]_{q \circ p}  &  E \ar@{ >->}[r]^i  &  \soc^X
				}$$
				Then $\exists_X$ is the characteristic map of $i$.\footnote{However complicated this construction may seem, the intuition is simple. $P$ is the set of pairs $(A, x)$ where $x \in A \subseteq X$, so the image of its projection are precisely inhabited subsets of $X$.}
		\end{itemize}
		
		We already mentioned that topoi have finite limits and colimits, but in general they need not be complete or cocomplete, \ie have arbitrary set-indexed (that is to say, external) (co)limits. However, they always have all internal ones~\cite{Johnstone_PT_2002:_sketches_of_an_elephant_a_topos_theory_compendium}. Here is a demonstration just for products and coproducts. Suppose $(X_i)_{i \in I}$ is an $I$-indexed family of objects where $I$ is itself an object of the topos. This means that this family is given by a map $p\colon X \to I$ where $X_i$s are its fibers. Then the internal coproduct of this family is just $X$ itself while the product is given internally as $\st{f \in X^I}{p \circ f = \id[I]}$, or diagrammatically as the equalizer
		$$\xymatrix@+1em{
			\prod_{i \in I} X_i \ar[r]  &  X^I \ppair{p^I}{\widehat{\id[I]} \circ \trm[X^I]}  &  I^I
		}$$
		where $\widehat{\id[I]}\colon \one \to I^I$ is the transpose of $\one \times I \ism I \stackrel{\id[I]}{\longrightarrow} I$.
		
		\intermission
		
		A topos need not possess an object of natural numbers (for example, finite sets and maps between them form a topos). A \df{natural numbers object} is concisely defined as the initial algebra for the functor $\insarg + \one$; with more words (and not requiring anything more from the ambient category than a terminal object), an object $\NN$, together with maps $z\colon \one \to \NN$, $s\colon \NN \to \NN$, is a natural numbers object when it is initial among the diagrams of the form $\one \stackrel{q}{\longrightarrow} A \stackrel{f}{\longrightarrow} A$ in the sense that there exists a unique $u\colon \NN \to A$ such that the diagram
		$$\xymatrix@+1em{
			\one \ar[r]^z \ar[d]_{\id[\one]}  &  \NN \ar[r]^s \ar[d]^u  &  \NN \ar[d]^u  \\
			\one \ar[r]_q  &  A \ar[r]_f  &  A
		}$$
		commutes. The map $z$ represents the element zero while $s$ represents the successor map. It can be seen that this universal property precisely captures the Peano axioms (and determines natural numbers up to isomorphism). Thus in a topos with natural numbers object other number sets can be constructed as well, as per previous section. We also note that the object $\finseq{A}$, representing the set of finite sequences of $A$, \ie the internal coproduct $\coprod_{n \in \NN} A^n$, then exists (it can be constructed for example as $\finseq{A} = \st{\alpha \in (\one + A)^\NN}{\xsome{n}{\NN}\all{k}{\NN}{\alpha_k \in A \iff k < n}}$).

   \chapter{Synthetic Topology}\label{Chapter: synthetic_topology}

	In this chapter we start with the actual subject of synthetic topology. We redefine basic topological notions in synthetic and constructive manner, and derive elementary results. Synthetic topology is given axiomatically, but we defer this approach to the second part of the first section, preferring to start with an introduction from the point of view of a classical topologist.

	\section{From Point-Set to Synthetic Topology}\label{Section: to_synthetic_topology}
   
      Recall that in classical set theory, the subsets of a set $X$ correspond to the characteristic maps on $X$, \ie maps $X \to \two$ where $\two = \{\top, \bot\}$ is the set of truth values. Consider what characteristic maps classify in the category of topological spaces. There are four possible topologies on $\two$:
      \begin{itemize}
         \item
            The topology $\{\emptyset, \two\}$. This is the trivial (= indiscrete) topology, and we will use the symbol $\trivtwo$ for $\two$ equipped with it. Every map $X \to \trivtwo$ (where $X$ is a topological space) is continuous, so maps into $\trivtwo$ classify subspaces\footnote{This should not be read that $\trivtwo$ is the subobject classifier in $\Top$. Subspace inclusions are a rather special kind of injective continuous maps, namely the ones with the minimal possible topology on the domain, such that the underlying map is still continuous (this property can be translated into categorical terms, obtaining the general notion of a \df{minimal morphism}, used also outside topology). In fact, $\Top$ cannot possibly have a subobject classifier, as not all its monomorphism are regular. However, since regular monomorphisms in $\Top$ are precisely subspace inclusions, $\trivtwo$ is a regular subobject classifier.} of $X$.
         \item
            The topology $\{\emptyset, \{\top\}, \two\}$. The set $\two$ with this topology is called the \df{\Sier space}, and we denote it by $\sier$. A map $f\colon X \to \sier$ is continuous if and only if $f^{-1}(\top)$ is open in $X$. As such, maps into $\sier$ classify open subspaces. Let $\tp(X)$ denote the topology on $X$, and $\C(X, Y)$ the set of continuous maps between spaces $X$ and $Y$. Then we have a bijective correspondence $\optp(X) \ism \C(X, \sier)$ for every topological space $X$.
         \item
            The topology $\{\emptyset, \{\bot\}, \two\}$. We will call this \df{inverted \Sier space}, and denote it by $\twsier$. Observe that continuous maps into $\twsier$ classify the complements of open subspaces, \ie closed subspaces. If we denote by $\ctp(X)$ the set of closed subspaces of a topological space $X$, we then have the bijection $\cltp(X) \ism \C(X, \twsier)$.
         \item
            The topology $\{\emptyset, \{\top\}, \{\bot\}, \two\}$. This is the discrete topology, and we shall denote $\two$ equipped with it again by $\two$ since this is the coproduct $\one + \one$ in $\Top$. The continuous maps into $\two$ classify \df{clopen} (= both closed and open) subsets. Alternatively, maps $\C(X, \two)$ represent \df{separations} of the space $X$, \ie pairs of open subsets $(U, V)$ of $X$ such that $U \cap V = \emptyset$ and $U \cup V = X$ (a continuous map $f\colon X \to \two$ induces the separation $(f^{-1}(\top), f^{-1}(\bot))$). Recall that a space is connected when it has no non-trivial separations (no separations in which $U$ and $V$ are both inhabited).
      \end{itemize}
      However, note that the spaces $\sier$ and $\twsier$ are homeomorphic via the \df{twist map} which interchanges $\top$ and $\bot$, so up to homeomorphism there are three possible topologies on any two-element set. The most interesting is the non-extremal one --- the one homeomorphic to $\sier$. The importance of the \Sier space is that it allows us to redefine topological notions in terms of spaces (read: objects) and continuous maps (read: morphisms), as opposed to spaces and open subsets. This makes us view topology through categorical glasses, and beside the whole subject of synthetic topology which emerges, we also get an immediate pay-off in classical topology since many standard theorems acquire a simpler and more elegant proof. One can find more exposition and details in~\cite{Escardo_M_2004:_synthetic_topology_of_data_types_and_classical_spaces} and (for those who understand Slovene)~\cite{Pretnar_M_2005:_sinteticna_topologija}; here we give an overview.
      
      We identify open subsets in a topological space $X$ with continuous maps $X \to \sier$ which we call ``open predicates'' on $X$ because we view $\sier$ as a ``topological space of open truth values'' in this context. Similarly, closed subsets in $X$ are identified with ``closed predicates'' $X \to \twsier$.
      
      Topological properties are also reinterpreted. Recall that classically a topological space $X$ is Hausdorff if and only if the diagonal $\diag[X] \subseteq X \times X$ is a closed subset of the topological product $X \times X$. Since elements on the diagonal are precisely those with equal components, we can say that $X$ is Hausdorff when its equality relation is a closed predicate.
      
      Consider the dual version, the diagonal being open in the product. It is easily seen that this is the case precisely for discrete topological spaces, so we interpret discreteness as the equality being an open predicate.
      
      One of the most important topological properties is compactness. Interestingly, its reinterpretation can be made \emph{simpler} than the classical ``every open cover has a finite subcover''. We need a little preparation, though.
      
      Topologists, and in particular analysts, have learnt from practice that the most useful topology on the set $\C(X, Y)$ is compact-open topology, \ie the one generated by the subbase
      $$\st[1]{V(K, U)}{K \text{ compact} \subseteq X,\ U \text{ open} \subseteq Y}$$
      where
      $$V(K, U) \dfeq \st{f \in \C(X, Y)}{f(K) \subseteq U}.$$
      Category theory provides a theoretical reason for this: if $X$ is Hausdorff and the exponential of $X$ and $Y$ in $\Top$ exists, it is the set $\C(X, Y)$, equipped with compact-open topology, and even when the exponential does not exist, this is its closest approximation.\footnote{For general topological spaces, a better approximation is the \df{Isbell topology}, and the best is the so-called \df{natural topology}. Both match the exponential topology when it exists (and in the case of Hausdorff spaces, they match the compact-open topology). For more on the subject see~\cite{Escardo_M_Lawson_JD_Simpson_A_2004:_comparing_cartesian_closed_categories_of_core_compactly_generated_spaces, Gierz_G_Hoffmann_KH_Keimel_K_Lawson_JD_Mislove_MW_Scott_DS_2003:_continuous_lattices_and_domains, Lowen_Colebunders_E_Richter_G_2001:_an_elementary_approach_to_exponential_spaces}, but since this section is meant to be merely an introduction, overview and motivation for synthetic topology for classical topologists, we do not go into further details, and stick to the far better known compact-open topology.}
      
      In particular, this enables us to put a natural topology on topology itself via the bijection $\optp(X) \ism \C(X, \sier)$. Similarly, we equip the set of closed subsets with topological structure via $\cltp(X) \ism \C(X, \twsier)$. Also note that the twist homeomorphism between $\sier$ and $\twsier$ induces a homeomorphism $\C(X, \sier) \ism \C(X, \twsier)$, namely the postcomposition with the twist map. On the level of open and closed subsets, this homeomorphism is (in both directions) just the complementation
      $$\xymatrix{
         \optp(X) \ar@/^1.5ex/[rr]^{U \mapsto U^C}  &  {\ism}  &  \cltp(X) \ar@/^1.5ex/[ll]^{F^C \mapsfrom F}
      }.$$
      
      Let us examine the topologies on $\optp(X)$ and $\cltp(X)$ closer.
      \begin{lemma}\label{Lemma: classical_compact-open_basis_for_topology}
         Let $X$ be a topological space.
         \begin{enumerate}
            \item
               The family $\st[1]{V(K, \{\top\})}{K \text{ compact} \subseteq X}$ is a basis for $\C(X, \sier)$. Transferring this to $\optp(X)$, its base is $\st[1]{\upward{K}}{K \text{ compact} \subseteq X}$ where $\upward{K} \dfeq \st{U \in \optp(X)}{K \subseteq U}$.
            \item
               Similarly, the topology on $\C(X, \twsier)$ is generated by the base $\st{V(K, \{\bot\})}{K \text{ compact} \subseteq X}$ which yields a base $\st[1]{\st{F \in \cltp(X)}{K \cap F = \emptyset}}{K \text{ compact} \subseteq X}$ for $\cltp(X)$.
         \end{enumerate}
      \end{lemma}
      \begin{proof}
      	\begin{enumerate}
      		\item
      			Denote $\basis = \st{V(K, \{\top\})}{K \text{ compact} \subseteq X}$. To see that $\basis$ is a subbasis for $\C(X, \sier)$, we still have to take care of subsets of the form $V(K, \emptyset)$ and $V(K, \sier)$, but these are either $\emptyset$ which we can drop from the (sub)basis, or the whole $\C(X, \sier)$ which can also be given as $V(\emptyset, \{\top\})$. To see that $\basis$ is in fact a basis, it is sufficient to verify that binary intersections of elements in $\basis$ are again in $\basis$ which is clear since $V(K, \{\top\}) \cap V(L, \{\top\}) = V(K \cup L, \{\top\})$ (and binary unions of compact sets are compact).
      			
      			As for the basis on $\optp(X)$, just recall the bijection $\C(X, \sier) \ism \optp(X)$.
      		\item
      			Use the previous item, as well as the fact that $\C(X, \sier)$ and $\C(X, \twsier)$ are homeomorphic via composition with the twist map, and $\optp(X)$ and $\cltp(X)$ are homeomorphic via complementation.
      	\end{enumerate}
      \end{proof}
      
      This now allows us to classify compact spaces; see~\cite{Escardo_M_2004:_notes_on_synthetic_topology},~\cite{Escardo_M_2004:_synthetic_topology_of_data_types_and_classical_spaces},~\cite{Escardo_M_2009:_intersections_of_compactly_many_open_sets_are_open}.
      
      \begin{theorem}\label{Theorem: characterization_of_classical_compactness}
         The following statements are equivalent for any topological space $X$.
         \begin{enumerate}
            \item $X$ is compact.
            \item The subset $\{X\}$ is open in $\optp(X)$.
            \item The subset $\st[1]{F \in \cltp(X)}{\xsome{x}{X}{x \in F}}$ of inhabited closed subsets of $X$ is closed in $\cltp(X)$.
            \item For any topological space $Y$ the projection $p\colon X \times Y \to Y$ is a closed map\footnote{Recall that a map is \df{closed} when its image of every closed subset of the domain is a closed subset of the codomain.}.
         \end{enumerate}
      \end{theorem}
      \begin{proof}
         In the following items keep in mind the results of Lemma~\ref{Lemma: classical_compact-open_basis_for_topology}.
         \begin{itemize}
            \item\proven{$(1 \Rightarrow 2)$}
            	Since $\{X\} = \upward{X}$.
            \item\proven{$(2 \Rightarrow 1)$}
            	Because $\{X\}$ is an open singleton, it must be a basic set, so there exists a compact $K \subseteq X$ such that $\upward{K} = \{X\}$. This means that any open subset of $X$ which contains $K$ must also contain $X$ (\ie be equal to it). Take any open cover of $X$. Then it is in particular an open cover of $K$ for which there exists a finite subcover. Its (open) union covers $K$, and therefore $X$ as well.
            \item\proven{$(2 \Leftrightarrow 3)$}
            	Use the homeomorphism $\optp(X) \ism \cltp(X)$, given by complementation.
            \item\proven{$(1 \Rightarrow 4)$}
            	Take any closed set $F \in \cltp(X \times Y)$, and any $y \in p(F)^C \subseteq Y$. For each $x \in X$ there exists a product basic open neighbourhood of $(x, y)$ in  $F^C$, \ie there are open subsets $U \subseteq X$ and $V \subseteq Y$ such that $(x, y) \in U \times V \subseteq F^C$. The $U$s then form an open cover of $X$ which is compact, so there exist $(U_0, V_0), \ldots, (U_{n-1}, V_{n-1})$ such that $X = \bigcup_{k \in \NN_{< n}} U_k$. Then $\bigcap_{k \in \NN_{< n}} V_k$ is an open neighbourhood of $y$ in $p(F)^C$. Since $y \in p(F)^C$ was arbitrary, $p(F)^C$ is open, therefore $p(F)$ is closed.
            \item\proven{$(4 \Rightarrow 1)$}
            	This proof follows the one in~\cite{Escardo_M_2009:_intersections_of_compactly_many_open_sets_are_open}. Let $\mathcal{C} \subseteq \tp(X)$ be an open cover of $X$. We define the following topology on $\tp(X)$: a subset $\mathcal{U} \subseteq \tp(X)$ is open when it is an upper set (if $U \in \mathcal{U}$ and $U \subseteq V \in \tp(X)$, then $V \in \mathcal{U}$) and when for every $\mathcal{A} \subseteq \mathcal{C}$ with properties
            	$$\all{U}{\mathcal{C}}{U \subseteq \bigcup\mathcal{A} \implies U \in \mathcal{A}} \qquad \text{and} \qquad \bigcup\mathcal{A} \in \mathcal{U}$$
            	there exists a finite subset $\mathcal{F} \subseteq \mathcal{A}$ such that $\bigcup\mathcal{F} \in \mathcal{U}$. It is easy to see that this is indeed a topology on $\tp(X)$.\footnote{Compare it with the Scott topology~\cite{Gierz_G_Hoffmann_KH_Keimel_K_Lawson_JD_Mislove_MW_Scott_DS_2003:_continuous_lattices_and_domains}.} Notice that for every $U \in \mathcal{C}$ the set
            	$$\upward{U} \dfeq \st{V \in \tp(X)}{U \subseteq V}$$
            	is open in $\tp(X)$.
            	
            	Define
            	$$\mathcal{V} \dfeq \st{(x, U) \in X \times \tp(X)}{U \in \mathcal{C} \land x \in U}.$$
            	This is an open subset of $X \times \tp(X)$ since every $(x, U) \in \mathcal{V}$ has a basic product open neighbourhood $U \times \upward{U}$ which is contained in $\mathcal{V}$. By assumption the projection $p\colon X \times \tp(X) \to \tp(X)$ is a closed map, so $p(\mathcal{V}^C)^C$ is open in $\tp(X)$. Since $\mathcal{C}$ is a cover, we have $p(\mathcal{V}^C)^C = \{X\}$, openness of which implies compactness of $X$ (by taking $\mathcal{A} = \mathcal{C}$).
         \end{itemize}
      \end{proof}
      
      \begin{corollary}\label{Corollary: compactness_with_quantifiers}
      	The following statements are equivalent for a topological space $X$.
      	\begin{enumerate}
      		\item
      			$X$ is compact.
      		\item
      			The map $\forall_X\colon \optp(X) \to \sier$, given by the following equivalent formulations
      			$$\forall_X(U) \dfeq X \subseteq U \sepeq \xall{x}{X}{x \in U} \sepeq \begin{cases} \top & \text{if } \xall{x}{X}{x \in U},\\ \bot & \text{otherwise}, \end{cases}$$
      			is continuous.
      		\item
      			The map $\exists_X\colon \cltp(X) \to \twsier$, given by
      			$$\exists_X(F) \dfeq X \between F \sepeq \xsome{x}{X}{x \in F} \sepeq \begin{cases} \top & \text{if } \xsome{x}{X}{x \in F},\\ \bot & \text{otherwise}, \end{cases}$$
      			is continuous.
      	\end{enumerate}
      \end{corollary}
      \begin{proof}
      	Easy from Theorem~\ref{Theorem: characterization_of_classical_compactness}.
      \end{proof}
		
		We can interpret this corollary in the following way. The identity maps $\sier \to \trivtwo$ and $\twsier \to \trivtwo$ are injective and continuous, and so represent subobjects of $\trivtwo$ in the category of topological spaces. The maps $\forall_X$ and $\exists_X$ from the corollary can be extended to the whole powerset of $X$ which is equipped with the trivial topology via $\pst(X) \ism \C(X, \trivtwo)$. Replace $\sier$ and $\twsier$ with $\trivtwo$ in the codomain of $\forall_X$ and $\exists_X$ as well, to ensure that the quantifiers are continuous. We can then characterize compactness of $X$ as the ability to restrict these maps (in the category of topological spaces) in either of the following diagrams.
		$$\xymatrix@+3em{
			\pst(X) \ar[r]^{\forall_X}  &  \trivtwo  &&  \pst(X) \ar[r]^{\exists_X}  &  \trivtwo  \\
			\optp(X) \ar@{-->}[r] \ar@{ >->}[u]  &  \sier \ar@{ >->}[u]  &&  \cltp(X) \ar@{-->}[r] \ar@{ >->}[u]  &  \twsier \ar@{ >->}[u]
		}$$
		
		With this preparation we give the definition of synthetic topology, following~\cite{Escardo_M_2004:_notes_on_synthetic_topology}. We want a model in which sufficient amount of logic can be interpreted, and in which we have a way of defining for every object what its topological structure is.
		
		\begin{definition}
			A model of \df{synthetic topology} is a topos $\topos$ with a chosen subobject $\opn \subseteq \soc$, called the \df{\Sier object}. A subobject $U \subseteq X$ in $\topos$ is (\df{intrinsically}) \df{open} when its characteristic map factors through (\ie its image is contained in) $\opn$. We define (in the internal language of the topos) the set of open subsets of $X$
			$$\tp(X) \dfeq \st{U \in \pst(X)}{U \text{ intrinsically open in } X}$$
			which we call the \df{intrinsic topology} of $X$.
		\end{definition}
		We do not emphasize topoi as categories, however, but rather view them as models of constructive mathematics, and develop our theory in this sense.
		
		The idea of the \Sier object $\opn$ is, of course, that it should play the same role in $\topos$ as the \Sier space $\sier$ does in the case of topological spaces, in the sense that open subsets of $X$ are classified by maps $X \to \opn$. It follows that the intrinsic topology $\tp(X)$ is isomorphic to the exponent $\opn^X$ (in the same way as $\pst(X)$ is isomorphic to $\soc^X$). The special case $X = \one$ provides an interesting insight: since $\opn \ism \opn^\one$, we can interpret $\opn$ as the topology on a one-element set. Thus, in words of \MEsc[], to equip all objects with the topology, it is sufficient to prescribe the topology only on $\one$.\footnote{This clearly shows the importance of using constructive logic in the case of synthetic topology, as classically the only subsets of $\one$ are the empty set and the whole set. Constructively, the subsets of $\one$ can be a lot more diverse and interesting.}
		
		Because $\opn$ is a subset of $\soc$, we can describe openness also by truth values. The elements of $\opn$ are called \df{open truth values}. Since the characteristic map $\chi_U\colon X \to \soc$ of a subset $U \subseteq X$ is given by $\chi_U(x) = x \in U$, we have that $U$ is open in $X$ if and only if for every $x \in X$ the truth value of the statement $x \in U$ is open.
		
		Using logical description of openness, one can often prove topological theorems in a strikingly simple way. We demonstrate this here with examples, and provide synthetic definitions of familiar topological properties along the way.
		
		For start, note that in any model of synthetic topology every map $f\colon X \to Y$ is continuous, by which we mean the usual thing, namely that $f$-preimages of open subsets of $Y$ are open subsets of $X$. For the proof we need only to write the definition of the preimage,
		$$x \in f^{-1}(U) \iff f(x) \in U$$
		for all $x \in X$, $U \in \tp(Y)$. One should understand this proof as follows: to verify that $f^{-1}(U)$ is open in $X$, we have to see that for every $x \in X$ the truth value of $x \in f^{-1}(U)$ is open. But this is equivalent to $f(x) \in U$ which is an open statement since by assumption $U$ is open in $Y$, therefore for every $y \in Y$, in particular $y = f(x)$, the statement $y \in U$ is open. This is the general strategy of synthetic topological proofs: write a statement for which we have to verify openness, and find an equivalent statement of which openness follows from assumptions.
		
		Discussion about compactness above suggests what the synthetic definition of compactness should be.
		\begin{definition}
			A set $X$ is \df{compact} when its universal quantifier $\forall_X\colon \pst(X) \to \soc$ restricts to a map $\tp(X) \to \opn$.\footnote{This is almost a transcription of the characterization of compactness below Corollary~\ref{Corollary: compactness_with_quantifiers}. We note however that there is actually a difference in detail: $\pst(X)$ for a topological space means the powerset of its underlying set while for an object $X \in \obj{\topos}$ it means the set of all subobjects of $X$. The category of topological spaces $\Top$ is not a topos, but the point is that subobjects in $\Top$ are represented by arbitrary continuous injections, rather than just subspace embeddings. This difference is nicely seen in Section~\ref{Section: gros_topos_model} where (a small sub-)category of topological spaces is embedded into a topos, and there the subobjects of a representable object, represented by subspace inclusions, are just the stable ones.} The logical formulation of this condition is that for every $U \in \tp(X)$ the truth value of the statement $\xall{x}{X}{x \in U}$ is open.
		\end{definition}
		
		It is often convenient to replace the open subsets by their characteristic maps, \ie open predicates. An \df{open predicate} on $X$ is a map $X \to \opn$ (or, if you will, it is a predicate on $X$ which restricts to such a map --- its range is contained among open truth values). The open subset, given by such a predicate $\phi$, is of course reconstructed as $\phi^{-1}(\top)$. We see that compactness can be equivalently reformulated as follows: $X$ is compact when for every open predicate $\phi$ on $X$ the truth value of $\xall{x}{X}{\phi(x)}$ is open.
		
		With this it is easy to prove an analogue of the classical theorem that finite products of compact topological spaces are compact.\footnote{In general, arbitrary products of compact sets need not be compact, however (\eg in Section~\ref{Section: russian_constructivism_model} we see that $\two^\NN$ is not compact in the effective topos). This should not come as a surprise, as even classically the Tychonoff theorem~\cite{Dugundji_J_1974:_topology} is equivalent to the axiom of choice.}
		\begin{proposition}
			Finite products of compact sets are compact.
		\end{proposition}
		\begin{proof}
			It is sufficient to prove the preservation of compactness by nullary and binary products.
			\begin{itemize}
				\item\proven{$\one$ is compact}
					Let $\phi\colon \one \to \opn$ be an open predicate on $\one$. Then
					$$\xall{x}{\one}{\phi(x)} \iff \phi(*).$$
				\item\proven{if $X$ and $Y$ are compact, then so is $X \times Y$}
					Let $\phi \in \opn^{X \times Y}$. Then
					$$\xall{a}{X \times Y}{\phi(a)} \iff \xall{x}{X}\xall{y}{Y}{\phi(x, y)}.$$
			\end{itemize}
		\end{proof}
		
		Here is another familiar proposition from classical topology, translated to synthetic setting.
		\begin{proposition}
			An image of a compact set is compact.
		\end{proposition}
		\begin{proof}
			Let $X$ be compact, and $f\colon X \to Y$ surjective. Then for $\phi \in \opn^Y$ we have
			$$\xall{y}{Y}{\phi(y)} \iff \xall{x}{X}{\phi(f(x))}.$$
		\end{proof}
		
		Recall from the discussion above that classically a topological space $X$ is Hausdorff if and only if the diagonal in $X \times X$ is closed, and is discrete when this diagonal is open. We defer the discussion of synthetic Hausdorffness until the next section where we introduce closed subsets, but we can say something about discreteness. Synthetically, a set $X$ is called \df{discrete} when the diagonal $\diag[X]$ is an open subset of $X \times X$, or equivalently, when its characteristic map, the equality, is an open predicate. In logical formulation this means that $X$ is discrete when for every $x, y \in X$ the truth value of the statement $x = y$ is open.
		
		Classically, we are more used to discreteness to mean that singletons are open, but it is all the same.
		\begin{proposition}
			A set is discrete if and only if all its singletons are open.
		\end{proposition}
		\begin{proof}
			Because for any $x, y \in X$
			$$x = y \iff x \in \{y\}.$$
		\end{proof}
		
		Let us connect the two properties we have thus far.
		\begin{proposition}
			A compact product of discrete sets is discrete.
		\end{proposition}
		\begin{proof}
			Let $(X_i)_{i \in K}$ be a family of discrete spaces, indexed by a compact set $K$. Then for any $x = (x_i)_{i \in K}, y = (y_i)_{i \in K} \in \prod_{i \in K} X_i$
			$$x = y \iff \xall{i}{K}{x_i = y_i}.$$
		\end{proof}
		This can be seen as a generalization of the classical fact that finite products of discrete spaces are discrete. Classically it is difficult to say what a ``compactly indexed family'' of spaces is (but see~\cite{Escardo_M_2009:_intersections_of_compactly_many_open_sets_are_open}), however synthetically there is no problem, as all sets come with topology included. A hint that this has meaning in classical topology as well is a special case of this proposition when all factors are the same, and the product is actually an exponential. Indeed, recall that if $K$ is compact and $X$ is discrete, then $\C(K, X)$, equipped with the compact-open topology, is a discrete space.
		
		Actually, this proposition is just a special case of the following one since a singleton in a product is the intersection of preimages of singletons in individual factors via projections.
		\begin{proposition}
			\
			\begin{enumerate}
				\item
					Let $K$ be compact and $X$ arbitrary. Then any $K$-indexed intersection of open subsets of $X$ is open.
				\item
					Conversely, let $K$ be such that any $K$-indexed intersection of open subsets of $\one$ is again open. Then $K$ is compact.
			\end{enumerate}
		\end{proposition}
		\begin{proof}
			\begin{enumerate}
				\item
					Let $(U_i)_{i \in K}$ be a family of open subsets of $X$. Then for any $x \in X$
					$$x \in \bigcap_{i \in K} U_i \iff \xall{i}{K}{x \in U_i}.$$
				\item
					Tautological, as the intersection of subsets of $\one$ matches the universal quantifier (via identification $\pst(\one) \ism \soc$).
			\end{enumerate}
		\end{proof}
		What this proposition tells us is that a set $K$ is compact if and only if all $K$-indexed intersections of opens are open, and that it is sufficient to check this condition already on the singleton $\one$.
		
		In classical topology finite intersections and arbitrary unions of opens are open. We saw that synthetically the first part generalizes to the fact that compact intersections of opens are open. What happens with unions, synthetically? It quickly turns out that requiring arbitrary unions of opens to be open is a far too stringent condition; indeed $\top \in \opn$ (which amounts to saying that a set is always open in itself, a very reasonable condition) would then imply $\opn = \soc$, \ie arbitrary subset is open, as any subset $A \subseteq \one$ can be written as a union $A = \bigcup_{* \in A} \one$.
		
		This suggests that the dual notion of compactness, one that captures the property which unions of opens are open, should be scrutinized.
		\begin{definition}
			A set $X$ is \df{overt} when any of the following equivalent statements hold.
			\begin{itemize}
				\item
					The existential quantifier $\exists_X\colon \pst(X) \to \soc$ restricts to $\tp(X) \to \opn$.
				\item
					For every open subset $U \in \tp(X)$ the statement $\xsome{x}{X}{x \in U}$ (\ie the statement ``$U$ is inhabited'') is open.
				\item
					For every open predicate $\phi \in \opn^X$ the statement $\xsome{x}{X}{\phi(x)}$ (``$\phi$ attains truth'') is open.
			\end{itemize}
		\end{definition}
		
		Just as for compactness we can prove that finite products and images of overt sets are overt. As for two further example of the duality between compactness and overtness, recall first that classically a compact subset of a Hausdorff space is closed, and second that a space is compact precisely when projections along it are closed maps (Theorem~\ref{Theorem: characterization_of_classical_compactness}).
		\begin{proposition}\label{Proposition: overt_subset_of_discrete}
			An overt subset of a discrete set is open.
		\end{proposition}
		\begin{proof}
			Let $X$ be discrete and $U \subseteq X$ overt. Then for any $x \in X$
			$$x \in U \iff \xsome{y}{U}{x = y}.$$
			Alternatively, simply view $U$ as an overt union of its (open) singletons.
		\end{proof}
		
		\begin{proposition}\label{Proposition: overt_iff_projections_open}
			A set $X$ is overt if and only if for every $Y$ the projection $p\colon X \times Y \to Y$ is an open map.
		\end{proposition}
		\begin{proof}
			\begin{itemize}
				\item\proven{$(\Rightarrow)$}
					Let $X$ be overt and $U \subseteq X \times Y$ open. Then for any $y \in Y$
					$$y \in p(U) \iff \xsome{x}{X}{(x, y) \in U}.$$
				\item\proven{$(\Leftarrow)$}
					Define
					$$V \dfeq \st{(x, U) \in X \times \tp(X)}{x \in U};$$
					since $x \in U$ is an open statement, $V$ is open in $X \times \tp(X)$. By assumption $p(V)$ is open in $\tp(X)$, so for every $U \in \tp(X)$ the statement $U \in p(V)$, or equivalently $\xsome{x}{X}{x \in U}$, is open.
			\end{itemize}
		\end{proof}
		
		The notion of overtness comes from locale theory. Locales are, roughly speaking, spaces, given by their lattice of open subsets, rather than as a set of points with additional structure~\cite{Johnstone_PT_2002:_stone_spaces}. The point-free view of topology lends itself far better to constructive mathematics and algebra whereas for classical topology one relies on the excluded middle and the axiom of choice.\footnote{For example, the Tychonoff theorem~\cite{Johnstone_PT_2002:_stone_spaces} and the Gelfand duality~\cite{Coquand_T_Spitters_B_2009:_constructive_gelfand_duality_for_c_star_algebras} can be proven for locales without using the axiom of choice.} It was observed by Joyal, Tierney, Johnstone, Sambin, Martin-L\"of and others~\cite{Joyal_A_Tierney_M_1984:_an_extension_of_the_galois_theory_of_grothendieck, Johnstone_PT_1984:_open_locales_and_exponentiation, Sambin_G_1987:_intuitionistic_formal_spaces_a_first_communication} that stating inhabitedness (``non-emptiness in a positive way'') of a locale is a non-trivial property, and the locales for which this could be done in a suitable way were called \df{open} because this property is equivalent to the unique map into $\one$ being an open map. Needless to say, with openness being an overloaded term, confusion arose, and Paul Taylor renamed the notion to `overt' (and developed it in the framework of his ASD\footnote{Actually there is a certain difference between his notion of overtness and ours. We define the existential quantifier as a map $\soc^X \to \soc$ (if we use the predicate version, as Taylor does) which is characterized as the left adjoint of the transpose of the projection $\soc \times X \to \soc$, and we call $X$ overt when this map restricts to $\opn^X \to \opn$ where it is again the left adjoint to the transpose of the projection. However Taylor works in a framework without powersets, and so defines that $X$ is overt when the left adjoint $\opn^X \to \opn$ exists. It can happen that such an adjoint on a subset of $\soc$ exists, but is not the restriction of the usual existential quantifier; for example, $\nnst$ is a complete Boolean lattice, but in general not a sublattice of $\soc$. Taylor assumes some further axioms in ASD though, and it is a matter of further research whether they ensure that both notions of overtness coincide in cases which could be reasonably called models of both ASD and synthetic topology.}~\cite{Taylor_P_2000:_geometric_and_higher_order_logic_in_terms_of_abstract_stone_duality}). 
		\begin{proposition}
			A set $X$ is overt if and only if $\trm[X]\colon X \to \one$ is an open map.
		\end{proposition}
		\begin{proof}
			Since for any $U \in \tp(X)$ we have $\trm[X](U) = \st{* \in \one}{\xsome{x}{X}{x \in U}}$. (Compare also with the previous proposition.)
		\end{proof}
		Compactness is one of the most important topological properties while overtness is unfamiliar to all but a few classical topologists. The reason for this is because overtness, in whatever form we choose, is satisfied by all topological spaces. Classically, arbitrary unions of open subsets are open. Every map into $\one$ is open since all subsets of $\one$ are open. For every $X$ the set $\st{U \in \tp(X)}{\xsome{x}{X}{x \in U}}$ is open in $\tp(X)$ (recall Lemma~\ref{Lemma: classical_compact-open_basis_for_topology}, and consider the union $\bigcup_{x \in X} \upward{\{x\}}$). A projection from a product to one of its factors is always an open map.
		
		Even for compactness there is an issue which is classically invisible. For a topological space $X$, when saying that every open cover has a finite subcover, it does not matter whether the cover is open in $X$, or in some larger space. This is because for a subset of a topological space it is automatically assumed that it has the subspace (also called the relative, the inherited, or the induced) topology, in the sense that if $X \subseteq Y$ and $U \in \tp(X)$, then there exists an open subset $V \in \tp(Y)$ such that $U = X \cap V$. In other words, there is no difference in saying that a space is compact in some larger space, or compact on its own. Synthetically such difference exists, so we introduce new definitions.
		\begin{definition}
			A subset $X \subseteq Y$ is called
			\begin{itemize}
				\item
					\df{subcompact}\footnote{Usually it is simply called ``compact in $Y$'' while if $X$ is compact in the preceding sense, it is said that it is ``compact in itself'', or ``on its own''. This terminology feels unwieldy to me, and as one quickly starts replacing ``open in this and this'' simply with ``open'', this can lead to confusion in the case of compactness. For this reason I suggest the term `subcompact' (to mean a ``\textbf{sub}set, \textbf{compact} in a given larger set'').} in $Y$ when for every $U \in \tp(Y)$ the statement $\xall{x}{X}{x \in U}$ (\ie $X \subseteq U$) is open (equivalently, when for every open predicate $\phi\colon Y \to \opn$ the statement $\xall{x}{X}{\phi(x)}$ is open);
				\item
					\df{subovert} in $Y$ when for every $U \in \tp(Y)$ the statement $\xsome{x}{X}{x \in U}$ (\ie $X \between U$) is open (equivalently, when for every open predicate $\phi\colon Y \to \opn$ the statement $\xsome{x}{X}{\phi(x)}$ is open).
			\end{itemize}
			We denote the set of subcompact subsets of $X$ by $\cp(X)$, and the set of subovert subsets of $X$ by $\ov(X)$.
		\end{definition}
		
		As we will promptly see, subcompactness and subovertness are weaker properties than compactness and overtness, but they usually suffice, so it makes sense to use them to obtain more general propositions. To give an example how this works in practice, we generalize Proposition~\ref{Proposition: overt_subset_of_discrete}.
		\begin{proposition}\label{Proposition: subovert_subset_of_discrete}
			A subovert subset of a discrete set is open.
		\end{proposition}
		\begin{proof}
			Let $X$ be discrete and $U \subseteq X$ subovert. Take any $x \in X$. Then
			$$x \in U \iff \xsome{y}{U}{x \in \{y\}}.$$
			We see that the proof is exactly the same as in Proposition~\ref{Proposition: overt_subset_of_discrete}, we just had to write it in such a way that it is clear that quantification goes over $x$ belonging to a set which is open in $X$, not just in $U$.
		\end{proof}
		
		Subovert and subcompact subsets have analogous properties to the ones above shown for overt and compact sets.
		\begin{proposition}\label{Proposition: images_and_finite_products_of_subovert/subcompact_subsets}
			\
			\begin{enumerate}
				\item
					The image of a subovert/subcompact subset is subovert/subcompact.
				\item
					Finite product of subovert/subcompact subsets is subovert/subcompact.
			\end{enumerate}
		\end{proposition}
		\begin{proof}
			\begin{enumerate}
				\item
					Let $f\colon X \to Y$ be a map, and $A \subseteq X$ subovert/subcompact. Take any $\phi \in \opn^Y$. We have
					$$\xsome{y}{f(X)}{\phi(y)} \iff \xsome{x}{X}{\phi(f(x))},$$
					and similarly,
					$$\xall{y}{f(X)}{\phi(y)} \iff \xall{x}{X}{\phi(f(x))}.$$
				\item
					We have to verify the condition for nullary and binary products. Let $\phi$ be an open predicate on $\one$.
					Then
					$$\xsome{x}{\one}{\phi(x)} \iff \phi(*) \iff \xall{x}{\one}{\phi(x)}.$$
					Now let $\phi$ be an open predicate on $X \times Y$, and let $A \subseteq X$, $B \subseteq Y$ be subovert/subcompact subsets. Then
					$$\xsome{a}{A \times B}{\phi(a)} \iff \xsome{x}{A}\xsome{y}{B}{\phi(x, y)},$$
					and likewise
					$$\xall{a}{A \times B}{\phi(a)} \iff \xall{x}{A}\xall{y}{B}{\phi(x, y)}.$$
			\end{enumerate}
		\end{proof}
		The consequence of the first item is that an overt/compact set is subovert/subcompact in any larger set (as it is the image of its inclusion). On the other hand, it is clear from the definition that a set is overt/compact if and only if it is subovert/subcompact in itself. In these senses subovertness/subcompactness generalize overtness/compactness.
		
		For overtness the (sub)overtness of a dense subset is sufficient.
		\begin{definition}
			A subset $D \subseteq X$ is \df{dense} in $X$ when every inhabited $U \in \tp(X)$ intersects it.
		\end{definition}
		\begin{proposition}\label{Proposition: testing_overtness_on_dense_subset}
			Suppose $D \subseteq X$ is a subovert dense subset in $X$. Then $X$ is overt.
		\end{proposition}
		\begin{proof}
			By definition
			$$\xsome{x}{X}{x \in U} \iff \xsome{x}{D}{x \in U}$$
			for every $U \in \tp(X)$.
		\end{proof}
		
		So far we have not made any assumptions on $\opn$; in principle it can be an arbitrary subset of $\soc$. In practice we need more information than that. One of the basic assumptions on $\opn$ is that it is a bounded $\land$-subsemilattice of $\soc$, \ie that it is closed under finite meets (and consequently, that any topology $\tp(X)$ is closed under finite intersections). This is just saying that finite sets are compact.
		
		\begin{proposition}\label{Proposition: intersection_of_subovert_and_open_is_open_(original)}
			Assume that $\opn$ is closed under binary meets. Then the intersection of a subovert subset and an open subset is subovert.
		\end{proposition}
		\begin{proof}
			Let $A \subseteq X$ be subovert, and $U, V \subseteq X$ open. Then
			$$\xsome{x}{A \cap U}{x \in V} \iff \xsome{x}{A}{x \in U \cap V}.$$
		\end{proof}
		
		\begin{corollary}\label{Corollary: open_subset_of_overt_is_subovert_(original)}
			If $\opn$ is closed under binary meets, then an open subset of an overt set is subovert.
		\end{corollary}
		\begin{proof}
			In (the proof of) the previous proposition take $A = X$.
		\end{proof}
		
		This corollary can be seen as the dual statement to the classical proposition that a closed subset of a compact space is compact. The proof of this proposition actually resorts to the fact that classically `compactness' and `subcompactness' are the same (so a more precise statement would in fact be that a closed subset of a compact space is subcompact). We already mentioned that this is because we assume subsets of a topological space to be equipped with the subspace topology, something that synthetically we generally don't have. We consider this issue now.
		
		First of all recall that all maps are continuous. Consider the following special case: let $A \subseteq X$, and denote its inclusion $i\colon A \hookrightarrow X$. Continuity of $i$ means that for every $U \in \tp(X)$ the set $i^{-1}(U)$, \ie $A \cap U$, is open in $A$. The subset $A$ is called a subspace in $X$ when all its opens are obtained in this way. Here is the precise definition.
		
		\begin{definition}
			Let $A \subseteq X$. We say that
			\begin{itemize}
				\item
					$A$ is a \df{subspace} in $X$ (or \df{has the subspace topology}, inherited from $X$) when for every $U \in \tp(A)$ there exists $V \in \tp(X)$ such that $U = A \cap V$;
				\item
					$A$ is a \df{well-embedded subspace} in $X$ when there exists a map $e\colon \tp(A) \to \tp(X)$ such that $U = A \cap e(U)$ for all $U \in \tp(A)$.\footnote{This terminology is taken from \Esc['s]~\cite{Escardo_M_2004:_notes_on_synthetic_topology}. \Esc also considers two stronger notions, with $e$ being a left or a right adjoint to $i^{-1}\colon \tp(X) \to \tp(A)$, but we won't need these conditions.}
			\end{itemize}
		\end{definition}
		Clearly a well-embedded subspace is also a subspace (for the converse we would in general require some choice). Examples of (well-embedded) subspaces are retracts: if $r\colon X \to A$ is a retraction and $U \in \tp(A)$, then $r^{-1}(U) \in \tp(X)$ and $U = A \cap r^{-1}(U)$. In particular, singletons are (well-embedded) subspaces. See Chapter~\ref{Chapter: models} for examples of subsets which are not subspaces (\eg the reals $\RR$ having the Euclidean topology, and the rationals $\QQ$ the discrete topology).
		
		\begin{proposition}\label{Proposition: subovert/subcompact_to_overt/compact}
			A subovert/subcompact well-embedded subspace is overt/compact.
		\end{proposition}
		\begin{proof}
			Let $A \subseteq X$, and let $e\colon \tp(A) \to \tp(X)$ witness that $A$ is well-embedded in $X$. Take any $U \in \tp(A)$. The proposition follows from equivalences
			$$\xsome{x}{A}{x \in U} \iff \xsome{x}{A}{x \in e(U)} \quad \text{and} \quad \xall{x}{A}{x \in U} \iff \xall{x}{A}{x \in e(U)}.$$
		\end{proof}
		
		From this analysis we see that if we want an open subset of an overt set to be overt, it makes sense to require that open subsets are well-embedded subspaces. The following notion is due to Rosolini~\cite{Rosolini_G_1986:_continuity_and_effectiveness_in_topoi}.
		\begin{definition}
			$\opn$ is a \df{dominance} when $\top \in \opn$ and the following \df{dominance axiom} holds:
			$$\xall{u}{\opn}\xall{p}{\soc}{\Big(\big(u \implies (p \in \opn)\big) \implies (u \land p) \in \opn\Big)}.$$
		\end{definition}
		This might seem a complicated condition, but what it says is that if the part where $p$ meets $u$ is open in $u$, then it is open; in other words, it captures the notion of \df{transitivity of openness} (subset, open in an open subset, is open in the whole space), restricted to truth values, or subsets of $\one$ if you will. Since they determine the topology on all sets, we may wager that it is enough. Indeed, we have the following proposition.
		
		\begin{proposition}\label{Proposition: dominance}
			The following statements are equivalent.
			\begin{enumerate}
				\item
					$\opn$ is a dominance.
				\item
					$\top \in \opn$ and openness is transitive.
				\item
					Finite sets are compact and open subsets are well-embedded subspaces.
				\item
					$\opn$ is a closed under finite meets and open subsets are subspaces.
			\end{enumerate}
		\end{proposition}
		\begin{proof}
			\begin{itemize}
				\item\proven{$(1 \impl 2)$}
					We have $\top \in \opn$ by assumption. Let $V \subseteq U \subseteq X$ so that $U$ is open in $X$ and $V$ is open in $U$. Take any $x \in X$. Then the statement $x \in U$ is open, and it implies openness of $x \in V$, so by the dominance axiom the statement $(x \in U) \land (x \in V)$ --- which is just $x \in V$ --- is open.
				\item\proven{$(2 \impl 3)$}
					We know that compactness of finite sets means that $\opn$ is a bounded $\land$-subsemilattice of $\soc$. The assumption $\top \in \opn$ takes care of nullary meets. As for binary meets, take any $u, v \in \opn$, and view them as open subsets of $\one$. Recall that continuity of inclusion $u \hookrightarrow \one$ implies that $u \land v$ is open in $u$, and hence by transitivity also open in $\one$.
					
					Let now $U \subseteq X$ be an open subset. Transitivity of openness implies that the map $e\colon \tp(U) \to \tp(X)$, defined by $e(V) \dfeq V$, witnesses that $U$ is a well-embedded subspace.
				\item\proven{$(3 \impl 4)$}
					Immediate.
				\item\proven{$(4 \impl 1)$}
					We obtain $\top \in \opn$ because $\opn$ is closed under nullary meets. Take now any $u \in \opn$ and $p \in \soc$ such that $u$ implies openness of $p$. If we interpret truth values as subsets of $\opn$, this amounts to saying that $u \land p$ is an open subset of $u$, so by assumption there exists $v \in \opn$ such that $u \land p = u \land v$. Because $\opn$ is closed under binary meets, $u \land v$ is open, and then so is $u \land p$.
			\end{itemize}
		\end{proof}
		
		\begin{corollary}\label{Corollary: if_dominance_then_open_subset_of_overt_is_overt_(original)}
			If $\opn$ is a dominance, then open subsets of overt sets are overt.
		\end{corollary}
		\begin{proof}
			By the previous proposition $\opn$ is closed under binary meets, and open subsets are well-embedded subspaces. Now use Corollary~\ref{Corollary: open_subset_of_overt_is_subovert_(original)} and Proposition~\ref{Proposition: subovert/subcompact_to_overt/compact}.
		\end{proof}
		
		Results such as this corollary put a certain issue into perspective. There are statements which hold in all models of synthetic topology (such as Proposition~\ref{Proposition: images_and_finite_products_of_subovert/subcompact_subsets}). Generally though we deal with statements which hold only with some additional assumptions. One way to cope with this is to consider a specific model of synthetic topology (we do this in Chapter~\ref{Chapter: models}), but when a property is common enough, it is better to identify suitable sufficient conditions under which we can prove it, the benefit being that this holds for a larger class of models, and gives us a deeper insight. Such conditions are given as \emph{principles}, a common occurrence in constructive mathematics. The \Sier object $\opn$ being a dominance is one example, and there are others scattered throughout the thesis. Here we recall just one more (and some of its consequences); we won't need it later, but it deserves a mention in its own right.
		
		\begin{definition}
			A model of synthetic topology satisfies \df{Phoa's principle}~\cite{Phoa_W_1990:_domain_theory_in_realizability_toposes} when $\opn$ is a bounded sublattice of $\soc$, and for every map $f\colon \opn \to \opn$ and $u \in \opn$ the condition
			$$f(u) = (f(\top) \land u) \lor f(\bot), \quad \text{or equivalently,} \quad f(u) = (f(\bot) \lor u) \land f(\top),$$
			holds.
		\end{definition}
		
		The \Sier set $\opn \subseteq \soc$ can in general be arbitrarily bizarre; Phoa's principle states however, that it actually behaves the way the \Sier space $\sier$ does (examine the continuous maps $\sier \to \sier$).
		
		\begin{proposition}
			If Phoa's principle holds, then every map $\opn \to \opn$ is monotone.
		\end{proposition}
		\begin{proof}
			Take any $f\colon \opn \to \opn$ and any $u, v \in \opn$ such that $u \impl v$ which is equivalent to $u \lor v = v$. By Phoa's principle
			$$f(u) \lor f(v) = (f(\top) \land u) \lor f(\bot) \lor (f(\top) \land v) \lor f(\bot) =$$
			$$= \big(f(\top) \land (u \lor v)\big) \lor f(\bot) = (f(\top) \land v) \lor f(\bot) = f(v),$$
			so $f(u) \impl f(v)$.
		\end{proof}
		We used the $f(u) = (f(\top) \land u) \lor f(\bot)$ version of Phoa's principle, but the proof from $f(u) = (f(\bot) \lor u) \land f(\top)$ is analogous. This enables us to see that the two versions are indeed equivalent since
		$$(f(\top) \land u) \lor f(\bot) = \big(f(\top) \lor f(\bot)\big) \land \big(u \lor f(\bot)\big) = (f(\bot) \lor u) \land f(\top)$$
		because $f(\bot) \impl f(\top)$.
		
		\begin{proposition}
			If Phoa's principle holds, then for any $X$ the set $\tp(X) \ism \opn^X$ is overt and compact (in particular, $\opn$ itself is overt and compact).
		\end{proposition}
		\begin{proof}
			Take any $U \in \tp(\opn^X)$ (with the characteristic map $\chi_U\colon \opn^X \to \opn$).
			
			Let $\psi \in U$. Define a map $g\colon \opn \to \opn^X$ by $g(u)(x) \dfeq u \lor \psi(x)$, and let $f\colon \opn \to \opn$ be $f \dfeq \chi_U \circ g$. By the previous proposition $f(\bot) \impl f(\top)$, but $f(\bot) = \chi_U(g(\bot)) = \chi_U(\psi) = \top$ which means
			$$\top = f(\top) = \chi_U(g(\top)) = \chi_U(x \mapsto \top).$$
			Therefore $U$ is inhabited if and only if $(x \mapsto \top) \in U$ which is an open statement, thus proving overtness of $\opn^X$.
			
			Assume now that $(x \mapsto \bot) \in U$, and take arbitrary $\psi \in \opn^X$. Define $g\colon \opn \to \opn^X$ by $g(u)(x) \dfeq u \land \psi(x)$, and $f\colon \opn \to \opn$ again by $f \dfeq \chi_U \circ g$. We have $f(\bot) = \chi_U(g(\bot)) = \chi_U(x \mapsto \bot) = \top$ and $f(\bot) \impl f(\top)$, so
			$$\top = f(\top) = \chi_U(g(\top)) = \chi_U(\psi).$$
			So as soon as $U$ contains the characteristic map of the empty subset of $X$, it contains all of $\opn^X$. Since $(x \mapsto \bot) \in U$ is an open statement, $\opn^X$ is compact.
		\end{proof}
		
		\intermission
		
		We close the section by considering some examples of $\opn$. We restrict to cases when $\opn$ is closed under finite $\land$.
		\begin{itemize}
			\item\proven{$\opn = \soc$}
				This is the largest possible $\opn$, and in this case every subset (and every predicate) is open. This corresponds classically to discrete topology. We see that every set is overt and compact, and $\opn$ is trivially a dominance.
			\item\proven{$\opn = \{\top\}$}
				This is the smallest possible $\opn$ (if we require it to contain finite, in particular nullary, meets). In every set $X$ only $X$ itself is open. This corresponds classically to trivial (indiscrete) topology. Every set is compact, and precisely inhabited ones are overt. It is easily seen that this $\opn$ is a dominance.
			\item\proven{$\opn = \two$}
				In this case precisely decidable subsets are open, and at least finite sets are overt and compact. This $\opn$, and more generally every one contained in $\two$, is a dominance.
			\item\proven{$\opn = \st{p \in \soc}{\lnot{p} \lor \lnot\lnot{p}}$}
				This is a little relaxed previous case: open subsets are those of which membership satisfies weak law of excluded middle (as opposed to the previous case where it had to satisfy the full one). Consider what it means for this $\opn$ to be a dominance. In this case the dominance axiom states
				$$\xall{u}{\opn}\xall{p}{\soc}{\Big(\big(u \implies (\lnot{p} \lor \lnot\lnot{p})\big) \implies \big(\lnot(u \land p) \lor \lnot\lnot(u \land p)\big)\Big)}.$$
				It is easily seen that the condition is satisfied when $\lnot{u}$. Assume now $\lnot\lnot{u}$. Since then
				$$\lnot(u \land p) = \lnot\lnot\lnot(u \land p) = \lnot(\lnot\lnot{u} \land \lnot\lnot{p}) = \lnot\lnot\lnot{p} = \lnot{p},$$
				the dominance condition amounts to
				$$\xall{u, p}{\soc}{\Big(\big(\lnot\lnot{u} \land \big(u \implies (\lnot{p} \lor \lnot\lnot{p})\big)\big) \implies (\lnot{p} \lor \lnot\lnot{p})\Big)}.$$
				We see that this $\opn$ is a dominance if and only if the weak law of excluded middle \wlem holds (in which case $\opn = \soc$): for the left-to-right direction take $u = (\lnot{p} \lor \lnot\lnot{p})$, and the other direction is trivial.
			\item\proven{$\opn = \nnst$}
				In this case the open subsets are the stable ones. (At least) finite sets are compact, and the empty set and the singletons are overt. As for the overtness of other finite sets, it is easy to see that the following statements are equivalent:
				\begin{itemize}
					\item
						$\opn$ is a bounded sublattice of $\soc$,
					\item
						$\wlem$ holds,
					\item
						$\nnst = \two$.
				\end{itemize}
				
				We also show that this $\opn$ is a dominance. Take any $u \in \opn$ and $p \in \soc$ such that $u \implies (p \in \nnst)$, or equivalently, $u \land \lnot\lnot{p} \implies p$. Then
				$$\lnot\lnot(u \land p) = \lnot\lnot{u} \land \lnot\lnot{p} = u \land \lnot\lnot{p} = u \land p.$$
			\item\proven{$\opn = \Ros$}
				In this case the open subsets are the semidecidable ones, \ie those of which membership is in the so-called \df{Rosolini dominance}
				$$\Ros = \st{u \in \soc}{\xsome{\alpha}{\two^\NN}{\big(u \iff \xsome{n}{\NN}{\alpha_n = 0}\big)}}.$$
				As the name suggests, this is indeed a dominance which was proven by Rosolini~\cite{Rosolini_G_1986:_continuity_and_effectiveness_in_topoi}. Here is the idea of the proof. Take any $u \in \Ros$ and $p \in \soc$ such that $u \implies (p \in \Ros)$. There exists $\alpha \in \two^\NN$ for which $u \iff \xsome{n}{\NN}{\alpha_n = 0}$. We need to find $\beta \in \two^\NN$ such that $u \land p \iff \xsome{n}{\NN}{\beta_n = 0}$. Define $\beta$ to be all $1$s as long as the terms of $\alpha$ are $1$. If however we ever come across a term of $\alpha$ which is $0$, then $u$ holds, so there exists $\gamma \in \two^\NN$ such that $p \iff \xsome{n}{\NN}{\gamma_n = 0}$. We take this $\gamma$ for the rest of $\beta$.
		\end{itemize}
		
		We'll spend a few more words about the Rosolini dominance.
		
		\begin{lemma}\label{Lemma: overtness_of_countable_sets}
			For any $\opn$ if $\NN$ and $\emptyset$ are overt, then all countable sets are overt.
		\end{lemma}
		\begin{proof}
			Let $X$ be countable which means we have a surjection $\NN \to \one + X$. If $\NN$ is overt, then so is its image $\one + X$. Take any $U \in \tp(X)$ and an arbitrary $x \in \one + X$. If $x \in \one$, then $x \in U$ equals $\bot$ which is an open truth value because $\emptyset$ is overt. If $x \in X$, then $x \in U$ is open. We conclude that $U$ is open also in $\one + X$, and since
			$$\xsome{x}{X}{x \in U} \iff \xsome{x}{\one + X}{x \in U}$$
			and $\one + X$ is overt, the result follows.
		\end{proof}
		
		\begin{lemma}\label{Lemma: Rosolini_dominance_and_countable_joins}
			\
			\begin{enumerate}
				\item
					$\Ros$ is a bounded sublattice of $\soc$.
				\item
					If $\opn \subseteq \soc$ is closed under finite meets and countable joins, then $\Ros \subseteq \opn$.
				\item
					Assuming $\ACRos$, $\Ros$ is closed under countable joins.
			\end{enumerate}
		\end{lemma}
		\begin{proof}
			\begin{enumerate}
				\item
					The nullary meet $\top$ and join $\bot$ are represented by the constant sequences of $0$s and $1$s respectively. Now take $u, v \in \Ros$, represented by $\alpha, \beta \in \two^\NN$. Then $u \land v$ is represented by $(\sup\{\alpha_m, \beta_n\})_{m, n \in \NN}$ (precomposed by a bijection $\NN \ism \NN \times \NN$ to obtain a sequence), and to represent $u \lor v$, interleave $\alpha$ and $\beta$.
				\item
					Since $\opn$ contains nullary meet and join, $\two \subseteq \opn$. Countable joins of the decidable truth values are precisely the semidecidable ones, so $\Ros \subseteq \opn$.
				\item
					We already know $\bot \in \Ros$, so by Lemma~\ref{Lemma: overtness_of_countable_sets} it is sufficient to see that $\Ros$ is closed under $\NN$-indexed joins. Let $(u_i)_{i \in \NN}$ be a sequence of semidecidable truth values; this gives a map $u\colon \NN \to \Ros$. By $\ACRos$ this map lifts to $\tilde{u}\colon \NN \to \two^\NN$. Transpose it and precompose with $\NN \ism \NN \times \NN$ to obtain a sequence which represents $\bigvee_{i \in \NN} u_i$.
			\end{enumerate}
		\end{proof}
		
		The consequence of this lemma is that (given sufficient countable choice) $\Ros$ is the smallest subset of $\soc$ (and the smallest dominance), closed under finite meets and countable joins.

	\section{Redefinition of Basic Notions in Synthetic Topology}\label{Section: redefinition_of_synthetic_topology}
   
   	The introduction of the \Sier object $\opn$ equips all sets with an intrinsic topology, in the sense that it determines what open sets are (in a natural way, \ie all maps are continuous). However, we know from classical topology that closed sets play an important role as well, something about which $\opn$ stays silent. The question what closed subsets in synthetic topology should be was posed by \MEsc in~\cite{Escardo_M_2004:_notes_on_synthetic_topology}.
   	
   	One would naturally first try the classical definition, and if that does not work, some definitions which are classically equivalent to it. \Esc considers the following possibilities:
   	\begin{itemize}
   		\item
   			a subset is closed when it contains all its adherent points ($x \in X$ is an \df{adherent point} of $A \subseteq X$ when every open set containing $x$ intersects $A$),
   		\item
   			a subset is closed when it is the complement of an open subset,
   		\item
   			a subset is closed when its complement is open.
   	\end{itemize}
   	He argues however that none of these options seem to suit in general. Indeed, they don't satisfy expected theorems, \eg that closed subsets of compact sets are (sub)compact (in classical analogy and analogy with Corollary~\ref{Corollary: open_subset_of_overt_is_subovert_(original)}), or behave analogously to synthetically open subsets. In view of this \Esc argues for the following definition:
   	\begin{itemize}
   		\item
   			a subset $F \subseteq X$ is closed in $X$ when the set $F \setimpl[X] U = \st{x \in X}{x \in F \implies x \in U}$ is open in $X$ for every $U \subseteq F$ open in $F$.
   	\end{itemize}
   	We will call such subsets \df{\Esc closed}.
   	
   	To motivate such a definition, notice that it is exactly what we need to prove the following proposition.
   	\begin{proposition}\label{Proposition: Escardo_closed_inherits_compactness}
   		An \Esc closed subset of a compact set is compact.
   	\end{proposition}
   	\begin{proof}
   		For any \Esc closed $F \subseteq X$ and $U \in \tp(F)$ we have
   		$$\xall{x}{F}{x \in U} \iff \xall{x}{X}{x \in (F \setimpl[X] U)}.$$
   	\end{proof}
   	
   	In the previous section we recalled that classically a topological space $X$ is Hausdorff if and only if the diagonal $\diag[X] \subseteq X \times X$ is closed in the topological product $X \times X$. We prefer this for the synthetic definition of Hausdorffness (as opposed to the existence of separating open sets of distinct points). In this case we obtain yet another theorem, familiar from classical topology, and analogous to Proposition~\ref{Proposition: overt_subset_of_discrete}.
   	\begin{proposition}\label{Proposition: compact_in_Hausdorff_is_Escardo_closed}
   		A compact subset in a Hausdorff set is \Esc closed.
   	\end{proposition}
   	\begin{proof}
   		Let $X$ be Hausdorff, $K \subseteq X$ compact, and $U \in \tp(K)$. Then for any $x \in X$
   		\begin{align*}
   			& x \in \big(K \setimpl[X] U\big)\\
   			\iff\quad& x \in K \implies x \in U\\
   			\iff\quad& \big(\xsome{k}{K}{x = k}\big) \implies x \in U\\
   			\iff\quad& \all{k}{K}{x = k \implies x \in U}\\
   			\iff\quad& \all{k}{K}{(x, k) \in \diag[K] \implies (x, k) \in \st{(a, a) \in \diag[K]}{a \in U}}.
   		\end{align*}
   		The last condition is open, so $K \setimpl[X] U$ is open in $X$.
   	\end{proof}
   	
   	Here is another one (compare with Theorem~\ref{Theorem: characterization_of_classical_compactness} and Proposition~\ref{Proposition: overt_iff_projections_open}).
   	\begin{proposition}
   		If $K$ is compact and $X$ any set, then the projection $p\colon K \times X \to X$ maps \Esc closed subsets of $K \times U$ to \Esc closed subsets of $X$.
   	\end{proposition}
   	\begin{proof}
   		Let $F \subseteq K \times X$ be \Esc closed, and take any $U \in \tp(p(F))$. Then for $x \in X$
   		\begin{align*}
   			& x \in p(F) \implies x \in U\\
   			\iff\quad& \big(\xsome{k}{K}{(k, x) \in F}\big) \implies x \in U\\
   			\iff\quad& \all{k}{K}{(k, x) \in F \implies x \in U}\\
   			\iff\quad& \all{k}{K}{(k, x) \in F \implies (k, x) \in p^{-1}(U) \cap F}.
   		\end{align*}
   		Since $p^{-1}(U) \cap F$ is open in $F$, the result follows.
   	\end{proof}
   	
   	In view of these \Esc['s] results, it would appear that we have our candidate for closed subsets in synthetic topology. However, I feel that the following issues (the first two to be proven later, after Proposition~\ref{Proposition: classification_of_strongly_open/closed}) are still problematic.
   	\begin{itemize}
   		\item
   			All maps are continuous with regard to open subsets, but in general not with regard to \Esc closed subsets.
   		\item
   			Open subsets are classified by $\opn$, but \Esc closed ones are not classified by anything in general.
   		\item
   			\Esc closed subsets are subspaces by their definition while for open subsets we require the dominance axiom. In particular, note the difference between Proposition~\ref{Proposition: Escardo_closed_inherits_compactness} and Corollaries~\ref{Corollary: open_subset_of_overt_is_subovert_(original)} and~\ref{Corollary: if_dominance_then_open_subset_of_overt_is_overt_(original)}.
   		\item
   			Classically, open and closed subsets are dual, and they are both suitable for the definition of topology, as knowing one determines the other. Synthetically, \Esc closed subsets are determined by open ones, but not vice versa, and overall there is not much duality between them in general.
   	\end{itemize}
   	
      For these reasons we decide to redefine the basic notions of synthetic topology, and put openness and closedness on equal terms. Our starting point is the last remark. Classically open and closed sets determine each other via complementation, something that constructively we don't have. \Esc['s] direction to define closed sets seems to be the right one, but it does not seem we can force the open sets to be determined by the closed ones. If it does not seem plausible to resolve the discrepancy that openness is a basic notion, but closedness is the derived one, in such a way that both determine the other, another solution then is to have \emph{both} openness and closedness as derived notions.
      
      By \Esc[], we test closedness via implication into open subsets; replace open subsets here by something general. The idea how to test openness hides in the dominance axiom: via conjunction. Recall that for any $p \in \soc$, the rules
      $$x \mapsto (x \land p) \quad \text{and} \quad x \mapsto (x \impl p)$$
      (where $x \in \soc$) are adjoint. From these observations we proceed with the redefinition.
      
      \intermission
      
      Let $\tst$ be an arbitrary subset of the set of truth values $\soc$. Call its elements \df{test truth values}, and subsets classified by it \df{test subsets} (meaning that $A \subseteq X$ is a test subset of $X$ when for every $x \in X$ the truth value of the statement $x \in A$ is in $\tst$). For example, if $\top \in \tst$, then every set is a test subset of itself. If $\two \subseteq \tst$, then every decidable subset is test. If $\tst = \emptyset$, then the only test subset anywhere is $\emptyset$ as a subset of itself. We denote the set of test subsets of a set $X$ by $\tetp(X)$. Of course, via identification $\pst(\one) \ism \soc$, the test subsets of $\one$ are precisely the test truth values, $\tetp(\one) \ism \tst$. In general, $\tetp(X) \ism \tst^X$.
      
      Define
      $$\opn \dfeq \st{u \in \soc}{(u \land t) \in \tst \text{ for all } t \in \tst},$$
      $$\cld \dfeq \st{f \in \soc}{(f \impl t) \in \tst \text{ for all } t \in \tst}.$$
      Call elements of $\opn$ and $\cld$ \df{open} and \df{closed} \df{truth values} respectively. Call subsets, classified by $\opn$, \df{open subsets}, and those classified by $\cld$ \df{closed subsets}\footnote{It seemed fitting to denote the set of closed truth values by $\cld$ (capital Greek zeta), partially because closed subsets in some ways behave like the open ones, but in some other ways quite the opposite --- and we get $\cld$ by inverting the upper half of $\opn$ --- and partially because in my native tongue the term for ‘closed’ starts with the letter z. On a side note, $\tst$ is also meant to be a capital Greek letter (namely tau).}. Let $\optp(X)$ denote the set of open subsets of a set $X$, and $\cltp(X)$ the set of closed subsets of $X$. We have identifications $\optp(X) \ism \opn^X$, $\cltp(X) \ism \cld^X$. In the particular case of a singleton, we identify open (resp.~closed) truth values with open (resp.~closed) subsets of $\one$.
      
      Since test, open and closed subsets are classified by the corresponding truth values, every map is continuous in all the following senses: preimages of test/open/closed subsets are test/open/closed. In particular, if we consider an inclusion $i\colon A \hookrightarrow X$, we see that $A \cap B = i^{-1}(B)$ is a test/open/ closed subset in $A$ if $B \subseteq X$ was in $X$.
      
      We can test openness and closedness of subsets directly, without referring to truth values.
      \begin{lemma}\label{Lemma: open_closed_subsets_characterization}
         A subset $U \subseteq X$ is open in $X$ if and only if for every $A \in \tetp(X)$ also $(U \cap A) \in \tetp(X)$. Similarly, $F \subseteq X$ is closed in $X$ if and only if for every $A \in \tetp(X)$ also $(F \setimpl[X] A) \in \tetp(X)$.
      \end{lemma}
      \begin{proof}
         For the `if' direction take any $x \in X$ and $t \in \tst \ism \tetp(\one)$. Then
         $$(x \in U) \land t \sepeq (x \in U \cap \trm[X]^{-1}(t)).$$
         For the `only if' direction just note $(x \in U \cap A) \sepeq (x \in U) \land (x \in A)$. In the case of closed sets, replace $\cap$, $\land$ with $\setimpl$, $\impl$.
      \end{proof}
      
      Notice that $\opn$ and $\cld$ are bounded $\land$-subsemilattices of $\soc$, ie.~they are closed under finite infima. It is sufficient to verify the condition for nullary and binary meets:
      \begin{align*}
         \top \land t &\sepeq t, & (u \land v) \land t &\sepeq u \land (v \land t),\\
         \top \impl t &\sepeq t, & (f \land g) \impl t &\sepeq f \impl (g \impl t).
      \end{align*}
      In fact, being a bounded $\land$-subsemilattice of $\soc$ is characteristic for $\opn$, as any such is the set of open truth values for some $\tst$, namely for itself. This follows from the following observations:
      $$\opn \subseteq \tst \iff \top \in \tst \qquad \text{and} \qquad \tst \subseteq \opn \iff \text{for any $t, s \in \tst$ also $t \land s \in \tst$}.$$
      
      The $\land$-semilattice structure on $\opn$ and $\cld$ induces the $\cap$-semilattice structure on $\optp(X)$ and $\cltp(X)$ for any $X$, \ie a set is always open and closed in itself, and binary intersections of open (resp.~closed) subsets are again open (resp.~closed).
      
      So we know something of meets of open and closed sets, but not about their joins, nor about meets or joins of test sets. Examining them leads us to the notions of overtness and compactness. We define:
		\begin{itemize}
			\item
				a set $X$ is \df{strongly overt} when for all $A \in \tetp(X)$ the truth value of $\xsome{x}{X}{x \in A}$ is test (equivalently, when unions of any $X$-indexed family of test subsets of some set is again a test subset in it);
			\item
				a set $X$ is \df{overt} when for all $U \in \optp(X)$ the truth value of $\xsome{x}{X}{x \in U}$ is open ($X$-indexed unions of opens are open).
		\end{itemize}
		
		The definition of overtness is the usual one. The name `strongly overt' is appropriate: if a set $X$ is strongly overt, it is also overt by the Frobenius law
		$$(\xsome{x}{X}{x \in U}) \land t \sepeq \some{x}{X}{(x \in U) \land t}.$$
		On the other hand we can find a trivial counterexample for the reverse implication. Take $\tst = \emptyset$; then $\opn = \soc$, so all sets are overt, but the empty one is not strongly overt (otherwise $\emptyset$ as an empty union of test sets would have to be test in any set). Though obviously the reverse implication holds when $\tst = \opn$.
		
		Compactness is not so direct, however. By analogy we would say that $X$ is strongly compact when for all $A \in \tetp(X)$ the truth value of $\xall{x}{X}{x \in A}$ is test, and compact when for all $U \in \optp(X)$ the truth value of $\xall{x}{X}{x \in U}$ is open. But the statements $(\xall{x}{X}{x \in U}) \land t$ and $\all{x}{X}{(x \in U) \land t}$ are not equivalent in general (they are for inhabited $X$). What testness of $\xall{x}{X}{x \in A}$ naturally implies is that unions of closed sets are closed, as opposed to that intersections of opens are open. Judging from Theorem~\ref{Theorem: characterization_of_classical_compactness} (and Corollary~\ref{Corollary: compactness_with_quantifiers}, and the subsequent discussion), it would just as well match the classical notion of compactness to consider which unions of closed subsets are closed, and use that for the definition. Indeed, we can suspect (and will demonstrate in short order) that the definition of compactness via closedness lends better to prove synthetic versions of classical theorems connecting compact and closed sets.
		
		Regardless, I did not decide for a redefinition, and consequently an introduction of yet another notion of compactness (there are already enough of them as it is). Also, it will prove useful for us to know which intersections of opens are open (in particular in Section~\ref{Section: (co)limits_and_bases} and Chapter~\ref{Chapter: metrization_theorems}), so we keep the usual definition of compactness while for unions of closed sets we introduce a new term, but one reminiscent of compactness, to suggest that the notions are closely connected. We define:
		\begin{itemize}
			\item
				a set $X$ is \df{strongly condensed} when for all $A \in \tetp(X)$ the truth value of $\xall{x}{X}{x \in A}$ is test ($X$-indexed intersections of test subsets are test).
			\item
				a set $X$ is \df{condensed} when for all $F \in \cltp(X)$ the truth value of $\xsome{x}{X}{x \in F}$ is closed ($X$-indexed unions of closed subsets are closed),
			\item
				a set $X$ is \df{compact} when for all $U \in \optp(X)$ the truth value of $\xall{x}{X}{x \in U}$ is open ($X$-indexed intersections of open subsets are open); thus in particular finite sets are compact.
		\end{itemize}
		
		Strong condenseness implies condenseness, as promised. The proof is analogous to the one above (strongly overt implies overt):
		$$(\xsome{x}{X}{x \in F}) \impl t \sepeq \all{x}{X}{(x \in F) \impl t}.$$

      As usual, these definitions can be generalized to subsets. We say that $A \subseteq X$ is a \df{subovert} subset of $X$ when for all $U \in \optp(X)$ the truth value of $\xsome{x}{X}{x \in U}$ is open, and analogously for other variations. Clearly a set is (strongly) overt/condensed/compact if and only if it is (strongly) subovert/subcondensed/subcompact in itself.
      
      We prove that condenseness has expected properties.
      
      \begin{proposition}
         A set $X$ is condensed if and only if for every set $Y$ the projection $p\colon X \times Y \to Y$ is a closed map.
      \end{proposition}
      \begin{proof}
         The proof is completely analogous to the proof of Proposition~\ref{Proposition: overt_iff_projections_open}. Suppose $X$ is condensed, and let $F \subseteq X \times Y$ be a closed subset of the product. Then for any $y \in Y$
         $$y \in p(F) \iff \xsome{x}{X}{(x, y) \in F},$$
         so $p(F)$ is closed in $Y$.
         
         Conversely, let the projection $p\colon X \times \cltp(X) \to \cltp(X)$ be a closed map. Define
         $$A \dfeq \st{(x, F) \in X \times \cltp(X)}{x \in F}$$
         which is evidently a closed subset of $X \times \cltp(X)$. But then $p(A)$ must be closed in $\cltp(X)$, and we have
         $$F \in p(A) \iff \xsome{x}{X}{x \in F}.$$
         So the condition on the right must be closed.
      \end{proof}
      The proof of this proposition is rather tautological with condenseness instead of compactness. In fact, one can actually see the equivalence: just imagine a closed subset of the product $X \times Y$ as an $X$-indexed union of slices parallel to $Y$ (and consider their union if you identify each $\{x\} \times Y$ with $Y$).
      
      \begin{proposition}
         Let $X$ be a set, $C \subseteq X$ subcondensed in $X$, and $A \subseteq X$ a closed subset of $X$. Then $C \cap A$ is a subcondensed subset of $X$. (In particular, a closed subset of a condensed set is subcondensed).
      \end{proposition}
      \begin{proof}
         For any $F \in \cltp(X)$ we have
         $$\xsome{x}{C \cap A}{x \in F} \iff \xsome{x}{C}{x \in A \cap F}.$$
         For the second part of the proposition just take $C = X$.
      \end{proof}
      
      Recall that it makes sense in synthetic topology to define that a space $X$ is Hausdorff when its diagonal $\diag[X] = \st{(x, y) \in X \times X}{x = y}$ is a closed subset of the product $X \times X$, or equivalently, that equality is a closed predicate. Because closed sets are classified by $\cld$, we obtain, similarly as in the case of discreteness, that a set is Hausdorff if and only if its singletons are closed since $(x, y) \in \diag[X] \iff x = y \iff x \in \{y\}$.
      \begin{proposition}
         Let $X$ be a Hausdorff set. If $C \subseteq X$ is subcondensed in $X$, it is closed in $X$.
      \end{proposition}
      \begin{proof}
         Because $x \in C \iff \xsome{c}{C}{(x, c) \in \diag[X]}$.
      \end{proof}
      
      The analogous propositions for open sets of course still hold. A set is overt if and only if every projection along it is an open map. The intersection of a subovert and an open subset is subovert. A subovert subset of a discrete set is open.
      
      We repeated these notions for sets in order to add condenseness to the mix, and the next step would be to repeat Proposition~\ref{Proposition: images_and_finite_products_of_subovert/subcompact_subsets}, but we will not do so directly. From the original overtness and compactness we already got ten separate notions, and there would have been twelve if we bothered to name the sets such that intersections of closed sets, indexed by them, are closed (but we didn't since we won't need this notion in the thesis). We can cut them by half by taking on the categorical mindset, and argue that overtness/condenseness/compactness should be the properties of maps rather than objects.
      \begin{definition}
      	A map $f\colon X \to Y$ is
      	\begin{itemize}
      		\item
      			\df{overt} when the statement $\xsome{x}{X}{f(x) \in U}$ is open for every $U \in \optp(Y)$,
      		\item
      			\df{condensed} when the statement $\xsome{x}{X}{f(x) \in F}$ is closed under every $F \in \cltp(Y)$,
      		\item
      			\df{compact} when the statement $\xall{x}{X}{f(x) \in U}$ is open for every $U \in \optp(Y)$.
      	\end{itemize}
      \end{definition}
		Observe:
		\begin{itemize}
			\item
				a map $f\colon X \to Y$ is overt/condensed/compact if and only if its image is subovert/sub\-con\-densed/subcompact in $Y$,
			\item
				a set $X$ is overt/condensed/compact if and only if the identity map $\id[X]$ is an overt/condensed/ compact map,
			\item
				a subset $A \subseteq X$ is subovert/subcondensed/subcompact in $X$ if and only if the inclusion $A \hookrightarrow X$ is an overt/condensed/compact map.
		\end{itemize}
		
		Propositions~\ref{Proposition: images_and_finite_products_of_subovert/subcompact_subsets} and~\ref{Proposition: intersection_of_subovert_and_open_is_open_(original)}, item by item, generalize thusly.
		\begin{proposition}\label{Proposition: constructions_of_overt_maps}
			\
			\begin{enumerate}
				\item
					Composing an overt/condensed/compact map with any map on the left yields an overt/con\-densed/compact map.\footnote{In categorical terms, overt/condensed/compact maps form a cosieve.}
				\item
					Finite products of overt/condensed/compact maps are overt/condensed/compact.
				\item
					For any overt/condensed map $f\colon X \to Y$, and any map $g\colon A \to Y$ with an open/closed image in $Y$, the diagonal pullback map
					$$\xymatrix@+1em{
						P \ar[r] \ar[d] \ar@{-->}[rd]  &  X \ar[d]^f  \\
						A \ar[r]_g  &  Y
					}$$
					is an overt/condensed map. Also, in the case $\tst = \opn$, if $f$ is compact and $g$ has a closed image, the diagonal pullback map is compact.
			\end{enumerate}
		\end{proposition}
		\begin{proof}
			\begin{enumerate}
				\item
					Let $f\colon X \to Y$, $g\colon Y \to Z$, and $A \subseteq Z$. Then
					$$\some{x}{X}{(g \circ f)(x) \in A} \iff \some{x}{X}{f(x) \in g^{-1}(A)},$$
					and likewise
					$$\all{x}{X}{(g \circ f)(x) \in A} \iff \all{x}{X}{f(x) \in g^{-1}(A)}.$$
				\item
					Let $f_i\colon X_i \to Y_i$ be maps for $i \in \NN_{< n}$. The image of $\prod_{i \in \NN_{< n}} f_i$ is the (finite) product of images of $f_i$s, thus subovert/subcondensed/subcompact by Proposition~\ref{Proposition: images_and_finite_products_of_subovert/subcompact_subsets}.
				\item
					We deal with overtness and condenseness by noting that for $B \subseteq Y$
					$$\some{(a, x)}{A \times X}{g(a) = f(x) \land f(x) \in B} \iff \some{x}{X}{f(x) \in B \cap g(A)}.$$
					Now take $f$ compact, $g$ with a closed image, and $U \in \tp(Y)$. We have
					$$\all{(a, x)}{A \times X}{g(a) = f(x) \implies f(x) \in U} \iff \all{x}{X}{f(x) \in g(A) \setimpl[Y] U},$$
					and under the assumption $\tst = \opn$, the set $g(A) \setimpl[Y] U$ is open in $Y$, so the whole statement is open.
			\end{enumerate}
		\end{proof}
		
		In practice we often need a combination of all these items.
		\begin{corollary}\label{Corollary: overt_combination}
			Let $n \in \NN$ be a natural number, $\big(f_k\colon X_k \to Y_k\big)_{k \in \NN_{< n}}$ a finite sequence of overt maps, $g\colon \prod_{k \in \NN_{< n}} Y_k \to Z$ any map, and $U \in \tp\Big(\prod_{k \in \NN_{< k}} X_k\Big)$ such that there exists $V \in \tp\Big(\prod_{k \in \NN_{< k}} Y_k\Big)$ for which $U = \Big(\prod_{k \in \NN_{< k}} f_k\Big)^{-1}(V)$. Define $h\colon U \to Z$ by
			$$h \dfeq g \circ \rstr{\Big(\prod_{k \in \NN_{< k}} f_k\Big)}_U.$$
			Then $h$ is an overt map.
		\end{corollary}
		\begin{proof}
			By the previous proposition since
			$$\xymatrix@+4em{
				U \ar@{^(->}[r] \ar[d]_{\rstr{\Big(\prod_{k \in \NN_{< k}} f_k\Big)}_U^V} \ar@{-->}[rd]|{\rstr{\Big(\prod_{k \in \NN_{< k}} f_k\Big)}_U}  &  \prod_{k \in \NN_{< k}} X_k \ar[d]^{\prod_{k \in \NN_{< k}} f_k}  \\
				V \ar@{^(->}[r]  &  \prod_{k \in \NN_{< k}} Y_k
			}$$
			is a pullback square.
		\end{proof}
      
      \intermission
      
      It still remains to reexamine the notion of dominance. But for that we first need to look at subspaces. As with most other notions, we redefine them with test subsets.
      \begin{definition}
         \
         \begin{itemize}
            \item
               A subset $A \subseteq X$ is a \df{subspace} of $X$ when for every test set $B \in \tetp(A)$ there exists a test set $C \in \tetp(X)$ such that $B = A \cap C$.
            \item
               An open subset $U \subseteq X$ which is a subspace, is called \df{strongly open}.
            \item
               A closed subset $F \subseteq X$ which is a subspace, is called \df{strongly closed}.
         \end{itemize}
      \end{definition}
      Actually, for open and closed subspaces we can extend test subsets in a canonical way.
      \begin{lemma}\label{lemma:opcl_subspaces}
         \
         \begin{enumerate}
            \item\label{lemma:opcl_subspaces_open} Let $A \subseteq U \subseteq X$. If $U$ is strongly open in $X$ and $A$ is test in $U$, then $A$ ($= U \cap A$) is test in $X$.
            \item\label{lemma:opcl_subspaces_closed} Let $A \subseteq F \subseteq X$. If $F$ is strongly closed in $X$ and $A$ is test in $F$, then $F \setimpl[X] A$ is test in $X$.
         \end{enumerate}
      \end{lemma}
      \begin{proof}
         Recall Lemma~\ref{Lemma: open_closed_subsets_characterization}.
         \begin{enumerate}
            \item
               Because $U$ is a subspace, there exists $B \in \tetp(X)$ such that $A = U \cap B$. But then $A$ is test in $X$.
            \item
               Similarly, we obtain $B \in \tetp(X)$ such that $A = F \cap B$. We then have
               $$F \setimpl[X] A \sepeq F \setimpl[X] F \cap B \sepeq F \setimpl[X] B,$$
               so $F \setimpl[X] A$ is test in $X$.
         \end{enumerate}
      \end{proof}
      
      Strongly open and strongly closed subsets also have a nice inheritance property.
      \begin{proposition}\label{Proposition: open/closed_subsets_of_strongly_open/closed_are_open/closed}
         Open/closed subsets of a strongly open/strongly closed subset are open/closed.
      \end{proposition}
      \begin{proof}
         Use Lemma~\ref{Lemma: open_closed_subsets_characterization}. Let $V \subseteq U \subseteq X$ where $V$ is open in $U$ and $U$ is strongly open in $X$. Take any $A \in \tetp(X)$. Then $U \cap A \in \tetp(U)$, so $V \cap U \cap A = V \cap A \in \tetp(U)$. By Lemma~\ref{lemma:opcl_subspaces}(\ref{lemma:opcl_subspaces_open}) $V \cap A \in \tetp(X)$, so $V$ is open in $X$.
         
         Do similarly for closed subsets. Let $G \subseteq F \subseteq X$ where $G$ is closed in $F$ which is strongly closed in $X$. For any $A \in \tetp(X)$ we have
         $$G \setimpl[X] A \sepeq (F \cap G) \setimpl[X] A \sepeq F \setimpl[X] (G \setimpl[X] A) \sepeq$$
         $$\sepeq F \setimpl[X] \big(G \setimpl[X] A \cap F\big) \sepeq F \setimpl[X] \big(G \setimpl[F] A \cap F\big).$$
         Since $A \cap F \in \tetp(F)$ and $G$ is closed in $F$, we have $(G \setimpl[F] A \cap F) \in \tetp(F)$. By Lemma~\ref{lemma:opcl_subspaces}(\ref{lemma:opcl_subspaces_closed}), $\big(F \setimpl[X] (G \setimpl[F] A \cap F)\big) \in \tetp(X)$, so $G$ is closed in $X$.
      \end{proof}
      
      \begin{corollary}\label{Corollary: strongly open/closed in overt/condensed/compact}
      	A strongly open subset of an overt set is overt. A strongly closed subset of a condensed set is condensed. If $\tst = \opn$, then a strongly closed subset of a compact set is compact.
      \end{corollary}
      \begin{proof}
      	For $U \subseteq V \subseteq X$ we have $\xsome{x}{V}{x \in U} \iff \xsome{x}{X}{x \in U}$, and by the previous lemma, if $V$ is strongly open in $X$ and $U$ open in $V$, then $U$ is also open in $X$. Same for condenseness and closed sets.
      	
      	Assume $\tst = \opn$, let $X$ be compact, $F \subseteq X$ closed in $X$, and $U \subseteq F$ an open = test subset of $F$. Then $\xall{x}{F}{x \in U} \iff \all{x}{X}{x \in F \implies x \in U}$, and the latter statement is test, consequently open, by Lemma~\ref{lemma:opcl_subspaces}.
      \end{proof}
      
      Consider strongly open (resp.~closed) subsets of $\one$; call the corresponding truth values (via $\pst(\one) \ism \soc$) \df{strongly open} (resp.~\df{closed}). Denote the set of strongly open (resp.~closed) truth values by $\overline{\opn}$ (resp.~$\overline{\cld}$). Of course $\overline{\opn} \subseteq \opn$, $\overline{\cld} \subseteq \cld$.
      
      Here is an explicit characterization for strongly open and closed truth values.
      \begin{proposition}\label{Proposition: characterization_of_strongly_open/closed_truth_values}
      	A truth value $u \in \soc$ is strongly open if and only if the condition
      	\begin{equation}\label{DO}\tag{DO}
      		\all{p}{\soc}{(u \impl (p \in \tst)) \implies (u \land p) \in \tst}
      	\end{equation}
      	holds. A truth value $f \in \soc$ is strongly closed if and only if the condition
      	\begin{equation}\label{DC}\tag{DC}
      		\all{p}{\soc}{(f \impl (p \in \tst)) \implies (f \impl p) \in \tst}
      	\end{equation}
      	holds.
      \end{proposition}
      \begin{proof}
      	Suppose $u \in \overline{\opn}$. Take any $p \in \soc$, and assume $u \impl (p \in \tst)$. Viewing truth values as subsets of $\one$, notice that $u \land p$ is a test subset of $u$. This is because for any $x \in u$ we have $x = *$ and $u$ holds, so $p \in \tst \ism \tetp(\one)$, and therefore $x \in u \land p$, or equivalently $* \in p$, is a test statement. Now use Lemma~\ref{lemma:opcl_subspaces} to conclude that $u \land p$ is a test subset of $\one$, ie.~a test truth value.
      	
      	Conversely, suppose~(\ref{DO}) holds for $u$, and take any $t \in \tetp(u)$. The antecedent in~(\ref{DO}) is satisfied for $p = t$, and so $u \land t = t \in \tst$.
      	
      	Similarly for closed truth values.
      \end{proof}
      
      Notice that here we took a different strategy than before: we first defined the subsets with the property that interests us, and then the truth values. Actually, as Lemma~\ref{Lemma: open_closed_subsets_characterization} tells us, we could have done the same with open and closed subsets. But we do not have a choice in this case, as strongly open/closed subsets are not classified by $\overline{\opn}$ and $\overline{\cld}$ in general (see Proposition~\ref{Proposition: classification_of_strongly_open/closed} below). Still, we have one direction.
      \begin{lemma}\label{Lemma: strongly_open/closed_one-way_classified}
         Let $U \subseteq X$. If for every $x \in X$ the truth value of $x \in U$ is strongly open, then $U$ is a strongly open subset of $X$. Analogously for closedness.
      \end{lemma}
      \begin{proof}
         Of course $U$ is open. Take any $A \in \tetp(U)$ and any $x \in X$. The inclusion $A \subseteq U$ implies $x \in A \implies x \in U$ which when truth values are interpreted as subsets of $\one$ means $(x \in A) \subseteq (x \in U)$. Use Lemma~\ref{lemma:opcl_subspaces}(\ref{lemma:opcl_subspaces_open}) to conclude that $(x \in A) \in \tetp(\one)$, ie.~$x \in A$ is a test truth value. Since $x \in X$ was arbitrary, $A$ is a test subset of $X$ which proves the  first part of the lemma. Using Lemma~\ref{lemma:opcl_subspaces}(\ref{lemma:opcl_subspaces_closed}) we can write an analogous proof for strongly closed subsets.
      \end{proof}
      
      Consider what this all means for decidable sets and truth values.
      
      \begin{lemma}\label{Lemma: decidable_subsets_open/closed_subspaces}
      	If $\bot \in \tst$, decidable subsets are strongly open. If $\top \in \tst$, decidable subsets are strongly closed.
      \end{lemma}
      \begin{proof}
      	It follows easily from Proposition~\ref{Proposition: characterization_of_strongly_open/closed_truth_values} that $\bot \in \tst$ implies $\two \subseteq \overline{\opn}$, and $\top \in \tst$ implies $\two \subseteq \overline{\cld}$ (in fact, the converses also hold). Now use Lemma~\ref{Lemma: strongly_open/closed_one-way_classified}.
      \end{proof}
      
      \begin{corollary}\label{Corollary: decidable_subsets_of_overt/compact_sets}
      	If $\two \subseteq \tst = \opn$, then decidable subsets of overt/condensed/compact sets are overt/condensed/compact.
      \end{corollary}
      \begin{proof}
      	By Lemma~\ref{Lemma: decidable_subsets_open/closed_subspaces} and Corollary~\ref{Corollary: strongly open/closed in overt/condensed/compact}.
      \end{proof}
      
      We are now ready to give the redefinition of the dominance property.
      \begin{definition}
      	\
      	\begin{itemize}
				\item The set $\opn$ is a \df{dominance} for $\tst$ when~(\ref{DO}) holds for all $u \in \opn$, ie.~when the following condition holds for $\opn$ and $\tst$.
				$$\xall{u}{\opn}\all{p}{\soc}{(u \impl (p \in \tst)) \implies (u \land p) \in \tst}$$
				\item The set $\cld$ is a \df{codominance} for $\tst$ when~(\ref{DC}) holds for all $f \in \cld$, ie.~when the following condition holds for $\cld$ and $\tst$.
				$$\xall{f}{\cld}\all{p}{\soc}{(f \impl (p \in \tst)) \implies (f \impl p) \in \tst}$$
			\end{itemize}
      \end{definition}
      We did not require $\top$ to be in $\opn$ and $\cld$ as in the original definition since we get that for free. From Proposition~\ref{Proposition: characterization_of_strongly_open/closed_truth_values} it is clear that $\opn$ is a dominance for $\tst$ precisely when $\opn = \overline{\opn}$ (compare with Proposition~\ref{Proposition: dominance}), and $\cld$ is a codominance for $\tst$ when $\cld = \overline{\cld}$.
      
      As expected, the dominance axiom implies desirable properties of synthetic topology.
      \begin{proposition}\label{Proposition: consequences_of_(co)dominance}
      	If $\opn$ is a dominance for $\tst$, then:
      	\begin{itemize}
      		\item open subsets are subspaces (ie.~open and strongly open subsets match),
      		\item openness is a transitive property (in the sense that open subset of an open set is open),
      		\item open subsets of an overt set are overt.
      	\end{itemize}
      	Similarly, if $\cld$ is a codominance for $\tst$, then closed subsets are subspaces, closedness is transitive, and closed subsets of a condensed set are condensed.
      \end{proposition}
      \begin{proof}
      	Suppose $U \subseteq X$ is an open subset of $X$. That means that for every $x \in X$, the truth value of $x \in U$ is in $\opn$, hence in $\overline{\opn}$. By Lemma~\ref{Lemma: strongly_open/closed_one-way_classified}, $U$ is strongly closed in $X$. The rest now follows from Proposition~\ref{Proposition: open/closed_subsets_of_strongly_open/closed_are_open/closed} and Corollary~\ref{Corollary: strongly open/closed in overt/condensed/compact}. The proof is the same for $\cld$.
      \end{proof}
      
      To bring the redefinitions full circle, we return to the discussion of \Esc closed subsets.
      \begin{proposition}\label{Proposition: classification_of_strongly_open/closed}
      	The following statements are equivalent.
      	\begin{enumerate}
      		\item
      			$\opn$ is a dominance.
      		\item
      			Strongly open subsets are classified by $\overline{\opn}$.
      		\item
      			Strongly open subsets are classified by a subset of $\soc$.
      		\item
      			Maps are continuous with regard to strongly open subsets.
      	\end{enumerate}
      	Analogously, the following statements are equivalent as well.
      	\begin{enumerate}
      		\item
      			$\cld$ is a codominance.
      		\item
      			Strongly closed subsets are classified by $\overline{\cld}$.
      		\item
      			Strongly closed subsets are classified by a subset of $\soc$.
      		\item
      			Maps are continuous with regard to strongly closed subsets.
      	\end{enumerate}
      \end{proposition}
      \begin{proof}
      	\begin{itemize}
      		\item\proven{$(1 \impl 2)$}
      			If $\opn$ is a dominance, strongly open and open subsets match, so strongly open subsets are classified by $\opn = \overline{\opn}$.
      		\item\proven{$(2 \impl 3)$}
      			Of course.
      		\item\proven{$(3 \impl 4)$}
      			Maps are continuous with regard to any family of classified subsets.
      		\item\proven{$(4 \impl 1)$}
      			The subset $\{\top\} \subseteq \opn$ is open, being classified by $\id[\opn]$. As a singleton it is a retract --- and therefore a subspace --- of $\opn$, so strongly open. Any open subset is a preimage of $\{\top\}$ (via its characteristic map), and by the assumption of continuity must therefore be strongly open.
      	\end{itemize}
      	The equivalence for closed subsets is proved exactly the same way.
      \end{proof}
      
      Note that in the case $\opn = \tst$ \Esc closed and strongly closed subsets match. This proposition then proves the existence of difficulties we lined up in the beginning of the section, as $\cld$ is not always a codominance.

		\intermission
   
   	So far we have studied the general theory of open and closed sets. However, in practical examples we often know more about $\opn$ and $\cld$. We end the section by considering typical properties which additionally hold.
   	
   	We introduced test truth values and subsets in order to put open and closed sets on equal footing. However if we restrict our attention to open sets only, as was the case in the usual models of synthetic topology up till now, we do not really need $\tst$ since, as already remarked, any $\opn$ is generated in particular by itself, \ie $\tst = \opn$. Indeed, by replacing all instances of `test' with `open' in this section, one obtains the standard definitions in synthetic topology. Another nice thing which is implied by $\tst = \opn$ is that we no longer have to differentiate between strong overtness and overtness, or strong condenseness and compactness.
   	
   	The next common property is that $\opn$ is a bounded $\lor$-subsemilattice of $\soc$ (\ie $\bot \in \opn$, and $\opn$ is closed under binary $\lor$), hence a bounded sublattice. This can be rephrased that finite sets are overt.
   	
   	Another thing that (at first glance perhaps surprisingly) often holds in practical examples is that open truth values are $\lnot\lnot$-stable. (Closed truth values typically aren't though; we denote their stable part by $\nnst[\cld] \dfeq \cld \cap \nnst$.)
   	
   	These properties taken together have several useful consequences.
   	
   	\begin{theorem}\label{Theorem: standard_special_case_of_synthetic_topology}
   		Suppose that $\tst = \opn \subseteq \nnst$, and that finite sets are overt. Then:
   		\begin{enumerate}
   			\item
   				$\cld = \st{f \in \soc}{\lnot{f} \in \opn}$ (closed sets are precisely those of which the complement is open),
   			\item
   				$\opn = \lnot\cld$ (all open sets can be expressed as the complements of the closed ones),\footnote{The first and second item together imply that open and closed subsets determine each other.}
   			\item
   				$\lnot\lnot\cld = \lnot\opn = \nnst[\cld]$ (complements of open sets are stable closed sets),
   			\item
   				strong condenseness, condenseness and compactness agree; moreover, condenseness can be tested just on stable closed subsets (in the sense that a map $f\colon X \to Y$ is condensed if and only if $\xsome{x}{X}{f(x) \in F}$ is closed under all stable closed subsets $F \subseteq Y$),
   			\item
   				$\cld$ is a bounded sublattice of $\soc$ (finite sets are condensed),
   			\item\label{thm:prac:Post}
   				$\opn \cap \cld = \two = \opn \cap \nnst[\cld]$ (clopen subsets are precisely the decidable ones),
   			\item
   				$\opn$ is a dominance if and only if openness is a transitive relation (open subsets of open subsets are open),
   			\item
   				$\cld$ is a codominance if and only if closedness is a transitive relation (closed subsets of closed subsets are closed), and similarly, $\lnot\lnot$-stable closed subsets are subspaces when stable closedness is a transitive relation.
   		\end{enumerate}
   	\end{theorem}
   	\begin{proof}
   		\begin{enumerate}
   			\item
   				Since $\emptyset$ is overt, $\bot \in \opn = \tst$, and so for every $f \in \cld$ the truth value of $f \impl \bot$ ($= \lnot f$) is test, hence open. Conversely, take any $f \in \soc$ such that $\lnot{f} \in \opn$. We have to verify $f \in \cld$, \ie $f \impl u$ is open for every $u \in \opn$. Calculate
   				$$f \impl u \sepeq f \impl \lnot\lnot{u} \sepeq \lnot\lnot(f \impl u) \sepeq \lnot\lnot(\lnot{f} \lor u).$$
   				Since $\opn$ is a $\lor$-semilattice, $(\lnot{f} \lor u) \in \opn$, and so $\lnot\lnot(\lnot{f} \lor u) \in \opn$.
   			\item
   				Follows from the previous item since for any $u \in \opn$ we have $u = \lnot\lnot{u}$ which means $\lnot{u} \in \cld$.
   			\item
   				Follows from previous two items.
   			\item
   				We already know that strong condenseness implies condenseness, and if $\tst = \opn$, then strong condenseness and compactness agree. Now take any map $f\colon X \to Y$, ``condensed for stable closed subsets'', and a test (= open) subset $U \subseteq Y$. Then
   				$$\xall{x}{X}{f(x) \in U} \iff \xall{x}{X}{\lnot\lnot(f(x) \in U)} \iff \lnot\xsome{x}{X}{f(x) \in U^C}.$$
   				Use previous items. The set $U^C$ is stable closed, so the truth value of the existential quantifier is closed. Therefore its negation is open, so $f$ is strongly condensed.
   			\item
   				Clearly $\bot \in \cld$ since $\bot = \lnot\top$ and $\top \in \opn$. Now take any $f, g \in \cld$. Since $\lnot(f \lor g) = (\lnot{f} \land \lnot{g}) \in \opn$, we conclude by the first item that $(f \lor g) \in \cld$. Alternatively, just recall that finite sets are compact, and then use the previous item.
   			\item
   				Clearly $\opn \cap \cld = \opn \cap \nnst[\cld]$ since $\opn \subseteq \nnst$, so it remains to verify $\opn \cap \cld = \two$. We already know that $\two \subseteq \opn$, $\two \subseteq \cld$. Conversely, take any $p \in \opn \cap \cld$. Using previous items, we see that also $\lnot{p} \in \opn \cap \cld$, and then $(p \lor \lnot{p}) \in \opn \cap \cld \subseteq \soc_{\lnot\lnot}$. Therefore $p \lor \lnot{p} \sepeq \lnot\lnot(p \lor \lnot{p}) \sepeq \lnot(\lnot{p} \land \lnot\lnot{p}) = \lnot\bot = \top$. So $p$ is decidable.
   			\item
   				We already know that one direction holds by Proposition~\ref{Proposition: consequences_of_(co)dominance}. The other follows from definitions since $\tst = \opn$.
   			\item
   				One direction is entailed in Proposition~\ref{Proposition: consequences_of_(co)dominance}. Conversely, let $X$ be a set, $F \subseteq X$ a closed subset, and $U \in \tp(F)$ a test = open subset of $F$. Then $F \setminus U$ is a closed subset of $F$, thus by transitivity of closedness also closed in $X$. Consequently, $X \setminus (F \setminus U)$ is open in $X$. For arbitrary $x \in X$ we have
   				\begin{align*}
   					& x \in X \setminus (F \setminus U)\\
   					\iff\quad& \lnot\big(x \in F \land \lnot(x \in U)\big)\\
   					\iff\quad& \lnot\lnot\lnot\big(x \in F \land \lnot(x \in U)\big)\\
   					\iff\quad& \lnot\big(\lnot\lnot(x \in F) \land \lnot(x \in U)\big)\\
   					\iff\quad& \lnot\lnot\big(x \in F \implies x \in U\big)\\
   					\iff\quad& x \in F \implies \lnot\lnot(x \in U)\\
   					\iff\quad& x \in F \implies x \in U.
   				\end{align*}
   				We explicitly prove the last equivalence. Assume $x \in F \implies \lnot\lnot(x \in U)$ and $x \in F$. Then $\lnot\lnot(x \in U)$, but $x \in F$ also implies that $x \in U$ is an open, and therefore $\lnot\lnot$-stable, statement; thus $x \in U$.
   				
   				We proved $X \setminus (F \setminus U) = F \setimpl[X] U$, and since $X \setminus (F \setminus U)$ is open, so is $F \setimpl[X] U$. The proof for stable closed subsets is the same.
   		\end{enumerate}
   	\end{proof}
   	
   	All examples of models we consider in Chapter~\ref{Chapter: models} satisfy the conditions of this theorem. It shouldn't come as a surprise that finite unions of opens are open, nor that $\tst = \opn$, given that only open subsets were considered thus far. Going a bit philosophical, we try to give some intuition for the remaining condition $\opn \subseteq \nnst$.
   	
   	A possible intuition for statements which $\lnot\lnot$-hold is that ``they hold in principle'', but we do not necessarily have a witness or reason for that. Thus, $\lnot\lnot$-stable statements would be those which insofar they hold, we do not need to be given a witness for it, we can provide one ourselves. Open subsets can be thought of as those for which membership is observable, \ie if for $U \in \tp(X)$ and $x \in X$ we consider whether $x$ belongs to $U$, then in case it does, we will sooner or later figure this out.\footnote{To illustrate this intuition, suppose you measure a physical quantity which is a real number. You can only measure it with certain precision, inexactly, meaning that you can only ever tell that it belongs to some interval. Classically, a subset $U \subseteq \RR$ is open precisely when for every $x \in U$ there is some interval around $x$ which is contained in $U$ --- which means that for any point in $U$, if we measure it with sufficient precision, we can confirm that it is in $U$ (but we couldn't do that for a point in the boundary of the subset).} So as soon as $x \in U$ ``in principle'', we can eventually confirm it.\footnote{When I discussed this with my advisor Andrej Bauer, he put it as follows. If an element is in the open set, we can make an observation that it is so. But we do not need another observation that we observed it; we simply know.}

	\section{(Co)limits and Bases}\label{Section: (co)limits_and_bases}
		
		In this section we consider how, given the intrinsic topology of a family of sets, we characterize the topology of a set constructed from the given ones. Specifically, we want to express the topology of limits and colimits.
		
		Colimits are easy. Recall that for any cartesian closed category, in particular any topos $\topos$, and any object $A$ from $\topos$, the contravariant exponentiation functor $A^\insarg\colon \topos \to \topos$ maps colimits to limits (because of adjointness; see Section~\ref{Section: topoi}). We conclude that given a diagram $D$ in $\topos$ such that its colimit $\colimit{D}$ exists, its topology is expressed by limits
		$$\tetp(\colimit{D}) \ism \tst^{\colimit{D}} \ism \limit{\tst^D} \ism \limit{\tetp(D)},$$
		$$\optp(\colimit{D}) \ism \opn^{\colimit{D}} \ism \limit{\opn^D} \ism \limit{\optp(D)}, \qquad \cltp(\colimit{D}) \ism \cld^{\colimit{D}} \ism \limit{\cld^D} \ism \limit{\cltp(D)}.$$
		
		We consider this formula in the special case of coproducts and coequalizers (recall that every colimit can be constructed from these). In the remainder of the section we restrict our attention only to open sets, in order not to triple the formulae, and because we will later need the results just for open sets anyway.
		
		If $\{X_i\}_{i \in I}$ is an $I$-indexed family of sets, then $\tp(\coprod_{i \in I} X_i) \ism \prod_{i \in I} \tp(X_i)$ by the result above. In particular $\tp(X + Y) \ism \tp(X) \times \tp(Y)$. In words, open subsets in a coproduct are given by their restrictions on individual summands. The topology is the standard coproduct topology we are used to from classical topology.
		
		In a topos, every epimorphism is regular (\ie a coequalizer of some parallel pair), or in other words, every surjective map is a quotient map. By the result above, if $f\colon X \to Y$ is surjective, then $Y$ has the \df{quotient topology} as we know it from classical topology, meaning that for any $V \subseteq Y$ the subset $V$ is open in $Y$ if and only if $f^{-1}(V)$ is open in $X$. To formally verify this, let $P \dfeq \st{(a, b) \in X \times X}{f(a) = f(b)}$, and let $p_0, p_1\colon P \to X$ be the projections. Then
		$$\xymatrix@+1em{
			P \ppair{p_0}{p_1}  &  X \ar[r]^f  &  Y
		}$$
		is a coequalizer diagram, so
		$$
		\xymatrix@+1em{
			\opn^Y \ar[r]^{\opn^f}  &  \opn^X \ppair{\opn^{p_0}}{\opn^{p_1}}  &  \opn^P,
		} \quad \text{or equivalently}, \quad
		\xymatrix@+1em{
			\tp(Y) \ar[r]^{f^{-1}}  &  \tp(X) \ppair{p_0^{-1}}{p_1^{-1}}  &  \tp(P),
		}
		$$
		is an equalizer diagram. Thus
		$$\tp(Y) \ism \st{U \in \tp(X)}{p_0^{-1}(U) = p_1^{-1}(U)} =$$
		$$= \st{U \in \tp(X)}{\all{(a, b)}{P}{a \in U \iff b \in U}} = \st{U \in \tp(X)}{U = f^{-1}(f(U))}.$$
		The sets with the property $U = f^{-1}(f(U))$ are called \df{saturated} (with respect to $f$). The isomorphism between open subsets of $Y$ and saturated open subsets of $X$ is of course $V \mapsto f^{-1}(V)$. The isomorphism in the other direction is $U \mapsto f(U)$, so the image of a saturated open set is always open. On the logical level we see this by the equivalence $[x] \in f(U) \iff x \in U$. This works since if $U$ is saturated, it does not matter which $x$ we take as a representative of the equivalence class $[x]$.
		
		While colimits are simple, the limits are quite a bit more work. There is no general rule which would ensure the topology of limits to be expected and well behaved. Since limits can be expressed by products and equalizers, it is sufficient to consider just these two cases. All monomorphisms, \ie all injective maps, are equalizers. Given an injective map $e\colon A \to X$, the expected topology on $A$ is the subspace topology relative to $X$, but we already know that this does not happen in general. We will not explore this question further than we already did (in Sections~\ref{Section: to_synthetic_topology} and~\ref{Section: redefinition_of_synthetic_topology}). We thus focus on products, for start on binary products. Actually, for them we have the isomorphism $\tp(X \times Y) \ism \opn^{X \times Y} \ism (\opn^X)^Y \ism (\tp(X))^Y$, so we can interpret open subsets of the product $X \times Y$ as $Y$-indexed families of open subsets of $X$. However, we wish to express the topology of the product in terms of the topologies $\tp(X)$ and $\tp(Y)$, as is custom in classical topology. This leads us to the notion of basis, as classically the binary product topology is given by the basis of products of open sets.
		
		Classically a basis for the topology is a subset of the topology (the open subsets in the basis are called basic open subsets) such that every open subset is expressible as a union of basic ones. Our current task is to adapt this definition to constructive and synthetic topological setting.
		
		The first alteration is that basic open subsets are understood to be indexed by some set, rather than given as a subset of the topology, so a single basic open subset can potentially be indexed several times. This point of view is natural; for example, if the empty set is open, it appears as the binary product of open sets whenever not both sets are inhabited. Later in Chapter~\ref{Chapter: metrization_theorems} we consider collections of metric balls as bases for metric spaces, and balls with different centers and/or radii can coincide.
		
		Similarly, when considering the unions of basic subsets, we also allow noninjective indexing. Thus we come to the following definition.
		\begin{definition}
			A \df{weak basis} for a set $X$ is a map $b\colon \basis \to \tp(X)$ such that for every $U \in \tp(X)$ there exists a set $I$ and a map (an \df{indexing}) $a\colon I \to \basis$ such that $U = \bigcup \st{b(a(i))}{i \in I}$.
		\end{definition}
		
		Some remarks are in order. First of all, the notation $b(a(i))$ is somewhat unintuitive. In practice we consider the indexing $I \to \basis$ to be given as $i \mapsto B_i$, and then write the union as $U = \bigcup_{i \in I} B_i$. Second, just as classically, this condition is equivalent to saying that for every $U \in X$ and $x \in U$ there exists $B \in \basis$ such that $x \in b(B) \subseteq U$. If $U$ is given as a union of weakly basic sets, the existence of $B$ is immediate; conversely, write $U$ as a union of \emph{all} weakly basic subsets contained in it.
		
		The third remark refers to the reason why we call $b$ merely a \emph{weak} basis. Classically a set is open precisely when it is a union of basic sets. This is because arbitrary unions of opens are open. In synthetic topology we can satisfy this condition only in the trivial case $\opn = \soc$. In general it makes sense to consider only overtly indexed unions.
		\begin{definition}
			A \df{basis} for a set $X$ is a map $b\colon \basis \to \tp(X)$ such that for every $U \in \tp(X)$ there exists a set $I$ and an overt map (an \df{overt indexing}) $a\colon I \to \basis$ such that $U = \bigcup \st{b(a(i))}{i \in I}$.\footnote{This is a slight generalization of the definition of basis we provided in~\cite{Bauer_A_Lesnik_D_2010:_metric_spaces_in_synthetic_topology}.}
		\end{definition}
		The corollary is that now basic subsets determine the open ones: a subset of $X$ is open if and only if it is an overtly indexed union of basic open subsets (since overtly indexed unions of open sets are open).
		
		We will mostly use this notion of a basis, but we note that we can strengthen the condition further.
		\begin{definition}
			A \df{canonical basis} for a set $X$ is a map $b\colon \basis \to \tp(X)$ such that for every $U \in \tp(X)$ there exists a canonical choice of a set $I$ and an overt map (a \df{canonical overt indexing}) $a\colon I \to \basis$ such that $U = \bigcup \st{b(a(i))}{i \in I}$.
		\end{definition}
		It makes sense to use canonical bases when otherwise we would be forced to invoke the axiom of choice. We note that there is no reason to define a canonical weak basis, as in the case of weak basis we can always make a canonical choice of the indexing map, namely we take the union of all basic subsets in a given open set.
		
		A reader occupied with foundational issues may have noticed that our definition of (weak, canonical) basis is problematic due to existential quantification over all morphisms with a fixed codomain which is generally not a set. However, this problem can easily be fixed with an equivalent definition stating that $b\colon \basis \to \tp(X)$ is a basis when for every $U \in \tp(X)$ there exists a subovert subset $A \subseteq \basis$ such that $U = \bigcup b(A)$ (in the case of weak basis we drop the condition that $A$ is subovert, and in the case of canonical basis we require a map $\tp(X) \to \ov(\basis)$ which gives the desired $A$). Thus we do not trouble ourselves with this issue.
		
		\begin{definition}
			We say that $X \times Y$ has (\df{weak}, \df{canonical}) \df{product topology} when the map $P\colon \tp(X) \times \tp(Y) \to \tp(X \times Y)$, $P(U, V) \dfeq U \times V$, is a (weak, canonical) basis for $X \times Y$.
		\end{definition}
		
		It is not necessary that binary products have the product topology.\footnote{See Example~{5.1} in~\cite{Escardo_M_Lawson_JD_Simpson_A_2004:_comparing_cartesian_closed_categories_of_core_compactly_generated_spaces} for an example of a full subcategory of $\Top$ which is cartesian closed, and in which $\NN^{\NN^\NN} \times \NN^\NN$ has the topology strictly stronger than the usual product topology. Consequently, in a gros topos (see Section~\ref{Section: gros_topos_model}) over this category this product won't have the product topology either.} It is therefore useful to have some sufficient conditions to ensure that they do.
		
		Call a set $X$ \df{locally compact} when for every $U \in \tp(X)$ and every $x \in U$ there exists an open set $V \in \tp(X)$ and a subcompact set $K \in \cp(X)$ such that $x \in V \subseteq K \subseteq U$.
		
		\begin{proposition}\label{Proposition: local_compactness_implies_weak_product_topology}
			If $X$ is locally compact, then $X \times Y$ has weak product topology.
		\end{proposition}
		\begin{proof}
			Take any $U \subseteq X \times Y$ open in the product, and $(a, b) \in U$. Let $e\colon X \to X \times \{b\}$ be the isomorphism $e(x) \dfeq (x, b)$. Because $X$ is locally compact, there exist a subcompact set $K \subseteq X$ and an open set $V \subseteq X$ such that $a \in V \subseteq K \subseteq e^{-1}\big(U \cap (X \times \{b\})\big)$. Let $W \subseteq Y$ be $W \dfeq \st{y \in Y}{\xall{x}{K}{(x, y) \in U}}$ which is an open subset of $Y$. Clearly $(a, b) \in V \times W \subseteq U$.
		\end{proof}
		
		In specific examples it is often not difficult to improve this proposition to conclude that the product has the product topology (not just the weak one); for example, if the set of real numbers $\RR$ has the expected topology, \ie generated by open intervals, and is locally compact (which amounts to saying that the interval $\II$ is compact, at least when $\RR$ is Hausdorff), then it is easy to see that the product $\RR \times \RR$ has the product topology. Here is an attempt to generalize this.
		
		Call a set $X$ \df{strongly locally compact} when there is a map $b = (b_O, b_K)\colon \basis \to \tp(X) \times \cp(X)$ (\ie $b_O$ maps to open subsets, and $b_K$ maps to subcompact subsets) such that:
		\begin{itemize}
			\item
				$\basis$ is an overt set,
			\item
				$b_O(j) \subseteq b_K(j)$ for all $j \in \basis$,
			\item
				for every $U \in \tp(X)$ and every $x \in U$ there exists $j \in \basis$ such that $x \in b_O(j) \subseteq b_K(j) \subseteq U$.
		\end{itemize}
		An example are overt discrete sets (with $b\colon X \to \tp(X) \times \cp(X)$, $b(x) = (\{x\}, \{x\})$). Another example are locally compact metric spaces (for definitions, see Section~\ref{Section: metric_spaces}) in which an overt family of balls of the form $\ball{x_j}{r_j}$ generate the intrinsic topology while balls of the form $\bball{x_j}{r_j}$ are subcompact.
		
		\begin{proposition}
			\
			\begin{enumerate}
				\item
					Let $X$ be strongly locally compact, witnessed by the map $b = (b_O, b_K)\colon \basis \to \tp(X) \times \cp(X)$. Then $b_O$ is a canonical basis for $X$.
				\item
					Binary product of strongly locally compact sets is strongly locally compact, and has canonical product topology.
			\end{enumerate}
		\end{proposition}
		\begin{proof}
			\begin{enumerate}
				\item
					For any $U \in \tp(X)$ let $I \dfeq \st{i \in \basis}{b_K(i) \subseteq U}$. We have $U = \bigcup_{i \in I} b_O(i)$, and since $\basis$ is overt, and $I$ is an open --- therefore subovert --- subset of $\basis$, the indexing is overt.
				\item
					Let $X$ and $Y$ be strongly locally compact, witnessed by maps
					$$b' = (b'_O, b'_K)\colon \basis' \to \tp(X) \times \cp(X) \quad \text{and} \quad b'' = (b''_O, b''_K)\colon \basis'' \to \tp(Y) \times \cp(Y),$$
					respectively. Define $\basis \dfeq \basis' \times \basis''$, and $b = (b_O, b_K)\colon \basis \to \tp(X \times Y) \times \cp(X \times Y)$ by
					$$b_O(i, j) \dfeq b'_O(i) \times b''_O(j), \qquad b_K(i, j) \dfeq b'_K(i) \times b''_K(j).$$
					This works since the binary product of open (resp.~subcompact) subsets is an open (resp.~subcompact) subset.
					
					The set $\basis$ is overt, being the binary product of overt sets. Clearly $b_O(i, j) \subseteq b_K(i, j)$ for all $(i, j) \in \basis$. Now take $U \in \tp(X \times Y)$ and $(x, y) \in U$. By Proposition~\ref{Proposition: local_compactness_implies_weak_product_topology} there are $V \in \tp(X)$, $W \in \tp(Y)$ such that $(x, y) \in V \times W \subseteq U$. By strong local compactness we may find $i \in \basis'$ and $j \in \basis''$ such that $x \in b'_O(i) \subseteq b'_K(i) \subseteq V$ and $y \in b''_O(j) \subseteq b''_K(j) \subseteq W$, but then $(x, y) \in b_O(i, j) \subseteq b_K(i, j) \subseteq V \times W \subseteq U$.
					
					We proved that $b = (b_O, b_K)$ witnesses strong local compactness of $X \times Y$, and so by the previous item the map $b_O$ is a canonical basis. From its proof we gather that for any $U \in \tp(X \times Y)$ there exists a canonical choice of an open subset $I \subseteq \basis = \basis' \times \basis''$ such that the union, indexed by the composition of the inclusion $I \hookrightarrow \basis$ and the map $b_O$, equals $U$.
					$$\xymatrix@+2.5em{
						&  \tp(X) \times \tp(Y) \ar[rr]^P  &&  \tp(X \times Y)  \\
						I \ar@{^(->}[r]^{\text{open}}  &  \basis' \times \basis'' \ar[u]^{b'_O \times b''_O} \ar[rru]_{b_O}  &&
					}$$
					Since $I$ is open in overt $\basis$, it is subovert in $\basis$, so the composition of $I \hookrightarrow \basis$ and $b'_O \times b''_O$ is an overt map. We conclude that $X \times Y$ has canonical product topology.
			\end{enumerate}
		\end{proof}
		
		For general products we cannot take simply products of open subsets of factors as basic opens of the product --- we need not obtain an open subset; we know this already from classical topology. Recall that classically the basis for a product $\prod_{i \in I} X_i$ of topological spaces consists of all finite intersections of preimages $p_i^{-1}(U)$ where $p_i$ are projections, and $U$ an open subset of the $i$-th factor. In the spirit of synthetic topology and our definition of the basis where ``arbitrary union of opens'' was replaced by ``overtly indexed union of opens'', we replace ``finite intersection of opens'' by ``compactly indexed intersection of opens''.
		
		Thus we say that the product $\prod_{i \in I} X_i$ \df{has (weak, canonical) product topology} when the map
		$$P\colon \cp(I) \times \prod_{i \in I} \tp(X_i) \to \tp(\prod_{i \in I} X_i),$$
		defined by
		$$P\big(K, (U_i)_{i \in I}\big) \dfeq \bigcap_{i \in K} p_i^{-1}(U_i),$$
		is a (weak, canonical) basis for the product $\prod_{i \in I} X_i$. Here $P$ indeed maps into $\tp(\prod_{i \in I} X_i)$ since for any $x \in \prod_{i \in I} X_i$ we have
		$$x \in P\big(K, (U_i)_{i \in I}\big) \iff \xall{i}{K}{p_i(x) \in U_i}.$$
		Notice that since we assume $K$ to be only subcompact in $I$, we need to choose $U_i$s for all $i \in I$ even though the choice is irrelevant for $i \notin K$.
		
		\begin{proposition}\label{Proposition: intersection_of_preimages_by_projections_as_a_product}
			Let $X = \prod_{i \in I} X_i$ be an $I$-indexed product with projections $p_i\colon X \to X_i$. Suppose that at least one of the following holds: $I$ has decidable equality, or $T = \opn$ and $I$ is Hausdorff.
			\begin{enumerate}
				\item
					For every family $(U_i)_{i \in I} \in \prod_{i \in I} \tp(X_i)$, every subcompact $K \in \cp(I)$ and every $i \in I$ the set
					$$V_i \dfeq \st{x \in X_i}{i \in K \implies x \in U_i} = \st{x \in X_i}{\all{j}{K}{i = j \implies x \in U_i}}$$
					is open in $X_i$. Moreover, $\bigcap_{k \in K} p_k^{-1}(U_k) = \prod_{i \in I} V_i$.
				\item
					If $I$ is compact, then $X$ has (weak, canonical) product topology if and only if the map $P\colon \prod_{i \in I} \tp(X_i) \to \tp(\prod_{i \in I} X_i)$, defined by
					$$P\big((U_i)_{i \in I}\big) \dfeq \prod_{i \in I} U_i = \bigcap_{i \in I} p_i^{-1}(U_i),$$
					is a (weak, canonical) basis for $X$.
			\end{enumerate}
		\end{proposition}
		\begin{proof}
			\begin{enumerate}
				\item
					We claim that for an arbitrary $j \in I$ the statement $i = j \implies x \in U_i$ is open. This clearly holds if $T = \opn$ and $I$ is Hausdorff. Assume now that $I$ has decidable equality. If $i = j$, then our statement is equivalent to $x \in U_i$, and if $i \neq j$, then it is equivalent to $\top$. In either case it is open. Since $K$ is subcompact in $I$, the set $V_i$ is open in $X_i$.
					
					The proof of the equality $\bigcap_{k \in K} p_k^{-1}(U_k) = \prod_{i \in I} V_i$ is straightforward.
				\item
					One direction is obvious since a compact set is subcompact in itself. The other direction follows from the previous item since any $\bigcap_{k \in K} p_k^{-1}(U_k)$ can be written as $\prod_{i \in I} V_i$.
			\end{enumerate}
		\end{proof}
		
		This proposition tells us that when $I$ is compact we may (under a few other conditions) simplify the definition of product topology. In particular our definitions of product topology for binary and general products are consistent.
		
		When there are many (sub)compact sets, the product topology might differ from what we would expect from classical intuition. For example, if we view classical sets as discrete (classical) topological spaces, finite products have classical product topology while products of infinitely many factors with at least two points are not discrete. On the other hand, taking $\opn = \two$ in $\Set$ (or more generally $\opn = \soc$ in any topos), all sets are discrete, compact and have decidable equality, so all (necessarily compactly indexed) products have (canonical) product topology.
		
		We can also simplify the definition of product topology in the case of exponentials. For sets $X$, $Y$, a subset $U \subseteq X \times Y$ and an element $x \in X$ let $U_x$ denote the ``slice'' of $U$ at coordinate $x$,
		$$U_x \dfeq \st{y \in Y}{(x, y) \in U}.$$
		Clearly if $U$ is open in $X \times Y$, then $U_x$ is open in $Y$ for every $x \in X$.
		
		Say that $Y^X$ \df{has (weak, canonical) exponential topology} when the map
		$$P\colon \cp(X) \times \tp(X \times Y) \to \tp(Y^X),$$
		defined by
		$$P(K, U) \dfeq \bigcap_{x \in K} \st{f \in Y^X}{f(x) \in U_x},$$
		is a (weak, canonical) basis for the exponential $Y^X$. One should see this as the analogue of the classical compact-open topology where the basis of $\C(X, Y)$ is taken to be the finite intersections of sets of the form $\st{f \in \C(X, Y)}{f(K) \subseteq V}$ where $K \subseteq X$ is compact, and $V \in \tp(Y)$. We of course replace finite intersections with compactly indexed ones, but notice that $\st{f \in \C(X, Y)}{f(K) \subseteq V} = \bigcap_{x \in K} \st{f \in \C(X, Y)}{f(x) \in V}$, so we can say that $Y^X$ has exponential topology when compactly indexed intersections of sets of the form $\st{f \in Y^X}{f(x) \in V}$ comprise a basis. This matches our definition above after we recall $\tp(X \times Y) \ism \tp(Y)^X$.
		
		Notice that $Y^X$ has the exponential topology if and only if, when interpreted as an $X$-indexed product of factors $Y$, it has the product topology.

	\section{Predicative Setting}\label{Section: predicative_version}
		
		As already mentioned, we strive in this thesis to prove theorems with as few assumptions (such as choice principles) as possible. There is one exception though: we cling to the cosiness of topoi and higher-order logic. This preference, I feel, simplifies proofs significantly (and lets us avoid hardcore category theory and endless diagrams). We spend a section, however, to demonstrate that at least basic notions of synthetic topology can be developed in general setting. The reader should find the ideas in this section reminiscent of Rosolini's dominions~\cite{Rosolini_G_1986:_continuity_and_effectiveness_in_topoi}.
		
		\begin{definition}
			Let $\cat$ be a category. We say that a family\footnote{The term `family' is used loosely in this section, and should be suitably interpreted in each specific case. For example, if our background setting is the set theory and $\cat$ is a small category, then a `family' is a set in the external sense.} $\mathcal{A}$ of morphisms in $\cat$ is a \df{continuous family} when for every $f \in \mrph{\cat}$ and $h \in \mathcal{A}$ with the common codomain there exists (some) pullback of $h$ along $f$, and any such pullback is again in $\mathcal{A}$.
		\end{definition}
		As the name suggests, a continuous family should capture the notion that all morphisms are ``continuous with regard to morphisms in it''.
		
		Recall that morphisms $f\colon X \to Y$, $f'\colon X' \to Y'$ are called \df{isomorphic} when there exist isomorphisms $X' \to X$ and $Y' \to Y$ such that the diagram
		$$\xymatrix@+1em{
			X' \ar[r]^\ism \ar[d]_{f'}  &  X \ar[d]^f  \\
			Y' \ar[r]_\ism  &  Y
		}$$
		commutes. Since this is necessarily a pullback diagram, it is clear from the definition that for any morphism $f$ in a continuous family, morphisms isomorphic to it are also in the family.
		
		To introduce synthetic topology to a category $\cat$, fix a continuous family $\tetp$ of morphisms in $\cat$, and call its members \df{test morphisms}.
		
		In analogy with the definitions and results in Section~\ref{Section: redefinition_of_synthetic_topology}, open morphisms should be those which, intersected by test morphisms, yield test morphisms. It makes sense to define the \df{intersection} $f \cap g$ of morphisms $f\colon A \to X$, $g\colon B \to X$ as the diagonal map in the pullback diagram below.
		$$\xymatrix@+2em{
			A \cap B \pbcorner \ar[r] \ar[d] \ar[rd]^{f \cap g}  &  B \ar[d]^g  \\
			A \ar[r]_f  &  X
		}$$
		We denote the domain of $f \cap g$ by $A \cap B$ (but note that it depend on the morphisms $f$, $g$, not just on their domains). Actually, a more precise definition would define the intersection on the equivalence classes of morphisms where morphisms are equivalent when they are isomorphic and have the same codomain (in analogy how we define the intersection on subobjects, as opposed to individual monomorphisms in a category), as then we could talk about \emph{the} intersection (it would be unique). This is yet another technical detail we won't trouble ourselves with, and will work simply with morphisms, not their equivalence classes (a more diligent reader is free to add square brackets everywhere, but we note that it doesn't really matter, as the families of morphisms we work with are saturated with respect to the equivalence relation, and constructions we perform on them also respect it).
		
		Clearly the intersection is commutative. We claim that it is also associative. Let $f\colon A \to X$, $g\colon B \to X$, $h\colon C \to X$ be morphisms for which all relevant intersections exist. Consider the following diagram.
		$$
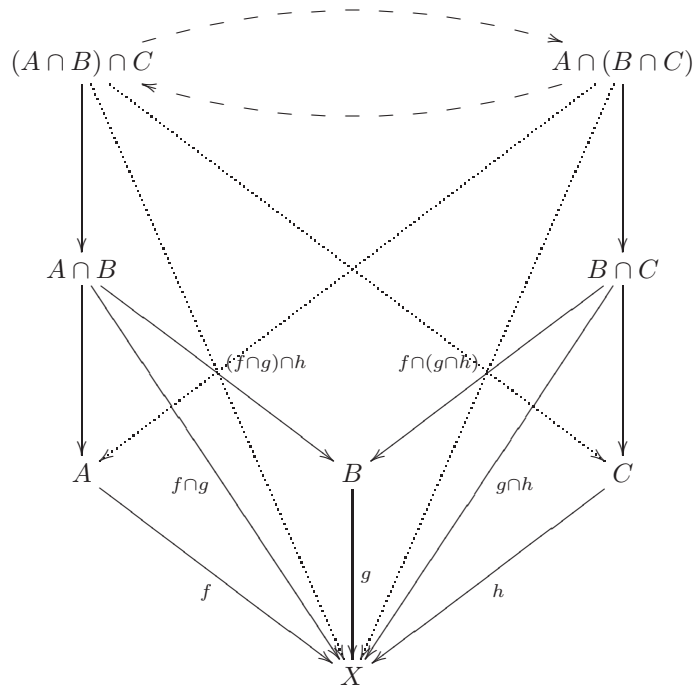
\xymatrix@+4em{
			(A \cap B) \cap C \ar[d] \ar@{..>}[rrdd] \ar@{..>}[rddd]^{(f \cap g) \cap h} \ar@{-->}@/^2em/[rr]  &&  A \cap (B \cap C) \ar[d] \ar@{..>}[lldd] \ar@{..>}[lddd]_{f \cap (g \cap h)} \ar@{-->}@/^2em/[ll]  \\
			A \cap B \ar[d] \ar[rd] \ar[rdd]_{f \cap g}  &&  B \cap C \ar[d] \ar[ld] \ar[ldd]^{g \cap h}  \\
			A \ar[rd]_f  &  B \ar[d]^g  &  C \ar[ld]^h  \\
			&  X  &
		}$$
		The universal property of pullbacks implies the existence of the morphisms between $(A \cap B) \cap C$ and $A \cap (B \cap C)$ in the upper row, and moreover, that their compositions are identities, proving our claim.
		
		Back to synthetic topology. Since $\tetp$ is a continuous family, there exists the intersection of any morphism with any test morphism. We call $u\colon U \to X$ an \df{open morphism} when for every test morphism $t \in \ms{X}$ the intersection $u \cap t$ is also a test morphism. We denote the family of open morphisms by $\optp$.
		
		Clearly identities (and more generally isomorphisms) are open. Because of associativity of intersections, binary intersections of open morphisms are open. Furthermore, if identities (and hence all isomorphisms) are test, then all open morphisms are test, and if intersections of test morphisms are test, then all test morphisms are open. Consequently any continuous family, closed under identities and binary intersections, is a family of open morphisms for some continuous family of test morphisms, namely for itself.
		
		The parallel with development in Section~\ref{Section: redefinition_of_synthetic_topology} stops a bit when we notice that in this generality of definition of synthetic topology open morphisms need not form a continuous family. While if a morphism is open, so is its any isomorphic morphism, we can provide an artificial example of a category in which not all pullbacks of open morphisms are open.
		$$
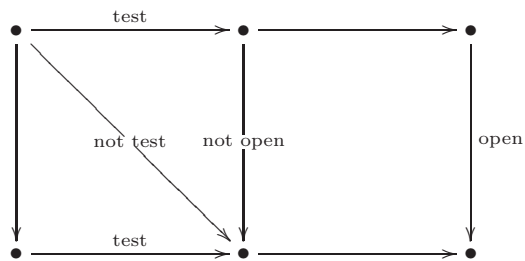
\xymatrix@+5em{
			\bullet \ar[r]^{\text{test}} \ar[d] \ar[rd]|{\text{not test}}  &  \bullet \ar[r] \ar[d]|{\text{not open}}  &  \bullet \ar[d]^{\text{open}}  \\
			\bullet \ar[r]^{\text{test}}  &  \bullet \ar[r]  &  \bullet
		}$$
		Consider the partial order in the diagram above, as a category. This category has all pullbacks (and they are unique), the object part of which are simply the infima in the partial order; moreover, the two squares are pullbacks.
		
		Let $\tetp$ be the smallest continuous family for which the morphisms declared test in the diagram belong to it. Then the rightmost vertical morphism is (trivially) open, but the middle one is not since the diagonal morphism in the left square is not test.
		
		Continuity with regard to opens followed in a topos because opens were classified by $\opn$. We thus define for a category with a terminal object and a chosen continuous family $\tetp$ that an object $\opn$, together with the \df{truth morphism} $\top\colon \one \to \opn$, is the \df{\Sier object} when open morphisms are classified by it. More precisely, we require:
		\begin{itemize}
			\item
				all pullbacks of the truth morphism exist, and they are open morphisms,
			\item
				given any open morphisms $u\colon U \to X$, there exists a unique morphism $\chi_u\colon X \to \opn$ such that the diagram
				$$\xymatrix{
					U \pbcorner \ar[r]^{\trm[U]} \ar[d]_u  &  \one \ar[d]^\top  \\
					X \ar[r]_{\chi_u}  &  \opn
				}$$
				is a pullback.
		\end{itemize}
		Since the truth morphism is also a pullback of itself, it must be an open morphism. It is easy to see that the existence of a \Sier object implies the existence of pullbacks of open morphisms along any morphisms, and that they must be open as well; in sum, open morphisms form a continuous family.
		
		A well-developed example of generalized synthetic topology is Paul Taylor's Abstract Stone Duality~(ASD)~\cite{Taylor_P_2009:_foundations_for_computable_topology}. A model of ASD is a category $\cat$ with finite products and a \Sier object $\opn$. Moreover, it is assumed that for any object $X$ the exponential $\opn^X$ exists (in other words, for each object we also want an object which represents its topology). The final assumption, the one which gives ASD its power, is that the adjunction $\opn^\insarg\colon \cat \to \cat\op \dashv \opn^\insarg\colon \cat\op \to \cat$ is monadic. For more on the subject the reader is referred to~\cite{Taylor_P_2009:_foundations_for_computable_topology, Taylor_P_2000:_geometric_and_higher_order_logic_in_terms_of_abstract_stone_duality, Bauer_A_Taylor_P_2009:_the_dedekind_reals_in_abstract_stone_duality}.
		
		It still remains to characterize closed morphisms. Recall that $A \setimpl[X] B$ is the largest subobject of $X$ which, intersected by $A$, yields $A \cap B$. This suggests what the definition of implication of morphisms $f\colon A \to X$ and $g\colon B \to X$ (for which their intersection exists) would be. Consider all composites $A \cap B \to S \to X$ which make the diagram
		$$\xymatrix@+2em{
			A \cap B \pbcorner \ar[r] \ar[d] \ar[rd]_{f \cap g}  &  S \ar[d]  \\
			A \ar[r]_f  &  X
		}$$
		a pullback. The \df{implication} $f \setimpl g\colon A \setimpl B \to X$ (we don't need to write $f \setimpl[X] g$, as $X$ is already encoded in morphisms $f$, $g$ as their codomain) is the second morphism in the terminal such composite, \ie the diagram
		$$\xymatrix@+2em{
			A \cap B \pbcorner \ar[r] \ar[d] \ar[rd]_{f \cap g}  &  A \setimpl B \ar[d]^{f \setimpldiagram g}  \\
			A \ar[r]_f  &  X
		}$$
		is a pullback, and for any diagram
		$$
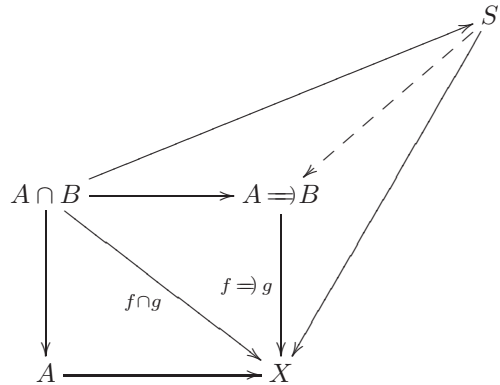
\xymatrix@+3em{
			&&  S \ar[ldd] \ar@{-->}[ld]  \\
			A \cap B \ar[r] \ar[d] \ar[rd]_{f \cap g} \ar[rru]  &  A \setimpl B \ar[d]_{f \setimpldiagram g}  &\\
			A \ar[r]  &  X  &
		}$$
		in which the outer quadrilateral is a pullback there exists a unique morphism $S \to A \setimpl B$ which makes the diagram commute.
		
		The family $\cltp$ of \df{closed morphisms} is then defined to contain all $f\colon F \to X$ such that for every test morphism $t \in \ms{X}$ the implication $f \setimpl t$ exists, and is also a test morphism.
		
		To define the notions of dominance and codominance, we still need to generalize the notion of subspaces. Call a morphism $h\colon X \to Y$ \df{hereditary}\footnote{In classical topology if $X \subseteq Y$ has the subspace topology, it is said that it \df{inherits} the topology from $Y$.} when every test morphism $f \in \ms{X}$ is a pullback of some test morphism $g \in \ms{Y}$.
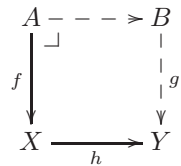
		$$\xymatrix@+1em{
			A \pbcorner \ar@{-->}[r] \ar[d]_f  &  B \ar@{-->}[d]^g  \\
			X \ar[r]_h  &  Y
		}$$
		
		A hereditary open (resp.~closed) morphism is called \df{strongly open} (resp.~\df{closed}). In analogy with results in Section~\ref{Section: redefinition_of_synthetic_topology} we can ``extend'' test morphisms into strongly open/closed morphisms in a canonical way. Suppose $h\colon X \to Y$ is strongly open, and $f\colon A \to X$ is test. The following diagram is a pullback.
		$$
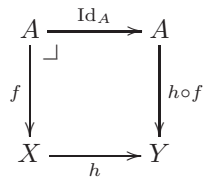
\xymatrix@+1em{
			A \pbcorner \ar[r]^{\id[A]} \ar[d]_f  &  A \ar[d]^{h \circ f}  \\
			X \ar[r]_h  &  Y
		}$$
		We claim that $h \circ f$ is test. We know that there exists a test morphism $g\colon B \to Y$ which makes the diagram
		$$
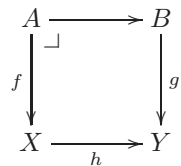
\xymatrix@+1em{
			A \pbcorner \ar[r] \ar[d]_f  &  B \ar[d]^g  \\
			X \ar[r]_h  &  Y
		}$$
		a pullback. This means $h \circ f \ism h \cap g$ which is a test morphism. Similarly, if $h\colon X \to Y$ is a strongly closed morphism and $f\colon A \to X$ a test morphism, then the diagram
		$$\xymatrix@+0em{
			A \pbcorner \ar@{ (->}[rrr] \ar[dd]_f  &&&  X \setimpl[Y] A \ar[dd]^{h \setimpldiagram (h \circ f)}  \\
			&&&\\
			X \ar[rrr]_h  &&&  Y
		}$$
		is a pullback, and the morphism on the right is test.
		
		We say that open morphisms form a \df{dominance} for $\tetp$ when all open morphisms are strongly open. Analogously, closed morphisms form a \df{codominance} for $\tetp$ when all closed morphisms are strongly closed. Note that in the case when test and open morphisms match, open morphisms form a dominance precisely when they are a subcategory (we know that identities are open, and openness of compositions of opens is equivalent to the dominance condition).
		
		\intermission
		
		We got a taste of how the basics of synthetic topology might look like in a general category. This can serve as a starting point for one who wishes to develop the subject, but chooses to avoid the topoi. Commonly regarded as the most stringent condition in a topos is the existence of powersets. Eschewing them leads one to \df{predicative} version of constructive mathematics~\cite{Sambin_G_Smith_JM_1998:_twentyfive_years_of_constructive_type_theory, Aczel_P_Rathjen_M_2001:_notes_on_constructive_set_theory, Myhill_J_1975:_constructive_set_theory, Aczel_P_1978:_typetheoretic_interpretation_constructive_set_theory}. Actually, predicative mathematicians still speak of power objects, but they are considered classes, \ie larger than sets. Consequently, when dealing with topological spaces, most topologies end up being proper classes~\cite{Giovanni_C_2006:_on_the_collection_of_points_of_a_formal_space}. Thus predicative topologists usually focus on topological spaces which at least have a basis which is a set, and have developed tools how to work with such bases efficiently~\cite{Coquand_T_Sambin_G_Smith_JM_Valentini_S_2003:_inductively_generated_formal_topologies}.
		
		Synthetic topology might provide an alternative way of how to do predicative topology. Usually the existence of (sufficiently many) exponentials is assumed, and if we have a \Sier set $\opn$, then $\opn^X$ is a \emph{set} (not a proper class) which represents the topology of a space $X$. The reason topology is suddenly a set is because we no longer assume that arbitrary unions of open subsets are open, but restrict this condition to overtly indexed unions. It might be interesting to develop a predicative version of topology with this restriction, so that no problems with topologies being proper classes occur.
		
		Our path however continues in quite the other direction; we return to topoi and full higher order logic, and even add further assumptions, most notably the existence of natural numbers object, in order to examine metric spaces.

   \chapter{Real Numbers and Metric Spaces}\label{Chapter: real_numbers_metric_spaces}

	Topological language is crucial especially in analysis. In this chapter and the next we consider the implications of synthetic topology for where the analysis starts --- real numbers and metric spaces. Therefore, we now assume the existence of the natural numbers object. This enables us to construct integers, rationals and reals.
	
	In order to connect real numbers and metric spaces with our theory of synthetic topology, the principal thing we want is the strict order $<$ to be an open relation on reals which is equivalent to open balls in metric spaces being intrinsically open.
	
	In the first section of this chapter we therefore reconstruct (the Dedekind version of) the real numbers to suit our purposes; we take only open Dedekind cuts. In the second section we define metric spaces and their generalizations, and consider their basic properties. The third section is devoted to completions of metric (and more general) spaces.
	
	For this theory to work we need to assume that countable sets are overt. The general reason for this is that open sets in metric spaces typically require an infinite number of balls to be covered by (but countably many is usually enough; we are mostly interested in separable metric spaces). However, we also need it already in the construction of the real numbers, to assure that $<$ is an open relation on Dedekind cuts.
	
	We therefore postulate in this chapter and the next:
	\postulate{$\emptyset$ and $\NN$ are overt;}
	recall from Lemma~\ref{Lemma: overtness_of_countable_sets} that this is equivalent to countable sets being overt. In particular, $\opn$ is a bounded lattice.

	\section{Real Numbers}\label{Section: real_numbers}
	
		Once we have the rational numbers, we can construct the reals as some sort of completion of rationals. As mentioned in Section~\ref{Section: intuitionism}, while there is little doubt what the set of real numbers is classically, constructively different versions of completions yield not necessarily isomorphic sets. Our immediate task is to find out which is the most suitable for our needs.
		
		Recall that the two most common constructions of real numbers are the Cauchy reals $\RR_c$ which is the set of all Cauchy sequences of rational numbers (with a given modulus of convergence) in which we identify Cauchy sequences which ``approach the same point'', and the Dedekind reals which is the set of all Dedekind cuts of rational numbers. As already mentioned however, in addition to a suitable completeness, in order to link real numbers and metric spaces to synthetic topology, we also want the open balls to be intrinsically open (which amounts to saying that $<$ is an open relation on $\RR$).
		
		At this point we shall purposefully complicate the situation somewhat; rather than outright saying that open Dedekind cuts would suit our needs we shall axiomatize what would make sense to call real numbers, notice that up to isomorphism we have a unique model, and verify that open Dedekind cuts form one.
		
		Much of our complication stems from the fact that we do not want $\RR$ to be built necessarily from $\QQ$. Intuitively, any dense subset of reals would serve, and its algebraic and order structure would extend to reals. Beside the mathematical hygiene (which theoretical mathematicians would appreciate) of identifying what assumptions we actually need, and assume only those, there are also at least two practical reasons for our approach. One is that the sets suitable for the construction of reals can also be used for the possible distances in metric spaces (which will serve us in Section~\ref{Section: metrization_via_embeddings}, but is useful for anyone who wants to speak about metric spaces before they have the reals --- for example to define reals as a metric completion). The other is that in practice we often want to have greater freedom in representing reals. For example, on a computer we might want to implement reals via dyadic rationals, not general rational numbers; 
 or perhaps we might choose to construct algebraic numbers first, and then derive real numbers from them.
		
		So we want to characterize the suitable dense subsets of real numbers, without already referring to reals. We start with the order structure.
		
		\begin{definition}
			A relation $<$ on a set $X$ is a \df{strict linear order} when the following holds for all $a, b \in A$.
			\begin{itemize}
				\item $\lnot(a < b \land b < a)$ \qquad (asymmetry)
				\item $a < b \implies \all{c}{A}{a < c \lor c < b}$ \qquad (cotransitivity)
				\item $\lnot(a < b) \land \lnot(b < a) \implies a = b$ \qquad (tightness)
			\end{itemize}
			Notice that asymmetry implies ireflexivity $\lnot(a < a)$, and together with cotransitivity implies transitivity $a < b \land b < c \implies a < c$. Define also $a \leq b \dfeq \lnot(b < a)$ and $a \apart b \dfeq a < b \lor b < a$. Then $\leq$ is a partial order on $A$ while $\apart$ is a tight apartness. Also, asymmetry can be restated $a < b \implies a \leq b$, and the stronger form of transitivity holds: $a < b \land b \leq c \implies a < c$ and $a \leq b \land b < c \implies a < c$.
		\end{definition}
		
		Number sets $\NN$, $\ZZ$, $\QQ$ (and their subsets) are all strictly ordered by the usual relation $<$. For Dedekind reals, let $a = (L_a, U_a)$ and $b = (L_b, U_b)$; then
		$$a < b \dfeq U_a \between L_b.$$
		Another example is the strict lexicographic order on $\NN^\NN$; for $\alpha, \beta \in \NN^\NN$ we define
		$$\alpha < \beta \dfeq \some{N}{\NN}{\all{n}{\NN_{< N}}{\alpha_n = \beta_n} \land \alpha_N < \beta_N}.$$
		
		Recall that for any partial order $(X, \leq)$, a supremum $\sup A$ of a subset $A \subseteq X$ is an element $s \in X$ with the property
		$$\all{x}{X}{s \leq x \iff \xall{a}{A}{a \leq x}}.$$
		The left-to-right implication states $s$ is an upper bound for $A$ (it is equivalent to $\xall{a}{A}{a \leq s}$), the other implication tells it is the least such. A supremum of a set need not exist, but if it does, antisymmetry of $\leq$ implies it is unique; moreover, we have $a \leq b \iff \sup\{a, b\} = b$. We define infimum $\inf A$ analogously.
		
		In the case of a strict order we have a strengthening of the notion of supremum and infimum. We define $s \in X$ to be the \df{strict supremum} of $A \subseteq X$ when
		$$\all{x}{X}{x < s \iff \xsome{a}{A}{x < a}}$$
		(and analogously the \df{strict infimum}). The left-to-right implication is again equivalent to $s$ being an upper bound, but the other one is a genuine strengthening of the previous condition, so a strict supremum is also a supremum, but not vice versa, unless we have classical logic. To see this, let $X = \{0, 1\}$ and $A = \{0\} \cup \st{1}{p}$ where $\lnot\lnot p$ holds. We claim $\sup A = 1$. Clearly, $1$ is an upper bound of $A$. Now let $x \in X$ be an upper bound for $A$. It cannot be $x = 0$ since that would imply $\lnot(1 \in A)$, \ie $\lnot\lnot\lnot p = \lnot p = \bot$. So $x = 1$. However, if $1$ is also the strict supremum of $A$, then $0 < 1$ implies $\xsome{a}{A}{0 < a}$, meaning $p$. We obtained $\lnot\lnot p \implies p$ for an arbitrary truth value $p$.
		
		We may calculate (strict) suprema and infima per parts.
		\begin{lemma}
			Let $\{A_i\}_{i \in I}$ be a family of subsets $A_i \subseteq X$, and $<$ a strict order on $X$. Then
			$$\sup \bigcup_{i \in I} A_i = \sup \st{\sup A_i}{i \in I}$$
			if all above suprema exist. Analogous formulae hold for infima and their strict versions.
		\end{lemma}
		\begin{proof}
			We prove the formula only in the case of strict suprema. Denote $s \dfeq \sup\st{\sup A_i}{i \in I}$, and assume the suprema in the claim are strict. Take any $a \in \bigcup_{i \in I} A_i$. Then there is an $i \in I$, such that $a \in A_i$, so $a \leq \sup A_i \leq s$. Now take any $x \in X$, $x < s$. That means there exists $i \in I$ for which $x < \sup A_i$, hence there is $a \in A_i$ for which $x < a$. This proves the claim.
		\end{proof}
		
		In particular, this means $\sup\{a, b, c\} = \sup\{\sup\{a, b\}, c\}$ and in general for $n \in \NN_{\geq 2}$,
		$$\sup\{a_0, a_1, \ldots, a_{n-1}\} = \sup\{\sup\{\ldots \sup\{a_0, a_1\}, a_2\}, \ldots, a_{n-1}\}.$$
		Moreover, $\sup\{a\} = a$, so typically when we check something for finite suprema (and infima and their strict versions), we need to verify the condition only for the supremum of the empty set (\ie the least element) and binary suprema.
		
		Actually, in the case of finite suprema and infima, there is no difference between the usual and the strict version.
		\begin{proposition}
			Let $<$ be a strict order on $X$, $n \in \NN$ and $a_0, \ldots, a_{n-1} \in X$. Suppose $s \dfeq \sup\{a_0, \ldots, a_{n-1}\}$ exists.
			\begin{enumerate}
				\item For $n \geq 1$, it is not true that for all $i \in \NN_{<n}$, $a_i < s$.
				\item The supremum $s$ is also strict.
			\end{enumerate}
			Similarly for infima.
		\end{proposition}
		\begin{proof}
			The claim of the proposition is obvious for $n = 0, 1$. It is sufficient to prove it for $n = 2$. Assume $a_0 < s$, $a_1 < s$. Suppose $a_1 < a_0$. Then also $a_1 \leq a_0$, so $s = a_0$, a contradiction to $a_0 < s$. So $a_0 \leq a_1$, but then $s = a_1$, a contradiction to $a_1 < s$.
			
			As for the second part, take any $x \in X$, $x < s$. By cotransitivity, $x < a_0 \lor a_0 < s$ and $x < a_1 \lor a_1 < s$. It cannot be both $a_0 < s$ and $a_1 < s$, so $x < a_0 \lor x < a_1$ which proves $s$ is strict.
		\end{proof}
		
		We call $(X, <)$ a lattice if it has binary (hence inhabited finite) suprema and infima (which are then automatically strict).
		
		\begin{lemma}\label{Lemma: decidability_of_(in)equalities}
			Consider the following statements for $(X, <)$.
			\begin{enumerate}
				\item $<$ is decidable.
				\item $\apart$ is decidable.
				\item $\leq$ is decidable.
				\item $=$ is decidable.
				\item $(X, <)$ is a lattice.
			\end{enumerate}
			Then
			\begin{itemize}
				\item $(1) \iff (2) \land (5)$,
				\item $(3) \iff (4) \land (5)$,
				\item $(1) \implies (3)$,
				\item $(2) \implies (4)$.
			\end{itemize}
			We do not have $(3) \implies (1)$, $(4) \implies (2)$ in general, unless classical logic is assumed.
		\end{lemma}
		\begin{proof}
			Simple. We supply only the following few details. Decidability of $\leq$ implies, for all $a, b \in X$, $a \leq b \lor b \leq a$ or equivalently $\sup\{a, b\} = b \lor \sup\{a, b\} = a$, making $X$ a lattice. On the other hand, if $X$ is a lattice, then for all $a, b \in X$, $a < b \iff \sup\{a, b\} \apart a$. As for the latter two implications not reversible, let $X = \{\bot, p\}$ where $p \in \soc$ such that $\lnot\lnot p$, and let $a < b \dfeq \lnot a \land b$. This is a strict order on $X$ which makes it a lattice, and we have $\bot \neq p$, $\bot \leq p$, $p \nleq \bot$. However, decidability of $<$ or of $\apart$ implies $p$.
		\end{proof}
		
		If we want to define the real numbers as a ``completion'' of certain strict orders, we need a property that will ensure ``density'' of the order.
		\begin{definition}
			A strict linear order $(X, <)$ is called \df{interpolating} when for every $a, b \in X$, $a < b$, there exists $x \in X$ such that $a < x < b$.
		\end{definition}
		
		We also need additional algebraic structure. There are several reasons for that. First, to define the reals, we still need the Archimedean property. Second, whatever the reals might be in a particular category, we will automatically know what algebraic structure they possess. Third, these strict orders will serve us as distances in metric spaces, and we want to be able to express all the conditions for the metric.
		
		\begin{definition}
			If $(X, <, +, 0)$ is a strictly linearly ordered set and a commutative monoid, \ie a semigroup for $+$ and with $0$ as the unit for addition, and the condition
			$$a < b \iff a + x < b + x$$
			holds for all elements $a, b, x \in X$, we call it a \df{strictly linearly ordered commutative monoid}.
		\end{definition}
		
		\begin{proposition}
			In a strictly linearly ordered commutative monoid $(X, <, +, 0)$, the following holds for all $a, b, c, d, x \in X$.
			\begin{itemize}
				\item $a < b \land c < d \implies a + c < b + d$
				\item $a > 0 \land b > 0 \implies a + b > 0$
				\item $a + x \leq b + x \iff a \leq b$
				\item $a < b \land c \leq d \implies a + c < b + d$
				\item $a \leq b \land c \leq d \implies a + c \leq b + d$
				\item $a \geq 0 \land b \geq 0 \implies a + b \geq 0$
				\item $a + x = b + x \iff a = b$ \quad (cancellation property)
			\end{itemize}
		\end{proposition}
		\begin{proof}
			Immediate from the definition.
		\end{proof}
		
		\begin{definition}
			Let $(X, <, +, 0, \cdot, 1)$ be a strictly linearly ordered commutative monoid with an additional operation (\df{multiplication}) on positive elements $\cdot\colon X_{> 0} \times X_{> 0} \to X_{> 0}$ such that $(X_{> 0}, \cdot, 1)$ is a commutative monoid, and which is distributive over addition, \ie $(a + b) \cdot c = a \cdot c + b \cdot c$ for all positive elements $a, b, c \in X_{> 0}$ (recall that $a + b$ is then also positive). If moreover the condition
			$$a < b \implies a \cdot x < b \cdot x$$
			holds for all $a, b, x \in X_{> 0}$, we call $(X, <, +, 0, \cdot)$ a \df{strictly linearly ordered unital commutative\footnote{The terms `unital' and `commutative' refer to the multiplication, as is the usual terminology for rings. For a general protoring we would require $X_{> 0}$ to be only a semigroup, and we would also have to write distributivity and preservation of $<$ in the case of multiplication from the left. Also, notice that a `unital' protoring has at least two elements since by definition $0 < 1$, so $0 \apart 1$.} protoring}.
		\end{definition}
		
		It quickly becomes tedious to write these long names, so whenever we say `ordered' or `strictly ordered', we mean strictly linearly ordered, and protorings will always be unital and commutative. Also, when referring to a set $X$ with additional order and algebraic structure, we will typically write simply $X$, omitting the symbols for relations and operations when they are understood.
		
		In a protoring we don't in general have additive inverses, though we can still define subtraction as a partial operation. If $X$ is a protoring and $a, b \in X$, we define $b-a$ to be $x \in X$ for which $a + x = b$ (if it exists). Since for any $a, b, x, y \in X$ we have $a + x = b \land a + y = b \implies x = y$, the difference $b-a$ is uniquely determined.
		
		For any $a \in X$ we can always calculate the differences $a-a = 0$ and $a-0 = a$. Moreover, given $a, b, c, d \in X$, if the differences $b-a$ and $d-c$ exist, then the difference $(b+d)-(a+c)$ exists also, namely $(b+d)-(a+c) = (b-a) + (d-c)$; to see this, simply calculate $a+c + ((b-a) + (d-c)) = a + (b-a) + c + (d-c) = b+d$.
		
		We will mix the multiplication within a protoring with another, external one. On any monoid $(X, +, 0)$ we have the action of natural numbers on $X$, inductively defined by $0 \cdot x \dfeq 0$ and $(n+1) \cdot x \dfeq n \cdot x + x$ for any $x \in X$, $n \in \NN$; intuitively, $n \cdot x$ is just $x$ summed $n$ times, in particular $1 \cdot x = x$. Notice also $n \cdot (a+b) = n \cdot a + n \cdot b$ for any $a, b \in X$, $n \in \NN$.
		
		\begin{lemma}\label{Lemma: multiplication_by_natural_numbers_in_a_strict_order}
			Let $X$ be a strictly linearly ordered unital commutative protoring. The following holds for all $a, b \in X$ and $n \in \NN_{> 0}$.
			\begin{enumerate}
				\item $a < b \iff n \cdot a < n \cdot b$
				\item If $a > 0$, then $n \cdot a > 0$.
				\item If $a, b > 0$, then $(n \cdot a) \cdot b = n \cdot (a \cdot b)$.
			\end{enumerate}
		\end{lemma}
		\begin{proof}
			All statements are proved by induction on $n \in \NN_{> 0}$.
			\begin{enumerate}
				\item The statement obviously holds for $n = 1$. For the `only if' direction, if $a < b$, then by induction hypothesis $n \cdot a < n \cdot b$, and then $(n+1) \cdot a = n \cdot a + a < n \cdot b + b = (n+1) \cdot b$. Conversely, assume $(n+1) \cdot a < (n+1) \cdot b$. By cotransitivity, $(n+1) \cdot a < n \cdot a + b$ or $n \cdot a + b < (n+1) \cdot b$. Rewriting the first case, we obtain $n \cdot a + a < n \cdot a + b$, so $a < b$. In the second case we have $n \cdot a + b < n \cdot b + b$; cancel $b$, then use the induction hypothesis.
				\item Immediate.
				\item Obvious for $n = 1$. For $n+1$ we calculate, using the induction hypothesis,
				$$((n+1) \cdot a) \cdot b = (n \cdot a + a) \cdot b = (n \cdot a) \cdot b + a \cdot b = n \cdot (a \cdot b) + (a \cdot b) = (n+1) \cdot (a \cdot b).$$
			\end{enumerate}
		\end{proof}
		
		We can now state the Archimedean property for a protoring $X$. However, since we don't have subtraction, we must state it in a slightly more complicated form than usual.
		\begin{definition}
			A strictly linearly ordered unital commutative protoring $(X, <, +, 0, \cdot, 1)$ is \df{Archi\-me\-dean} when the following \df{Archimedean property} holds:
			$$\all{a, b, c, d}{X}{c < d \implies \some{n}{\NN}{b + n \cdot c < a + n \cdot d}}.$$
		\end{definition}
		The usual Archimedean property is a special case: for any positive $a, b \in X_{> 0}$ there exists $n \in \NN$ such that $a < n \cdot b$, or equivalently, for any positive $a \in X_{> 0}$ there exists $n \in \NN$ such that $1 < n \cdot a$ and $a < n \cdot 1$ (clearly, such $n$ must itself be positive). In particular, in an Archimedean protoring we can, for any $a$, find $n \in \NN$ such that $a < n \cdot 1$ since we have $0 < a \lor a < 1$; in the latter case, take $n = 1$, and in the former, use the Archimedean property.
		
		Finally, we add topology to the mix. The name for the structure with all these properties is becoming cumbersome, so we introduce a new definition.
		\begin{definition}\label{Definition: streaks}
			A \df{streak} (more precisely, a \df{$\opn$-streak}) is a strictly linearly ordered unital commutative Archimedean protoring $X$ in which:
			\begin{itemize}
				\item for every $x \in X$ there exists $a \in X_{> 0}$, such that $x + a > 0$; in other words, every element is expressible as a difference of two positive elements, $x = b-a$ for some $a, b \in X_{> 0}$,
				\item $<$ is an open relation.
			\end{itemize}
			A map between streaks $f\colon X \to Y$ is a \df{streak morphism} when it preserves addition, multiplication and order, \ie[]
			\begin{itemize}
				\item $f(a + b) = f(a) + f(b)$,
				\item if $a, b > 0$, then $f(a \cdot b) = f(a) \cdot f(b)$,
				\item $a < b \implies f(a) < f(b)$,
			\end{itemize}
			for all $a, b \in X$.
		\end{definition}
		
		We were minimalistic in the definition of streak morphisms; they actually preserve all the structure of streaks.\footnote{They also preserve the strict order in the other direction, \ie $f(a) < f(b) \implies a < b$; see Proposition~\ref{Proposition: streak_morphisms} below.}
		\begin{proposition}\label{Proposition: streak_morphisms_basic}
			The following holds for any streak morphism $f\colon X \to Y$.
			\begin{enumerate}
				\item $f(0) = 0$
				\item $f(1) = 1$
				\item $a \in X_{> 0} \implies f(a) \in Y_{> 0}$ (In particular, when stating the preservation of multiplication in Definition~\ref{Definition: streaks}, the product $f(a) \cdot f(b)$ was well defined.)
				\item\label{Proposition(streak_morphisms_basic)item: difference_mapped_by_streak_morphism} For $a, b \in X$, if the difference $b-a$ exists, then so does $f(b)-f(a)$, and $f(b-a) = f(b)-f(a)$.
				\item The map $f$ is injective, even in the strong sense: $a \apart b \implies f(a) \apart f(b)$ for all $a, b \in X$.
			\end{enumerate}
		\end{proposition}
		\begin{proof}
			\begin{enumerate}
				\item We have $f(0) + f(0) = f(0 + 0) = f(0)$, and then $f(0) = 0$ by the cancellation property.
				\item $1 > 0$ implies $f(1) > f(0) = 0$. Hence, $f(1) < 1$ or $f(1) > 1$ would imply $f(1) \cdot f(1) < f(1)$ or $f(1) \cdot f(1) > f(1)$; however $f(1) \cdot f(1) = f(1 \cdot 1) = f(1)$. By tightness of the strict order, $f(1) = 1$.
				\item Since $a > 0$, we have $f(a) > f(0) = 0$.
				\item Let $b-a$ exist in $X$. Then $f(a) + f(b-a) = f(a + (b-a)) = f(b)$, so $f(b)-f(a) = f(b-a)$.
				\item By the definition of $\apart$, and because $f$ preserves $<$.
			\end{enumerate}
		\end{proof}
				
		Intuitively, streaks are just subprotorings of the reals, and morphisms of streaks are inclusions between them. The reals should then be the largest streak.
		
		\begin{definition}
			We define the set of real numbers $\RR$ (more precisely, $\RR_\opn$) to be the terminal streak, \ie such that for every streak $X$ there exists a unique streak morphism $X \to \RR$.
		\end{definition}
		
		The characterization of the real numbers (and later complete metric spaces) via terminality, and using special order and algebraic structure to construct them, was inspired by~\cite{Richman_F_2008:_real_numbers_and_other_completions}. The terminal streak need not exist in an arbitrary category, of course, but we can construct it under our assumptions, namely the topos with natural numbers in which countable sets are overt. We prove that the set of open Dedekind cuts on any overt interpolating ring streak (such as $\QQ$) is in fact terminal among streaks. We need a little preparation first, though.
		
		For any streak $X$ we define the relation $<$ between elements $x \in X$ and rationals $q \in \QQ$ as follows: if $q = \tfrac{b-a}{c}$ where $a, b, c \in \NN$, $c > 0$, then
		$$x < q \dfeq c \cdot x + a \cdot 1 < b \cdot 1 \qquad \text{and analogously} \qquad q < x \dfeq b \cdot 1 < c \cdot x + a \cdot 1.$$
		We claim this is well-defined --- independent of the choices for $a, b, c$. Let $q = \tfrac{b-a}{c} = \tfrac{b'-a'}{c'}$ where $a, b, c, a', b', c' \in \NN$, $c, c' > 0$. That means $c' b + c a' = c b' + c' a$, so
		$$c \cdot x + a \cdot 1 < b \cdot 1 \iff c' \cdot (c \cdot x + a \cdot 1) < c' \cdot (b \cdot 1) \iff c' c \cdot x + c' a \cdot 1 < c' b \cdot 1 \iff$$
		$$\iff c' c \cdot x + c' a \cdot 1 + c a' \cdot 1 < c' b \cdot 1 + c a' \cdot 1 \iff c' c \cdot x + c' a \cdot 1 + c a' \cdot 1 < (c' b + c a') \cdot 1 \iff$$
		$$\iff c c' \cdot x + c a' \cdot 1 + c' a \cdot 1 < (c b' + c' a) \cdot 1 \iff c c' \cdot x + c a' \cdot 1 + c' a \cdot 1 < c b' \cdot 1 + c' a \cdot 1 \iff$$
		$$\iff c c' \cdot x + c a' \cdot 1 < c b' \cdot 1 \iff c' \cdot x + a' \cdot 1 < b' \cdot 1.$$
		
		Notice also that for any $x \in X$ and $n \in \NN \subseteq \QQ$,
		$$x < n \iff x < n \cdot 1 \qquad \text{and} \qquad n < x \iff n \cdot 1 < x.$$
		Moreover, if $X = \QQ$, then this order matches the usual one on the rationals, so there will be no ambiguity when using the symbol $<$.
		
		This relation $<$ has similar properties as a strict order (not surprisingly since intuitively it is just the strict order of the reals, restricted to $X$ and $\QQ$).
		\begin{proposition}\label{Proposition: rational_streak_order_basic}
			Let $X$ be a streak. The following holds for all $x, y \in X$, $q, r \in \QQ$:
			\begin{itemize}
				\item $\lnot(x < q \land q < x)$,
				\item $x < q \implies x < y \lor y < q$, \hspace{0.5ex} $q < x \implies q < y \lor y < x$,
				\item $x < q \implies x < r \lor r < q$, \hspace{0.5ex} $q < x \implies q < r \lor r < x$,
				\item $q < r \land r < x \implies q < x$, \hspace{0.5ex} $q \leq r \land r < x \implies q < x$,
				\item $x < q \land q < r \implies x < r$, \hspace{0.5ex} $x < q \land q \leq r \implies x < r$.
			\end{itemize}
		\end{proposition}
		\begin{proof}
			Immediate from the definitions.
		\end{proof}
				
		\begin{lemma}\label{Lemma: streak_is_union_of_rational_intervals}
			Let $X$ be a streak. Then for any $b \in \ZZ_{> 0}$:
			\begin{enumerate}
				\item $\intoo[X]{0}{n} = \bigcup_{a \in \intoo[\NN]{0}{n b}} \intoo[X]{\tfrac{a-1}{b}}{\tfrac{a+1}{b}}$ for all $n \in \NN$ such that $n b \geq 2$,
				\item $X_{> 0} = \bigcup_{a \in \NN_{> 0}} \intoo[X]{\tfrac{a-1}{b}}{\tfrac{a+1}{b}}$,
				\item $X = \bigcup_{a \in \ZZ} \intoo[X]{\tfrac{a-1}{b}}{\tfrac{a+1}{b}}$.
			\end{enumerate}
		\end{lemma}
		\begin{proof}
			\begin{enumerate}
				\item
					For $x \in X$, $0 < x < n \cdot 1$, consider the disjunctions $\tfrac{m}{b} < x \lor x < \tfrac{m+1}{b}$ for $m \in \NN_{< n b}$. Write a finite sequence: in each of the disjunctions make a choice of a true disjunct, and write $0$ if the first one is chosen, and $1$ if the second one is. If the sequence contains no $1$s, let $a = n b - 1$; if the sequence contains no $0$s, let $a = 1$. Otherwise, let $a$ be that $m$ for which the first $1$ appears.
				\item
					By the previous item and the Archimedean property.
				\item
					Write $x = d-c$ where $c, d \in X_{> 0}$. By the previous item there are $a', a'' \in \NN_{> 0}$ such that $\tfrac{a'-1}{4 b} < c < \tfrac{a'+1}{4 b}$ and $\tfrac{a''-1}{4 b} < d < \tfrac{a''+1}{4 b}$. Let $a \in \ZZ$ be such that $a \leq \frac{a'' - a' + 2}{4} \leq a + 1$. Then
					$$(a'-1) \cdot 1 + 4 b \cdot x < 4 b \cdot c + 4 b \cdot x = 4 b \cdot d < (a''+1) \cdot 1,$$
					$$(a'+1) \cdot 1 + 4 b \cdot x > 4 b \cdot c + 4 b \cdot x = 4 b \cdot d > (a''-1) \cdot 1,$$
					so
					$$x < \frac{(a''+1) - (a'-1)}{4 b} = \frac{a'' - a' + 2}{4 b} \leq \frac{a+1}{b},$$
					$$x > \frac{(a''-1) - (a'+1)}{4 b} = \frac{a'' - a' - 2}{4 b} \geq \frac{a-1}{b}.$$
			\end{enumerate}
		\end{proof}
		
		\begin{lemma}\label{Lemma: order_comparison_between_streak_and_rational_elements}
			Let $X$ be a streak.
			\begin{enumerate}
				\item
					For an element $x \in X$ and a rational $q \in \QQ$ we have
					$$x < q \implies \xsome{r}{\QQ}{x < r < q} \qquad \text{and} \qquad q < x \implies \xsome{r}{\QQ}{q < r < x}.$$
				\item 
					For $x, y \in X$ we have
					$$x < y \iff \some{q}{\QQ}{x < q < y} \iff \some{q, r}{\QQ}{x < q < r < y}.$$
				\item
					Let $x, y \in X$. Then
					$$x = y \iff \all{q}{\QQ}{q < x \iff q < y} \iff \all{q}{\QQ}{x < q \iff y < q}.$$
			\end{enumerate}
		\end{lemma}
		\begin{proof}
			\begin{enumerate}
				\item
					Suppose $x < q$ and write $q = \tfrac{b-a}{c}$ where $a, b, c \in \NN$, $c > 0$. We then have $c \cdot x + a \cdot 1 < b \cdot 1$, so by the Archimedean property there exists $n \in \NN$ such that $c \cdot 1 + n \cdot (c \cdot x + a \cdot 1) < n b \cdot 1$. Take $r = q - \frac{1}{n} = \frac{n b - (n a + c)}{n c}$. The case $q < x$ is analogous.
				\item
					The implications $\Leftarrow$ are obvious, so it remains to prove $x < y \implies \some{q, r}{\QQ}{x < q < r < y}$. Assume $x < y$; then by the Archimedean property there is $n \in \NN$ (necessarily $>0$) with $1 + n \cdot x < n \cdot y$. By the previous lemma we have $a \in \ZZ$ such that $\tfrac{a-1}{3n} < x < \tfrac{a+1}{3n}$. We claim $x < \tfrac{a+1}{3n} < \tfrac{a+2}{3n} < y$. Only the last inequality needs proof. Write $a-1 = b - c$ where $b, c \in \NN$; then $a+2 = (b+3) - c$. We calculate
					$$(b+3) \cdot 1 = b \cdot 1 + 3 \cdot 1 < c \cdot 1 + 3n \cdot x + 3 \cdot 1 = c \cdot 1 + 3(n \cdot x + 1) < c \cdot 1 + 3n \cdot y.$$
				\item
					We prove only the first equivalence, the other one is analogous. The direction $\Rightarrow$ is obvious. Conversely, suppose $x \apart y$, \ie $x < y$ or $x > y$. In the first case we have $q \in \QQ$ such that $x < q < y$, a contradiction to the assumption and the asymmetry of $<$. We proceed similarly when $x > y$.
			\end{enumerate}
		\end{proof}
		
		\begin{proposition}\label{Proposition: streak_morphisms}
			Let $X$ and $Y$ be streaks.
			\begin{enumerate}
				\item
					For any streak morphism $f\colon X \to Y$, any $x, y \in X$ and $q \in \QQ$, we have
					$$x < q \iff f(x) < q, \qquad\qquad q < y \iff q < f(y),$$
					$$x < y \iff f(x) < f(y).$$
				\item
					There is at most one streak morphism $X \to Y$.
			\end{enumerate}
		\end{proposition}
		\begin{proof}
			\begin{enumerate}
				\item
					We can infer the implications in the direction $\Rightarrow$ immediately from the definitions, so we focus on the direction $\Leftarrow$. Suppose $f(x) < q$. By the previous lemma we have $r \in \QQ$ such that $f(x) < r < q$. Then $r < x \lor x < q$, but the first case leads to contradiction, so $x < q$. The second implication is similarly proved. As for the third, let $f(x) < f(y)$. By the previous lemma we have $q, r \in \QQ$ such that $f(x) < q < r < f(y)$. Then $q < x \lor x < \frac{q+r}{2}$ and $\frac{q+r}{2} < y \lor y < r$. The cases $q < x$ and $y < r$ lead to contradiction, so $x < \frac{q+r}{2} < y$, implying $x < y$.
				\item
					Let $f, g\colon X \to Y$ be streak morphisms, and take any $x \in X$. Using previous item we see that for all $q \in \QQ$, $q < f(x) \iff q < g(x)$ (as well as $f(x) < q \iff g(x) < q$), so $f(x) = g(x)$ by the previous lemma.
			\end{enumerate}
		\end{proof}
		
		So the category of streaks is actually a preorder, with $\RR$ as the last element.
		\begin{corollary}\label{Corollary: streak_terminality}
			Let $T$ be a streak.
			\begin{enumerate}
				\item If for every streak $X$ there is a streak morphism $X \to T$, then $T$ is a terminal streak.
				\item If $T'$ is a terminal streak and $T' \to T$ is a streak morphism, then it is an isomorphism, and $T$ is also terminal.
			\end{enumerate}
		\end{corollary}
		\begin{proof}
			\begin{enumerate}
				\item The existence of a streak morphism is assumed, and the uniqueness is supplied by the previous proposition.
				\item Either use the previous item and obtain $X \to T$ as the composition $X \to T' \to T$, or use category theory to verify that in a preorder category any morphism with a terminal object as a domain is an isomorphism.
			\end{enumerate}
		\end{proof}
		
		We make a short detour to fields, and then to lattices, as after Corollary~\ref{Corollary: streak_terminality}, we are in a position to prove that the real numbers, whatever they are, must be both. The stepping stone to fields, however, are the rings.
		
		\begin{definition}
			\
			\begin{itemize}
				\item
					We call the structure $(X, <, +, 0, -, \cdot, 1)$ a \df{strictly linearly ordered unital commutative ring} when $(X, <, +, 0)$ is a strictly linearly ordered commutative monoid, the subtraction $-$ is a total operation (\ie $(X, +, 0, -)$ is an abelian group), the multiplication $\cdot\colon X \times X \to X$ is a total commutative operation with unit $1$ and distributive over addition, and the condition
					$$a < b \iff a \cdot x < b \cdot x$$
					holds for all $a, b \in X$ and positive $x \in X_{> 0}$.
					
					We shorten the name of a strictly linearly ordered unital commutative ring which is also a streak to a \df{ring streak}.
				\item
					A strictly linearly ordered unital commutative ring $X$ is a \df{strictly linearly ordered field} when an element $a \in X$ has a multiplicative inverse $a^{-1}$ if and only if\footnote{Here the equivalence is assumed in order to exclude the trivial case $0 = 1$, \ie the trivial ring $\{0\}$, as that reduces the need for special cases in various theorems, similarly as we don't want to count $1$ as a prime number. In the case of streaks it makes no difference though since we have $0 \apart 1$. See Lemma~\ref{Lemma: simplify_conditions_for_field_streaks} below.} $a \apart 0$, and the condition $0 < a < b \implies 0 < b^{-1} < a^{-1}$ holds for all $a, b \in X$.
					
					We shorten the name of a strictly linearly ordered field which is also a streak to a \df{field streak}.
			\end{itemize}
		\end{definition}
		
		\begin{lemma}\label{Lemma: simplify_conditions_when_subtraction_is_total}
			Let $(X, <, +, 0, \cdot, 1)$ be a strictly linearly ordered unital commutative protoring with the following properties:
			\begin{itemize}
				\item $X$ is an abelian group for $+$ (\ie the subtraction is a total operation),
				\item $<$ is an open relation on $X$,
				\item for all $x \in X_{> 0}$ there is $n \in \NN$ such that $x < n \cdot 1$ and $1 < n \cdot x$.
			\end{itemize}
			Then $(X, <, +, 0, \cdot, 1)$ is a streak.
		\end{lemma}
		\begin{proof}
			We first prove the Archimedean property. Take any $a, b, c, d \in X$, $c < d$. We have $0 < b-a \lor b-a < 1$; in both cases we can find $n' \in \NN$ such that $b-a < n' \cdot 1$. Moreover, there is $n'' \in \NN$ such that $1 < n'' \cdot (d-c)$. Let $n \dfeq n' n'' \in \NN$. Then $b-a < n' \cdot 1 < n' n'' \cdot (d-c) = n \cdot (d-c)$, so $b + n \cdot c < a + n \cdot d$.
			
			To prove that $X$ is a streak, the only thing left is to write an arbitrary element as a difference of two positive ones. Take $x \in X$. We know that there is $n \in \NN_{> 0}$ such that $x < n \cdot 1$, but then $x = (n \cdot 1) - (n \cdot 1 - x)$.
		\end{proof}
				
		Recall that there is a canonical way to transform a commutative semigroup into a group, called the \df{Grothendieck construction}. The group consists of formal differences of elements of the semigroup. This makes the category of abelian groups a reflective subcategory in the category of abelian semigroups. We slightly adopt the Grothendieck construction for streaks (making formal differences of only positive elements), and verify that it turns a streak into a ring streak.
		
		Let $G(X) \dfeq (X_{> 0} \times X_{> 0})/_\equ$ where the equivalence relation $\equ$ is defined
		$$(a, b) \equ (c, d) \dfeq a + d = b + c$$
		for $a, b, c, d \in X_{> 0}$. As usual, the equivalence class $[(a, b)]$ represents the (formal) difference between $a$ and $b$. Thus, the order and the algebra on $G(X)$ is defined:
		$$[(a, b)] < [(c, d)] \dfeq a + d < b + c,$$
		$$[(a, b)] + [(c, d)] \dfeq [(a + c, b + d)], \qquad 0 = [(1, 1)], \qquad -[(a, b)] = [(b, a)],$$
		$$[(a, b)] \cdot [(c, d)] \dfeq [(a \cdot c + b \cdot d, a \cdot d + b \cdot c)], \qquad 1 = [(1+1, 1)].$$
			
		\begin{proposition}\label{Proposition: Grothendieck_construction_on_a_streak}
			For any streak $X$, its Grothendieck group $G(X)$ is again a streak, as well as a strictly ordered unital commutative ring, and the embedding $i\colon X \to G(X)$ is a streak morphism which is initial among streak morphisms from $X$ to ring streaks. The map $i$ is a streak isomorphism if and only if subtraction on $X$ is a total operation.
		\end{proposition}
		\begin{proof}
			It is easy to see that the order and the operation on $G(X)$ are well defined, and that they make $G(X)$ a strictly linearly ordered unital commutative ring. The relation $<$ is obviously open. By Lemma~\ref{Lemma: simplify_conditions_when_subtraction_is_total} it remains to verify that for all $a, b, c, d \in X_{> 0}$ such that $[(c, d)] > 0$, \ie $c > d$, there is $n \in \NN$ with the property $[(a, b)] < n \cdot [(c, d)]$. It is easy to see by induction that $n \cdot [(c, d)] = [(n \cdot c, n \cdot d)]$, so the condition becomes $a + n \cdot d < b + n \cdot c$, exactly the Archimedean property in $X$.
			
			Finally, define $i\colon X \to G(X)$ as follows. For any $x \in X$ write it as $x = b-a$ where $a, b \in X_{> 0}$, and define $i(x) \dfeq [(b, a)]$. It is easy (though somewhat lengthy) to see that it is well defined, and a streak morphism. To prove $i$ is initial, let $f\colon X \to Y$ be a streak morphism into a ring streak $Y$. Define the map $g\colon G(X) \to Y$ by $g([(a, b)]) \dfeq f(a) - f(b)$. Then $g$ is a well-defined streak morphism, we have $f = g \circ i$ by Proposition~\ref{Proposition: streak_morphisms_basic}(\ref{Proposition(streak_morphisms_basic)item: difference_mapped_by_streak_morphism}), and $g$ is unique by Proposition~\ref{Proposition: streak_morphisms}.
			
			If $i$ is an isomorphism, then subtraction on $X$ is total since it is total on $G(X)$. Conversely, if $X$ has total subtraction, define the map $G(X) \to X$, $[(a, b)] \mapsto a-b$. It is easy to see that this is a streak morphism, and that it is the inverse of $i$.
		\end{proof}
		
		\begin{corollary}
			If $(X, <, +, 0, \cdot, 1)$ is a streak in which subtraction is a total operation, then the multiplication $\cdot\colon X_{> 0} \times X_{> 0} \to X_{> 0}$ uniquely extends to the whole of $X$ so that $(X, <, +, 0, -, \cdot, 1)$ is a strictly linearly ordered unital commutative ring.
		\end{corollary}
		\begin{proof}
			By the previous proposition $i\colon X \to G(X)$ is a streak isomorphism via which the total multiplication on $G(X)$ can be transferred to $X$. By the properties of $i$ this is indeed the extension of the multiplication on $X_{> 0}$.
			
			Take now any $x, y \in X$, and write them $x = b - a$, $y = d - c$ where $a, b, c, d \in X_{> 0}$. If $\cdot$ is indeed an extension of the multiplication on positive elements, and satisfies the properties of the ring, then necessarily $x \cdot y = (b - a) \cdot (d - c) = (b \cdot d + a \cdot c) - (a \cdot d + b \cdot c)$. This proves uniqueness of the extension.
		\end{proof}
		Consequently, whenever we have a streak $X$ with total subtraction, we will consider the multiplication to be a total operation also, and thus the streak $X$ to be a ring streak.
		
		We saw that the conditions for $X$ to be a strictly linearly ordered unital commutative ring simplify in the case when $X$ is a streak. This also happens in the case of fields.
		\begin{lemma}\label{Lemma: simplify_conditions_for_field_streaks}
			Let $X$ be a ring streak in which the condition
			$$x \apart 0 \implies \some{y}{X}{x \cdot y = 1}$$
			holds for all $x \in X$. Then $X$ is a field streak.
		\end{lemma}
		\begin{proof}
			We have to verify two missing properties. First, assume that $x \in X$ has a multiplicative inverse $x^{-1}$, and let $n \in \NN_{> 0}$ be large enough such that $x^{-1} < n \cdot 1$ and $-x^{-1} < n \cdot 1$. Use Lemma~\ref{Lemma: streak_is_union_of_rational_intervals} to find $a\in \ZZ$ such that $\frac{a-1}{n} < x < \frac{a+1}{n}$. Clearly, if $a \geq 1$ or $a \leq -1$, then $x \apart 0$. The only case left is when $-\frac{1}{n} < x < \frac{1}{n}$, but then all of the expressions $n \cdot 1 - x^{-1}$, $n \cdot 1 + x^{-1}$, $1 - n \cdot x$, $1 + n \cdot x$ are positive, so
			$$0 < (n \cdot 1 - x^{-1}) \cdot (1 + n \cdot x) = n \cdot 1 + n^2 \cdot x - x^{-1} - n \cdot 1 = n^2 \cdot x - x^{-1},$$
			$$0 < (n \cdot 1 + x^{-1}) \cdot (1 - n \cdot x) = n \cdot 1 - n^2 \cdot x + x^{-1} - n \cdot 1 = -n^2 \cdot x + x^{-1},$$
			a contradiction.
			
			Next, take $x, y \in X$, $0 < x < y$. Use the Archimedean property to find $n \in \NN_{> 0}$ such that $y < n \cdot 1$. The statement $y^{-1} < \frac{1}{n}$ leads to contradiction since we can multiply both sides by $y$ and $n$. Thus $y^{-1} \geq \frac{1}{n} > 0$. Now use Lemma~\ref{Lemma: order_comparison_between_streak_and_rational_elements} to find the rationals $q, r \in \QQ$ such that $0 < x < q < r < y$. The statements $x^{-1} < q^{-1}$ and $r^{-1} < y^{-1}$ similarly lead to contradiction, so $y^{-1} \leq r^{-1} < q^{-1} \leq x^{-1}$.
		\end{proof}
		
		Recall that there is also a canonical way to turn a commutative ring into a field, namely its field of fractions --- at least when the ring has no nontrivial zero divisors. Constructively, we say that a ring $X$ has no nontrivial zero divisors when
		$$x \apart 0 \land y \apart 0 \implies x \cdot y \apart 0$$
		for all $x, y \in X$.
		
		We prove that the field of fractions of a ring streak is again a streak.
		
		\begin{lemma}
			A ring streak has no zero divisors.
		\end{lemma}
		\begin{proof}
			Let $X$ be a ring streak, $x, y \in X$ and $x, y \apart 0$. We have four possibilities: $x, y > 0$ or $x > 0, y < 0$ or $x < 0, y > 0$ or $x, y < 0$. Rewrite $x < 0$ and $y < 0$ as $0 < -x$ and $0 < -y$. In every case it then follows from the definition of an ordered ring that $x \cdot y \apart 0$.
		\end{proof}
		
		\begin{proposition}\label{Proposition: field_of_fractions_of_a_streak}
			For any ring streak $X$, its field of fractions $XX^{-1}$ is a field streak, and the embedding $i\colon X \to XX^{-1}$ is a streak morphism which is initial among streak morphisms from $X$ to field streaks.
		\end{proposition}
		\begin{proof}
			We can afford to take only positive elements as denominators, so we define the set $Y = \st{(a, b) \in X \times X}{b > 0}$, operations on it
			$$(a, b) + (c, d) \dfeq (a d + b c, b d), \qquad -(a, b) \dfeq (-a, b),$$
			$$(a, b) \cdot (c, d) \dfeq (a c, b d), \qquad (a, b)^{-1} \dfeq \begin{cases} (b, a) & \text{if } a > 0, \\ (-b, -a) & \text{if } a < 0, \end{cases}$$
			and the relation
			$$(a, b) \sim (c, d) \dfeq a d = b c.$$
			Clearly, $\sim$ is reflexive and symmetric. Now assume $(a, b) \sim (c, d)$ and $(c, d) \sim (e, f)$. Then $a d f = b c f  = b d e$, so $d (a f - b e) = 0$ whence $a f = b e$ because $d \apart 0$ and $X$ has no zero divisors by the previous lemma. This proves $\sim$ is transitive, and hence an equivalence relation. All above operations commute with $\sim$, so they induce operations on $XX^{-1} \dfeq Y/_\sim$. We verify $XX^{-1}$ is a field as usual, with $[(0, 1)]$ as the unit for addition, $[(1, 1)]$ as the unit for multiplication, and the apartness relation induced by $<$ defined as follows: $[(a, b)] < [(c, d)] \dfeq a d < b c$ (here it is convenient that we restricted the denominators to positive elements, else we would have to separate the cases). After somewhat lengthy but straightforward verification we see that $<$ is well-defined, that it is a strict order on $XX^{-1}$, and that it makes it into an ordered ring. Because $<$ is open on $X$, the strict order is an open relation also on $XX^{-1}$. For $[(a, b)], [(c, d)] \in XX^{-1}$, $[(c, d)] > 0$ (\ie $c > 0$) there exists $n \in \NN$ such that $a d < n b c$ because $X$ is Archimedean; this means $[(a, b)] < n \cdot [(c, d)]$, so $XX^{-1}$ is Archimedean by Lemma~\ref{Lemma: simplify_conditions_when_subtraction_is_total}. We conclude $XX^{-1}$ is a streak. Also, $[(0, 1)] \apart [(a, b)]$ is equivalent to $0 \apart a$, so every element apart from zero has a multiplicative inverse. By Lemma~\ref{Lemma: simplify_conditions_for_field_streaks} this suffices to prove that $XX^{-1}$ is a field streak.
			
			Recall that the embedding $i\colon X \to XX^{-1}$ is given by $i(x) = [(x, 1)]$. The fact that it is a streak morphism is immediate.
		\end{proof}
		
		This tells us something about the preorder of the streaks. For any streak $X$, the map $\NN \to X$, $n \mapsto n \cdot 1$, is a streak morphism, and unique by Proposition~\ref{Proposition: streak_morphisms}. Thus, $\NN$ is the initial streak; since streak morphisms are injective, we can also say that $\NN$ is the smallest streak, and all streaks contain it. If $X$ is a ring streak, then by Proposition~\ref{Proposition: Grothendieck_construction_on_a_streak} the map $\NN \to X$ extends to the map $\ZZ \to X$ which defines $n \cdot 1$ for all $n \in \ZZ$, and we see that $\ZZ$ is the initial ring streak. Moreover, for any $x \in X$, $a \in \ZZ$, $b \in \NN_{> 0}$, we have $x < \frac{a}{b}$ if and only if $b \cdot x < a \cdot 1$ (and similarly for $x > \frac{a}{b}$). Finally, if $X$ is a field streak, then the morphism $\ZZ \to X$ can be uniquely extended to the streak morphism $\QQ \to X$, so $\QQ$ is the initial field streak.
		
		Just as we can add formal differences and quotients to obtain the ring and the field structure, so we can add the formal suprema and infima to obtain the lattice structure on a streak.
		\begin{proposition}\label{Proposition: streak_to_lattice}
			For a streak $X$ there exists a streak $X^\lor$ which is a join-semilattice, and a streak morphism $j\colon X \to X^\lor$ such that every streak morphism $f\colon X \to Y$ where $Y$ is also a join-semilattice factors through it, \ie there is a (by Proposition~\ref{Proposition: streak_morphisms} necessarily unique) streak morphism $\bar{f}\colon X^\lor \to Y$ for which $\bar{f} \circ i = f$. Analogously, we can construct a streak morphism $m\colon X \to X^\land$ where $X^\land$ is a meet-semilattice, and $m$ has a similar universal property.
		\end{proposition}
		\begin{proof}
			Let $Y = \inhfin(X)$ be the set of inhabited finite subsets of $X$. We define a relation $<$ on $Y$ by
			$$A < B \dfeq \xsome{b}{B}\xall{a}{A}{a < b}$$
			for $A, B \in Y$. Because finite sets are both overt and compact, $<$ is open. It is asymmetric (in particular irreflexive) because the strict order on $X$ is. Now take any $A, B, C \in Y$ and assume $A < B$; then there is $b \in B$ which is larger than all elements in $A$. For all finitely many $a \in A$ and finitely many $c \in C$ we have $a < c \lor c < b$. If for some $c \in C$ we always get the first disjunct, then $A < C$. Otherwise, if for all $c \in C$ we get the second disjunct, then $C < B$. Hence, $<$ is cotransitive. Irreflexivity and cotransitivity of $<$ imply that the relation $\sim$, defined on $Y$ by
			$$A \sim B \dfeq \lnot(A < B) \land \lnot(B < A),$$
			is reflexive and transitive. Clearly, it is also symmetric, so an equivalence relation. The relation $<$ induces an open strict order on the quotient $X^\lor \dfeq Y/_\sim$. The semilattice structure is given by $\sup\{[A], [B]\} = [A \cup B]$.
			
			We define the addition on $X^\lor$ by $[A] + [B] \dfeq [A+B]$, and skip the proof that it is well-defined, and that $[\{0\}]$ is the neutral element $0$ for it.
			
			Observe that $[A] > 0 \iff (A_{> 0} \text{ inhabited})$. For $[A], [B] > 0$ we define the multiplication
			$$[A] \cdot [B] \dfeq [A_{> 0} \cdot B_{> 0}].$$
			It is well-defined, distributive over addition, and it determines a commutative monoid structure (with the unit $[\{1\}]$) on the positive elements.
			
			It remains to verify the Archimedean property. Notice that for $n \in \NN$ and $[A] \in X^\lor$ we have $n \cdot [A] = [n \cdot A]$. Take any $[A], [B], [C], [D] \in X^\lor$, $[C] < [D]$. Let $d \in D$ be such that $c < d$ for all $c \in C$, and fix some $a \in A$. For all $b \in B$, $c \in C$ find some $m \in \NN$ such that $b + m \cdot c < a + m \cdot d$, and let $n \in \NN$ be the maximum of these $m$s. Then $a + n \cdot d$ witnesses the fact that $[B] + m \cdot [C] < [A] + m \cdot [D]$.
			
			Finally, define $j\colon X \to X^\lor$ by $j(x) \dfeq [\{x\}]$. It is straightforward to verify that $j$ is a streak morphism.
			
			The construction of $X^\land$ and $m$ is the same, just with the strict order inverted: $[A] < [B] \dfeq \xsome{a}{A}\xall{b}{B}{a < b}$.
		\end{proof}
		
		\begin{corollary}
			The set of real numbers is an ordered field and a lattice.
		\end{corollary}
		\begin{proof}
			By Corollary~\ref{Corollary: streak_terminality}, taking in account the existence of the streak morphisms $\RR \to G(\RR)$ by Proposition~\ref{Proposition: Grothendieck_construction_on_a_streak}, $\RR \to \RR\RR^{-1}$ by Proposition~\ref{Proposition: field_of_fractions_of_a_streak}, and $\RR \to \RR^\lor$, $\RR \to \RR^\land$ by Proposition~\ref{Proposition: streak_to_lattice}.
		\end{proof}
		
		We return to expanding the definition of $<$, this time between arbitrary streaks $X$ and $Y$. For elements $x \in X$, $y \in Y$ we define
		$$x < y \dfeq \some{q}{\QQ}{x < q \land q < y}.$$
		Again, there is no ambiguity --- by Proposition~\ref{Proposition: rational_streak_order_basic}, if $X = Y$, then $<$ matches the strict order on $X$, and if either of $X$ and $Y$ is $\QQ$, it matches the previously defined order. The usual properties of $<$ again hold.
		
		We now want to see that interpolation property can also be expressed by the order between different streaks. We say that a streak $X$ is \df{interpolating with respect to a streak $Y$} when for any $a, b \in Y$, $a < b$, there exists $x \in X$ such that $a < x < b$. Of course, that means $X$ is interpolating when it is interpolating with respect to itself.
		\begin{lemma}\label{Lemma: interpolating_property_of_ring_streaks}
			The following is equivalent for a ring streak $X$.
			\begin{enumerate}
				\item $X$ is interpolating.
				\item $X$ is interpolating with respect to some streak.
				\item $X$ is interpolating with respect to $\NN$.
				\item There is $z \in X$ such that $0 < z < 1$.
				\item $X$ is interpolating with respect to $\QQ$.
				\item $X$ is interpolating with respect to all streaks.
			\end{enumerate}
		\end{lemma}
		\begin{proof}
			\begin{itemize}
				\item
					\itemimpl{1}{2} It is interpolating with respect to itself.
				\item
					\itemimpl{2}{3} All streaks contain $\NN$.
				\item
					\itemimpl{3}{4} Since $0, 1 \in \NN$.
				\item
					\itemimpl{4}{5} Take any $q, r \in \QQ$, $q < r$. Introduce the following notation, in order:
					\begin{itemize}
						\item $a, b, c, d \in \ZZ$ such that $d > 0$, $c \neq 0$ (so that its absolute value $|c| > 0$), $a < c < b$ and $q = \frac{a}{d}$, $r = \frac{b}{d}$;
						\item use Lemma~\ref{Lemma: order_comparison_between_streak_and_rational_elements} to obtain $f, g, h \in \QQ$, $0 < f < z < g < h < 1$;
						\item $m \in \NN$ large enough that $h^m < \frac{1}{2 |c| d}$;
						\item $t \in \NN$ large enough so that $t > \frac{2}{h-g}$ (in particular $t \geq 3$) and $t > \frac{2 |c| d m (m+1)}{f^2}$;
						\item $u \in \NN_{\geq 1}$ such that $\frac{u-1}{t} < z < \frac{u+1}{t}$ (Lemma~\ref{Lemma: streak_is_union_of_rational_intervals}); denote $w \dfeq \frac{u+1}{t}$ for short, and notice $w < h$ since $u+1 < t \cdot z + 2 \cdot 1 < t g + 2 < t h$;
						\item let the finite sequence $s\colon \NN_{\leq m} \to \NN_{< w^{-1}}$, essentially the beginning of the fractional expansion of $\frac{1}{d}$ in base $w^{-1}$, be inductively defined as follows: for $n \in \NN_{\leq m}$, let $s_n$ be the largest natural number such that $\sum_{k \in \NN_{\leq n}} s_k w^k \leq \frac{1}{d}$; in particular, $s_0 = 0$;
						\item define $p(x) \dfeq \sum_{k \in \NN_{\leq m}} s_k \cdot x^k$, to be interpreted either as a map $\QQ \to \QQ$ or a map $X \to X$, as necessary.
					\end{itemize}
					
					We claim that $c \cdot p(z)$ is an element of $X$ which is between $q$ and $r$. The idea of the proof is that we anticipate $X$ and $\QQ$ to be subsets of the reals which are a field, a lattice and a metric space with the metric induced by the absolute value as the norm, and then we could calculate
					$$\big|p(z) - \frac{1}{d}\big| \leq \big|p(z) - p(w)\big| + \big|p(w) - \frac{1}{d}\big| < \sum_{k \in \NN_{\leq m}} \Big(s_k \big|z^k - w^k\big|\Big) + w^m =$$
					$$= \sum_{k \in \NN_{\leq m}} \Big(s_k \big|z - w\big| \sum_{l \in \NN_{< k}} \big(z^l w^{k-1-l}\big)\Big) + w^m < \sum_{k \in \NN_{\leq m}} \Big(w^{-1} \ \frac{2}{t} \ k \ w^{k-1}\Big) + w^m <$$
					$$< \frac{m (m+1)}{w^2 t} + w^m < \frac{m (m+1)}{f^2 t} + h^m < \frac{1}{2 |c| d} + \frac{1}{2 |c| d} = \frac{1}{|c| d}$$
					whence $c \cdot p(z) \in \intoo{\frac{c-1}{d}}{\frac{c+1}{d}} \subseteq \intoo{q}{r}$. To obtain an actual proof, separate the cases $c < 0$, $c > 0$, rewrite the inequality with the absolute value as two separate inequalities, multiply with the common denominator of all fractions appearing in the calculation, and substitute any stand-alone integer with that integer multiplied by the multiplicative unit $1$ in $X$.
				\item
					\itemimpl{5}{6} If $a < b$ are elements of an arbitrary streak $Y$, we obtain, by Lemma~\ref{Lemma: order_comparison_between_streak_and_rational_elements}, $q, r \in \QQ$ such that $a < q < r < b$. Since $X$ is interpolating with respect to $\QQ$ by assumption, there is $x \in X$ such that $q < x < r$, and so $a < x < b$.
				\item
					\itemimpl{6}{1} If $X$ is interpolating with respect to all streaks, it is also with respect to itself.
			\end{itemize}
		\end{proof}
		
		Recall now the definition of (Dedekind) cuts.
		\begin{definition}
			Let $X$ be any strict linear order.
			\begin{itemize}
				\item A \df{lower cut} on $X$ is a subset $L \subseteq X$ with the property $x \in L \iff \some{y}{L}{x < y}$ for all $x \in X$.
				\item An \df{upper cut} on $X$ is a subset $U \subseteq X$ with the property $x \in U \iff \some{y}{L}{y < x}$ for all $x \in X$.
				\item A \df{Dedekind cut} on $X$ is a pair $(L, U)$ where $L$ is an inhabited lower cut, $U$ is an inhabited upper cut, $L$ and $U$ are disjoint, and the following property (called \df{locatedness} of the cut) holds: for all $x, y \in X$, if $x < y$, then $x \in L \lor y \in U$. We say that the Dedekind cut $(L, U)$ is \df{open} when $L$ and $U$ are open subsets of $X$.
				\item We denote the sets of open cuts by
				$$\lowercuts(X) \dfeq \st{L \in \pst(X)}{\text{$L$ an open lower cut on $X$}},$$
				$$\uppercuts(X) \dfeq \st{U \in \pst(X)}{\text{$U$ an open upper cut on $X$}},$$
				$$\Dedcuts(X) \dfeq \st{(L, U) \in \pst(X) \times \pst(X)}{\text{$(L, U)$ an open Dedekind cut on $X$}}.$$
				(Note that the sets of all cuts of a kind can be written as $\lowercuts[\soc](X)$, $\uppercuts[\soc](X)$ and $\Dedcuts[\soc](X)$.)
			\end{itemize}
		\end{definition}
		
		Finally, we have the construction of the reals.
		\begin{theorem}\label{Theorem: Dedekind_cuts_are_the_real_numbers}
			For any overt interpolating ring streak $\oirs$, the set $\Dedcuts(X)$ is a field streak and a lattice in which, for $a = (L_a, U_a), b = (L_b, U_b) \in \Dedcuts(\oirs)$, the order and algebraic structure is given by
			$$(L_a, U_a) < (L_b, U_b) \dfeq U_a \between L_b,$$
			$$a \leq b \iff \lnot(a > b) \iff L_a \cap U_b = \emptyset \iff L_a \subseteq L_b \iff U_a \supseteq U_b,$$
			$$\sup\{(L_a, U_a), (L_b, U_b)\} = (L_a \cup L_b, U_a \cap U_b), \qquad \inf\{(L_a, U_a), (L_b, U_b)\} = (L_a \cap L_b, U_a \cup U_b),$$
			$$(L_a, U_a) + (L_b, U_b) \dfeq (L_a + L_b, U_a + U_b),$$
			$$0 = (\oirs_{< 0}, \oirs_{> 0}), \qquad\qquad -(L_a, U_a) = (-U_a, -L_a),$$
			$$(L_a, U_a) \cdot (L_b, U_b) \dfeq (\st{q \in \oirs}{\xsome{l}{L_a}\some{k}{L_b}{l > 0 \land k > 0 \land q < l \cdot k}}, U_a \cdot U_b) \quad \text{for} \quad a, b > 0,$$
			$$(L_a, U_a)^{-1} = (\st{q \in \oirs}{\some{l}{L_a}{l > 0 \land q \cdot l < 1}}, \st{r \in \oirs}{\xsome{u}{U_a}{r \cdot u > 1}}) \quad \text{for} \quad a > 0.$$
			Moreover, $\Dedcuts(\oirs)$ is a terminal streak, with the unique morphism between streaks $X$ and $\Dedcuts(\oirs)$ given by
			$$x \mapsto \big(\st{q \in \oirs}{q < x}, \st{r \in \oirs}{x < r}\big).$$
			In particular, the sets of open Dedekind cuts are isomorphic for all overt interpolating ring streaks. The above formula also gives a canonical map $\oirs \to \Dedcuts(\oirs)$.
		\end{theorem}
		\begin{proof}
			The order and the algebraic part of the proof is the same as for the usual Dedekind cuts, so we focus on the new topological part. Rewrite the above formulae as
			\begin{itemize}
				\item
					$a < b \iff \some{q}{\oirs}{q \in U_a \land q \in L_b}$,
				\item
					$\sup\{a, b\} = (L_a \cup L_b, U_a \cap U_b)$, $\inf\{a, b\} = (L_a \cap L_b, U_a \cup U_b)$,
				\item
					$a + b = \big(\st{q \in \oirs}{\some{l, k}{\oirs}{l \in L_a \land k \in L_b \land q < l + k}},$\\
					$~\hspace{4em}\st{r \in \oirs}{\some{u, t}{\oirs}{u \in U_a \land t \in U_b \land r > u + t}}\big)$,
				\item
					$0 = \big(\st{q \in \oirs}{q < 0}, \st{r \in \oirs}{r > 0}\big)$,
				\item
					$-a = \big(\st{q \in \oirs}{\some{u}{\oirs}{u \in U_a \land q + u < 0}}, \st{r \in \oirs}{\some{l}{\oirs}{l \in L_a \land r + l > 0}}\big)$,
				\item
					$a \cdot b = \big(\st{q \in \oirs}{\some{l, k}{\oirs}{l \in L_a \land l > 0 \land k \in L_b \land k > 0 \land q < l \cdot k}},$\\
					$~\hspace{3.6em}\st{r \in \oirs}{\some{u, t}{\oirs}{u \in U_a \land t \in U_b \land r > u \cdot t}}\big)$ for $a, b > 0$,
				\item
					$a^{-1} = \big(\st{q \in \oirs}{\some{l}{\oirs}{l \in L_a \land l > 0 \land q \cdot l < 1}}, \st{r \in \oirs}{\xsome{u}{\oirs}{u \in U_a \land r \cdot u > 1}}\big)$ for $a > 0$;
			\end{itemize}
			since $\oirs$ is overt, all cuts are open.
			
			What remains is to verify that for any streak $X$, the map $f\colon X \to \Dedcuts(\oirs)$, $f(x) = (L, U)$ where $L = \st{q \in \oirs}{q < x}$, $U = \st{r \in \oirs}{x < r}$, is a well-defined streak morphism. The Archimedean property implies inhabitedness of $L$, $U$ (for $L$ we also use the fact that $\oirs$ is a ring, and so contains $\ZZ$). Clearly, $L$ is a lower and $U$ is an upper set. By the interpolation property of $\oirs$ and Lemmas~\ref{Lemma: order_comparison_between_streak_and_rational_elements} and~\ref{Lemma: interpolating_property_of_ring_streaks} we can conclude that $L$ is a lower and $U$ is an upper cut. They are open in $\oirs$ since $<$ is an open relation. Next, take any $q, r \in \oirs$, $q < r$. Playing with the definition of the strict order between possibly different streaks yields the contransitivity of $<$, namely $q < x \lor x < r$, proving the locatedness of $(L, U)$. Thus $(L, U)$ is an open Dedekind cut, meaning that $f$ indeed maps into $
 \Dedcuts(\oirs)$.
			
			Take any $x, y \in X$, $x < y$. Since $\oirs$ is interpolating, by Lemma~\ref{Lemma: interpolating_property_of_ring_streaks} there is $q \in \oirs$, $x < q < y$, implying $f(x) < f(y)$. Preservation of addition and multiplication (the latter on positive elements) is clear from the formulae above. Thus $f$ is a streak morphism, making $\Dedcuts(\oirs)$ the terminal streak by Corollary~\ref{Corollary: streak_terminality}.
		\end{proof}
		In view of this theorem we hereafter assume that the real numbers are constructed as open Dedekind cuts on some fixed overt interpolating ring streak $\oirs$.
		
		Having the real numbers now, perhaps some explanation why we defined the streaks as we did, is in order; among other things it might seem that the reasonable choice of the structure among which the reals are terminal would be the one from which we actually construct the reals, namely overt interpolating ring streaks. There is a catch, though: while the reals certainly are an interpolating ring streak, they are not in general overt.
		
		There are also some other reasons why we prefer to require less structure on streaks. One is that it allows us to succinctly characterize other number sets as well --- recall that $\NN$ is the initial streak, $\ZZ$ is the initial ring streak, and $\QQ$ is the initial field streak. Furthermore, requiring the multiplication only on positive elements was carefully chosen so that we could prove Proposition~\ref{Proposition: streak_to_lattice}, and consequently that the reals are a lattice. Actually, we could have proven that even if we assumed streaks are rings, but the construction would be more complicated: we would have to add formal suprema, then formal infima and verify that this preserves upper semilattice structure, and then perform the Grothendieck construction and verify that it preserves the lattice structure. We preferred to keep these three constructions separate.
		
		Of course, one might argue that we do not need these results at all, we can just directly show that open Dedekind cuts form a lattice and a field. But the way we did it is more general: we now know that the terminal streak is a lattice and a field even when we cannot perform the Dedekind construction, such as is the predicative setting. The point is, we used the existence of power sets in the definition of cuts. Actually, this can be mollified: since we are only interested in open cuts anyway, we can define them to be elements of $\tp(\oirs)$, or equivalently, of $\opn^\oirs$ (this is in fact how it is done in~ASD~\cite{Bauer_A_Taylor_P_2009:_the_dedekind_reals_in_abstract_stone_duality} --- they are elements of $\opn^\QQ$).
		
		We provide one more reason why considering structures with multiplication only on positive elements is worthwhile in the context of reals. Notice that in the case of streaks it would be equivalent by Lemma~\ref{Lemma: order_comparison_between_streak_and_rational_elements} to define comparison with rational numbers on both sides, \ie relations $<\colon \QQ \times X \to \opn$, $<\colon X \times \QQ \to \opn$, and then have $<\colon X \times X \to \opn$ as the derived notion. In view of this, define $X$ to be a \df{lower streak} when we can compare its elements with rational numbers just on one side, namely we have the relation $<\colon \QQ \times X \to \opn$, and $X$ is an Archimedean protoring (positive elements of $X$ are defined to be those which are greater than the rational $0$), plus the expected axioms for order and protoring operations. Terminal among lower streaks are, as one can expect, the lower reals, the model for which are inhabited open lower cuts (of overt interpolative ring streaks) with inhabited complements\footnote{Inhabitedness excludes the cuts which would represent plus and minus infinity. Sometimes it is convenient to include infinity as well, \ie to consider the set of all (open) lower cuts --- this is called the \df{extended lower reals}. Similarly, (open) upper cuts are \df{extended upper reals}, and (open) Dedekind cuts with the inhabitedness condition dropped are \df{extended reals}.}, with the multiplication on positive elements the same as in Theorem~\ref{Theorem: Dedekind_cuts_are_the_real_numbers}. However, this multiplication can in general not be extended to all cuts to obtain a ring. One can analogously define \df{upper streaks} and \df{upper reals}. The only problem is that we do not have the definition of positiveness; a way out is to require the multiplication on nonnegative elements.
		
		Actually, it is possible to take an even more general structure than streaks to construct the reals --- the operations need not even be exact, only approximative, as is known to those who implement reals on the computer~\cite{Lambov_B_2007:_reallib_an_efficient_implementation_of_exact_real_arithmetic, Muller_N_2001:_the_irram_exact_arithmetic_in_c++}. Consider the dyadic rationals, for example, and say that the \df{complexity} of a dyadic rational is the lowest $n \in \NN$ such that it can be written in the form $\frac{a}{2^n}$ where $a \in \ZZ$. Notice that the complexity of the sum does not exceed the largest complexity of the summands, but the complexity of the product is the sum of the complexities of the factors. However, even if we round the result of the product to the largest complexity of the of the factors, in the limit we still obtain exact multiplication on the reals. In view of this, we could define the streaks to have only approximative multiplication, \ie a streak would be a union of different levels of complexity which would measure to what extent the laws of associativity, distributivity \etc hold, with one level having exact arithmetic, and containing $\NN$. However, I feel that the current version is technical enough as it is, and chose not to make this theory even more technically complicated.
		
		\intermission
		
		We end the section about real numbers by comparing them with the standard Dedekind and Cauchy reals. It is obvious that for two \Sier objects $\opn \subseteq \opn'$ we have $\RR_\opn \subseteq \RR_{\opn'}$; in particular $\RR_\opn \subseteq \RR_\soc = \RR_d$.
		
		We want to define what the Cauchy sequences (with their moduli of convergence) in streaks are, but since we do not have subtraction and the absolute value in general, we make do as follows: a Cauchy sequence in a streak $X$ is a pair $(s\colon \NN \to X, m\colon \NN \to \NN)$ with the property $\xall{n}{\NN}\all{i, j}{\NN_{\geq m(n)}}{2^n \cdot s_i < 1 + 2^n \cdot s_j}$. We denote the set of Cauchy sequences in $X$ by $\Cauchy(X)$.
		
		As usual, we define the equivalence relation $\equ$ on $\Cauchy(X)$ by
		$$(s, m) \equ (s', m') \dfeq \xsome{M}{\NN^\NN}\xall{n}{\NN}\all{i, j}{\NN_{\geq M(n)}}{2^n \cdot s_i < 1 + 2^n \cdot s'_j \land 2^n \cdot s'_i < 1 + 2^n \cdot s_j}.$$
		For an interpolating ring streak $X$ we define the \df{Cauchy reals, generated from $X$}, to be the quotient set $\RR_c(X) \dfeq \Cauchy(X)/_\equ$. The usual Cauchy reals are of course $\RR_c(\QQ)$.
		
		We claim that $\RR_c(X)$ is an interpolating field lattice streak. Most of the proof is the same as for the usual Cauchy reals, the only new thing is openness of $<$, but this is obvious from the definition of $<$, namely for $[s, m], [s', m'] \in \RR_c(X)$,
		$$[s, m] < [s', m'] \dfeq \xsome{n}{\NN}\some{i}{\NN_{\geq \sup\{m(n+1), m'(n+1)\}}}{2^n \cdot s_i + 1 < 2^n \cdot s'_i}.$$
		Moreover, notice that any map $M \in \NN^\NN$ (we can always take \eg $M = \id[\NN]$) serves as a modulus of convergence for a constant sequence, so if $\cs\colon \RR \to \Cauchy(\RR)$ is the map $\cs(x) \dfeq (n \mapsto x, M)$, then the standard map $X \to \RR_c(X)$, mapping $x \in X$ to the equivalence class of $\cs(x)$, is a streak morphism.
		
		The consequence is that for any interpolating ring streak $X$ we have $X \subseteq \RR_c(X) \subseteq \RR$. The interesting special cases are $\RR_c \subseteq \RR$, and $\RR_c(\RR) = \RR$, \ie $\RR$ is \df{Cauchy complete}.
		
		\begin{proposition}
			Let $X$ be an interpolating ring streak.
			\begin{enumerate}
				\item If $\AC[\opn]{\NN, X}$, then $\RR_c(X) = \RR$; in particular, if $\ACopn$, then $\RR_c = \RR$.
				\item If $\AC{\NN, X}$, then $\RR_c(X) = \RR = \RR_d$; in particular, if $\AC{\NN, \NN}$, then $\RR_c = \RR = \RR_d$.
			\end{enumerate}
		\end{proposition}
		\begin{proof}
			\begin{enumerate}
				\item
					Take any $x \in \RR$. Since $X$ is interpolating, we can, for any $n \in \NN$, find $s_n \in X$ such that $x < s_n < x + 2^{-n}$. Since $<$ is open, by $\AC[\opn]{\NN, X}$ this defines a map $s\colon \NN \to X$. Then $[s, \id[\NN]] = x$.
				\item
					Use the previous item for $\opn = \soc$.
			\end{enumerate}
		\end{proof}
		
		Cauchy completeness of $\RR$ implies the existence of the map $\lim\colon \Cauchy(\RR) \to \RR$ which takes a Cauchy sequence to its limit; explicitly, it is the quotient map $\Cauchy(\RR) \to \Cauchy(\RR)/_\equ$, composed with the isomorphism $\RR_c(\RR) \ism \RR$. This limit operator has the usual properties; we mention only a few that we need.
		
		First, notice that if two Cauchy sequences differ only by their modulus of convergence, then they are in the same equivalence class, and therefore have the same limit.
		
		Second, $\lim$ is a retraction, as $\lim \circ \cs = \id[\RR]$. This is of course just stating the fact that the limit of a constant sequence is whatever the terms are. To formally verify this, let $l\colon \RR_c(\RR) \stackrel{\ism}{\longrightarrow} \RR$ be the isomorphism, and let $q\colon \Cauchy(\RR) \to \RR_c(\RR)$ be the quotient map. Since $q \circ \cs\colon \RR \to \RR_c(\RR)$ is a streak morphism, and these are unique, we have $q \circ \cs = l^{-1}$. Thus,
		$$\lim \circ \cs = l \circ q \circ \cs = \id[\RR].$$
		The consequence is that for any Cauchy sequence $(s, m)$ we have $(s, m) \equ \cs(\lim(s, m))$, witnessed by some $M\colon \NN \to \NN$ which means for $n, i \in \NN$, $i \geq M(n)$,
		$$2^n \cdot s_i < 1 + 2^n \cdot \lim(s, m) \qquad \text{and} \qquad 2^n \cdot \lim(s, m) < 1 + 2^n \cdot s_i.$$
		We have come back to the usual classical definition of a limit: it is a number such that the terms of the sequence are eventually contained in its arbitrarily small neighbourhoods.
		
		Third, if $(s, m)$ and $(s', m')$ are Cauchy sequences, and $s_k \leq s'_k$ for all $k \in \NN$, then $\lim(s, m) \leq \lim(s', m')$. Assume to the contrary, $\lim(s, m) > \lim(s', m')$; then there is some $n \in \NN$ such that $2^{-n} < \lim(s, m) - \lim(s', m')$. Let $M, M' \in \NN^\NN$ be moduli for $s, s'$ as in the above argument, and $k \dfeq \sup\{M(n+1), M'(n+1)\}$. Then $s_k > \lim(s, m) - 2^{-n-1} > \lim(s', m') + 2^{-n-1} > s'_k$, a contradiction.
		
		Fourth, if $(s, m)$ and $(s', m')$ are Cauchy sequences, so is their sum
		$$(s, m) + (s', m') \dfeq \big(s+s', n \mapsto \sup\{m(n+1), m'(n+1)\}\big),$$
		and
		$$\lim\big((s, m) + (s', m')\big) = \lim(s, m) + \lim(s', m').$$
		To see this, we need to prove $(s, m) + (s', m') \equ \cs(\lim(s, m) + \lim(s', m'))$, \ie
		$$2^n \cdot (s_i + s'_i) < 1 + 2^n \cdot (\lim(s, m) + \lim(s', m')),$$
		$$2^n \cdot (\lim(s, m) + \lim(s', m')) < 1 + 2^n \cdot (s_i + s'_i)$$
		for all $n, i \in \NN$, $i \geq M(n)$ for some modulus $M\colon \NN \to \NN$, but such $M$ exists by the argument above: take $M', M'' \in \NN^\NN$ such that $2^n \cdot s_i < 1 + 2^n \cdot \lim(s, m)$ and $2^n \cdot \lim(s, m) < 1 + 2^n \cdot s_i$ for $n, i \in \NN$, $i \geq M'(n)$, and $2^n \cdot s'_i < 1 + 2^n \cdot \lim(s', m')$ and $2^n \cdot \lim(s', m') < 1 + 2^n \cdot s'_i$ for $n, i \in \NN$, $i \geq M''(n)$, then let $M(n) \dfeq \sup\{M'(n+1), M''(n+1)\}$.

	\section{Metric Spaces}\label{Section: metric_spaces}
	
		Having the real numbers, we can define metric spaces. We will need a generalization of this notion though, both in the sense of relaxing the conditions for a metric, and choosing possibly some other streak than the real numbers for distances.
		
		\begin{definition}
			Let $X$ be a set and $S$ a streak. A pair $\mtr{X} = (X, d)$ is an \df{$S$-protometric space} when
			\begin{itemize}
				\item $d$ is a map $d\colon X \times X \to S$, \qquad (distances lie in $S$)
				\item for all $x, y \in X$ we have $d(x, y) = d(y, x)$, \qquad (symmetry)
				\item for all $x, y, z \in X$ we have $d(x, y) + d(y, z) \geq d(x, z)$. \qquad (triangle inequality)
			\end{itemize}
			When additionally
			\begin{itemize}
				\item for all $x \in X$, $d(x, x) = 0$,
			\end{itemize}
			we call $\mtr{X}$ an \df{$S$-pseudometric space}. Finally, when we also have
			\begin{itemize}
				\item for all $x, y \in X$, $d(x, y) = 0 \implies x = y$, \qquad (nondegeneracy)
			\end{itemize}
			we call it an \df{$S$-metric space}. In these three cases, the map $d$ is called the \df{$S$-(proto-, pseudo-)metric}, and the value $d(x, y)$ is called the \df{distance} between $x$ and $y$.
		\end{definition}
		When the streak $S$ is understood (usually, when it is arbitrary or when it is $\RR$), we omit it, and call $\mtr{X}$ simply a (proto-, pseudo-)metric space.
		
		We didn't assume the nonnegativity of distances because it follows from other conditions.
		\begin{proposition}
			For any streak $S$, any $S$-protometric space $(X, d)$ and any $x, y \in X$, we have $d(x, y) \geq 0$.
		\end{proposition}
		\begin{proof}
			We first verify the condition for $x = y$. Assume $d(x, x) < 0$. Add $d(x, x)$ on both sides to obtain $d(x, x) + d(x, x) < d(x, x)$, a contradiction to the triangle inequality, so $d(x, x) \geq 0$.
			
			Now take any $x, y \in X$, and let $d(x, y) < 0$; adding two such inequalities together yields $d(x, y) + d(x, y) < 0$. This is again a contradiction since, by symmetry and triangle inequality,
			$$d(x, y) + d(x, y) = d(x, y) + d(y, x) \geq d(x, x) \geq 0.$$
		\end{proof}
		
		Even though the structure of a streak is sufficient for the definition of metric spaces (and their generalizations), we will be more interested in those which are in addition rings and lattices, in particular because they are metric spaces with distances in themselves. We will therefore use a specific symbol $\lrs$ to denote an arbitrary lattice ring streak. In any such $\lrs$ we can define the absolute value $|x| \dfeq \sup\{x, -x\}$ which satisfies the usual properties: $|x| \geq 0$, $|x| = 0 \iff x = 0$, $|x y| = |x| |y|$, $|x + y| \leq |x| + |y|$. Thus, a lattice ring streak $\lrs$ is a $\lrs$-metric space for the \df{Euclidean metric} $d_E(x, y) \dfeq |x - y| = \sup\{x-y, y-x\}$.
		
		In a general pseudometric space $(X, d)$, a (\df{metric}) \df{ball} $\ball{x}{r}$ with the \df{center} $x \in X$ and the \df{radius} $r \in \RR$ is the set
		$$\ball{x}{r} \dfeq \st{y \in X}{d(x, y) < r}.$$
		Similarly, a \df{bordered} (\df{metric}) \df{ball} is defined as
		$$\bball{x}{r} \dfeq \st{y \in X}{d(x, y) \leq r}.$$
		In principle we allow arbitrary real number to be the radius, but in practice usually only inhabited balls are of interest, so we sometimes restrict to $r > 0$ (or perhaps $r \geq 0$ in the case of bordered balls). Whether bordered or otherwise, an inhabited ball contains its center.
		
		These are of course the familiar open and closed balls from classical topology and analysis. There are several reasons why we avoid this terminology. First of all, we will work almost exclusively with open balls (using closed balls occasionally merely because they are more likely to be compact), so we might as well drop the qualifier. Second, even classically there is some potential confusion when proving theorems like ``open balls are open'' and ``closed balls are closed'', in part because it sounds weird (novices typically wonder why one must prove for balls specifically declared to be open that they are open), and in part because this becomes false when considering generalizations of metric spaces; \eg in quasimetric spaces\footnote{Quasimetric spaces are ``metric spaces without the symmetry axiom''. We won't consider them in this thesis, but they have important natural examples, such as when the distance is time, or if one wants to model a city with one-way streets.}, the ``closed balls'' need not be topologically closed.
		
		Third, we must consider the situation in the context of synthetic topology. The balls are obviously open; the whole point of the (re)construction of the set of real numbers was to ensure that the strict linear order $<$ is an open relation. Bordered balls need not be closed however, though it is easy to see when they are.
		\begin{proposition}\label{Proposition: Hausdorffness_of_real_numbers}
			The following statements are equivalent.
			\begin{enumerate}
				\item For any $x \in \RR$ and any test truth value $t \in \tst$, the truth value $(x = 0) \implies t$ is test.
				\item $\{0\}$ is a closed subset of $\RR$.
				\item $\RR$ is Hausdorff.
				\item The interval $\II$ is closed in $\RR$.
				\item For any $a, b \in \RR$, the intervals $\intcc{a}{b}$, $\RR_{\leq b}$, $\RR_{\geq a}$ are closed in $\RR$.
				\item The relation $\leq$ on $\RR$ is closed.
				\item All bordered balls (across all (pseudo)metric spaces) are closed.
			\end{enumerate}
		\end{proposition}
		\begin{proof}
			Obvious. We merely note that
			\begin{itemize}
				\item $\{a\} = f^{-1}(\{0\})$ where $f\colon \RR \to \RR$, $f(x) = x - a$,\footnote{The point is that $\RR$ is a homogeneous space, \ie for any pair of real numbers there is a bijection on $\RR$ which takes the first number to the second, so it is sufficient to verify Hausdorfness on just one singleton.}
				\item $\intcc{a}{b} = f^{-1}(\{0\})$ where $f\colon \RR \to \RR$, $f(x) = |x-a| + a + |x-b| - b$,
				\item $\RR_{\leq b} = f^{-1}(\{0\})$ where $f\colon \RR \to \RR$, $f(x) = \sup\{x-b, 0\}$,
				\item $\RR_{\geq a} = f^{-1}(\{0\})$ where $f\colon \RR \to \RR$, $f(x) = \sup\{a-x, 0\}$.
			\end{itemize}
		\end{proof}
		
		In typical practical examples $\RR$ is Hausdorff, as expected.
		\begin{corollary}\label{Corollary: reals_are_Hausdorff}
			If $\tst = \opn \subseteq \nnst$, then $\RR$ is Hausdorff, and bordered balls are closed.
		\end{corollary}
		\begin{proof}
			The set $\opn$ is a bounded sublattice of $\soc$ (\ie finite sets are overt) since countable sets are overt. The singleton $\{0\}$ is a complement of an open set, namely $\{0\} = \RR \setminus \st{x \in \RR}{x \apart 0}$, so closed by Theorem~\ref{Theorem: standard_special_case_of_synthetic_topology}.
		\end{proof}
		
		Most of topological notions have their (pseudo)metric counterparts where an `open set' is replaced by a `ball'\footnote{This is simply because most topological notions can be equivalently rephrased by just basic instead of all open sets, and metric balls are of course our candidate for a topological basis in a metric space. This is the main subject of Chapter~\ref{Chapter: metrization_theorems}.}. These will be identified with a qualifier `metric' in front of the name while the usual topological notions will be called `intrinsic' if necessary.
		
		Thus we define, for a pseudometric space $\mtr{X} = (X, d)$, that the subset $A \subseteq X$ is \df{metrically dense} in $\mtr{X}$ when it intersects every inhabited ball (with symbols, $\xall{x}{X}\xall{r}{\RR_{> 0}}{A \between \ball{x}{r}}$). A pseudometric space is \df{metrically separable} when it contains a countable metrically dense subset.
		
		One should understand this definition in the sense that a countable dense subset, together with its enumeration, is actually given, so it would be more precise to say that a metrically separable pseudometric space is a triple $(X, d, s)$ where $(X, d)$ is a pseudometric space, and $s\colon \NN \to X + \one$ is a map with the property $\xall{x}{X}\xall{r}{\RR_{> 0}}\xsome{n}{s^{-1}(X)}{d(x, s_n) < r}$. This is to avoid the need for choice principles; for example, we want an $\NN$-indexed product of separable spaces to be separable. Most of the time, though, we will work with only finitely many, or even a single space, and then there isn't really a distinction.
		
		We will pay more attention to this matter in the case of totally bounded spaces, an important subclass of metrically separable ones. The classical definition~\cite{Dugundji_J_1974:_topology} says that a metric space $(X, d)$ is \df{totally bounded} when for every $\epsilon > 0$ there exist a finite subset $F \subseteq X$ such that every $x \in X$ is at less than $\epsilon$ distance to some point in $F$ (in other words, $X$ is covered by finitely many balls of radius $\epsilon$). This is also the standard constructive definition~\cite{Bishop_E_Bridges_D_1985:_constructive_analysis}, but only if we assume some form of countable choice --- otherwise, we cannot even prove that a totally bounded space is metrically separable.
		
		We therefore adopt the definition for our choiceless environment thusly: $(X, d, s, a)$ is a \df{totally bounded} pseudometric space when $(X, d)$ is a pseudometric space and maps $s\colon \NN \to X + \one$, $a\colon \NN \to \NN$ have the property
		$$\xall{n}{\NN}\xall{x}{X}\xsome{k}{(s^{-1}(X) \cap \NN_{< a_n})}{d(x, s_k) < 2^{-n}}.$$
		This definition clearly implies the usual one, and it is easy to verify that they are equivalent in the presence of countable choice. We also see that $s$ has a metrically dense image in $X$, so by this definition, a totally bounded pseudometric space is also metrically separable.
		
		Also, as expected, when $X$ is inhabited, we can take $s$ in the definitions of metrically separable and totally bounded metric spaces to be a map $\NN \to X$ instead of $\NN \to X + \one$. In fact, inhabitedness of a totally bounded space, even by the usual definition, is a decidable property: just consider how many balls of radius $1$ it takes to cover the space, and whether this number is zero or not.
		
		Recall that for a pseudometric space $(X, d)$, a point $x \in X$ and a subset $A \subseteq X$, the distance $d(A, x)$ is defined as the strict infimum of distances from points of $A$ to $x$. In general this is an upper real (more precisely, an extended upper real --- it is less then $\infty$ if and only if $A$ is inhabited). When $d(A, x)$ is a real number for all $x \in X$, the subset $A$ is called \df{located} in $X$. Note that a located subset must be inhabited if $X$ is.
		
		We know from standard constructive theory of metric spaces~\cite{Bishop_E_Bridges_D_1985:_constructive_analysis} that locatedness and total boundedness are closely connected. Specifically, if $\mtr{X} = (X, d)$ is a metric space, $A \subseteq X$ is inhabited, and $\mtr{A} = (A, \rstr{d}_{A \times A}, s, a)$ is totally bounded, then $A$ is located. This is so because
		$$\Big(\st{q \in \oirs}{\xsome{n}{\NN}\xall{k}{(s^{-1}(X) \cap \NN_{< a_n})}{q + 2^{-n} < d(x, s_k)}},$$
		$$\st{r \in \oirs}{\xsome{n}{s^{-1}(X)}{r > d(x, s_k)}}\Big)$$
		is a Dedekind cut for every $x \in X$ (and whichever overt interpolating ring streak $\oirs$ we chose to build the real numbers from), and it equals the distance $d(A, x)$. The proof for this is standard, and it works for pseudometric spaces as well. We only have to add that the cuts are open since $s^{-1}(X) \cap \NN_{< a_n}$ is a finite set, hence compact (while $\NN$, $s^{-1}(X)$ are overt, and $<$ is an open relation).
		
		Generally, a sort of a converse also holds: if $\mtr{X}$ is totally bounded, and $A$ is located in $\mtr{X}$, then $\mtr{A}$ is also totally bounded. Unfortunately, while one doesn't need choice to prove that if the usual definition of total boundedness is assumed, we would need it to produce the sequences $s$ and $a$. This is a bit of a deficiency of a choiceless treatment of totally bounded (pseudo)metric spaces, but conveniently we will not require this result.
		
		Another issue that transfers from standard theory of totally bounded spaces is the existence of their diameter. Recall that the \df{diameter} $\diam(\mtr{X})$ of a metric space $\mtr{X} = (X, d)$ is the supremum of the image of $d$, \ie of the distances between points in $\mtr{X}$. The supremum is understood to be nonnegative, so the diameter of the empty space is $0$ (alternatively, one can simply add $0$ to the image of $d$, and then take the supremum). In general, the diameter is a lower cut (if $\mtr{X}$ contains a metrically dense subovert subset, it is an open lower cut, so an extended lower real in our terminology). For totally bounded spaces however, it is always a real (also by our definition). To verify this, let $\mtr{X} = (X, d, s, a)$ be a totally bounded metric space. Then $\diam(\mtr{X}) = (L, U)$ where
		$$L \dfeq \st{q \in \oirs}{q < 0 \lor \xsome{i, j}{s^{-1}(X)}{q < d(s_i, s_j)}},$$
		$$U \dfeq \st{r \in \oirs}{r > 0 \land \xsome{n}{\NN}\xall{i, j}{(s^{-1}(X) \cap \NN_{< a_n})}{r > 2^{-n+1} + d(s_i, s_j)}}.$$
		It is easy to see that $(L, U)$ is an open Dedekind cut, and that it equals the supremum of distances.
				
		We continue with properties of metric spaces. If in a metric space $\mtr{X} = (X, d)$ the \df{strong triangle inequality}
		$$\sup\{d(x, y), d(y, z)\} \geq d(x, z)$$
		holds for all $x, y, z \in X$, we say that $\mtr{X}$ is an \df{ultrametric space}. In an ultrametric space, every point in a ball is its center, \ie for every $x, y \in X$, $r \in \RR$, if $d(x, y) < r$, then $\ball{x}{r} = \ball{y}{r}$, and if $d(x, y) \leq r$, then $\bball{x}{r} = \bball{y}{r}$. An example of an ultrametric space is any set $X$ with decidable equality, equipped with the \df{discrete metric}
		$$d_D(x, y) \dfeq \begin{cases} 0 & \text{if } x = y,\\ 1 & \text{if } x \neq y.\end{cases}$$
		
		Next, we want to consider products of (pseudo)metric spaces. From the classical topology we know that it makes sense to consider only countable products, and then we have various product metrics to choose from. For us this means considering binary and countably infinite products, and we will be particular in our choice of the product metric --- we will always take the supremum metric\footnote{We are restricting ourselves here. Recall that in classical topology we can equip, say, the space $\C(\II, \RR)$ with supremum metric as well because $\II$ is compact, and $\C(\II, \RR)$ could be interpreted as the $\II$-indexed product of the spaces of real numbers. Indeed, in our setting, when conditions of Theorem~\ref{Theorem: compact_maps_into_reals} are satisfied, we could take any set which is a countable union of compact overt subsets to be the index set of the product of metric spaces, and still obtain a metric space. We won't need this generality, however.}.
		
		We therefore define: if $(X, d_X)$, $(Y, d_Y)$ are pseudometric spaces, then their binary product is the pseudometric space $(X \times Y, d)$ where
		$$d\big((x, y), (z, w)\big) \dfeq \sup\{d_X(x, z), d_Y(y, w)\}.$$
		
		Let now $(X_n, d_n)_{n \in \NN}$ be a sequence of pseudometric spaces, and let $h\colon A \to \II$ be a map with the following properties:
		\begin{itemize}
			\item $A \subseteq \RR$ contains images of all pseudometrics $d_n$,
			\item $0 \in A$ (this already follows from the previous item if any $X_n$ is inhabited) and $h(0) = 0$,
			\item $h$ is \ed-continuous (see Definition~\ref{Definition: morphisms_of_metric_spaces} below),
			\item $h$ preserves $<$, \ie $a < b \iff h(a) < h(b)$ for all $a, b \in A$.
		\end{itemize}
		Then the product is the pseudometric space $\big(\prod_{n \in \NN} X_n, d\big)$ where
		$$d\big((x_n)_{n \in \NN}, (y_n)_{n \in \NN}\big) \dfeq \sup\st{2^{-n} h(d_n(x_n, y_n))}{n \in \NN}.$$
		Using $h$ in this definition is the standard trick how to compensate for the possible unboundedness of factors; we will call it the \df{gauge map}. One possibility for $h$ that always works is $h\colon \RR_{\geq 0} \to \II$, $h(t) = \frac{t}{1+t}$. However, we will typically consider countable products of spaces in which the distances are already bounded above by $1$, and then we can simply take for $h$ the identity on $\II$.
		
		It might not be obvious that the supremum in the definition of the product pseudometric is in fact a real number, so we rewrite the definition
		$$d\big((x_n)_{n \in \NN}, (y_n)_{n \in \NN}\big) = (L, U)$$
		where
		$$L \dfeq \st{q \in \oirs}{\xsome{n}{\NN}{q < 2^{-n} h(d_n(x_n, y_n))}},$$
		$$U \dfeq \st{r \in \oirs}{\some{n}{\NN}{r > 2^{-n} \land \xall{k}{\NN_{< n}}{r > 2^{-k} h(d_k(x_k, y_k))}}}.$$
		Because $\oirs$ is interpolative ring streak, $L$ is a lower and $U$ is an upper cut; they are clearly disjoint, and they are open since $\NN$ is overt, $\NN_{< n}$ is compact and $<$ is an open relation. Finally, take any $q, r \in \oirs$, $q < r$. By interpolation property, find $s \in \oirs$, $q < s < r$. Consider the options $q < 0 \lor 0 < s$. If $q < 0$, then $q \in L$, and we are done. Suppose $0 < s$. Let $n \in \NN$ be such that $2^{-n} < r-s$. By cotransitivity of $<$ we can construct a finite sequence $a\colon \NN_{< n} \to \two$ with the property $a_k = 0 \implies q < 2^{-k} h(d_k(x_k, y_k))$ and $a_k = 1 \implies 2^{-k} h(d_k(x_k, y_k)) < r$ for all $k \in \NN_{< n}$. If $a$ has a zero, then $q \in L$. If $a$ has only ones, then  the condition $\xall{k}{\NN_{< n}}{r > 2^{-k} h(d_k(x_k, y_k))}$ is satisfied, and moreover $r > 2^{-n} + s > 2^{-n}$.
		
		\begin{proposition}\label{Proposition: metric_multiplicative_properties}
			If pseudometric spaces are inhabited/metrically separable/totally bounded/met\-ric/ultrametric, then so is their binary product.
			
			If pseudometric spaces are inhabited metrically separable/inhabited totally bounded\footnote{The issue of inhabitedness is unavoidable here. A single totally bounded space is either inhabited or not, but inhabitedness of a countably infinite product of totally bounded spaces can have an arbitrary semidecidable truth value; we can thus prove that such a product is also totally bounded (if and) only if we know whether it is inhabited or not.}/met\-ric/ultra\-met\-ric, then so is their countably infinite product.
		\end{proposition}
		\begin{proof}
			Binary product of inhabited sets is of course inhabited, and clearly products of metric spaces are again metric. We focus on other properties.
			\begin{itemize}
				\item \proven{total boundedness}
					Let $(X, d', s', a')$, $(Y, d'', s'', a'')$ be totally bounded. Let $r\colon \NN \to \NN$,
					$$r(n) \dfeq \sup\st{k \in \NN_{\leq n}}{k^2 \leq n} = \inf\st{k \in \NN}{n < (k+1)^2},$$
					map a natural number to the integer part of its square root. Then the map $b = (b', b'')\colon \NN \to \NN \times \NN$,
					$$b(n) \dfeq \big(\inf\{n - (r(n))^2, r(n)\}, \inf\{(r(n) + 1)^2 - n - 1, r(n)\}\big),$$
					is a bijection (with the inverse $(i, j) \mapsto r(\sup\{i, j\})^2 + i + \sup\{i, j\} - j$), filling successively larger squares in $\NN \times \NN$. Declare $s\colon \NN \to \one + X \times Y$,
					$$s(n) \dfeq \begin{cases} \big(s'(b'(n)), s''(b''(n))\big) & \text{ if } s'(b'(n)) \in X \text{ and }  s''(b''(n)) \in Y, \\ \unit & \text{ otherwise}, \end{cases}$$
					and $a\colon \NN \to \NN$, $a(n) \dfeq (\sup\{a'(n), a''(n)\})^2$.
					
					Let $(X_n, d_n, s_n, a_n)_{n \in \NN}$ be a sequence of inhabited totally bounded pseudometric spaces, with $s_n$s mapping into $X_n$s. For all $n \in \NN_{\geq 1}$ let $b_n\colon \NN \to \NN^n$ generalize the bijection $b$, in the sense that they are bijections with the property $b_n(\NN_{< k^n}) = (\NN_{< k})^n$ for all $k \in \NN$. Let $p\colon \coprod_{n \in \NN_{\geq 1}} \NN^n \to \prod_{n \in \NN} X_n$,
					$$\Big(p\big((n_k)_{k \in \NN_{< m}}\big)\Big)_l \dfeq \begin{cases} s_l(n_l) & \text{ if } l < m \\ s_l(0) & \text{ if } l \geq m \end{cases}$$
					for all $l \in \NN$. Declare $s\colon \NN \to \prod_{n \in \NN} X_n$ to be
					$$s(n) \dfeq p\Big(b_{b'(n)}(b''(n))\Big),$$
					and $a\colon \NN \to \NN$, $a(n) \dfeq \big(\sup\st{a_k(n)}{k \in \NN_{< a_n(n)}}\big)^n$.
				\item \proven{separability}
					Same as in the previous item, we just don't have to care about $a$.
				\item \proven{ultrametric}
					Let $(X, d')$ and $(Y, d'')$ be ultrametric spaces. Then
					$$\sup\Big\{d\big((x, y), (z, w)\big), d\big((z, w), (u, v)\big)\Big\} = \sup\{d'(x, z), d''(y, w), d'(z, u), d''(w, v)\} \geq$$
					$$\geq \sup\{d'(x, u), d''(y, v)\} = d\big((x, y), (u, v)\big).$$
					
					Let $(X_n, d_n)_{n \in \NN}$ be a sequence of ultrametric spaces. Then
					$$\sup\Big\{d\big((x_n)_{n \in \NN}, (y_n)_{n \in \NN}\big), d\big((y_n)_{n \in \NN}, (z_n)_{n \in \NN}\big)\Big\} =$$
					$$= \sup\Big\{\sup\st{2^{-n} h(d_n(x_n, y_n))}{n \in \NN}, \sup\st{2^{-n} h(d_n(y_n, z_n))}{n \in \NN}\Big\} =$$
					$$= \sup\st{2^{-n} h\Big(\sup\big\{d_n(x_n, y_n), d_n(y_n, z_n)\big\}\Big)}{n \in \NN} \geq$$
					$$\geq \sup\st{2^{-n} h\big(d_n(x_n, z_n)\big)}{n \in \NN} = d\big((x_n)_{n \in \NN}, (z_n)_{n \in \NN}\big).$$
			\end{itemize}
		\end{proof}
		
		Even though we discussed products, we haven't really said what we consider to be morphisms of metric spaces. One possibility is to simply take all maps, but we usually need more than that.
		\begin{definition}\label{Definition: morphisms_of_metric_spaces}
			Let $(X, d_X)$, $(Y, d_Y)$ be (pseudo)metric spaces. We say that a map $f\colon X \to Y$ between them is:
			\begin{itemize}
				\item
					\df{\ed-continuous} when
					$$\xall{x}{X}\xall{\epsilon}{\RR_{> 0}}\xsome{\delta}{\RR_{> 0}}\xall{y}{X}{d_X(x, y) < \delta \implies d_Y(f(x), f(y)) < \epsilon},$$
					or equivalently, preimages of balls in $Y$ are unions of balls in $X$;
				\item
					\df{uniformly continuous} when
					$$\xall{\epsilon}{\RR_{> 0}}\xsome{\delta}{\RR_{> 0}}\xall{x, y}{X}{d_X(x, y) < \delta \implies d_Y(f(x), f(y)) < \epsilon};$$
				\item
					\df{boundedly uniformly continuous} when it is uniformly continuous on all bounded subsets of $X$, or equivalently, on all balls in $X$;
				\item
					\df{Lipschitz} when
					$$\xsome{C}{\RR_{> 0}}\xall{x, y}{X}{d_Y(f(x), f(y)) \leq C \cdot d_X(x, y)};$$
					any $C$ satisfying this condition is called a \df{Lipschitz constant} or a \df{Lipschitz coefficient};
				\item
					\df{nonexpansive} when it is Lipschitz with coefficient $1$;
				\item
					an \df{isometry} when
					$$\xall{x, y}{X}{d_Y(f(x), f(y)) = d_X(x, y)}.$$
			\end{itemize}
		\end{definition}
		There are familiar implications between these properties. Which one will be most useful to us will depend on case-by-case basis. For instance, in the case of products, we want something between \ed-continuous and nonexpansive maps for finite products, and something between \ed-continuous maps and Lipschitz maps for countable products, in order to accommodate the projections.
		
		\begin{lemma}
			Let $f, g\colon \mtr{X} \to \mtr{Y}$ be \ed-continuous maps between pseudometric spaces $\mtr{X} = (X, d_\mtr{X})$ and $\mtr{Y} = (Y, d_\mtr{Y})$, and let $D \subseteq X$ be a metrically dense subset in $\mtr{X}$. If there exist $x \in X$ such that $d_Y(f(x), g(x)) > 0$, then there exists such an $x$ already in $D$.
		\end{lemma}
		\begin{proof}
			Let $\epsilon \dfeq \frac{d_Y(f(x), g(x))}{2}$. By \ed-continuity of $f$, $g$ there is $\delta \in \RR_{> 0}$ such that
			$$d_\mtr{Y}(f(x), f(y)) < \epsilon \qquad \text{and} \qquad d_\mtr{Y}(g(x), g(y)) < \epsilon$$
			for all $y \in \ball[\mtr{X}]{x}{\delta}$. Pick an $x' \in D \cap \ball[\mtr{X}]{x}{\delta}$. Then
			$$d_Y(f(x'), g(x')) \geq d_Y(f(x), g(x)) - d_Y(f(x), f(x')) - d_Y(g(x), g(x')) > 2 \cdot \epsilon - \epsilon - \epsilon = 0.$$
		\end{proof}
		
		\begin{corollary}\label{Corollary: uniqueness_of_ed_extension_of_a_map_into_a_metric_space}
			Let $\mtr{X} = (X, d_\mtr{X})$ be a pseudometric space, $\mtr{Y} = (Y, d_\mtr{Y})$ a metric space, $D \subseteq X$ a metrically dense subset in $\mtr{X}$, and $f\colon D \to Y$ a map. Then there exists at most one \ed-continuous extension of $f$ to the whole of $X$.
		\end{corollary}
		\begin{proof}
			Use the contrapositive of the previous lemma and the fact that points at zero distance in a metric space are the same.
		\end{proof}
		
		\begin{proposition}\label{Proposition: retract_of_ms_tb}
			Let $\mtr{X} = (X, d_X)$ be a pseudometric space, $A \subseteq X$, and $r\colon X \to A$ a retraction. Let $d_A \dfeq \rstr{d_X}_{A \times A}$, and $\mtr{A} = (A, d_A)$.
			\begin{enumerate}
				\item
					If $\mtr{X}$ is metrically separable and $r$ \ed-continuous, then $\mtr{A}$ is also metrically separable.
				\item
					If $\mtr{X}$ is totally bounded and $r$ nonexpansive, then $\mtr{A}$ is also totally bounded.
			\end{enumerate}
		\end{proposition}
		\begin{proof}
			\begin{enumerate}
				\item
					Let $s_X\colon \NN \to \one + X$ witness separability of $\mtr{X}$. Define $s_A\colon \NN \to \one + A$, $s_A \dfeq (\id[\one] + r) \circ s_X$. We claim that the part of the image of $s_A$ under $A$ is metrically dense in $A$. Take any inhabited ball $B_A(a, \epsilon) \subseteq A$, $a \in A$, $\epsilon \in \RR_{> 0}$. Then $r^{-1}(B_A(a, \epsilon))$ is a union of balls, and since $B_A(a, \epsilon) \subseteq r^{-1}(B_A(a, \epsilon))$, it is inhabited. Therefore there exists $n \in \NN$ such that $s_X(n) \in r^{-1}(B_A(a, \epsilon))$, so $s_A(n) \in B_A(a, \epsilon)$.
				\item
					Suppose $s_X\colon \NN \to \one + X$, $a_X\colon \NN \to \NN$ witness total boundedness of $\mtr{X}$. Let $s_A\colon \NN \to \one + A$, $s_A \dfeq (\id[\one] + r) \circ s_X$, and $a_A\colon \NN \to \NN$, $a_A \dfeq a_X$. Take any $n \in \NN$ and $x \in A$. Then there exists $k \in \NN_{< a_X(n)}$ such that $d_X(s_X(k), x) < 2^{-n}$. But then
					$$d_A(s_A(k), x) = d_A(r(s_X(k)), r(x)) \leq d_X(s_X(k), x) < 2^{-n}.$$
			\end{enumerate}
		\end{proof}
		
		We are used to isometries between metric spaces being injective, and indeed isometries with metric domain are injective even in the strong sense. To see this, recall that for any pseudometric space $(X, d)$, the relation $\apart$, defined by $x \apart y \dfeq d(x, y) > 0$ for $x, y \in X$, is an apartness relation on $X$ which is tight precisely when $d$ is a metric (if $X$ is a subset of $\RR$, equipped with the Euclidean topology, this apartness relation coincides with the one induced by $<$). If $i\colon (X, d') \to (Y, d'')$ is an isometry where $(X, d')$ is a metric and $(Y, d'')$ a pseudometric space, then for all $x, y \in X$,
		$$x \apart y \implies d'(x, y) > 0 \implies d''(i(x), i(y)) > 0 \implies i(x) \apart i(y),$$
		and the contrapositive indeed implies injectivity since $\apart$ is tight on $X$.
		
		Isometries between pseudometric spaces need not be injective in general, and there is an important example of this. Recall that for any pseudometric space $\mtr{X} = (X, d)$ the relation $\equ$, defined for $x, y \in X$ by $x \equ y \dfeq (d(x, y) = 0)$, is an equivalence relation, and the pseudometric $d$ induces a metric $\Kqm$ on $X/_\equ$, defined by $\Kqm([x], [y]) \dfeq d(x, y)$. The metric space $\Kq{\mtr{X}} \dfeq (X/_\equ, \Kqm)$ is called the \df{Kolmogorov quotient}\footnote{In classical general topology the Kolmogorov quotient of a topological space is constructed by identifying points which have the same neighbourhoods, thus obtaining a $T_0$ space. In pseudometric spaces the points with the same neighbourhoods are precisely those at zero distance.} of $\mtr{X}$. The Kolmogorov quotient map $q\colon \mtr{X} \to \Kq{\mtr{X}}$ is always a surjective isometry, and it is injective (hence bijective) if and only if $\mtr{X}$ is a metric space. In fact, it has the following universal property: every map $f\colon \mtr{X} \to \mtr{Y}$ of any type from Definition~\ref{Definition: morphisms_of_metric_spaces} with a metric space for the codomain $\mtr{Y}$ factors through $q$, in the sense that there exists a map $g\colon \Kq{\mtr{X}} \to \mtr{Y}$ such that $f = g \circ q$; moreover, such $g$ is unique (since $q$ is an epimorphism), and of the same type as $f$. This implies that the category of metric spaces is a reflective subcategory of the category of pseudometric spaces, with the Kolmogorov quotient being the object part of the left adjoint functor to the inclusion of metric spaces among pseudometric ones.
		
		Isometries with some further properties will be of special interest.
		\begin{definition}
			Let $\mtr{X} = (X, d_\mtr{X})$, $\mtr{Y} = (Y, d_\mtr{Y})$ be pseudometric spaces, and $i\colon \mtr{X} \to \mtr{Y}$ an isometry. We say that $i$ is
			\begin{itemize}
				\item
					a \df{dense isometry} when it has a metrically dense image in $Y$,
				\item
					a \df{space-filling isometry} when $q \circ i$ is surjective (where $q$ is the Kolmogorov quotient map); equivalently, when the condition $\xall{y}{Y}\xsome{x}{X}{d_\mtr{Y}(i(x), y) = 0}$ holds,
				\item
					an \df{isometric embedding} when it is injective,
				\item
					an \df{isometric isomorphism} when it is bijective.
			\end{itemize}
		\end{definition}
		Let $A \subseteq X$, and let $\mtr{X} = (X, d_\mtr{X})$, $\mtr{A} = (A, d_\mtr{A})$ be pseudometric spaces. We say that $\mtr{A}$ is a \df{metric subspace} of $\mtr{X}$ when the pseudometric on $\mtr{A}$ is a restriction of the pseudometric on $\mtr{X}$ onto $A$, \ie $d_\mtr{A} = \rstr{d_\mtr{X}}_{A \times A}$; equivalently, when the inclusion $\mtr{A} \hookrightarrow \mtr{X}$ is an isometric embedding.
		
		Isometric isomorphisms deserve their name since if an isometry is bijective, its inverse is also an isometry. Thus they preserve all (pseudo)metric properties both ways, so when we are given a specific isometric isomorphism $\mtr{X} \to \mtr{Y}$, we can identify spaces $\mtr{X}$ and $\mtr{Y}$, and when we are given a specific isometric embedding $\mtr{X} \to \mtr{Y}$, we can identify $\mtr{X}$ with a metric subspace of $\mtr{Y}$ (and write $\mtr{X} \subseteq \mtr{Y}$).

	\section{Metric Completions}\label{Section: completions}
	
		One of the most important properties of metric spaces is completeness, and it warrants a separate section. We start by defining what a completion of a metric (and more generally, pseudometric) space is first. The idea is similar to how we define real numbers: the completion of a metric space $\mtr{X}$ should be the largest metric space into which we can densely isometrically embed $\mtr{X}$. In this section we assume metrics to have $\RR$ as the codomain, otherwise we have an additional technical complication of completing the codomain of a metric as well.
		
		\begin{lemma}\label{Lemma: factorization_of_dense_isometries}
			Let $(X, d)$, $(Y, d')$, $(Z, d'')$ be pseudometric spaces, and let $f\colon X \to Y$, $g\colon X \to Z$ be dense isometries. Then any \ed-continuous map $h\colon Y \to Z$ which satisfies $g = h \circ f$, is a dense isometry. Moreover, if $(Z, d'')$ is a metric space, then if such a map exists, it is unique.
		\end{lemma}
		\begin{proof}
			Let us first prove that $h$ is an isometry. Suppose there exist $x, y \in Y$ such that $d''(h(x), h(y)) < d'(x, y)$. Let $\epsilon \dfeq \frac{1}{4}\big(d'(x, y) - d''(h(x), h(y))\big) > 0$, and let $\delta, \gamma \in \intoc{0}{\epsilon}$ be such that for all $z \in Y$,
			$$d'(x, z) < \delta \implies d''(h(x), h(z)) < \epsilon \qquad \text{and} \qquad d'(y, z) < \gamma \implies d''(h(y), h(z)) < \epsilon.$$
			Since $f$ has a dense image in $Z$, there exist $a, b \in X$ such that $d'(x, f(a)) < \delta$ and $d'(y, f(b)) < \gamma$. But then
			$$d'(x, y) \leq d'(x, f(a)) + d'(f(a), f(b)) + d'(f(b), y) < \delta + d(a, b) + \gamma \leq d''(g(a), g(b)) + 2 \epsilon =$$
			$$= d''(h(f(a)), h(f(b))) + 2 \epsilon \leq d''(h(x), h(f(a))) + d''(h(x), h(y)) + d''(h(y), h(f(b))) + 2 \epsilon \leq$$
			$$\leq d''(h(x), h(y)) + 4\epsilon = d'(x, y),$$
			a contradiction. The case $d''(h(x), h(y)) > d'(x, y)$ is dealt with in a similar fashion, so $d''(h(x), h(y)) = d'(x, y)$ by tightness.
			
			To see that $h$ has a dense image, take any $z \in Z$ and $r \in \RR_{> 0}$. There exists $a \in X$ such that $d''(z, g(a)) < r$, but then the element $f(a) \in Y$ is such that $d(z, h(f(a))) < r$.
			
         Uniqueness of $h$ follows from Corollary~\ref{Corollary: uniqueness_of_ed_extension_of_a_map_into_a_metric_space}.
		\end{proof}
		
		\begin{definition}
			Let $\mtr{X}$ be a pseudometric space. We say that a pseudometric space $\cmtr{X}$, together with a dense isometry $i\colon \mtr{X} \to \cmtr{X}$, is the \df{completion} of $\mtr{X}$ when any dense isometry with domain $\mtr{X}$ uniquely factors through it, \ie for any pseudometric space $\mtr{Y}$ and a dense isometry $f\colon \mtr{X} \to \mtr{Y}$ there exists a unique isometry $h\colon \mtr{Y} \to \cmtr{X}$ such that $i = h \circ f$. Since $h$ must have a dense image, we can rephrase this as follows: if $\mathcal{M}$ is the category of pseudometric spaces and dense isometries, then the completion of $\mtr{X}$ is the terminal object in the coslice category $\mtr{X}/\mathcal{M}$.
			
			A pseudometric space is \df{complete} when its identity map is its completion; equivalently, when its completion is an isomorphism.
		\end{definition}
		
		When the dense isometry $i\colon \mtr{X} \to \cmtr{X}$ is understood, we often simply say that the completion of $\mtr{X}$ is just the space $\cmtr{X}$. Clearly, if $f\colon \mtr{X} \to \mtr{Y}$ is a dense isometry, then $\mtr{X}$ and $\mtr{Y}$ have the ``same'' completion, in the sense that if $j\colon \mtr{Y} \to \mtr{Z}$ is the completion of $\mtr{Y}$, then $j \circ f\colon \mtr{X} \to \mtr{Z}$ is the completion of $\mtr{X}$, and if $i\colon \mtr{X} \to \mtr{Z}$ is the completion of $\mtr{X}$ and $j\colon \mtr{Y} \to \mtr{Z}$ the unique isometry for which $i = j \circ f$, then $j$ is the completion of $\mtr{Y}$. In particular, a completion of a space is complete. Also, a completion is always a metric space since the Kolmogorov quotient map is a surjective, hence dense, isometry.
		
		\begin{proposition}\label{Proposition: characterization_of_completeness}
			The following statements are equivalent for a metric space $\mtr{X} = (X, d)$.
			\begin{enumerate}
				\item
					$\mtr{X}$ is complete.
				\item
					For every metric space $\mtr{Y} = (Y, d')$, every dense isometry $\mtr{X} \to \mtr{Y}$ is an isometric isomorphism.
				\item
					For every pseudometric space $\mtr{Y} = (Y, d')$, every dense isometry $\mtr{X} \to \mtr{Y}$ is a space-filling isometry.
			\end{enumerate}
		\end{proposition}
		\begin{proof}
			\begin{itemize}
				\item\itemimpl{1}{2}
					If $\mtr{X}$ is complete, then $\id[X]$ is its completion, so for every dense isometry $f\colon \mtr{X} \to \mtr{Y}$ there exists a dense isometry $g\colon \mtr{Y} \to \mtr{X}$ such that $g \circ f = \id[X]$. This implies $g$ is surjective, and if $\mtr{Y}$ is a metric space, it is also injective. It is therefore an isometric isomorphism, but then so is $f$.
				\item\itemimpl{2}{3}
					Compose the dense isometry $\mtr{X} \to \mtr{Y}$ with the Kolmogorov quotient map, then use the assumption.
				\item\itemimpl{3}{1}
					Take a dense isometry $f\colon \mtr{X} \to \mtr{Y}$. Define $g\colon \mtr{Y} \to \mtr{X}$ by declaring $g(y)$ to be that $x \in X$ for which $d'(f(x), y) = 0$. This map is total since $f$ is space-filling, and well-defined: if for $x' \in X$ we also have $d'(f(x'), y) = 0$, then $0 \leq d(x, x') = d'(f(x), f(x')) \leq d'(f(x), y) + d'(f(x'), y) = 0$, so $x = x'$ because $\mtr{X}$ is metric. The map $g$ is an isometry since for $y, y' \in Y$,
					$$d(g(y), g(y')) = d'(f(g(y)), f(g(y'))) \leq d'(f(g(y)), y) + d'(y, y') + d'(y', f(g(y'))) = d'(y, y'),$$
					$$d(g(y), g(y')) = d'(f(g(y)), f(g(y'))) \geq d'(y, y') - d'(f(g(y)), y) - d'(y', f(g(y'))) = d'(y, y').$$
					Clearly $g \circ f = \id[X]$ which implies that $g$ is a surjective, therefore dense, isometry, and $\id[X]$ is the completion of $\mtr{X}$.
			\end{itemize}
		\end{proof}
		
		\begin{lemma}\label{Lemma: decidable_subsets_in_complete_spaces}
			Let $\mtr{Z} = (X + Y, d)$ be a pseudometric space, $d'$ the restriction of $d$ onto $X$, $d''$ the restriction of $d$ onto $Y$, and $\mtr{X} = (X, d')$, $\mtr{Y} = (Y, d'')$. Assume that there exists $\delta \in \RR_{> 0}$ such that for all $x \in X$ and $y \in Y$ we have $d(x, y) \geq \delta$. Then:
			\begin{enumerate}
				\item
					If both $\mtr{X}$ and $\mtr{Y}$ are complete, so is $\mtr{Z}$.
				\item
					If $\mtr{Z}$ is complete and $\mtr{X}$ contains a metrically dense subovert subset, then $\mtr{X}$ is complete.
         \end{enumerate}
		\end{lemma}
		\begin{proof}
			\begin{enumerate}
            \item
               Assume $\mtr{X}$ and $\mtr{Y}$ are complete, and consider a dense isometry $f\colon \mtr{Z} \to \mtr{W}$ where $\mtr{W} = (W, d_\mtr{W})$ is some pseudometric space. Let
               $$A \dfeq \st{w \in W}{\xall{\epsilon}{\RR_{> 0}}\xsome{x}{X}{d_\mtr{W}(f(x), w) < \epsilon}}.$$
               We claim that $A$ is a decidable subset of $W$. Take any $w \in W$. Because $f$ is dense, there exists $z \in X + Y$ such that $d_\mtr{W}(f(z), w) < \frac{\delta}{2}$. If $z \in Y$, then for any $x \in X$,
               $$d_\mtr{W}(f(x), w) \geq d_\mtr{W}(f(x), f(z)) - d_\mtr{W}(f(z), w) = d(x, z) - d_\mtr{W}(f(z), w) > \delta - \frac{\delta}{2} = \frac{\delta}{2},$$
               implying that $w \notin A$. Assume now $z \in X$, take any $\epsilon \in \RR_{> 0}$, and let $\epsilon' \dfeq \inf\{\epsilon, \frac{\delta}{2}\}$. Let $x \in X + Y$ be such that $d_\mtr{W}(f(x), w) < \epsilon' \leq \epsilon$. Then $x \in X$ since
               $$d(x, z) = d_\mtr{W}(f(x), f(z)) \leq d_\mtr{W}(f(x), w) + d_\mtr{W}(f(z), w) < \frac{\delta}{2} + \frac{\delta}{2} = \delta,$$
               proving that $w \in A$. We can similarly prove
               $$\scom{A} = \st{w \in W}{\xall{\epsilon}{\RR_{> 0}}\xsome{y}{Y}{d_\mtr{W}(f(y), w) < \epsilon}}.$$
               Thus the restrictions $\rstr{f}_X^A$ and $\rstr{f}_Y^{\scom{A}}$ are dense isometries, and therefore space-filling by Proposition~\ref{Proposition: characterization_of_completeness} since $\mtr{X}$ and $\mtr{Y}$ are complete. Then $f$ is a space-filling isometry as well, so $\mtr{Z}$ is complete.
            \item
               Assume that $\mtr{Z}$ is complete, and that $D \subseteq X$ is a metrically dense subovert in $\mtr{X}$. Let $\mtr{W} = (W, d_\mtr{W})$ be a pseudometric space, and $f\colon \mtr{X} \to \mtr{W}$ a dense isometry. Let $\mtr{V}$ be the pseudometric space with the underlying set $W + Y$ and the pseudometric which restricts to $d_\mtr{W}$ on $W$ and to $d''$ on $Y$ while the distance between $a \in W$, $b \in Y$ is defined as $(L, U)$ where
               $$L \dfeq \st{q \in \oirs}{\xsome{x}{D}{q + d_\mtr{W}(a, f(x)) < d(x, b)}},$$
               $$U \dfeq \st{r \in \oirs}{\xsome{x}{D}{r > d_\mtr{W}(a, f(x)) + d(x, b)}}.$$
               One can verify that $(L, U)$ is an open Dedekind cut in the standard way, that we defined a pseudometric on $W + Y$, and that the distance between points in $W$ and points in $Y$ is at least $\delta$. Moreover, $f + \id[Y]$ is a dense isometry, and since $\mtr{Z}$ is complete, it is space-filling. Then $f$ is a space-filling isometry itself, thus proving that $\mtr{X}$ is complete.
         \end{enumerate}
		\end{proof}
		
		Having now all the metric properties that we need, we prove a technical lemma that will be useful later on.
		\begin{lemma}\label{Lemma: wlog_space_inhabited}
			Let $\mtr{X} = (X, d, s)$ be a metrically separable pseudometric space. Then there exists a canonical choice\footnote{`Canonical' in this case means that we have an actual functor from the category of separable pseudometric spaces (with morphisms at least \ed-continuous and at most Lipschitz) to the category of separable pseudometric spaces with a chosen basepoint and maps of the same type which preserve the basepoint. In fact, this functor is the left adjoint to the forgetful functor which forgets the basepoint, and the isometric embedding $e$ is (a component of) the unit of the adjunction.} of an inhabited pseudometric space $\mtr{Y}$, and an isometric embedding $e\colon \mtr{X} \to \mtr{Y}$ such that the image of $e$ is a decidable subset (\ie $\mtr{X}$ can be identified with a decidable metric subspace of $\mtr{Y}$), the distances between elements in the image of $e$ and its complement are bounded below by $1$, and:
			\begin{enumerate}
				\item
					$\mtr{Y}$ is metrically separable,
				\item
					$\mtr{X}$ is totally bounded if and only if $\mtr{Y}$ is,
				\item
					$\mtr{X}$ is metric if and only if $\mtr{Y}$ is,
				\item
					$\mtr{X}$ is ultrametric if and only if $\mtr{Y}$ is,
				\item
					$\mtr{X}$ is complete if and only if $\mtr{Y}$ is.
			\end{enumerate}
		\end{lemma}
		\begin{proof}
			Let $Y \dfeq \one + X$, and define the map $d'\colon (\one + X) \times (\one + X) \to \RR$ for $a, b \in \one + X$ as follows:
			\begin{itemize}
				\item
					if $a = b = \unit \in \one$, then $d'(a, b) \dfeq 0$,
				\item
					if $a, b \in X$, then $d'(a, b) \dfeq d(a, b)$,
				\item
					if exactly one of $a$, $b$, say $a$, is in $X$, then the image of $s$ intersects $X$ which is a decidable subset of $\one + X$, so there exists a unique $m \in \NN$ such that $s_m \in X$, and for all $n \in \NN_{< m}$, $s_n = \unit$; define $d'(a, b) \dfeq \sup\{d(a, s_m), 1\}$.
			\end{itemize}
			It is easy to see that $d'$ is a pseudometric on $Y$; we verify only the interesting examples of triangle inequality. Let $x, y \in X$; then
			$$d'(x, y) + d'(y, \unit) = d(x, y) + \sup\{d(y, s_m), 1\} =$$
			$$= \sup\{d(x, y) + d(y, s_m), d(x, y) + 1\} \geq \sup\{d(x, s_m), 1\} = d'(x, \unit),$$
			$$d'(x, \unit) + d'(\unit, y) = \sup\{d(x, s_m), 1\} + \sup\{d(y, s_m), 1\} \geq$$
			$$\geq d(x, s_m) + d(y, s_m) \geq d(x, y) = d'(x, y).$$
			Let $\mtr{Y} = (Y, d')$, and let $e$ be the canonical inclusion $X \hookrightarrow \one + X$; clearly it is an isometric embedding onto a decidable subset of $Y$.
			
			\begin{enumerate}
				\item
					The sequence $s'\colon \NN \to Y$,
					$$s'(n) \dfeq \begin{cases} \unit & \text{ if } n = 0,\\ s(n-1) & \text{ if } n \geq 1, \end{cases}$$
					demonstrates metric separability of $\mtr{Y}$.
				\item
					If $s\colon \NN \to \one + X$ and $a\colon \NN \to \NN$ witness total boundedness of $\mtr{X}$, then $s'$ defined in the previous item, and $a'\colon \NN \to \NN$, $a'(n) \dfeq a(n) + 1$, witness total boundedness of $\mtr{Y}$.
					
					Conversely, if $s'\colon \NN \to Y$ and $a'\colon \NN \to \NN$ make $\mtr{Y}$ totally bounded, then they demonstrate total boundedness of $\mtr{X}$ as well.
				\item
					Obvious.
				\item
					Assume $\mtr{X}$ is ultrametric. We verify only the interesting examples of strong triangle inequality for $\mtr{Y}$. Let $x, y \in X$; then
					$$\sup\{d'(x, y), d'(y, \unit)\} = \sup\{d(x, y), \sup\{d(y, s_m), 1\}\} =$$
					$$= \sup\{\sup\{d(x, y), d(y, s_m)\}, 1\} \geq \sup\{d(x, s_m), 1\} = d'(x, \unit),$$
					$$\sup\{d'(x, \unit), d'(\unit, y)\} = \sup\{\sup\{d(x, s_m), 1\}, \sup\{d(y, s_m), 1\}\} \geq$$
					$$\geq \sup\{d(x, s_m), d(y, s_m)\} \geq d(x, y) = d'(x, y).$$
					The rest is obvious.
				\item
					Clearly $\one$ with its only possible (pseudo)metric is complete, and $\mtr{X}$ contains a metrically dense countable, therefore overt, subset. The result then follows from Lemma~\ref{Lemma: decidable_subsets_in_complete_spaces}.
			\end{enumerate}
		\end{proof}
		
		We now recall some constructions of completions, and the fact we defined the notion of completion via universal property means the resulting metric spaces are isometrically isomorphic in a canonical way.
		
		The standard classical construction of completion is via Cauchy sequences. As is well-known, this still works constructively if countable choice is assumed~\cite{Troelstra_AS_Dalen_D_1988:_constructivism_in_mathematics_volume_1}. We recall the construction and the proof.
		
		Let $\mtr{X} = (X, d)$ be a pseudometric space. The set of Cauchy sequences in $\mtr{X}$ is defined
		$$\Cauchy(\mtr{X}) \dfeq \st{(s, m) \in X^\NN \times \NN^\NN}{\xall{n}{\NN}\xall{i, j}{\NN_{\geq m(n)}}{d(s_i, s_j) < 2^{-n}}}.$$
		Observe that this definition matches with the definition of Cauchy sequences in streaks if we equip them with the Euclidean metric; in particular, $\Cauchy(\RR)$ is unambiguous. The pseudometric $d\colon X \times X \to \RR$ induces the map
		$$\Cauchy(d)\colon \Cauchy(\mtr{X}) \times \Cauchy(\mtr{X}) \to \Cauchy(\RR),$$
		$$\Cauchy(d)\Big((s, m), (s', m')\Big) \dfeq \Big(d \circ (s, s'), \ n \mapsto \sup\{m(n+1), m'(n+1)\}\Big)$$
		since for $n \in \NN$ and $i, j \in \NN_{\geq \sup\{m(n+1), m'(n+1)\}}$,
		$$|d(s_i, s'_i) - d(s_j, s'_j)| \leq d(s_i, s_j) + d(s'_i, s'_j) < 2^{-n-1} + 2^{-n-1} = 2^{-n}.$$
      Recall the limit operator $\lim\colon \Cauchy(\RR) \to \RR$ from the end of Section~\ref{Section: real_numbers}. We claim that $\lim \circ \Cauchy(d)$ is a pseudometric on $\Cauchy(\mtr{X})$. Indeed, when calculating the distance of $(s, m)$ to itself, we obtain the zero sequence which, as we recall, has limit $0$. Symmetry is obvious. To verify triangle inequality, let $(s, m), (s', m'), (s'', m'') \in \Cauchy(\mtr{X})$. Then
		$$\lim\Big(d \circ (s, s'), \ n \mapsto \sup\{m(n+1), m'(n+1)\}\Big) + \lim\Big(d \circ (s', s''), \ n \mapsto \sup\{m'(n+1), m''(n+1)\}\Big) =$$
		$$= \lim\Big(d \circ (s, s') + d \circ (s', s''), n \mapsto \sup\big\{\sup\{m(n+2), m'(n+2)\}, \sup\{m'(n+2), m''(n+2)\}\big\}\Big) \geq$$
		$$\geq \lim\Big(d \circ (s, s''), n \mapsto \sup\big\{m(n+2), m'(n+2), m''(n+2)\big\}\Big) =$$
		$$= \lim\Big(d \circ (s, s''), n \mapsto \sup\big\{m(n+1), m''(n+1)\big\}\Big).$$
		
		Let $\Ccmpl(\mtr{X})$ be the Kolmogorov quotient of the pseudometric space $(\Cauchy(\mtr{X}), \lim \circ \Cauchy(d))$, and $q\colon \Cauchy(\mtr{X}) \to \Ccmpl(\mtr{X})$ the quotient map. Furthermore, let $\cs[X]\colon X \to \Cauchy(\mtr{X})$ be the map which assigns to $x \in X$ the constant sequence with terms $x$, along with an arbitrary modulus of convergence (say $\id[\NN]$), and let $i\colon \mtr{X} \to \Ccmpl(\mtr{X})$ be the composition $i \dfeq q \circ \cs$. Since both $\cs[X]$ and $q$ are isometries, so is $i$. Now take any $(s, m) \in \Cauchy(\mtr{X})$, and $\epsilon \in \RR_{> 0}$. Let $n \in \NN$ be such that $2^{-n} < \epsilon$. Then the distance between $(s, m)$ and $\cs(s_{m(n)+1})$ is $< \epsilon$ which proves that $i$ is a dense isometry.
      
      We say that $\mtr{X}$ is \df{Cauchy complete} when $i$ is a bijection. In this case we have the limit operator $\lim\colon \Cauchy(\mtr{X}) \to \mtr{X}$, defined as the composition $i^{-1} \circ q$. Notice that when equipped with the Euclidean metric, the reals $\RR$ are Cauchy complete by this definition also, and the limit operator we defined at the end of Section~\ref{Section: real_numbers} is a special case of this one, so there is no ambiguity in definitions.
		
      Since $i$ is a dense isometry, a complete space is also Cauchy complete. For a Cauchy completion to be a completion however, we need in general some form of countable choice.
		\begin{proposition}\label{Proposition: completion_by_Cauchy_sequences}
			Suppose $\AC[\opn]{\NN}$. Then for any metric space $\mtr{X}$, the space $\Ccmpl(\mtr{X})$, together with the map $i$, is the completion of $\mtr{X}$.
		\end{proposition}
		\begin{proof}
			Let $\mtr{X} = (X, d)$, $\mtr{Y} = (Y, d')$ be metric spaces, and $f\colon \mtr{X} \to \mtr{Y}$ a dense isometry. Fix $y \in Y$. Then for any $n \in \NN$ there exists $s^y_n \in X$ such that $d'(y, f(s^y_n)) < 2^{-n-1}$. By $\AC[\opn]{\NN}$ this defines a sequence $s^y\colon \NN \to X$. Notice that $(s^y, \id[\NN]) \in \Cauchy(\mtr{X})$ since for any $n \in \NN$, $i, j \in \NN_{\geq n}$ we have
			$$d(s^y_i, s^y_j) = d'(f(s^y_i), f(s^y_j)) \leq d'(f(s^y_i), y) + d'(y, f(s^y_j)) < 2^{-i} + 2^{-j} \leq 2^{-n-1} + 2^{-n-1} = 2^{-n}.$$
			Essentially the same proof shows also that if $t^y\colon \NN \to X$ is another such sequence, then $(s^y, \id[\NN]) \equ (t^y, \id[\NN])$. Hence, $g\colon \mtr{Y} \to \Cauchy(\mtr{X})/_\equ$, $g(y) \dfeq [s^y]$, is a well-defined map. Moreover, if $y = f(x)$ for some $x \in X$, then we can take for $s^y$ the constant sequence with terms $x$ which shows $i = g \circ f$.
		\end{proof}
		
		We want a more general construction, though, one that does not require choice principles. Notice that in the proof above we only required choice when selecting a point which maps into a ball around $y$. These elements of the Cauchy sequence are meant to be better and better approximations to $y$. However, if we allow multiple points to approximate $y$ at each stage, the need for choice disappears. This is what Fred Richman did in~\cite{Richman_F_2008:_real_numbers_and_other_completions}. We present (an adaptation of) his construction here. 
		
      Fix a set $\lvl \subseteq \RR_{> 0}$, $\lvl$ subovert in $\RR$, with a strict infimum $0$ (in particular $\lvl$ must be inhabited). This set will index the accuracy of the approximations. A simple choice for $\lvl$ is the set $\st{2^{-n}}{n \in \NN}$; this is essentially what we did with Cauchy sequences. On the other extreme we could take the whole $\RR_{> 0}$ (if $\RR$ is overt). Richman himself chose $\QQ_{> 0}$ because he wanted to define the real numbers as the completion of the rationals, without already referring to the real numbers in the process. If we followed suit, we could take $\lvl = \oirs_{> 0}$.
		
		Let $\mtr{X} = (X, d)$ be a pseudometric space, and let $\inhov(X)$ be the set of inhabited subovert subsets of $X$. Define the relation $\equ$ on the set $\inhov(X)^\lvl$ as follows: for $S, T\colon \lvl \to \inhov(X)$,
		$$S \equ T \dfeq \xall{p, q}{\lvl}\xall{s}{S_p}\xall{t}{T_q}{d(s, t) \leq p + q}.$$
		The relation $\equ$ is obviously symmetric; it is a simple exercise to verify that it is transitive. Thus it is a partial equivalence relation. Call the elements of its domain \df{Cauchy families}, and denote their set
		$$\Cf(\mtr{X}) \dfeq \st{S \in \inhov(X)^\lvl}{S \equ S}.$$
		Now define the map $d_\Cf\colon \Cf(\mtr{X}) \times \Cf(\mtr{X}) \to \RR$ by $d_\Cf(S, T) \dfeq (L, U)$ where
		$$L \dfeq \st{a \in \oirs}{\xsome{p, q}{\lvl}\xsome{s}{S_p}\some{t}{T_q}{a + p + q < d(s, t)}},$$
		$$U \dfeq \st{b \in \oirs}{\xsome{p, q}{\lvl}\xsome{s}{S_p}\some{t}{T_q}{d(s, t) + p + q < b}}.$$
		It is easy to see that $L$ is a lower and $U$ is an upper cut, and that they are inhabited. They are open because we assumed $\lvl$, $S_p$ and $T_q$ are subovert. Now assume that we have $x \in L \cap U$. Then there are $p, q, p', q' \in \lvl$, $s \in S_p$, $t \in T_q$, $s' \in S_{p'}$, $t' \in T_{q'}$ such that $x + p + q < d(s, t)$, $d(s', t') + p' + q' < x$. Consequently
		$$x + p + q < d(s, t) \leq d(s, s') + d(s', t') + d(t', t) \leq p + p' + x - p' - q' + q + q' = x + p + q,$$
		a contradiction, so $L \cap U = \emptyset$. Finally, let $a, b \in \oirs$, $a < b$. Because $0$ is the strict infimum of $\lvl$ there is $p \in \lvl$ such that $p < \frac{1}{5}(b-a)$. Pick some $s \in S_p$, $t \in T_p$. Then $a + \frac{2}{5}(b-a) < d(s, t)$ or $d(s, t) < a + \frac{3}{5}(b-a)$. In the first case, $a + 2 p < d(s, t)$, so $a \in L$. In the second case, $d(s, t) + 2 p < b$, so $b \in U$. This concludes the proof that $(L, U)$ is an open Dedekind cut.
		
		Clearly, $d_\Cf$ is symmetric. As for triangle inequality, take $S, T, U \in \Cf(\mtr{X})$. It is sufficient to prove for every $a, b \in \oirs$
		$$\xsome{p, q}{\lvl}\xsome{s}{S_p}\some{t}{T_q}{d(s, t) + p + q < a} \hspace{1ex} \land$$
		$$\land \hspace{1ex} \xsome{q', r'}{\lvl}\xsome{t'}{T_{q'}}\some{u}{U_{r'}}{d(t', u') + q' + r' < b} \hspace{1ex} \implies$$
		$$\implies \hspace{1ex} \xsome{p'', r''}{\lvl}\xsome{s''}{S_{p''}}\some{u''}{U_{r''}}{d(s'', u'') + p'' + r'' < a+b},$$
		but that is obvious since we can take $p'' = p$, $r'' = r'$, $s'' = s$, $u'' = u'$.
		
		Thus $d_\Cf$ is a protometric on $\Cf(\mtr{X})$. Notice that for $S, T \in \Cf(\mtr{X})$ we have
		$$S \equ T \iff \xall{p, q}{\lvl}\xall{s}{S_p}\xall{t}{T_q}{d(s, t) \leq p + q} \iff$$
		$$\iff \lnot\xsome{p, q}{\lvl}\xsome{s}{S_p}\xsome{t}{T_q}{d(s, t) > p + q} \iff$$
		$$\iff \lnot\xsome{\epsilon, p, q}{\lvl}\xsome{s}{S_p}\xsome{t}{T_q}{d(s, t) > \epsilon + p + q} \iff \lnot d(S, T) > 0 \iff d(S, T) = 0.$$
		We conclude that $(\Cf(\mtr{X}), d_\Cf)$ is a pseudometric space, with $\equ$ matching the equivalence relation by which we obtain the Kolmogorov quotient.
		
		Let $\bar{\i}\colon X \to \Cf(\mtr{X})$ map $x \in X$ to the constant map which takes every $a \in \lvl$ to $\{x\}$. It is easy to see that $\bar{\i}$ is a dense isometry.
		
		\begin{proposition}\label{Proposition: completion_by_Cauchy_families}
			Let $\mtr{X} = (X, d)$ be a pseudometric space with a subovert metrically dense subset $D \subseteq X$. Let $\Rcmpl(\mtr{X})$ be the Kolmogorov quotient of $(\Cf(\mtr{X}), d_\Cf)$, and $q\colon \Cf(\mtr{X}) \to \Rcmpl(\mtr{X})$ the quotient map.
			\begin{enumerate}
				\item
					The space $\Rcmpl(\mtr{X})$, together with the map $i \dfeq q \circ \bar{\i}\colon \mtr{X} \to \Rcmpl(\mtr{X})$, is the completion of $\mtr{X}$.
				\item
					Equivalence classes in $\Rcmpl(\mtr{X})$ have canonical representatives; in other words, $q$ is a retraction.
			\end{enumerate}
		\end{proposition}
		\begin{proof}
			\begin{enumerate}
				\item
					Let $\mtr{Y} = (Y, d')$ be a pseudometric space, and $f\colon \mtr{X} \to \mtr{Y}$ a dense isometry. Define $\bar{g}\colon \mtr{Y} \to \Cf(\mtr{X})$ by
					$$\bar{g}(y) \dfeq \Big(p \mapsto f^{-1}(B(y, p)) \cap D\Big) = \Big(p \mapsto \st{x \in D}{\xsome{a}{D}{d(x, a) + d'(f(a), y) < p}}\Big).$$
					Inhabitedness of $f^{-1}(B(y, p)) \cap D$ follows from the fact that $D$ is metrically dense in $Y$. Because it is an intersection of an open and a subovert subset, it is subovert in $X$. Thus $\bar{g}$ indeed maps into $\Cf(\mtr{X})$.
					
					It is obvious from the definition of the equivalence relation $\equ$ on $\Cf(\mtr{X})$ that $\bar{g}(f(x)) \equ \bar{\i}(x)$, \ie $q(\bar{g}(f(x))) = q(\bar{\i}(x))$, for all $x \in X$. Therefore the map $g\colon \mtr{Y} \to \Rcmpl(\mtr{X})$, $g \dfeq q \circ \bar{g}$, satisfies the required condition $i = g \circ f$.
				\item
					In the previous item, take $(Y, d') = \Rcmpl(\mtr{X})$ and $f = i$. The map $\bar{g}$ is a splitting of $q$ since $q \circ \bar{g} \circ i = i = \id[\Rcmpl(\mtr{X})] \circ i$ and $i$ is a dense isometry (use Lemma~\ref{Lemma: factorization_of_dense_isometries}).\footnote{Richman provides a ``more canonical'' splitting by taking the maximal Cauchy family from every equivalence class, namely $p \mapsto i^{-1}(\bball{[S]}{p})$. We chose a smaller family $p \mapsto i^{-1}(\ball{[S]}{p}) \cap D$ because we have an additional condition of subovertness to satisfy.}
			\end{enumerate}
		\end{proof}
		
		To examine the situation, while we have avoided the need for choice principles by using Cauchy families for completion, we could only complete pseudometric spaces with subovert metrically dense subsets, for much the same reason why we required overt streaks to make Dedekind cuts from. The situation is actually not as bad as it sound. First of all, since countable sets are overt, we can complete all metrically separable spaces. Second, just because the proof of the above proposition doesn't necessarily work for general (pseudo)metric spaces, that doesn't mean we couldn't in a particular instance of $\mtr{X}$ find some other way to construct subovert approximations; for example, if we can actually find Cauchy sequences in $\mtr{X}$ approximating points in $\mtr{Y}$, it is enough. And third, we only need subovert approximations in order to prove that the Dedekind cut $d_\Cf(S, T)$ is open. If $\RR = \RR_d$, then we can drop this condition, and can complete arbitrary metric spaces in this way.
		
		The above construction of completions is sufficient for our purposes since we are interested in complete metrically separable spaces in a topos. However, we remark that it wouldn't work in a more general environment, as the definition of Cauchy families is impredicative since in general there is no telling that inhabited overt subsets form a set. Actually, there is another construction of metric completion~\cite{Richman_F_2000:_the_fundamental_theorem_of_algebra_a_constructive_development_without_choice} which solves this problem, and doesn't even need quotients. The idea for it comes from two observations: one is that any point $x \in X$ in a metric space $\mtr{X} = (X, d)$ is uniquely determined by its distance function $d(x, \insarg)\colon X \to \RR$ as its sole zero, and the other is that an \ed-continuous map is determined by its restriction to a metrically dense subset of the domain. Combining these observations we see that points in the completion of $\mtr{X}$ could be represented as their distance functions to points in $\mtr{X}$, or even some metrically dense subset of $\mtr{X}$. What remains is to figure out how to characterize a distance function from a point. It certainly must satisfy suitable triangle inequalities, and it must have a zero, but it turns out this is enough. When restricting to a metrically dense subset, ``having a zero'' becomes ``zero is the strict infimum of the image''.
		
		Thus we define: a map $f\colon X \to \RR$ is a \df{location} on the pseudometric space $\mtr{X} = (X, d)$ when it satisfies the conditions
		\begin{itemize}
			\item $f(x) + d(x, y) \geq f(y)$ for all $x, y \in X$ (\ie $f$ is nonexpansive),
			\item $f(x) + f(y) \geq d(x, y)$ for all $x, y \in X$ (the inequalities from these two items can be captured by a single one, namely $f(x) \geq |f(y) - d(x, y)|$),
			\item $0$ is the strict infimum of the image of $f$ (in particular, $f$ is nonnegative).
		\end{itemize}
		Let $\loc(\mtr{X})$ denote the set of locations on $\mtr{X}$.
		
		Suppose the pseudometric space $\mtr{X} = (X, d)$ has a subovert metrically dense subset $D$ (\eg $\mtr{X}$ is metrically separable). Let $d_\loc\colon \loc(\mtr{X}) \times \loc(\mtr{X}) \to \RR$ be the map
		$$d_\loc(f, g) \dfeq \inf\st{f(x) + g(x)}{x \in X} = \sup\st{|f(x) - g(x)|}{x \in X}$$
		for $f, g \in \loc(\mtr{X})$. We claim this infimum and supremum actually exist (in fact, they are strict). To see this, let
		$$L \dfeq \st{a \in \oirs}{\xsome{x}{D}{a < |f(x) - g(x)|}},$$
		$$U \dfeq \st{b \in \oirs}{\xsome{x}{D}{b > f(x) + g(x)}}.$$
		Since $D$ is subovert in $X$, the sets $L$ and $U$ are open. It is clear that they are a lower and an upper cut, respectively. Suppose there is $q \in L \cap U$. Then there are $x, y \in D$ such that $q < |f(x) - g(x)|$ and $q > f(y) + g(y)$. Let $0 < \epsilon < \frac{1}{2} \inf\{|f(x) - g(x)| - q, q - f(y) - g(y)\}$, and find $u, v \in D$ so that $f(u), g(v) < \epsilon$. Then
		$$q < |f(x) - g(x)| - 2 \epsilon  \leq |f(x) - d(x, u)| + |d(x, u) - d(x, v)| + |g(x) - d(x, v)| - 2 \epsilon \leq$$
		$$\leq f(u) - \epsilon + g(v) - \epsilon + d(u, v) < d(u, v), \qquad \text{and}$$
		$$q > f(y) + g(y) + 2 \epsilon \geq d(y, u) - f(u) + d(y, v) - g(v) + 2 \epsilon \geq \epsilon - f(u) + \epsilon - f(v) + d(u, v) > d(u, v),$$
		a contradiction. Finally, if $a, b \in \oirs$, $a < b$, let $0 < \epsilon < \frac{1}{5} (b-a)$, and find $u, v \in D$ so that $f(u), g(v) < \epsilon$. Then $a + 2 \epsilon < d(u, v)$ or $d(u, v) < b - 2 \epsilon$. In the first case $|f(x) - g(x)| \leq d(u, v) - f(u) - g(v) > d(u, v) - 2 \epsilon > a$, so $a \in L$, and in the second, $f(x) + g(x) \leq f(u) + g(v) + d(u, v) < d(u, v) + 2 \epsilon < b$, so $b \in U$.
		
		Therefore $(L, U)$ is an open Dedekind cut, and it is clear from the definitions that
		$$\inf\st{f(x) + g(x)}{x \in X} = (L, U) = \sup\st{|f(x) - g(x)|}{x \in X}.$$
		Thus $d_\loc$ is a well-defined map. We claim it is a metric on $\loc(\mtr{X})$. It is clearly symmetric. For $f, g \in \loc(\mtr{X})$ we have
		$$d(f, g) = 0 \iff \sup\st{|f(x) - g(x)|}{x \in X} = 0 \iff \xall{x}{X}{|f(x) - g(x)| = 0} \iff f = g.$$
		Let $f, g, h \in \loc(\mtr{X})$. Then
		$$d_\loc(f, h) = \sup\st{|f(x) - h(x)|}{x \in X} \leq \sup\st{|f(x) - g(x)| + |g(x) - h(x)|}{x \in X} \leq$$
		$$\leq \sup\st{|f(x) - g(x)|}{x \in X} + \sup\st{|g(x) - h(x)|}{x \in X} = d_\loc(f, g) + d_\loc(g, h).$$
		
		\begin{proposition}\label{Proposition: completion_by_locations}
			Let $\mtr{X} = (X, d)$ be a metric space with a subovert metrically dense subset. Then the map $X \to \RR^X$ which is the transpose of the metric $d\colon X \times X \to \RR$, restricts to a map $i\colon X \to \loc(\mtr{X})$, and $\loc(\mtr{X})$, together with $i$, is the completion of $\mtr{X}$.
		\end{proposition}
		\begin{proof}
			The fact that $d(x, \insarg)$ is a location for every $x \in X$ follows directly from definitions; thus we obtain $i$. To see that it is an isometry, take $x, y \in X$, and calculate
			$$d_\loc\big(d(x, \insarg), d(y, \insarg)\big) = \inf\st{d(x, z) + d(y, z)}{z \in X} \leq d(x, y) + d(y, y) = d(x, y),$$
			$$d_\loc\big(d(x, \insarg), d(y, \insarg)\big) = \sup\st{|d(x, z) - d(y, z)|}{z \in X} \geq |d(x, y) - d(y, y)| = d(x, y).$$
			It is dense because for any $f \in \loc(\mtr{X})$ and $\epsilon > 0$ there is $x \in X$ such that $f(x) < \epsilon$, and then
			$$d_\loc\big(d(x, \insarg), f\big) = \inf\st{d(x, y) + f(y)}{y \in X} \leq d(x, x) + f(x) < \epsilon.$$
			Finally, for any metric space $\mtr{Y} = (Y, d')$ and dense isometry $f\colon \mtr{X} \to \mtr{Y}$ we define $g\colon \mtr{Y} \to \loc(\mtr{X})$ by
			$$g(y) \dfeq \big(x \mapsto d'(f(x), y)\big).$$
			It is easy to see that this works.
		\end{proof}
		
		\begin{theorem}\label{Theorem: extend_Lipschitz_maps_from_a_dense_subspace_into_a_completion}
			Let
			\begin{itemize}
				\item
					$\mtr{X} = (X, d_{\mtr{X}})$, $\mtr{Y} = (Y, d_{\mtr{Y}})$, $\mtr{Z} = (Z, d_{\mtr{Z}})$ be pseudometric spaces with subovert metrically dense subsets,
				\item
					$i\colon \mtr{Z} \to \cmpl{\mtr{Z}}$ the completion of $\mtr{Z}$,
				\item
					$e\colon \mtr{X} \to \mtr{Y}$ a dense isometry, and
				\item
					$f\colon \mtr{X} \to \mtr{Z}$ a Lipschitz map with coefficient $C \in \RR_{> 0}$.
			\end{itemize}
			Then there exists a unique \ed-continuous map $\bar{g}\colon \mtr{Y} \to \cmpl{\mtr{Z}}$ such that $\bar{g} \circ e = i \circ f$; moreover, it is also a Lipschitz map with the same coefficient $C$. If $f$ is an isometry, then so is $\bar{g}$.
		\end{theorem}
		\begin{proof}
			Uniqueness of $\bar{g}$ follows from Corollary~\ref{Corollary: uniqueness_of_ed_extension_of_a_map_into_a_metric_space}. As for the existence, let $D \subseteq X$ be a metrically dense subovert subset, and define $\bar{g}\colon Y \to \loc(Z)$ by
			$$\bar{g}(y) \dfeq \Big(z \mapsto (L, U)\Big)$$
			where
			$$L \dfeq \st{s \in \oirs}{\xsome{x}{D}{s + C \cdot d_{\mtr{Y}}(x, y) < d_{\mtr{Z}}(f(x), z)}},$$
			$$U \dfeq \st{s \in \oirs}{\xsome{x}{D}{s > C \cdot d_{\mtr{Y}}(x, y) + d_{\mtr{Z}}(f(x), z)}}.$$
			We verify that $(L, U)$ is an open Dedekind cut much the same way as in various previous propositions. Now take any $y, y' \in Y$, and $\epsilon \in \RR_{> 0}$. Then there are $x, x' \in X$ such that $d_{\mtr{Y}}(x, y) < \frac{\epsilon}{2 C}$, and $d_{\mtr{Y}}(x', y') < \frac{\epsilon}{2 C}$.
			$$d_\loc(\bar{g}(y), \bar{g}(y')) \leq \bar{g}(y)(f(x)) + \bar{g}(y')(f(x)) \leq C \cdot d_{\mtr{Y}}(x, y) + C \cdot d_{\mtr{Y}}(x, y') \leq$$
			$$\leq C \cdot d_{\mtr{Y}}(x, y) + C \cdot d_{\mtr{Y}}(x, y) + C \cdot d_{\mtr{Y}}(y, y') < C \cdot d_{\mtr{Y}}(y, y') + \epsilon,$$
			and since $\epsilon \in \RR_{> 0}$ was arbitrary, $d_\loc(\bar{g}(y), \bar{g}(y')) \leq C \cdot d_{\mtr{Y}}(y, y')$, so $\bar{g}$ is Lipschitz with coefficient $C$.
			
			We similarly prove that $d_\mtr{Z}(f(x), f(x')) \geq C \cdot d_\mtr{X}(x, x')$ for all $x, x' \in X$ implies $d_\loc(\bar{g}(y), \bar{g}(y')) \geq C \cdot d_{\mtr{Y}}(y, y')$ for all $y, y' \in Y$. Taking $C = 1$ we see that if $f$ is an isometry, so is $\bar{g}$.
		\end{proof}
		\begin{corollary}\label{Corollary: extend_isometries_from_a_dense_subspace_into_a_complete_space}
			If $\mtr{X}$ is a metrically separable metrically dense metric subspace of a metrically separable metric space $\mtr{Y}$, and $\mtr{Z}$ is a \cms[], then any isometry $\mtr{X} \to \mtr{Z}$ uniquely extends to an isometry $\mtr{Y} \to \mtr{Z}$.
		\end{corollary}
		\begin{proof}
			Immediate from the previous theorem.
		\end{proof}
		
		The usual version of this theorem is that any (boundedly) uniformly continuous map $\mtr{X} \to \mtr{Z}$ uniquely extend to a (boundedly) uniformly continuous map $\mtr{Y} \to \mtr{Z}$ (see~\cite{Richman_F_2008:_real_numbers_and_other_completions, Dugundji_J_1974:_topology}). If $\RR = \RR_d$, the same proof works, but in general there is a problem how to define the cuts $L$ and $U$ in the proof so as to see that they are open. However, in this thesis we only require the version with isometries. Regardless, this tells us something about what morphisms to choose when dealing with complete spaces.
		
		We haven't said anything yet about completions of protometric spaces, or morphisms between them, though. Recall that metric spaces are a reflective subcategory of pseudometric spaces via the Kolmogorov quotient, and that completions of a pseudometric space and its Kolmogorov quotient match. Similarly, there is a canonical way how to transform a protometric space into a pseudometric one. We define for a protometric space $\mtr{X} = (X, d)$ its \df{kernel}
		$$\ker(\mtr{X}) \dfeq \st{x \in X}{d(x, x) = 0}.$$
		Clearly $\ker(\mtr{X})$ is a pseudometric space, and it makes pseudometric spaces a coreflective subcategory of protometric ones (we assume of course that morphisms are at least \ed-continuous, so that kernels are mapped into kernels). Notice that if we perform any of the above constructions of completions on $\mtr{X}$, we obtain the same result as if we tried to complete $\ker(\mtr{X})$. To put this in line of the definition of completeness via terminality, one possible choice of morphisms between protometric spaces is to allow also partials maps, so long as their domain contains the kernel of the space. However, one wonders whether this is too much freedom --- it makes the categories of protometric and pseudometric spaces equivalent. Regardless what the optimal choice of generality of morphisms between protometric spaces (and protometric spaces themselves) is, it seems to make sense to consider the completion of a protometric space to be the completion of the Kolmogorov quotient of its kernel.
		
		\intermission
		
		We return now to familiar and comfortable waters, namely to two classes of spaces most interesting to us: let \cms stand for \textbf{C}omplete \textbf{M}etrically \textbf{S}eparable metric space, and \ctb for \textbf{C}omplete \textbf{T}otally \textbf{B}ounded metric space.\footnote{\ctb is the standard abbreviation, but \cms is not. The point is that in the context of metric spaces, `separable' means `metrically separable', and since three-letter acronyms are more attractive for some reason, the standard acronym is \textbf{CSM} for \textbf{C}omplete \textbf{S}eparable \textbf{M}etric space. My dislike for the discrepancy of adding \textbf{M} for `metric' to the acronym in the case of metrically separable, but not totally bounded spaces, and desire to distinguish between `separable' and `metrically separable' in this thesis (and in order to keep a three-letter acronym) led me to prefer \cms[]. Hopefully the readers will find \cms reminiscent enough of \textbf{CSM} so as to cause no difficulties.}
		
		\begin{proposition}\label{Proposition: retract_of_cms_ctb}
			Let $\mtr{X} = (X, d_X)$ be a metric space, $A \subseteq X$, $\mtr{A} = (A, d_A)$ a metric subspace of $\mtr{X}$, and $r\colon \mtr{X} \to \mtr{A}$ a nonexpansive retraction.
			\begin{enumerate}
				\item
					If $\mtr{X}$ is \cms[,] then so is $\mtr{A}$.
				\item
					If $\mtr{X}$ is \ctb[,] then so is $\mtr{A}$.
			\end{enumerate}
		\end{proposition}
		\begin{proof}
			The retract $\mtr{A}$ inherits metric separability/total boundedness by Proposition~\ref{Proposition: retract_of_ms_tb}. We focus on completeness. Since $\mtr{X}$ is complete, the isometric embedding $e\colon \mtr{A} \hookrightarrow \mtr{X}$ induces the isometric embedding $\cmpl{e}\colon \loc(\mtr{A}) \to \mtr{X}$ with the property $f = d_X(\cmpl{e}(f), \insarg)$ for all $f \in \loc(\mtr{A})$. The space $\mtr{A}$ is metrically dense in $\mtr{A} \cup \{\cmpl{e}(f)\}$ since $0$ is the strict infimum of $f$. For any $a \in A$ we have $r(a) = a$, thus by Corollary~\ref{Corollary: uniqueness_of_ed_extension_of_a_map_into_a_metric_space} $r(\cmpl{e}(f)) = \cmpl{e}(f)$, meaning $\cmpl{e}(f) \in A$.
		\end{proof}
		
		\begin{proposition}\label{Proposition: cms_products}
			Products of \cms[s] are again \cms[s]. More precisely:
			\begin{enumerate}
				\item Let $\mtr{X} = (X, d', s')$, $\mtr{Y} = (Y, d'', s'')$ be \cms[s]. Then so is their product $\mtr{X} \times \mtr{Y} = (X \times Y, d, s)$.
				\item Let $(\mtr{X}_n)_{n \in \NN} = (X_n, d_n, s_n)_{n \in \NN}$ be a sequence of inhabited \cms[s] with distances bounded by $1$, and gauge map\footnote{We don't really need this additional assumption on the gauge map (and boundedness of metric), but it makes the proof technically easier, and we will only require the result under these hypotheses.} $\id[\II]$. Then the product $\mtr{X} = (X, d, s) = \prod_{n \in \NN} \mtr{X}_n$ is a \cms also.
			\end{enumerate}
		\end{proposition}
		\begin{proof}
			Metric separability is preserved by products by Proposition~\ref{Proposition: metric_multiplicative_properties}, so we focus on completeness.
			\begin{enumerate}
				\item
					By Lemmas~\ref{Lemma: wlog_space_inhabited}~and~\ref{Lemma: decidable_subsets_in_complete_spaces} we may without loss of generality assume that $X$ and $Y$ are inhabited (since $(X + X') \times (Y + Y') \ism X \times Y + X \times Y' + X' \times Y + X' \times Y'$).
					
					Take a location $f \in \loc(\mtr{X} \times \mtr{Y})$, and define $g'\colon \mtr{X} \to \RR$ by $g'(x) \dfeq \inf\st{f(x, y)}{y \in Y}$. This (strict) infimum exists by the standard argument; to see that $g'$ is a location on $\mtr{X}$, take $x, x' \in X$, and calculate
					$$g'(x) + d'(x, x') = \inf\st{f(x, y)}{y \in Y} + d'(x, x') = \inf\st{f(x, y) + d'(x, x')}{y \in Y} =$$
					$$= \inf\st{f(x, y) + d((x, y), (x', y))}{y \in Y} \geq \inf\st{f(x', y)}{y \in Y} = g'(x'),$$
					
					$$g'(x) + g'(x') = \inf\st{f(x, y)}{y \in Y} + \inf\st{f(x', y)}{y \in Y} \geq$$
					$$\geq \inf\st{f(x, y) + f(x', y)}{y \in Y} \geq \inf\st{d((x, y), (x', y))}{y \in Y} = d'(x, x'),$$
					
					$$\inf\st{g'(x)}{x \in X} = \inf\st{\inf\st{f(x, y)}{y \in Y}}{x \in X} =$$
					$$= \inf\st{f(x, y)}{(x, y) \in X \times Y} = 0.$$
					Since $\mtr{X}$ is complete, there exists $a \in X$ such that $g' = d'(a, \insarg)$. Define the map $g''\colon Y \to \RR$, $g''(y) \dfeq f(a, y)$. For $y, y' \in Y$ we have
					$$|g''(y) - d''(y, y')| = |f(a, y) - d((a, y), (a, y'))| \geq f(a, y') = g''(y'),$$
					$$\inf\st{g''(y)}{y \in Y} = \inf\st{f(a, y)}{y \in Y} = g'(a) = 0,$$
					so $g''$ is a location on $\mtr{Y}$. From completeness of $\mtr{Y}$ we obtain $b \in Y$ such that $g'' = d''(b, \insarg)$. Then $f(a, b) = g''(b) = 0$, so $f = d((a, b), \insarg)$.
				\item
					From Proposition~\ref{Proposition: completion_by_Cauchy_families} we infer that it does not matter what we take for $\lvl$; if a space is complete for Cauchy families for some $\lvl$, it is for all of them. In this particular case we choose to take $\lvl = \st{2^k}{k \in \ZZ}$.
					
					Denote the projections $p_n\colon X \to X_n$, and let $S \in \Cf(\mtr{X})$. For all $n \in \NN$, define
					$$S_n \dfeq \Big(2^k \mapsto p_n\big(S(2^{k-n})\big)\Big);$$
					one can verify that $S_n$ is a Cauchy family on $\mtr{X}_n$. As such, there exists (a unique) $x_n \in X_n$ which is represented by $S_n$. Denote $x = (x_n)_{n \in \NN}$. We claim that $x$ is represented by $S$; it is sufficient to verify, for every $2^k \in \lvl$ and every $t \in S_{2^k}$, that $d(x, t) \leq 2^k$. Let $t_n = p_n(t)$; then $t_n \in S_n(2^{k+n})$, so $d_n(x_n, t_n) \leq 2^{k+n}$. We conclude
					$$d(x, t) = \sup\st{2^{-n} d_n(x_n, t_n)}{n \in \NN} \leq \sup\st{2^{-n} \cdot 2^{k+n}}{n \in \NN} = 2^k.$$
			\end{enumerate}
		\end{proof}
		
		\begin{corollary}\label{Corollary: ctb_products}
			Binary product of \ctb[s] is a \ctb[]. Countable product of inhabited \ctb[s] (with gauge map $\id[\II]$) is an inhabited \ctb[].
		\end{corollary}
		\begin{proof}
			By Propositions~\ref{Proposition: cms_products} and~\ref{Proposition: metric_multiplicative_properties}.
		\end{proof}
   
   \chapter{Metrization Theorems}\label{Chapter: metrization_theorems}

	In classical topology, a metric $d$ on a set $X$ induces a topology on $X$ by taking all balls as a basis. The existence of a metric which induces a given topology is a powerful tool, as demonstrated by a plethora of classical theorems which hold specifically in metrizable spaces. Thus a natural question to ask in our setting is: given a metric space, does its metric induce its intrinsic topology? In this chapter we give some answers when that happens, and what are the consequences. In addition to overtness of countable sets, we assume $\tst = \opn$.
	
	First we need to determine what in our context means that the intrinsic topology is generated by the basis of metric balls. Recall the relevant definitions from Section~\ref{Section: (co)limits_and_bases}.
	\begin{definition}
		Let $\mtr{X} = (X, d)$ be a metric space, with $B\colon X \times \RR \to \tp(X)$ mapping $(x, r) \in X \times \RR$ to the ball $B(x, r)$ (recall that metric balls are intrinsically open). The space $\mtr{X}$ is (\df{weakly}, \df{canonically}) \df{metrized} by the metric $d$ when $B$ is a (weak, canonical) basis for $X$.
	\end{definition}
	Since balls are inhabited if and only if they have positive radius, and $\RR_{> 0}$ is open in $\RR$, the definition of metrization would be equivalent if we restricted the domain of $B$ to $X \times \RR_{> 0}$, but we may as well allow arbitrary real radii, as that spares us another condition to check.
	
	The definition can also be rephrased as follows. Call $U \subseteq X$ a \df{metrically open} subset of $X$ when it is an overtly indexed union of balls. Let $\mtp(\mtr{X})$ denote the set of all metrically open subsets of $X$. Since metric balls are intrinsically open, we have $\mtp(\mtr{X}) \subseteq \tp(X)$. The space $\mtr{X}$ is metrized when the reverse inclusion also holds, so that the metric and the intrinsic topology match, \ie $\mtp(\mtr{X}) = \tp(X)$. Also, we define a map between metric spaces $f\colon \mtr{X} \to \mtr{Y}$ to be \df{metrically continuous} when for every $V \in \mtp(\mtr{Y})$ we have $f^{-1}(V) \in \mtp(\mtr{X})$. Since \ed-continuity can be phrased as ``the preimage of any ball is a union of balls'', and since a ball is certainly metrically open (being a $\one$-indexed union of balls), metric continuity implies \ed-continuity.
	
	For some properties weak metrization suffices.
	\begin{proposition}\label{Proposition: metric_to_intrinsic_separability_in_weakly_metrized_spaces}
		The following holds for a weakly metrized metric space $\mtr{X} = (X, d)$.
		\begin{enumerate}
			\item
				A metrically dense subset $D \subseteq X$ is also intrinsically dense in $X$.
			\item
				If $\mtr{X}$ is metrically separable, it is also intrinsically separable.
			\item
				If $\mtr{X}$ has a metrically dense subovert subset (\eg $\mtr{X}$ is metrically separable), then it is overt.
		\end{enumerate}
	\end{proposition}
	\begin{proof}
		\begin{enumerate}
			\item
				Let $U \in \tp(X)$ be an open set, inhabited by $a \in U$. Since $\mtr{X}$ is weakly metrized, we can write $U$ as a union $U = \bigcup_{i \in I} B(x_i, r_i)$, and there is some $i \in I$ such that $a \in \ball{x_i}{r_i}$. Consequently $r_i > 0$, so $D$ intersects $\ball{x_i}{r_i}$, and consequently also $U$.
			\item
				By the previous item.
			\item
				By the first item and Proposition~\ref{Proposition: testing_overtness_on_dense_subset}.
		\end{enumerate}
	\end{proof}
	
	Overt sets with decidable equality, such as $\NN$, $\ZZ$, or $\QQ$, provide simple examples of metrized spaces.
	\begin{proposition}\label{Proposition: overt_sets_with_decidable_equality_metrized}
		Let $X$ be an overt set with decidable equality, equipped with the discrete metric $d_D\colon X \times X \to \RR$,
		$$d_D(x, y) \dfeq \begin{cases} 1 & \text{ if } x = y,\\ 0 & \text{ if } x \neq y. \end{cases}$$
		Then $(X, d_D)$ is canonically metrized.
	\end{proposition}
	\begin{proof}
		For $U \in \tp(X)$ we have $U = \bigcup\st{B(x, 1)}{x \in U}$, and since $U$ is open in overt $X$, and therefore subovert, the indexing map is overt because it is the composition (recall Proposition~\ref{Proposition: constructions_of_overt_maps}) of the overt map $U \hookrightarrow X$ and the inclusion $X \ism X \times \{1\} \hookrightarrow X \times \RR$.
	\end{proof}
	
	We say that two metrics $d$, $d'$ on a given set $X$ are (\df{topologically}) \df{equivalent} when they induce the same metric topology, \ie $\mtp(X, d) = \mtp(X, d')$. Usually there are several non-equivalent metrics on a given set, so it is easy to also find examples of non-metrized metric spaces. For instance, in addition to the discrete metric $d_D$, we can also endow the rationals $\QQ$ with the Euclidean metric $d_E(p,q) = |p - q|$. The former metrizes $\QQ$ by Proposition~\ref{Proposition: overt_sets_with_decidable_equality_metrized}, while the latter is too coarse. Likewise, even though every map is intrinsically continuous, a map between metric spaces need not be metrically continuous (or even \ed-continuous), for example identity on $\QQ$ as a map $(\QQ, d_E) \to (\QQ, d_D)$.
	
	This example is not a coincidence. Since metric topology is contained in the intrinsic one, they are more likely to match when the metric is finer, and $\opn$ is small. Consider first the coarseness of metric and intrinsic topology. Taking subsets in a metric space does not change the metric, but can only enlarge the intrinsic topology. This suggests that ``larger'' metric spaces are more likely to be metrized, so complete metric spaces ought to be good candidates for metrized spaces, and indeed in the example above, $(\QQ, d_D)$ is complete while $(\QQ, d_E)$ is not. This intuition will pay off, as demonstrated by Theorems~\ref{Theorem: metrization_of_Baire/Cantor_implies_metrization_of_cmss_ctbs},~\ref{Theorem: metrization_of_Hilbert_cube_implies_metrization_of_ctbs}~and~\ref{Theorem: metrization_of_Urysohn_space_implies_metrization_of_cmss} (however, it should not be taken too far --- complete spaces are not necessarily metrized, and metrized spaces are not necessarily complete).
	
	As for the size of $\opn$, large $\opn$ is obviously a poor choice if we want a lot of metric spaces to be metrized; for example, in the extreme case $\opn = \soc$, all sets are overt and discrete, and only discrete metric spaces with decidable equality are metrized --- the bare minimum, as these always are (per Proposition~\ref{Proposition: overt_sets_with_decidable_equality_metrized}). On the other extreme we could take the smallest $\opn$ for which countable sets are overt (meaning the intersection of all such). Recall from Lemma~\ref{Lemma: Rosolini_dominance_and_countable_joins} that this means $\Ros \subseteq \opn$, and if $\ACRos$, then $\opn = \Ros$.
	
	However, even that might not be enough. Equip $\NN$ with the discrete metric (which is an ultrametric), and make its countable product (with the gauge map $\id[\II]$) $\prod_{n \in \NN} \NN \ism \NN^\NN$; this is called the \df{Baire space} $\Baire$, and it is a metrically separable ultrametric space by Proposition~\ref{Proposition: metric_multiplicative_properties}. Its metric $d_C\colon \NN^\NN \times \NN^\NN \to \RR$ is called the \df{comparison metric} since it tells us how far two sequences match; more precisely, for any $\alpha, \beta \in \NN^\NN$ and $n \in \NN$,
	$$\all{k}{\NN_{< n}}{\alpha_k = \beta_k} \iff d_C(\alpha, \beta) \leq 2^{-n} \iff d_C(\alpha, \beta) < 2^{-n+1}.$$
	The space $\Baire$ is also complete by Proposition~\ref{Proposition: cms_products}, but we can easily see that directly. Take a location $f \in \loc(\Baire)$. Observe that for any $\alpha, \beta \in \NN^\NN$ and $n \in \NN$, if $f(\alpha) < 2^{-n-1}$ and $f(\beta) < 2^{-n-1}$, then $d_C(\alpha, \beta) \leq f(\alpha) + f(\beta) < 2^{-n}$, so $\alpha_n = \beta_n$. Consequently we can define $\gamma\colon \NN \to \NN$ by $\gamma_n \dfeq \beta_n$ where $\beta \in \NN^\NN$ is any sequence such that $f(\beta) < 2^{-n-1}$. Then $f = d_C(\gamma, \insarg)$.
	
	Nice properties of the Baire space notwithstanding, Richard Friedberg still showed in~\cite{Friedberg_RM_1958:_un_contreexample_relatif_aux_fonctionnelles_recursives} what in our language would state that in Russian Constructivism the Baire space, even for $\opn = \Ros$, is not even weakly metrized. We expand this example in Section~\ref{Section: russian_constructivism_model} where we see that few metric spaces are metrized in Russian Constructivism.
	
	On a more positive note, we shall see that we can infer metrization of a space from metrization of another when they are connected via a suitably nice map. Here is an example of this.\footnote{This proposition is a special case of both Proposition~\ref{Proposition: transfer_of_metrization_via_quotients} and Theorem~\ref{Theorem: metrization_and_subspaces}.}
	\begin{proposition}\label{Proposition: retracts_inherit_metrization}
		A retract of an overt metrized space is overt and metrized.
	\end{proposition}
	\begin{proof}
		Let $\mtr{X} = (X, d_\mtr{X})$ be a metric space, and $r\colon X \to A$ a retraction onto the subset $A \subseteq X$. Then $A$ is overt since it is the image of an overt set.
		
		Let $\mtr{A} = (A, d_\mtr{A})$ denote the metric subspace of $\mtr{X}$ with the underlying set $A$. Take any $U \in \tp(A)$. Then $r^{-1}(U) \in \tp(X)$, so $r^{-1}(U) = \bigcup_{i \in I} \ball[\mtr{X}]{x_i}{r_i}$ for some overt indexing $I \to X \times \RR$. Observe
		$$U = \bigcup_{(a, i) \in A \times I} \ball[\mtr{A}]{a}{r_i - d_\mtr{X}(x_i, a)}.$$
		This is an overtly indexed union since it is the composition of the product of overt maps $\id[A]$ and $I \to X \times \RR$ with some further map.
	\end{proof}
	
	There are some metric spaces which have a wealth of maps connecting them to entire families of spaces. Consequently we can expect that metrization of these families hinges on metrizations of select few spaces. This presumption will pay off via Theorems~\ref{Theorem: metrization_of_Baire/Cantor_implies_metrization_of_cmss_ctbs},~\ref{Theorem: metrization_of_Hilbert_cube_implies_metrization_of_ctbs}~and~\ref{Theorem: metrization_of_Urysohn_space_implies_metrization_of_cmss}. Here we present the metric spaces which play an important role in our theory.
	\begin{itemize}
		\item
			The aforementioned Baire space $\NN^\NN$ with the comparison metric $d_C$. Recall that it is a complete metrically separable ultrametric space. We define the \df{shift maps} on it: for any $m \in \NN$ let $\shift{m}\colon \NN^\NN \to \NN^\NN$ be the map $\shift{m}(\alpha) \dfeq (n \mapsto \alpha_{n+m})$.
		\item
			The analogous construction as in the case of the Baire space, but with the factor $\two$ instead of $\NN$. The space $\two$ is equipped with the discrete metric which makes it into a complete totally bounded ultrametric space; consequently, the countable product  $\two^\NN$, called the \df{Cantor space}, has these properties also. It isometrically embeds into the Baire space in the obvious way, and so it is equipped with (the restriction of) the comparison metric $d_C$. Moreover, the map $\NN^\NN \to \two^\NN$,
			$$\alpha \mapsto \Bigg(n \mapsto \begin{cases} 0 & \text{ if } \alpha_n = 0,\\ 1 & \text{ otherwise, } \end{cases}\Bigg)$$
			is a nonexpansive retraction from the Baire space to the Cantor space.
		\item
			The space $\opcN \dfeq \st{\alpha \in \two^\NN}{\all{n}{\NN}{\alpha_n = 0 \implies \alpha_{n+1} = 0}}$. We equip it with the restriction of the comparison metric from the Cantor space. In fact, $\opcN$ is a nonexpansive retract of the Cantor space (and hence of the Baire space), courtesy of the map $\two^\NN \to \opcN$,
			$$\alpha \mapsto \Bigg(n \mapsto \begin{cases} 0 & \text{ if } \xsome{k}{\NN_{\leq n}}{\alpha_k = 0},\\ 1 & \text{ otherwise, } \end{cases}\Bigg)$$
			making it a complete totally bounded ultrametric space. A good way to visualize $\opcN$ is as a ``one-point compactification'' of $\NN$ since the map $\NN \to \opcN$,
			$$n \mapsto \Bigg(k \mapsto \begin{cases} 1 & \text{ if } k < n,\\ 0 & \text{ if } k \geq n, \end{cases}\Bigg)$$
			is injective, and the complement of its image is a single point, namely the constant sequence with all terms $1$ which in view of this intuition we denote $\infty$. Via this injection we consider $\NN$ a subset of $\opcN$; a sequence represents the number which is the sum of its terms.
			
			Note that the comparison metric, restricted to $\NN$, is given by
			$$d_C(i, j) = \begin{cases} 2^{-\inf\{i, j\}} & \text{ if } i \neq j,\\ 0 & \text{ if } i = j. \end{cases}$$
			The metric space $\opcN$, together with the above inclusion, is the completion of $\NN$ with this metric.
			
			The strict order $<$ on $\NN$ extends to $\opcN$,
			$$\alpha < \beta \dfeq \some{k}{\NN}{\alpha_k = 0 \land \beta_k = 1}.$$
			It is easy to see that this defines an open strict linear order on $\opcN$ (even decidable when at least one of $\alpha$, $\beta$ is in $\NN$). Note that $\NN = \st{\alpha \in \opcN}{\alpha < \infty}$ is an open subset of $\opcN$. The strict order induces a lattice structure on $\opcN$, the partial order given by
			$$\alpha \leq \beta \iff \lnot(\alpha > \beta) \iff \all{k}{\NN}{\alpha_k = 1 \implies \beta_k = 1},$$
			and suprema and infima computed pointwise,
			$$\sup\{\alpha, \beta\}_k = \sup\{\alpha_k, \beta_k\}, \qquad \inf\{\alpha, \beta\}_k = \inf\{\alpha_k, \beta_k\}.$$
			
			We define the \df{successor map} $\succm\colon \opcN \to \opcN$ by
			$$\succm(\alpha)_n \dfeq \begin{cases} 1 & \text{ if } n = 0,\\ \alpha_{n-1} & \text{ if } n \geq 1, \end{cases}$$
			and the \df{predecessor map} $\predm\colon \opcN \to \opcN$ by
			$$\predm(\alpha)_n \dfeq \alpha_{n+1}.$$
			Observe $\succm(n) = n+1$ for $n \in \NN$ while $\succm(\infty) = \infty$, and $\predm(0) = 0$, $\predm(n) = n-1$ for $n \in \NN_{\geq 1}$, $\predm(\infty) = \infty$.
		\item
			The unit interval $\II$, equipped with (the restriction of) the Euclidean metric. It is a complete metrically separable space since it is a nonexpansive retract of the reals via the map $\RR \to \II$, $x \mapsto \sup\{\inf\{x, 1\}, 0\}$. It is even totally bounded; for maps $s\colon \NN \to \II$, $a\colon \NN \to \NN$ witnessing this we can take $s$ to be the concatenation of finite sequences $\{\frac{k}{2^n}\}_{k \in \NN_{\leq 2^n}}$ for all $n \in \NN$, and $a_n \dfeq 2^{n+1} + n$.
		\item
			Another countable product (with the gauge map $\id[\II]$), this time with the factor $\II$ from the previous item. It is a complete totally bounded space, called the \df{Hilbert cube}.
		\item
			The Urysohn space, defined and constructed in Section~\ref{Section: metrization_via_embeddings}.
	\end{itemize}

	\section{The \wso Principle}\label{Section: wso}
	
		In Proposition~\ref{Proposition: overt_sets_with_decidable_equality_metrized} we have seen that overt sets with decidable equality are always metrized (by the discrete metric). This tells us what happens if points in a metric space are sufficiently spaced apart. The next step would be to examine what happens when we have accumulation points.
		\begin{definition}
			An element $a \in X$ in the metric space $\mtr{X} = (X, d)$ is called a (\df{metric}) \df{accumulation point} when for all $r \in \RR_{> 0}$ there exists $x \in X$ such that $0 < d(a, x) < r$ (in other words, $x \in \ball{a}{r}$ and $x \apart a$).
		\end{definition}
		
		We should start with a simple case, a space with a single accumulation point; intuitively, it should look like a convergent sequence, together with its limit. The space $\opcN$ satisfies these criteria --- it is arguably the simplest \ctb that has an accumulation point (in the metric sense), namely $\infty$, since every inhabited ball with the center in $\infty$ intersects $\NN$ while every point in $\NN$ is isolated.
		
		In order for the metric and the intrinsic topology to match, $\infty$ should be an accumulation point of $\NN$ also in intrinsic sense. This need not always happen; for example, if $\opn = \soc$, then $\infty$ is isolated. So we state this property as a principle.
		\principle{\wso[]}{For every $U \in \tp(\opcN)$, if $\infty \in U$, then there exists $n \in \NN$ such that $n \in U$.}
		The acronym \wso stands for `\textbf{W}eakly \textbf{S}equentially \textbf{O}pen'.
		
		\begin{proposition}
			The principle \wso is equivalent to $\NN$ being intrinsically dense in $\opcN$.
		\end{proposition}
		\begin{proof}
			Clearly if $\NN$ is dense in $\opcN$, then \wso holds. Conversely, let $U \in \tp(\opcN)$ be inhabited by $a \in U$. Define the map $f\colon \opcN \to \opcN$ by $f(x) \dfeq \inf\{x, a\}$. The subset $f^{-1}(U)$ is open in $\opcN$, and $\infty \in f^{-1}(U)$, so by \wso there is some $n \in \NN \cap f^{-1}(U)$. Then $f(n) \in \NN \cap U$.
		\end{proof}
		
		This principle is nonclassical --- it implies the failure of the law of excluded middle.
		\begin{proposition}\label{Proposition: wso_not_classical}
			The principles \wso and \lpo are not both true.
		\end{proposition}
		\begin{proof}
			Notice that \lpo is equivalent to $\opcN$ having decidable equality. If this holds, then $\{\infty\}$ is a decidable, hence open, subset of $\opcN$, a contradiction to \wso[].
		\end{proof}
		
		The principle \wso has implications for metrization of other spaces with accumulation points. This is because $\opcN$ can be used to classify convergent sequences in complete metric spaces.
		\begin{proposition}\label{Proposition: convergent_sequences_and_maps_from_No}
			Let $\mtr{X} = (X, d)$ be a complete metric space. Then convergent sequences in $\mtr{X}$ are in bijective correspondence to \ed-continuous maps $\opcN \to \mtr{X}$.
		\end{proposition}
		\begin{proof}
			For any \ed-continuous map $f\colon \opcN \to \mtr{X}$, the sequence $a\colon \NN \to X$, $a_n \dfeq f(n)$, converges to $f(\infty)$. Conversely, given a sequence $a\colon \NN \to X$, define the map $f\colon \opcN \to \mtr{X}$ by
			$$f(\alpha) \dfeq \lim(n \mapsto f(\inf\{n, \alpha\})).$$
		\end{proof}
		In particular, the map $\NN \to \RR$, $n \mapsto 2^{-n}$, extends to the whole of $\opcN$. Notice that for all $t \in \opcN$
		$$2^{-t} \geq 0, \qquad\quad 2^{-t} > 0 \iff t \in \NN, \qquad \text{and} \qquad 2^{-\infty} = 0.$$
		Exponentiation by elements of $\opcN$ is often useful; for example, if $\mathfrak{d}\colon \NN^\NN \times \NN^\NN \to \two^\NN$ is the comparison map
		$$\mathfrak{d}(\alpha, \beta)_n \dfeq \begin{cases} 1 & \text{ if } \alpha_n = \beta_n,\\ 0 & \text{ if } \alpha_n \neq \beta_n, \end{cases}$$
		and $r\colon \two^\NN \to \opcN$ the standard retraction, then the comparison metric $d_C$ on the Baire space (and its subsets) can be expressed as
		$$d_C(\alpha, \beta) = 2^{-r(\mathfrak{d}(\alpha, \beta))}.$$
		
		In~\cite{Bauer_A_Lesnik_D_2010:_metric_spaces_in_synthetic_topology} we proved that \wso has several desirable consequences. Among other things it ensures that \cms[s] are overt, but we assumed countable choice for this.
		\begin{proposition}
			If $\ACopn$ and \wso[], then every \cms is overt.
		\end{proposition}
		\begin{proof}
			Let $\mtr{X} = (X, d, s)$ be a \cms[]. Let $U \subseteq X$ be an arbitrary inhabited open subset of $X$, and let $x \in U$. By $\ACopn$ there is a sequence $a\colon \NN \to \NN$ such that for all $n \in \NN$, $d(x, s(a_n)) < 2^{-n}$. Then the sequence $n \mapsto s(a_n)$ converges to $x$, so it extends to a map $f\colon \opcN \to X$ so that for all $n \in \NN$, $f(n) = s(a_n)$, and $f(\infty) = x$. Then $f^{-1}(U)$ is an open neighbourhood of $\infty$ in $\opcN$, so by \wso it contains some $n \in \NN$. Consequently $s(a_n) \in U$, meaning that $X$ is intrinsically separable, and so overt.
		\end{proof}
		
		In this thesis however we try to do as much as possible without choice principles. We will obtain a weaker version of this proposition, but one that will work for the spaces we want.
		
		\begin{proposition}
			Let $\mtr{X} = (X, d)$ be a metric space with a subovert metrically dense subset $D$. Suppose that for every $x \in X$ there is a canonical choice of a map $s^x\colon \NN \to D$ with the property $d(x, s^x(n)) \leq 2^{-n}$ for all $n \in \NN$. Then \wso implies $X$ is overt.
		\end{proposition}
		\begin{proof}
			Essentially the same as in the previous proposition, post-choice. We establish that $D$ is intrinsically dense in $X$, so $X$ is overt.
		\end{proof}
		
		\begin{corollary}\label{Corollary: overtness_of_standard_ctbs}
			If \wso holds, then $\NN^\NN$, $2^\NN$ and $\opcN$ are overt.
		\end{corollary}
		\begin{proof}
			Define $D = \st{\alpha \in \NN^\NN}{\xsome{n}{\NN}\xall{k}{\NN_{\geq n}}{\alpha_k = 0}}$; since it is in bijection with $\NN$, it is overt, and it is metrically dense in the Baire space. For any $\alpha \in \NN^\NN$ define $s^\alpha\colon \NN \to D$ by
			$$s^\alpha(n)_k \dfeq \begin{cases} \alpha_k & \text{ if } k \leq n,\\ 0 & \text{ if } k > n. \end{cases}$$
			Now use the previous proposition. The same trick works for $2^\NN$ while for $\opcN$ we can take $D = \NN$, $s^t(n) \dfeq \inf\{n, t\}$. Alternatively, we can simply recall that they are retracts of $\NN^\NN$, and therefore inherit overtness.
		\end{proof}
		
		Our purpose now is to use \wso to help classify metrization of \ctb[s]. Recall that classically \ctb[s] are precisely compact metric spaces. We may thus suspect that compactness also plays the role here. This turns out to be true (as will be demonstrated by Theorems~\ref{Theorem: metrization_of_compact_spaces},~\ref{Theorem: metrization_of_No},~and~\ref{Theorem: metrization_of_Cantor}), but it is not so straightforward. Not every compact metric space need be totally bounded (for example, if $\opn = \soc$, all sets are compact), and not every \ctb need be compact (\cf Section~\ref{Section: russian_constructivism_model}).
		
		\begin{definition}
			A metric space $\mtr{X} = (X, d)$ is \df{Lebesgue} when every overtly indexed cover of $X$ by balls can be shrunk a little, and still cover $X$. More formally, if $X = \bigcup_{i \in I} B(x_i, r_i)$ for an overtly indexed family of centers and radii, then there exists $\epsilon \in \RR_{> 0}$ such that $X = \bigcup_{i \in I} B(x_i, r_i - \epsilon)$ (equivalently, there exists $n \in \NN$ such that $X = \bigcup_{i \in I} B(x_i, r_i - 2^{-n})$).
		\end{definition}
		The reason we call this property by Lebesgue is because it is equivalent to the existence of Lebesgue numbers for covers by balls. Recall from classical topology that the \df{Lebesgue number} of a cover is $\delta \in \RR_{> 0}$ such that any subset on which the metric is bounded by $\delta$ is contained in some element of the cover. Note that if $\delta$ is a Lebesgue number for a cover by balls, then the radii can be shrunk by $\frac{\delta}{2}$, and the balls still cover the space. Conversely, if the radii can be shrunk by $\epsilon \in \RR_{> 0}$, then $\epsilon$ is also a Lebesgue number for the cover. Classically, any open cover on a compact metric space has a Lebesgue number. We show that the Lebesgue property has something to do with compactness in the case of synthetic topology also.
		
		\begin{theorem}\label{Theorem: metrization_of_compact_spaces}
			Let $\mtr{X} = (X, d)$ be a metric space.
			\begin{enumerate}
				\item If $X$ is overt, bordered balls in $\mtr{X}$ with radii of the form $2^{-t}$, $t \in \opcN$, are subcompact, and \wso holds, then $\mtr{X}$ is canonically metrized.\footnote{It is easy to see from the proof that the assumption in the theorem can be weakened to a stronger form of local compactness (in the metric sense): there is a map $f\colon X \to \RR_{> 0}$ such that for all $x \in X$ and $t \in \opcN$ for which $2^{-t} < f(x)$, the bordered ball $\bball{x}{2^{-t}}$ is subcompact. We will not need this more general form of the theorem, however. Actually, it is not so uncommon for \emph{all} bordered balls in a metric space being subcompact --- at least classically, there are a good many such spaces, for example all closed subsets of Euclidean spaces $\RR^n$.}
				\item If $X$ is compact and \wso holds, then $\mtr{X}$ is Lebesgue.
				\item Assuming $\ACopn$, if $X$ is compact, $\mtr{X}$ metrically separable, and \wso holds, then $\mtr{X}$ is totally bounded.
				\item If $\mtr{X}$ is totally bounded, Lebesgue, and metrized, then it is overt and compact.
			\end{enumerate}
		\end{theorem}
		\begin{proof}
			\begin{enumerate}
				\item
					Take any $U \in \tp(X)$. We claim
					$$U = \bigcup\st{\ball{x}{2^{-k}}}{x \in X \land k \in \NN \land \bball{x}{2^{-k}} \subseteq U}.$$
					The right-hand side is clearly contained in $U$. Conversely, take any $x \in U$. The set $\st{t \in \opcN}{\bball{x}{2^{-t}} \subseteq U}$ is open in $\opcN$ because of compactness of bordered balls. It contains $\infty$, so by \wso there is some $k \in \NN$ such that $\bball{x}{2^{-k}} \subseteq U$. Therefore $x$ is contained in the right-hand side. It remains to be seen that the inclusion
					$$\st{(x, k) \in X \times \NN}{\bball{x}{2^{-k}} \subseteq U} \hookrightarrow X \times \NN \stackrel{\id[X] \times (k \mapsto 2^{-k})}{\verylongrightarrow} X \times \RR$$
					is overt, but this is clear since $X$ and $\NN$ are overt, and the statement $\bball{x}{2^{-k}} \subseteq U$ is open.
				\item
					Let $X = \bigcup_{i \in I} B(x_i, r_i)$ be an overtly indexed cover of $X$. Define
					$$U \dfeq \st[2]{t \in \opcN}{X \subseteq \bigcup_{i \in I} \ball{x_i}{r_i - 2^{-t}}}.$$
					Then $U$ is open in $\opcN$ (since $X$ is compact) and contains $\infty$, so by \wso there is some $n \in \NN \cap U$, proving the Lebesgue property.
				\item
					Let $s\colon \NN \to \one + X$ witness that $\mtr{X}$ is separable. For any $n \in \NN$ define $U \subseteq \opcN$ to be
					$$U \dfeq \st{t \in \opcN}{X \subseteq \bigcup\st{\ball{s_k}{2^{-n}}}{k \in s^{-1}(X) \cap \opcN_{< t}}}.$$
					The set $s^{-1}(X) \cap \opcN_{< t}$ is a decidable subset of $\NN$, thus overt, and since $X$ is compact, $U$ is open in $\opcN$. It contains $\infty$ because $s$ has metrically dense image in $\mtr{X}$. By $\wso$ there exists $a_n \in U \cap \NN$, and by $\ACopn$ this defines a sequence $a = (a_n)_{n \in \NN}$ which (together with $s$) witnesses that $\mtr{X}$ is totally bounded.
				\item
					The set $X$ is overt by Proposition~\ref{Proposition: metric_to_intrinsic_separability_in_weakly_metrized_spaces}. To see that it is compact, let $s\colon \NN \to \one + X$, $a\colon \NN \to \NN$ witness that $\mtr{X}$ is totally bounded, and take any $U \in \tp(X)$; we have to prove that the truth value of $\xall{x}{X}{x \in U}$ is open. Since $\mtr{X}$ is metrized, we can write $U$ as an overtly indexed union $U = \bigcup_{i \in I} \ball{x_i}{r_i}$. Observe
					$$\xall{x}{X}{x \in U} \implies \xall{x}{X}{x \in \bigcup_{i \in I} \ball{x_i}{r_i}} \implies$$
					$$\implies \xsome{n}{\NN}\xall{x}{X}{x \in \bigcup_{i \in I} \ball{x_i}{r_i - 2^{-n}}} \implies$$
					$$\implies \xsome{n}{\NN}\xall{k}{\big(s^{-1}(X) \cap \NN_{< a(n)}\big)}{s(k) \in \bigcup_{i \in I} \ball{x_i}{r_i - 2^{-n}}} \implies$$
					$$\implies \xall{x}{X}{x \in \bigcup_{i \in I} \ball{x_i}{r_i}} \implies \xall{x}{X}{x \in U}$$
					since $\mtr{X}$ is Lebesgue and totally bounded. The third line is an open statement (the set $s^{-1}(X) \cap \NN_{< a(n)}$ is finite, therefore compact), and then so is the equivalent statement $\xall{x}{X}{x \in U}$.
			\end{enumerate}
		\end{proof}
		
		\begin{corollary}\label{Corollary: overt_compact_iff_Lebesgue_metrized}
			Suppose \wso holds, and $\RR$ is Hausdorff. Then the following is equivalent for a totally bounded metric space $\mtr{X} = (X, d, s, a)$.
			\begin{enumerate}
				\item $X$ is overt and compact.
				\item $\mtr{X}$ is Lebesgue and metrized.
			\end{enumerate}
		\end{corollary}
		\begin{proof}
			Recall from Proposition~\ref{Proposition: Hausdorffness_of_real_numbers} that if $\RR$ is Hausdorff, bordered balls in $\mtr{X}$ must be closed, and so subcompact if $X$ is compact. The rest follows from Theorem~\ref{Theorem: metrization_of_compact_spaces}.
		\end{proof}
		Note that under the conditions of this corollary, an overt compact set is metrized by \emph{any} totally bounded metric on it, and hence these metrics must all be equivalent.
		
		From these results (and their proofs) the role of the principle \wso[], or at least one aspect of it, might have crystallized for the attentive reader --- intuitively, it makes (at least on the level of metrically separable spaces) compact sets behave the way we expect them to. Here is another confirmation of this.
		\begin{theorem}\label{Theorem: compact_maps_into_reals}
			Assume \wso[].
			\begin{enumerate}
				\item
					A compact map $f\colon X \to \RR$ is bounded.
				\item
					Suppose $\opn \subseteq \nnst$. If $X$ is inhabited, and $f\colon X \to \RR$ is an overt compact map, then the image of $f$ has a strict supremum and a strict infimum.\footnote{The conclusion of this statement is not true if we do not assume $\opn \subseteq \nnst$ or some such condition (even after we exclude pathological examples like $\opn = \cld = \soc$ because of \wso[]). To see this, let $p \in \soc$ be some open and closed truth value, or equivalently (since we assume $\tst = \opn$) a subovert subcompact one. Then the set $\{0\} \cup \st{1}{p}$ is inhabited subovert subcompact in $\RR$, but it has a strict supremum if and only if $p$ is decidable. Compare this with Theorem~\ref{Theorem: standard_special_case_of_synthetic_topology} which under its conditions implies $\opn \cap \cld = \two$.}
			\end{enumerate}
		\end{theorem}
		\begin{proof}
			We prove only the existence of the upper bound; finding the lower bound is analogous, or we can consider $-f$.
			\begin{enumerate}
				\item
					The idea is that we find a bound of $f$ in $\opcN$; surely $\infty$ works, and then we can use \wso to see that some natural number is a bound for the image of $f$ also. However, to do that, we need to be able to compare real numbers with elements in $\opcN$. For $x \in \RR$, $t \in \opcN$ define
					$$x < t \dfeq \some{n}{\NN}{x < n \land n \leq t}.$$
					This is an open relation since $\NN$ is overt, the strict order on $\RR$ is open, and $n \leq t$ is decidable (hence open) because $n \in \NN$. Consequently, since $f$ is a compact map, the set $\st{t \in \opcN}{\xall{x}{X}{f(x) < t}}$ is an open subset in $\opcN$, and it clearly contains $\infty$; by \wso it then contains some $n \in \NN$ which is an upper bound for $f$.
				\item
					Define
					$$L \dfeq \st{q \in \oirs}{\xsome{x}{X}{q < f(x)}},$$
					$$U \dfeq \st{r \in \oirs}{\xall{x}{X}{r > f(x)}}.$$
					Because $f$ is an overt and compact map, these two sets are open in $\oirs$. $L$ is inhabited because $X$ is, and $U$ is inhabited by the previous item. From interpolation property of $\oirs$ we conclude that they are a lower and an upper cut, respectively. They are obviously disjoint. Finally, take $q, r \in \oirs$, $q < r$. We wish to prove $q \in L \lor r \in U$, but this is an open statement, so it is equally good (since $\opn \subseteq \nnst$) to prove $\lnot\lnot(q \in L \lor r \in U)$.
					$$\lnot(q \in L \lor r \in U) \iff \lnot\Big(\xsome{x}{X}{q < f(x)} \lor \xall{x}{X}{r > f(x)}\Big) \iff$$
					$$\iff \lnot\xsome{x}{X}{q < f(x)} \land \lnot\xall{x}{X}{r > f(x)} \iff$$
					$$\iff \xall{x}{X}{q \geq f(x)} \land \lnot\xall{x}{X}{r > f(x)} \implies \bot$$
					So $(L, U)$ is a real number, and it is clearly the strict supremum of the image of $f$.
			\end{enumerate}
		\end{proof}
		
		\begin{corollary}
			If \wso holds, then compact (pseudo)metric spaces, and more generally, subcompact (pseudo)metric subspaces, are bounded.
		\end{corollary}
		\begin{proof}
			Let $\mtr{X} = (X, d_\mtr{X})$ be a pseudometric space, and $\mtr{A} = (A, d_\mtr{A})$ its subspace. If the inclusion $A \hookrightarrow X$ is a compact map, then so is the composition $d_\mtr{A} = A \times A \hookrightarrow X \times X \stackrel{d_\mtr{X}}{\longrightarrow} \RR$, and hence bounded by the preceding theorem.
		\end{proof}
		In particular, \wso implies that subcompact subsets in $\NN$ are bounded. In fact since subcompact and closed propositions match (because $\tst = \opn$), $\NN$ has decidable equality, and $\NN_{< n}$ is compact for every $n \in \NN$, a subset $K \subseteq \NN$ is subcompact if and only if it is closed and bounded. Classically we know that the same result holds in Euclidean spaces $\RR^k$, $k \in \NN$, in general; in our case this result follows if we additionally assume that $\RR$ is Hausdorff and $\II$ is compact.
		
		\intermission
		
		We now focus on metrization of some specific \ctb[s]. In constructive set theory, a sort of a substitute of compactness of the Cantor space is \df{Brouwer's Fan Principle}~\cite{Troelstra_AS_Dalen_D_1988:_constructivism_in_mathematics_volume_1, Bridges_DS_Richman_F_1987:_varieties_of_constructive_mathematics} which states that every decidable bar is uniform. We recall the relevant definitions.
		
		For a set $A$, let $\finseq{A}$ denote the set of finite sequences in $A$. We denote the length of a finite sequence $a \in \finseq{A}$ by $|a|$.
		
		If $\alpha \in A^\NN$, let $\prefix{\alpha}{n} \dfeq (\alpha_0, \ldots, \alpha_{n-1})$, and write $a \lex \alpha$ when $a$ is a prefix of $\alpha$, \ie $a = \prefix{\alpha}{|a|}$. Denote the concatenation of $a, b \in \finseq{A}$ by $a \cnct b$, and likewise of $a \in \finseq{A}$, $\beta \in A^\NN$ by $a \cnct \beta$. For $t \in \opcN$ and $\alpha, \beta \in A^\NN$ let $\concat{t}{\alpha}{\beta} \in A^\NN$ be
		$$(\concat{t}{\alpha}{\beta})_k = \begin{cases} \alpha_k & \text{ if } k < t,\\ \beta_{k-t} & \text{ if } t \leq k. \end{cases}$$
		If $t \in \NN$, then $\concat{t}{\alpha}{\beta} = \prefix{\alpha}{t} \cnct \beta$ while $\concat{\infty}{\alpha}{\beta} = \alpha$.
		
		Notice that the sets $\fbs$ and $\finseq{\NN}$ are in bijection with $\NN$, and thus overt.
		\begin{definition}
			A \df{bar} is a subset $S \subseteq \fbs$ such that $\xall{\alpha}{\two^\NN}{\xsome{a}{S}{a \lex \alpha}}$. A bar is \df{uniform} when there exists a bound $n \in \NN$ such that $\xall{\alpha}{\two^\NN}{\some{a}{S}{|a| \leq n \land a \lex \alpha}}$.
		\end{definition}
		We will consider the strengthening of the Brouwer's Fan Principle by replacing `decidable' with `subovert' (when we say that a bar is decidable, or subovert, we mean it is such as a subset of $\fbs$).
		
		First we prove a technical lemma saying that the Lebesgue property can be tested on certain smaller families of balls.
		\begin{lemma}\label{lemma: check_Lebesgue_on_restrictedly_indexed_covers}
			Let $\mtr{X} = (X, d)$ be a metric space, $D \subseteq X$ subovert metrically dense in $\mtr{X}$, and $\lvl \subseteq \RR_{> 0}$ a subset, subovert in $\RR$, with the strict infimum $0$. Then $\mtr{X}$ is Lebesgue if and only if every overtly indexed cover of $X$ by balls with centers from $D$ and radii from $\lvl$ can be shrunk by some $\epsilon \in \RR_{> 0}$.
		\end{lemma}
		\begin{proof}
			One direction is obvious. For the other, let $X = \bigcup_{i \in I} \ball{x_i}{r_i}$ be an overtly indexed cover. Define $J \dfeq \st{(a, s, i) \in D \times \lvl \times I}{d(x_i, a) + s < r_i}$. Then $X = \bigcup_{(a, s, i) \in J} \ball{a}{s}$, and this is an overtly indexed cover by Corollary~\ref{Corollary: overt_combination} since the inclusions $D \hookrightarrow X$, $\lvl \hookrightarrow \RR$ are overt maps, and $J$ is the preimage of
			$$\st{(a, s, x, r) \in X \times \RR \times X \times \RR}{d(x, a) + s < r}$$
			which is open in $X \times \RR \times X \times \RR$. By assumption the newly obtained cover can be shrunk by some $\epsilon \in \RR_{> 0}$, but then the original cover can be shrunk by the same $\epsilon$ also.
		\end{proof}
		
		\begin{proposition}\label{Proposition: Lebesgue_property_of_standard_ctbs}
			\
			\begin{enumerate}
				\item
					The metric space $\opcN$ is Lebesgue.
				\item
					The metric space $\two^\NN$ is Lebesgue if and only if every subovert bar is uniform.
				\item
					The metric space $\NN^\NN$ is not Lebesgue.
			\end{enumerate}
		\end{proposition}
		\begin{proof}
			By Lemma~\ref{lemma: check_Lebesgue_on_restrictedly_indexed_covers} we can restrict to covers by balls with radii of the form $2^{-k}$ for $k \in \NN$ (and centers in metrically dense countable subset if necessary).
			\begin{enumerate}
				\item
					Let $\opcN = \bigcup_{i \in I} \ball{x_i}{2^{-k_i}}$ for some overtly indexed cover. There is some $j \in I$ such that $\infty \in \ball{x_j}{2^{-k_j}}$. Let $m \dfeq \inf\{x_j, k_j\}$, and $j_0, j_1, \ldots, j_{m-1} \in I$ such that $l \in \ball{x_{j_l}}{2^{-k_{j_l}}}$ for all $l \in \NN_{< m}$. Let
					$$n \dfeq \sup\Big(\{k_j\} \cup \st{k_{j_l}}{l \in \NN_{< m}}\Big).$$
					Then the cover can be shrunk by any positive number smaller than $2^{-n-1}$, \eg $2^{-n-2}$.
				\item
					Let $S \subseteq \fbs$ be a subovert bar. Denote by $o \in \two^\NN$ the zero sequence, and let $f\colon S \to \two^\NN \times \RR$ be the map $f(a) \dfeq (a \cnct o, 2^{-|a|+1})$. Since $f$ can be written as the inclusion $S \hookrightarrow \fbs$, composed with another map, it is an overt map. Thus the union $\bigcup_{a \in S} \ball{a \cnct o}{2^{-|a|+1}}$ is overtly indexed, and it covers $\two^\NN$ since $S$ is a bar. By the Lebesgue property there is some $n \in \NN$ such that
					$$\two^\NN \subseteq \bigcup_{a \in S} \ball{a \cnct o}{2^{-|a|+1} - 2^{-n+1}} = \bigcup_{a \in S \land |a| < n} \ball{a \cnct o}{2^{-|a|+1} - 2^{-n+1}} \subseteq$$
					$$\subseteq \bigcup_{a \in S \land |a| < n} \ball{a \cnct o}{2^{-|a|+1}} \subseteq \two^\NN,$$
					implying that $S$ is a uniform bar with the bound $n$.
					
					Conversely, let $\bigcup_{i \in I} \ball{\alpha_i}{2^{-k_i}}$ be an overtly indexed cover of $\two^\NN$; then
					$$S \dfeq \st[1]{\prefix{\alpha_i}{k_i+1}}{i \in I}$$
					is a subovert bar. By assumption, it is uniform with some bound $n \in \NN$ which implies that the cover can be shrunk by $2^{-n+1}$.
				\item
					Let $I = \st{a \in \finseq{\NN}}{|a| \geq 1}$; since $I$ is in bijection with $\NN$, it is overt. Observe $\NN^\NN = \bigcup_{a \in I} \ball{a \cnct o}{2^{-a_0}}$ since $\alpha \in \ball{\prefix{\alpha}{\alpha_0 + 1} \cnct o}{2^{-\alpha_0}}$, but for any $n \in \NN$ the constant sequence with terms $n$ is not in $\bigcup_{a \in I} \ball{a \cnct o}{2^{-a_0} - 2^{-n}}$.
			\end{enumerate}
		\end{proof}
		
		\begin{lemma}\label{Lemma: bordered_balls_are_retracts_in_standard_ctbs}
			Let the set $A$ be $\NN^\NN$, $\two^\NN$, or $\opcN$, equipped with (the restriction of) the comparison metric $d_C$. Then for every $\beta \in A$ and $t \in \opcN$ the bordered ball $\bball{\beta}{2^{-t}}$ is a nonexpansive retract of $A$.
		\end{lemma}
		\begin{proof}
			Define the map $r\colon A \to \bball{\beta}{2^{-t}}$,
			$$r(\alpha)_n \dfeq \begin{cases} \beta_n & \text{ if } n < t,\\ \alpha_n & \text{ if } n \geq t. \end{cases}$$
			It is easy to see that this works for $\NN^\NN$ and $\two^\NN$. In the case of $\opcN$, adjust the formula to
			$$r(\alpha)_n \dfeq \begin{cases} \beta_n & \text{ if } n < t,\\ \inf\Big(\{\alpha_n\} \cup \st{\beta_k}{k \in \NN_{< t}}\Big) & \text{ if } n \geq t. \end{cases}$$
		\end{proof}
		
		We are now ready to characterize metrization of \ctb[s] $\opcN$ and $\two^\NN$.
		
		\begin{lemma}\label{Lemma: implies_wso}
			Let $\NN^\NN$, $\two^\NN$ and $\opcN$ be equipped with the comparison metric. Then each of the following statements implies the next.
			\begin{enumerate}
				\item
					$\NN^\NN$ is metrized.
				\item
					$\two^\NN$ is metrized.
				\item
					$\opcN$ is metrized.
				\item
					For every $U \in \tp(\opcN)$, if $\infty \in U$, then there exists $r \in \RR_{> 0}$ such that $\ball{\infty}{r} \subseteq U$.
				\item
					\wso holds.
			\end{enumerate}
		\end{lemma}
		\begin{proof}
			If $\NN^\NN$ or $\two^\NN$ is metrized, it is also overt by Proposition~\ref{Proposition: metric_to_intrinsic_separability_in_weakly_metrized_spaces}, therefore since $\two^\NN$ is a retract of $\NN^\NN$, and $\opcN$ is a retract of $\two^\NN$, metrization of the retract space follows from Proposition~\ref{Proposition: retracts_inherit_metrization}. The rest is immediate.
		\end{proof}
		
		\begin{theorem}\label{Theorem: metrization_of_No}
			Consider the following statements (where $\opcN$ is understood to be equipped with the comparison metric).
			\begin{enumerate}
				\item
					$\opcN$ is compact, and \wso holds.
				\item
					$\opcN$ is canonically metrized.
				\item
					$\opcN$ is metrized.
				\item
					$\opcN$ is weakly metrized.
				\item
					For every $U \in \tp(\opcN)$, if $\infty \in U$, then there exists $r \in \RR_{> 0}$ such that $\ball{\infty}{r} \subseteq U$.
			\end{enumerate}
			Then every statement implies the next one, the first three are equivalent, and if $\opn$ is a dominance, all are equivalent.
		\end{theorem}
		\begin{proof}
			Assume the first statement. Then $\opcN$ is overt by Corollary~\ref{Corollary: overtness_of_standard_ctbs}. Since we assume it is also compact, and bordered balls with radii $2^{-t}$, $t \in \opcN$, are its retracts by Lemma~\ref{Lemma: bordered_balls_are_retracts_in_standard_ctbs}, they are compact as well. Now use Theorem~\ref{Theorem: metrization_of_compact_spaces} to conclude that $\opcN$ is canonically metrized. Clearly then each statement implies the next.
			
			Assume $\opcN$ is metrized; this implies \wso by Lemma~\ref{Lemma: implies_wso}. It is Lebesgue by Proposition~\ref{Proposition: Lebesgue_property_of_standard_ctbs}. The first statement now follows from Theorem~\ref{Theorem: metrization_of_compact_spaces}.
			
			Finally, suppose $\opn$ is a dominance, and assume the fifth statement. It implies \wso by Lemma~\ref{Lemma: implies_wso}. To prove compactness, and hence the first statement, take any $U \in \tp(\opcN)$. Because
			$$\xall{t}{\opcN}{t \in U} \iff \infty \in U \land \xall{t}{\opcN}{t \in U},$$
			it suffices to show that the right-hand side of the equivalence is open which we derive from the dominance axiom. If $\infty \in U$, then there exists $k \in \NN$ such that $B(\infty, 2^{-k}) \subseteq U$, therefore $\xall{t}{\opcN}{t \in U}$ is equivalent to the open statement $\xall{t}{\NN_{\leq k}}{t \in U}$.
		\end{proof}
		
		\begin{theorem}\label{Theorem: metrization_of_Cantor}
			The following statements are equivalent for $\two^\NN$ (equipped with the comparison metric).
			\begin{enumerate}
				\item
					$\two^\NN$ is compact, and \wso holds.
				\item
					$\two^\NN$ is metrized, and every subovert bar is uniform.
			\end{enumerate}
		\end{theorem}
		\begin{proof}
			The space $\two^\NN$ is overt by Corollary~\ref{Corollary: overtness_of_standard_ctbs}, and its compactness implies compactness of bordered balls with radii $2^{-t}$, $t \in \opcN$, which are its retracts by Lemma~\ref{Lemma: bordered_balls_are_retracts_in_standard_ctbs}. Use Theorem~\ref{Theorem: metrization_of_compact_spaces} to conclude that $\two^\NN$ is (canonically) metrized.
			
			Conversely, if $\two^\NN$ is metrized, the principle \wso holds by Lemma~\ref{Lemma: implies_wso}. If every subovert bar is uniform, the space $\two^\NN$ is Lebesgue by Proposition~\ref{Proposition: Lebesgue_property_of_standard_ctbs}. Theorem~\ref{Theorem: metrization_of_compact_spaces} then implies compactness of $\two^\NN$.
		\end{proof}
		
		A more transparent connection between compactness of $\two^\NN$ and the Fan Principle is expressed by the following corollary.
		\begin{corollary}\label{Corollary: Fan_Principle_and_compactness_of_Cantor}
		  If $\two^\NN$ is metrized by the comparison metric, then the following are equivalent.
		  \begin{enumerate}
			  \item $\two^\NN$ is compact.
			  \item Every subovert bar is uniform.
		  \end{enumerate}
		  In the special case when $\ACRos$ and $\opn = \Ros$, the second statement may be replaced by Brouwer's Fan Principle ``every decidable bar is uniform''.
		\end{corollary}
		
		\begin{proof}
		  Metrization of $\two^\NN$ implies \wso by Lemma~\ref{Lemma: implies_wso}. The equivalence now follows from Theorem~\ref{Theorem: metrization_of_Cantor}.
		
		  If $\ACRos$ and $\opn = \Ros$ (therefore $\opn$ is a dominance), then the subovert bars, overt bars and semidecidable bars match. By~\cite{Ishihara_H_2005:_constructive_reverse_mathematics_compactness_properties} the Fan Principles for decidable and semidecidable bars are equivalent.
		\end{proof}
		
		\intermission
		
		The principle \wso also has consequences for (countable) products. When we defined the product metric, we simply took the one we are familiar with from classical topology, but there is little reason in general to expect that the product metric should metrize the product, even if all individual factors are metrized. For example, in $\Set$ (with $\opn = \two = \soc$) all sets are overt and have decidable equality, thus they are metrized by the discrete metric (by Proposition~\ref{Proposition: overt_sets_with_decidable_equality_metrized}), but the metric of a countable product of discrete metric spaces is not (equivalent to) the discrete metric. We claim however that $\wso$ under mild conditions ensures the match between the product topology (as defined in Section~\ref{Section: (co)limits_and_bases}) and the product metric topology.
		
		\begin{proposition}\label{Proposition: wso_implies_expected_countable_product_topology}
			If \wso holds, then $X = \prod_{i \in I} X_i$ has the (weak, canonical) product topology if and only if the map
			$$P\colon \coprod_{n \in \NN} \prod_{k \in \NN_{< n}} \tp(X_k) \to \tp(X),$$
			defined by
			$$P\big((U_k)_{k \in \NN_{< n}}\big) \dfeq \bigcap_{k \in \NN_{< n}} p_k^{-1}(U_k) = \prod_{k \in \NN_{< n}} U_k \times \prod_{j \in \NN_{\geq n}} X_j,$$
			is a (weak, canonical) basis for $X$.
		\end{proposition}
		\begin{proof}
			One direction is obvious since $\NN_{< n}$ is compact. For the other it is sufficient to present any subset of the form $\bigcap_{k \in K} p_k^{-1}(U_k)$ where $K$ is subcompact in $\NN$ as a (subovert) union of the sets of the form $\bigcap_{k \in \NN_{< n}} p_k^{-1}(U_k)$ (in a canonical way). Let
			$$J \dfeq \st{n \in \NN}{K \subseteq \NN_{< n}};$$
			since $\NN_{< n}$ is open and $K$ subcompact in $\NN$, $J$ is open, and therefore subovert in $\NN$. For every $i \in \NN$ define
			$$V_i \dfeq \st{x \in X_i}{i \in K \implies x \in U_i}.$$
			By Proposition~\ref{Proposition: intersection_of_preimages_by_projections_as_a_product} this is an open subset in $X_i$, and furthermore for every $n \in J$
			$$\bigcap_{k \in K} p_k^{-1}(U_k) = \prod_{i \in \NN} V_i = \bigcap_{i \in \NN} p_i^{-1}(V_i) = \bigcap_{k \in \NN_{< n}} p_k^{-1}(V_k)$$
			where the last equality holds since clearly $V_i = X_i$ for $i \in J$ (and $J$ is an upper set).
			
			Thus
			$$\bigcap_{k \in K} p_k^{-1}(U_k) = \bigcup_{n \in J} \bigcap_{k \in \NN_{< n}} p_k^{-1}(V_k);$$
			the inclusion $\supseteq$ is clear while $\subseteq$ holds because $J$ is inhabited by \wso (see Theorem~\ref{Theorem: compact_maps_into_reals} and the discussion after it).
		\end{proof}
		
		\begin{theorem}\label{Theorem: wso_implies_expected_countable_products}
			Let
			\begin{itemize}
				\item
					$\big(\mtr{X}_n = (X_n, d_n)\big)_{n \in \NN}$ be a sequence of metric spaces,
				\item
					$\mtr{X} = (X, d) \dfeq \prod_{n \in \NN} \mtr{X}_n$ the product metric space,
				\item
					$\big(p_n\colon X \to X_n\big)_{n \in \NN}$ the projections.
			\end{itemize}
			\begin{enumerate}
				\item
					Any metrically open subset in $U$ is an overtly indexed union of product basic subsets. Thus metrization of $\mtr{X}$ implies that $X$ has the product topology.
				\item
					Suppose in addition that \wso holds, and that $\mtr{X}_n$s are canonically metrized for all $n \in \NN$. Then the converse also holds, \ie if $X$ has the product topology, then $\mtr{X}$ is metrized.
			\end{enumerate}
		\end{theorem}
		\begin{proof}
			In order to simplify the proof we assume that the metrics $d_n$ are bounded above by $1$, and that the gauge map for the product $\mtr{X}$ is the identity $\id[\II]$ (we need the result only in this case).
			\begin{enumerate}
				\item
					Take any metrically open $U \subseteq X$ which can then be written as an overtly indexed union $U = \bigcup_{i \in I} \ball{x^i}{r^i}$. Let
					$$J \dfeq \st{(i, q) \in I \times \finseq{\oirs}}{2^{-|q|} < r^i \land \all{k}{\NN_{< |q|}}{2^{-k} q_k < r^i}}.$$
					We claim
					$$U = \bigcup_{(i, q) \in J} \bigcap_{k \in \NN_{< |q|}} p_k^{-1}\Big(\ball[k]{x^i_k}{q_k}\Big).$$
					Denote the right-hand side by $V$.
					\begin{itemize}
						\item\proven{$U \subseteq V$}
							Take any $y = (y_n)_{n \in \NN} \in U$. There exists $i \in I$ such that $y \in \ball{x^i}{r^i}$. Then $r^i > 0$, so we can find a small enough $m \in \NN$ for which $2^{-m} < r^i$, and we can also find $q_0, \ldots, q_{m-1} \in \oirs$ such that $d_k\big(x^i_k, y_k\big) < q_k < 2^k \cdot r^i$ for all $k \in \NN_{< m}$ since $\oirs$ is interpolating and $2^{-k} \cdot d_k\big(x^i_k, y_k\big) \leq d(x^i, y) < r^i$. We see that $\big(i, (q_k)_{k \in \NN_{< m}}\big) \in J$, and $p_k(y) \in \ball[k]{x^i_k}{q_k}$ for all $k \in \NN_{< m}$, so $y \in \bigcap_{k \in \NN_{< |q|}} p_k^{-1}\Big(\ball[k]{x^i_k}{q_k}\Big) \subseteq V$.
						\item\proven{$V \subseteq U$}
							Take $y \in V$, \ie there exist $i \in I$, $m \in \NN$ and $q \in \oirs^m$ such that $2^{-m} < r^i$, and for all $k \in \NN_{< m}$ we have $d_k\big(x^i_k, y_k\big) < q_k < 2^k \cdot r^i$. Let
							$$\epsilon \dfeq \inf\Big(\{r^i - 2^{-m}\} \cup \st{r^i - 2^{-k} \cdot q_k}{k \in \NN_{< n}}\Big) > 0;$$
							then $2^{-k} d_k(x^i_k, y_k) \leq r^i - \epsilon$ for all $k \in \NN_{< n}$. Also, for all $j \in \NN_{\geq n}$
							$$2^{-j} \cdot d_j(x^i_j, y_j) \leq 2^{-j} \leq 2^{-m} \leq r^i - \epsilon,$$
							so $d(x^i, y) = \sup\st{2^{-k} \cdot d_k(x^i_k, y_k)}{k \in \NN} \leq r^i - \epsilon < r^i$ whence we conclude $y \in U$.
						\item\proven{$V$ is an overtly indexed union of compactly indexed intersections}
							The intersections are compactly indexed, as the indices range over the finite sets $\NN_{< m}$. The union is overtly indexed by Corollary~\ref{Corollary: overt_combination} since
							\begin{itemize}
								\item
									the indexing $I \to X \times \RR$ is overt,
								\item
									$\finseq{\oirs} \ism \coprod_{n \in \NN} \oirs^n$ (hence its identity $\id[\finseq{\oirs}]$) is overt, being a countable (consequently overt) coproduct of overt sets, and
								\item
									$J$ is the preimage of the open subset
									$$\st{(x, r, q) \in X \times \RR \times \finseq{\oirs}}{2^{-|q|} < r \land \all{k}{\NN_{< |q|}}{2^{-k} q_k < r}}.$$
							\end{itemize}
					\end{itemize}
				\item
					Suppose $X$ has the product topology. By Proposition~\ref{Proposition: wso_implies_expected_countable_product_topology} any $U \in \tp(X)$ can be written as an overtly indexed union of sets of the form $\bigcap_{k \in \NN_{< n}} p_k^{-1}(U_k)$ where $U_k$s are open in canonically metrized $X_k$s, so we can write them as overtly indexed unions
					$$U_k = \bigcup_{i \in I_k} \ball[k]{x_{i,k}}{r_{i,k}}.$$
					Fix a set $\lvl \subseteq \RR_{> 0}$ with the strict infimum $0$, subovert in $\RR$, for example $\lvl = \st{2^{-n}}{n \in \NN}$. Define
					$$J \dfeq \st[2]{\big((i_k)_{k \in \NN_{< n}}, (a_j)_{j \in \NN}, s\big) \in \big(\prod_{k \in \NN_{< n}} I_k\big) \times D \times \lvl}{\all{k}{\NN_{< n}}{d_k(x_{i_k, k}, a_k) + 2^k s < r_{i_k, k}}}.$$
					We claim
					$$\bigcap_{k \in \NN_{< n}} p_k^{-1}(U_k) = \bigcup\st[2]{\ball{(a_j)_{j \in \NN}}{s}}{\big((i_k)_{k \in \NN_{< n}}, (a_j)_{j \in \NN}, s\big) \in \big(\prod_{k \in \NN_{< n}} I_k\big) \times D \times \lvl}.$$
					\begin{itemize}
						\item\proven{$(\supseteq)$}
							Easy from the definition of the product metric.
						\item\proven{$(\subseteq)$}
							Take any $(y_j)_{j \in \NN} \in \bigcap_{k \in \NN_{< n}} p_k^{-1}(U_k)$. For each $k \in \NN_{< n}$ pick $i_k \in I_k$ such that $d_k(y_k, x_{i_k, k}) < r_{i_k, k}$. Let $s \in \lvl$ be smaller than $2^{-k-1} \big(r_{i_k, k} - d_k(y_k, x_{i_k, k})\big)$ for all $k \in \NN_{< n}$. Pick $(a_j)_{j \in \NN} \in D \cap \ball{(y_j)_{j \in \NN}}{s}$. Then $\big((i_k)_{k \in \NN_{< n}}, (a_j)_{j \in \NN}, s\big) \in J$ since
							$$d_k(x_{i_k, k}, a_k) + 2^k s \leq d_k(x_{i_k, k}, y_k) + d_k(y_k, a_k) + 2^k s <$$
							$$< d_k(x_{i_k, k}, y_k) + 2^{k+1} s < d_k(x_{i_k, k}, y_k) + r_{i_k, k} - d_k(y_k, x_{i_k, k}) = r_{i_k, k},$$
							and we have $(y_j)_{j \in \NN} \in \ball{(a_j)_{j \in \NN}}{s}$, as required.
						\item\proven{the union is overtly indexed}
							By Corollary~\ref{Corollary: overt_combination} since the indexings with domains $I_k$ are overt maps, $D$ and $\lvl$ are subovert, and $J$ is the preimage of the open subset
							\begin{multline*}
								\Big\{\big((x_k)_{k \in \NN_{< n}}, (r_k)_{k \in \NN_{< n}}, (a_j)_{j \in \NN}, s\big) \in \big(\prod_{k \in \NN_{< n}} X_k\big) \times \RR^n \times X \times \RR \,\Big|\\
								\all{k}{\NN_{< n}}{d_k(x_k, a_k) + 2^k s < r_k}\Big\}.
							\end{multline*}
					\end{itemize}
					Since we expressed a basic open subset as an overtly indexed union of balls in a canonical way, the result follows.
			\end{enumerate}
		\end{proof}
		
		This has consequences for metrization of the Cantor and the Baire space.
		\begin{corollary}
			The following statements are equivalent for both the Baire space and the Cantor space.
			\begin{enumerate}
				\item
					The Baire space/the Cantor space is metrized.
				\item
					It has the product topology, and \wso holds.
			\end{enumerate}
		\end{corollary}
		\begin{proof}
			If the Baire space or the Cantor space is metrized, \wso holds by Lemma~\ref{Lemma: implies_wso}. The rest follows from Theorem~\ref{Theorem: wso_implies_expected_countable_products}.
		\end{proof}

	\section{Transfer of Metrization via Quotients}\label{Section: metrization_via_quotients}
	
		Our purpose in the following two sections is to characterize metrization of \cms[s] and \ctb[s]. We do so by providing, for both classes of spaces, their single member, such that its metrization is transfered to all spaces in their class.
		
		There are two natural ways how the topology of one space determines the topology of another when they are connected by a suitable map: when that map is a topological embedding, or when it is a topological quotient. We consider the metric version of the latter in this section, and of the former in the next.
		
		We are therefore interested in metric spaces such that \cms[s], or at least \ctb[s], are their quotient, as well as spaces that \cms[s], or \ctb[s], embed into them. Classically, and to some extent constructively, the existence of such spaces is well-known:
		\begin{itemize}
			\item
				every inhabited \ctb is a quotient of the Cantor space,
			\item
				every inhabited \cms is a quotient of the Baire space,
			\item
				every \ctb embeds into the Hilbert cube,
			\item
				every \cms embeds into the Urysohn space.
		\end{itemize}
		Our intention is to recall or provide constructive versions of these claims, and add the transfer of metrization.
		
		We start by recalling that a map $f\colon \mtr{X} \to \mtr{Y}$ between metric spaces $\mtr{X}$, $\mtr{Y}$ is metrically continuous when for every metrically open subset $U \subseteq \mtr{Y}$ the preimage $f^{-1}(U)$ is metrically open in $\mtr{X}$. Here are some examples.
		\begin{proposition}\label{Proposition: examples_of_metric_continuous_maps}
			\
			\begin{enumerate}
				\item\label{Proposition(examples_of_metric_continuous_maps)item: maps_from_metrized_to_metric_space_are_metrically_continuous}
					A map from a metrized space to a metric space is metrically continuous.
				\item\label{Proposition(examples_of_metric_continuous_maps)item: inclusion_of_metric_subspace_metrically_continuous}
					The inclusion of a metric subspace with a subovert metrically dense subset is metrically continuous.
				\item
					Let $\mtr{X} = (X, d_\mtr{X})$, $\mtr{Y} = (Y, d_\mtr{Y})$ be metric spaces, and $D \subseteq X$ a subovert metrically dense subset of $\mtr{X}$. Let $f\colon \mtr{X} \to \mtr{Y}$ be a map for which there exist $\epsilon, C \in \RR_{> 0}$ such that for all $x, y \in X$,
					$$d_\mtr{X}(x, y) < \epsilon \implies d_\mtr{Y}\big(f(x), f(y)\big) \leq C \cdot d_\mtr{X}(x, y)$$
					(\ie $f$ is Lipschitz with coefficient $C$ on sufficiently small balls). Then $f$ is metrically continuous. In particular, a Lipschitz map is metrically continuous.
			\end{enumerate}
		\end{proposition}
		\begin{proof}
			\begin{enumerate}
				\item
					A metrically open subset is also intrinsically open, and then so is its preimage which is in turn metrically open in a metrized space.
				\item
					This is a special case of the following item.
				\item
					Suppose $U \subseteq Y$ is an overtly indexed union of balls in $Y$, $U = \bigcup_{i \in I} \ball[\mtr{Y}]{y_i}{r_i}$. Define
					$$J \dfeq \st[1]{(i, a) \in I \times D}{f(a) \in \ball[\mtr{Y}]{y_i}{r_i}}.$$
					We claim
					$$f^{-1}(U) = \bigcup_{(i, a) \in J} \ball[\mtr{X}]{a}{\inf\left\{\epsilon, \frac{r_i - d_\mtr{Y}(y_i, f(a))}{C}\right\}}.$$
					Take any $x \in X$ which is in the set on the right-hand side, \ie $d_\mtr{X}(a, x) < \inf\left\{\epsilon, \frac{r_i - d_\mtr{Y}(y_i, f(a))}{C}\right\}$ for some $(i, a) \in J$. Then
					$$d_\mtr{Y}(y_i, f(x)) \leq d_\mtr{Y}(y_i, f(a)) + d_\mtr{Y}(f(a), f(x)) \leq$$
					$$\leq  d_\mtr{Y}(y_i, f(a)) + C \cdot d_\mtr{X}(a, x) < d_\mtr{Y}(y_i, f(a)) + r_i - d_\mtr{Y}(y_i, f(a)) = r_i,$$
					so $f(x) \in U$. Conversely, take any $x \in X$, $f(x) \in U$, and let $i \in I$ be such that $f(x) \in \ball[\mtr{Y}]{y_i}{r_i}$. Denote $\delta \dfeq \inf\left\{\epsilon, \frac{r_i - d_\mtr{Y}(y_i, f(x))}{2 C}\right\}$; then $\delta > 0$, and there exists $a \in D$ such that $d_\mtr{X}(x, a) < \delta$. Calculate
					$$d_\mtr{Y}(y_i, f(a)) \leq d_\mtr{Y}(y_i, f(x)) + d_\mtr{Y}(f(x), f(a)) \leq$$
					$$\leq  d_\mtr{Y}(y_i, f(x)) + C \cdot d_\mtr{X}(x, a) < d_\mtr{Y}(y_i, f(x)) + r_i - d_\mtr{Y}(y_i, f(x)) = r_i$$
					to see that $(i, a) \in J$, and then $x$ is in the right-hand side because $d_\mtr{X}(a, x) < \epsilon$ and
					$$d_\mtr{X}(a, x) < \frac{r_i - d_\mtr{Y}(y_i, f(x))}{2 C} \leq \frac{r_i - d_\mtr{Y}(y_i, f(a)) + d_\mtr{Y}(f(a), f(x))}{2 C} \leq$$
					$$\leq \frac{r_i - d_\mtr{Y}(y_i, f(a))}{2 C} + \frac{d_\mtr{X}(a, x)}{2}$$
					whence $d_\mtr{X}(a, x) < \frac{r_i - d_\mtr{Y}(y_i, f(a))}{C}$.
					
					Finally, we have to see that the above union is overtly indexed. This is true by Corollary~\ref{Corollary: overt_combination} because the indexing $I \to Y \times \RR$ and the inclusion $D \hookrightarrow X$ are overt maps, and $J$ is the preimage of the subset $\st{(y, r, x) \in Y \times \RR \times X}{f(x) \in \ball[\mtr{Y}]{y}{r}}$ which is open in $Y \times \RR \times X$.
			\end{enumerate}
		\end{proof}
		
		We now define the maps which transfer metrization of the domain to the codomain.
		\begin{definition}
			Let $\mtr{X}$, $\mtr{Y}$ be metric spaces.
			\begin{itemize}
				\item
					A map $f\colon \mtr{X} \to \mtr{Y}$ is a \df{metric semiquotient map} when for all $U \subseteq \mtr{Y}$ the implication
					$$f^{-1}(U) \text{ metrically open in } \mtr{X} \implies U \text{ metrically open in } \mtr{Y}$$
					holds.
				\item
					A map $f\colon \mtr{X} \to \mtr{Y}$ is a \df{metric quotient map} when it is surjective, metrically continuous and metric semiquotient\footnote{A metric semiquotient map need not be surjective by itself --- \eg in $\Set$ with $\opn = \two$, any map between discrete metric spaces is semiquotient. However unless $\opn = \soc$, the complement of the image of a semiquotient map $f\colon (X, d_\mtr{X}) \to (Y, d_\mtr{Y})$ must be empty since if $y \in Y \setminus f(X)$, then for any $p \in \soc$ the preimage $f^{-1}(\st{y}{p}) = \emptyset$ is metrically open in $X$ which implies that (arbitrary) $p$ is an open truth value.}, \ie when in addition to surjectivity the equivalence
					$$U \text{ metrically open in } \mtr{Y} \iff f^{-1}(U) \text{ metrically open in } \mtr{X}$$
					holds.
			\end{itemize}
		\end{definition}
		
		\begin{proposition}\label{Proposition: transfer_of_metrization_via_quotients}
			Let $f\colon \mtr{X} \to \mtr{Y}$ be a metric semiquotient map between metric spaces $\mtr{X}$, $\mtr{Y}$. Then if $\mtr{X}$ is metrized, so is $\mtr{Y}$.
		\end{proposition}
		\begin{proof}
			For every intrinsically open subset $U \subseteq Y$ its preimage $f^{-1}(U)$ is intrinsically open in $\mtr{X}$, thus metrically open because $\mtr{X}$ is metrized. Since $f$ is metric semiquotient, $U$ is metrically open in $\mtr{Y}$.
		\end{proof}
		Technically, a metric semiquotient map is sufficient for the transfer of metrization, but the maps for which we will use this proposition will be metric quotient. Note also that a retraction of an overt metrized space is a metric quotient map, so this proposition is a generalization of Proposition~\ref{Proposition: retracts_inherit_metrization}.
		
		It is well known that, assuming countable choice, every inhabited \cms is an image of the Baire space by an \ed-continuous map~\cite[7.2.4]{Troelstra_AS_Dalen_D_1988:_constructivism_in_mathematics_volume_2}, and similarly for \ctb[s] and the Cantor space. We reexamine these proofs in order to see that these maps are metric quotient. But first some preparation.
		
		\begin{lemma}
			A decidable metric subspace of a metrized space is metrized.\footnote{This lemma is later generalized by Theorem~\ref{Theorem: metrization_and_subspaces} (recall Lemma~\ref{Lemma: decidable_subsets_open/closed_subspaces}).}
		\end{lemma}
		\begin{proof}
			Let $\mtr{X} = (X, d)$ be a metric space and $A \subseteq X$ a decidable subset. Then if $U \subseteq A$ is intrinsically open in $A$, it is also open in $X$ (by Lemma~\ref{Lemma: decidable_subsets_open/closed_subspaces} and Proposition~\ref{Proposition: open/closed_subsets_of_strongly_open/closed_are_open/closed}), and can therefore be written as an overtly indexed union $U = \bigcup_{i \in I}{\ball{x_i}{r_i}}$. Define $J \dfeq \st{i \in I}{x_i \in A}$. Clearly
			$$U = \bigcup_{i \in J}{\ball{x_i}{r_i}}$$
			works.
		\end{proof}
		
		\begin{lemma}\label{Lemma: metrization_of_inhabited_spaces_sufficient}
			If all inhabited \cms[s] (resp.~\ctb[s]) are metrized, then all \cms[s] (resp.~\ctb[s]) are metrized.
		\end{lemma}
		\begin{proof}
			Any \cms[]/\ctb[] isometrically embeds into an inhabited one as a decidable subset by Lemma~\ref{Lemma: wlog_space_inhabited}. The result follows by the previous lemma.
		\end{proof}
		
		\begin{lemma}\label{Lemma: Lipschitz_bijection_reflects_metrization}
			Let $\mtr{X}$, $\mtr{Y}$ be metric spaces and $f\colon \mtr{X} \to \mtr{Y}$ a Lipschitz bijection. Then if $\mtr{Y}$ is metrized, so is $\mtr{X}$.
		\end{lemma}
		\begin{proof}
			If $f$ is Lipschitz, it is metric continuous by Proposition~\ref{Proposition: examples_of_metric_continuous_maps}, and then its inverse $f^{-1}\colon \mtr{Y} \to \mtr{X}$ is a metric semiquotient map. The result follows from Proposition~\ref{Proposition: transfer_of_metrization_via_quotients}.
		\end{proof}
		
		\begin{corollary}\label{Corollary: metrization_of_variants_of_Cantor_space_and_Hilbert_cube}
			Let $a\colon \NN \to \NN_{\geq 1}$ be a strictly increasing sequence, \ie $\xall{n}{\NN}{a_n < a_{n+1}}$. Define $b\colon \NN \to \NN$ by $b(n) \dfeq \inf\st{k \in \NN}{n < a_k}$, and $v\colon \NN \to \NN_{\geq 1}$ by $v(0) \dfeq a_0$, $v(n+1) \dfeq a_{n+1} - a_n$ (meaning that $v(n)$ is the number of elements in $b^{-1}(\{n\})$).
			\begin{enumerate}
				\item
					For any $A \subseteq \II$ the product $\prod_{n \in \NN} A^{v(n)}$ is isomorphic to $(A^\NN, d)$ where
					$$d(\alpha, \beta) = \sup\st{2^{-b(n)} \cdot |\alpha_n - \beta_n|}{n \in \NN},$$
					and the identity map $\prod_{n \in \NN} A^{v(n)} \to A^\NN$ is nonexpansive.
				\item
					Consider the special cases $A = \two$ and $A = \II$ where we obtain the variants of the Cantor space and the Hilbert cube,
					$$\Cantor_a \dfeq \prod_{n \in \NN} \two^{v(n)}, \qquad \Hilbert_a \dfeq \prod_{n \in \NN} \II^{v(n)}.$$
					Then if the standard Cantor space/Hilbert cube is metrized, so are their variants.
			\end{enumerate}
		\end{corollary}
		\begin{proof}
			\begin{enumerate}
				\item
					For $d$ simply recall the definition of the product metric for finite and countably infinite products (the latter with the gauge map $\id[\II]$). The second part holds because for all $n \in \NN$, $a_n > n$, therefore $b(n) \leq n$, therefore $2^{-n} \leq 2^{-b(n)}$.
				\item
					The identity from a variant of the Cantor space/Hilbert cube to the standard one is a nonexpansive bijection by the previous item, so the result follows from Lemma~\ref{Lemma: Lipschitz_bijection_reflects_metrization}.
			\end{enumerate}
		\end{proof}
		
		We write a complicated lemma because we want to deal with the Baire space and the Cantor space at once (we did a slightly simpler version with just the Baire space in~\cite{Bauer_A_Lesnik_D_2010:_metric_spaces_in_synthetic_topology}).
		\begin{lemma}
			Assume $\ACopn$. Suppose we are given
			\begin{itemize}
				\item
					$\mtr{X} = (X, d)$ a complete metric space,
				\item
					$s\colon \NN \to X$ some sequence of elements of $X$, and
				\item
					$a\colon \NN \to \opcN_{\geq 1}$ an increasing sequence, \ie $\xall{n}{\NN}{1 \leq a_n \leq a_{n+1}}$.
			\end{itemize}
			Define
			\begin{itemize}
				\item
					$A \dfeq \st{\alpha \in \NN^\NN}{\xall{n}{\NN}{\alpha(n) < a_n}}$ the subset of $\NN^\NN$, equipped with (the restriction of) the comparison metric $d_C$,
				\item
					$\delta\colon \NN \times \NN \times \NN \to \QQ$ to be such a map that for all $i, j, k \in \NN$,
					$$|d(s_i, s_j) - \delta(i, j, k)| < 2^{-k},$$
					\ie $\delta(i, j, k)$ is a rational $2^{-k}$-approximation of the distance between $s_i$ and $s_j$ (it exists by $\ACopn$),
				\item
					$T \dfeq \st{\alpha \in A}{\xall{k}{\NN}{\delta(\alpha(k), \alpha(k+1), k) < 2^{-k+2}}}$,
				\item
					$f\colon A \times \NN \to \NN$, $f(\alpha, k) \dfeq \inf\big(\{k\} \cup \st{j \in \NN}{\delta(\alpha(j), \alpha(j+1), j) \geq 2^{-k+2}}\big)$,
				\item
					$r\colon A \to T$, $r(\alpha)(k) \dfeq \alpha(f(\alpha, k))$.
			\end{itemize}
			Then the following holds.
			\begin{enumerate}
				\item
					The subset $A$ is a nonexpansive retract of the Baire space, and therefore metrically separable.
				\item
					The map $r$ is a nonexpansive (hence metric continuous) retraction, therefore $T$ is metrically separable.
				\item
					For every $\alpha \in T$ the sequence $k \mapsto s_{\alpha(k)}$ is Cauchy with the modulus of convergence $n \mapsto n+4$.
				\item
					The map $q\colon T \to X$, $q(\alpha) \dfeq \lim(k \mapsto s_{\alpha(k)})$, which is well defined by the previous item (and because $\mtr{X}$ is complete, therefore Cauchy complete), is metric continuous,
				\item
					Let $B \dfeq \st{\alpha \in T}{\shift{4}(\alpha) \in A}$ and $T' \dfeq \st{\shift{4}(\alpha)}{\alpha \in B}$. Then $B$ is a nonexpansive retract of $T$ (therefore metrically separable), $T' \subseteq T$, and for every $\alpha \in B$ we have $q(\alpha) = q(\shift{4}(\alpha))$.
				\item
					Hereafter assume that the restriction $\rstr{q}_{T'}\colon T' \to X$ is surjective. Then for any $\beta \in T'$ and any $r \in \intoc{0}{1}$ the following inclusion holds:
					$$\ball[\mtr{X}]{q(\beta)}{\frac{r}{8}} \subseteq q\Big(\ball[C]{\beta}{r}\Big).$$
				\item
					The map $q\colon T \to X$ is metric quotient.
				\item
					If $A$ is metrized (by the comparison metric), then so is $\mtr{X}$.
			\end{enumerate}
		\end{lemma}
		\begin{proof}
			\begin{enumerate}
				\item
					The map $\alpha \mapsto \big(n \mapsto \inf\{\alpha(n), \predm(a_n)\}\big)$ is a nonexpansive retraction $\NN^\NN \to A$.
				\item
					Notice that $f(\alpha, k)$ tells how far up to $k$ the sequence $\alpha$ still satisfies the defining condition for $T$. The sequence $r(\alpha)$ then mimics $\alpha$ up to the point where this condition is satisfied, and is constant thereafter, so $r(\alpha) \in T$ (actually we first need to see that $r(\alpha) \in A$, but this is so because $a$ is an increasing sequence). When $\alpha \in T$, we have $f(\alpha, k) = k$ for all $k \in \NN$, therefore $r(\alpha) = \alpha$, so $r$ is a retraction onto $T$. Notice that $r(\alpha)(k)$ depends only on the terms of $\alpha$ up to index $k$, so it is Lipschitz with the coefficient $1$ (\ie nonexpansive), and consequently metrically continuous by Proposition~\ref{Proposition: examples_of_metric_continuous_maps} (since $A$ is metrically separable).
				\item
					Take $\alpha \in T$ and $k \in \NN$. Then
					$$d_\mtr{X}(s_{\alpha(k)}, s_{\alpha(k+1)}) < \delta(\alpha(k), \alpha(k+1), k) + 2^{-k} < 2^{-k+2} + 2^{-k}  < 2^{-k+3},$$
					so for $n \in \NN$ and $i, j \in \NN_{\geq n+4}$ (without loss of generality $i \leq j$)
					$$d_\mtr{X}(s_{\alpha(i)}, s_{\alpha(j)}) \leq \sum_{k \in \intco[\NN]{i}{j}} d_\mtr{X}(s_{\alpha(k)}, s_{\alpha(k+1)}) \leq \sum_{k \in \intco[\NN]{i}{j}} 2^{-k+3} < 2^{-i+4} \leq 2^{-n}.$$
				\item
					Since for $\alpha \in T$ and $k \in \NN$ we have $d_\mtr{X}(s_{\alpha(k)}, s_{\alpha(k+1)}) < 2^{-k+3}$, it follows that for any $n \in \NN$ we have $d_\mtr{X}(q(\alpha), s_{\alpha(n)}) \leq 2^{-n+4}$. Consequently, for any $\alpha, \beta \in T$ and $n \in \NN$, if $d_C(\alpha, \beta) < 2^{-n}$ (so that the sequences $\alpha$, $\beta$ match at least up to the index $n$), then
					$$d_\mtr{X}(q(\alpha), q(\beta)) \leq d_\mtr{X}(q(\alpha), s_{\alpha(n)}) + d_\mtr{X}(s_{\alpha(n)}, q(\beta)) =$$
					$$= d_\mtr{X}(q(\alpha), s_{\alpha(n)}) + d_\mtr{X}(s_{\beta(n)}, q(\beta)) \leq 2^{-n+4} + 2^{-n+4} = 2^{-n+5}.$$
					Now take any $\alpha, \beta \in T$ such that $d_C(\alpha, \beta) < 1$. Use the Archimedean property of the reals and the cotransitivity of $<$ to find $n \in \NN$ such that $d_C(\alpha, \beta) < 2^{-n} < 3 \cdot d_C(\alpha, \beta)$. Then
					$$d_\mtr{X}(q(\alpha), q(\beta)) \leq 2^{-n+5} = 32 \cdot 2^{-n} < 96 \cdot d_C(\alpha, \beta),$$
					proving that $q$ is Lipschitz when restricted to balls with radius $1$, thus (since $T$ is metrically separable) metrically continuous by Proposition~\ref{Proposition: examples_of_metric_continuous_maps}.
				\item
					Since $B = \st{\alpha \in T}{\xall{n}{\NN}{\alpha(n+4) < a_n}}$, it is a nonexpansive retract of $T$ by the map
					$$\alpha \mapsto \left(n \mapsto \begin{cases} \alpha(n) & \text{ if } n < 4,\\ \inf\{\alpha(n), \predm(a_{n-4})\} & \text{ if } n \geq 4. \end{cases}\right)$$
					
					For any $\alpha \in B$ let $\beta \dfeq \shift{4}(\alpha)$. Then $\beta \in A$ by the definition of $B$, and for any $k \in \NN$,
					$$\delta(\beta(k), \beta(k+1), k) = \delta(\alpha(k+4), \alpha(k+5), k) < d_\mtr{X}(s_{\alpha(k+4)}, s_{\alpha(k+5)}) + 2^{-k} <$$
					$$< \delta(\alpha(k+4), \alpha(k+5), k+4) + 2^{-k-4} + 2^{-k} < 2^{-k-2} + 2^{-k-4} + 2^{-k}  = 21 \cdot 2^{-k-4} < 2^{-k+2},$$
					thus $\beta \in T$. Clearly $q(\alpha) = q(\beta)$.
				\item
					Let $\beta \dfeq \shift{4}(\alpha) \in T'$, $r \in \intoc{0}{1}$, and $x \in X$ such that $d_\mtr{X}(q(\beta), x) < \frac{r}{8}$. Since $\rstr{q}_{T'}$ is surjective, we have $\gamma \in T'$, $q(\gamma) = x$. Observe that for any $i, j \in \NN$,
					$$\sum_{k \in \intco[\NN]{i}{j}} d_\mtr{X}(s_{\gamma(k)}, s_{\gamma(k+1)}) \leq \sum_{k \in \intco[\NN]{i}{j}} \delta(\gamma(k), \gamma(k+1), k+4) + 2^{-k-4} \leq$$
					$$\leq \sum_{k \in \intco[\NN]{i}{j}} \big(2^{-k-2} + 2^{-k-4}\big) \leq \sum_{k \in \intco[\NN]{i}{j}} 2^{-k-1} < 2^{-i},$$
					so $d_\mtr{X}(s_{\gamma(i)}, x) \leq 2^{-i}$.
					
					By the Archimedean property and cotransitivity find $n \in \NN$ such that $\frac{r}{8} < 2^{-n-1} < 2^{-n} < r$. Define $\zeta\colon \NN \to \NN$ by
					$$\zeta(k) \dfeq \begin{cases} \beta(k) & \text{ if } k \leq n,\\ \gamma(k) & \text{ if } k > n. \end{cases}$$
					Since $\beta$ and $\gamma$ belong to $A$, so does $\zeta$. Moreover, since $\beta, \gamma \in T' \subseteq T$ and
					$$\delta(\beta(n), \gamma(n+1), n) < d_\mtr{X}(s_{\beta(n)}, s_{\gamma(n+1)}) + 2^{-n} \leq$$
					$$\leq d_\mtr{X}(s_{\beta(n)}, q(\beta)) + d_\mtr{X}(q(\beta), x) + d_\mtr{X}(x, s_{\gamma(n+1)}) + 2^{-n} <$$
					$$< 2^{-n} + 2^{-n-1} + 2^{-n-1} + 2^{-n} < 2^{-n+2},$$
					we have $\zeta \in T$. Clearly $q(\zeta) = x$ and $\zeta \in \ball[C]{\beta}{2^{-n}} \subseteq \ball[C]{\beta}{r}$, hence the result follows.
				\item
					If the restriction of $q$ is surjective, then $q$ must be as well. We already know that $q$ is metrically continuous. We prove it is also metric semiquotient.
					
					Take any $U \subseteq X$, and suppose $q^{-1}(U) \subseteq T$ can be represented as an overtly indexed union $q^{-1}(U) = \bigcup_{i \in I} \ball[C]{\alpha_i}{r_i}$. Recall that $B$ is metrically separable, and let $D \subseteq B$ be the countable (hence overt) metrically dense subset. Define
					$$J \dfeq \st[1]{(\alpha, i) \in D \times I}{\shift{4}(\alpha) \in \ball[C]{\alpha_i}{r_i}}.$$
					We claim
					$$U = \bigcup_{(\alpha, i) \in J} \ball[\mtr{X}]{q(\shift{4}(\alpha))}{\frac{\inf\{r_i, 1\}}{8}}.$$
					
					Take any $(\alpha, i) \in J$, and use the result in the previous item to calculate
					$$\ball[\mtr{X}]{q(\shift{4}(\alpha))}{\frac{\inf\{r_i, 1\}}{8}} \subseteq q\Big(\ball[C]{\shift{4}(\alpha)}{\inf\{r_i, 1\}}\Big) \subseteq$$
					$$\subseteq q\big(\ball[C]{\shift{4}(\alpha)}{r_i}\big) = q\big(\ball[C]{\alpha_i}{r_i}\big) \subseteq q(q^{-1}(U)) \subseteq U,$$
					thus proving that the right-hand side is contained in $U$.
					
					Conversely, take any $x \in U$. Since $\rstr{q}_{T'}$ is surjective, there exist $i \in I$ and $\beta \in T'$ such that $\beta \in \ball[C]{\alpha_i}{r_i}$ and $x = q(\beta)$. Let $\gamma \in B$ be such that $\beta = \shift{4}(\gamma)$. Recall that $q$ is metrically continuous, therefore \ed-continuous, so there exists $\eta \in \RR_{> 0}$ such that for all $\omega \in T$, if $d_C(\beta, \omega) < \eta$, then $d_\mtr{X}(q(\beta), q(\omega)) < \frac{\inf\{r_i, 1\}}{8}$. Now let $\alpha \in D$ be such that $d_C(\gamma, \alpha) < \frac{\inf\{r_i, \eta\}}{16}$; then $d_C(\shift{4}(\gamma), \shift{4}(\alpha)) < \inf\{r_i, \eta\}$. We claim that $(\alpha, i) \in J$; indeed, $\shift{4}(\alpha) \in \ball[C]{\shift{4}(\gamma)}{r_i} = \ball[C]{\beta}{r_i} = \ball[C]{\alpha_i}{r_i}$. To conclude that $U$ is contained in the right-hand side, observe $d_\mtr{X}(q(\shift{4}(\alpha)), x) = d_\mtr{X}(q(\shift{4}(\alpha)), q(\beta)) < \frac{\inf\{r_i, 1\}}{8}$ since $d_C(\beta, \shift{4}(\alpha)) < \eta$.
					
					Finally, we wish to see that the union on the right-hand side is overtly indexed. This is so by Corollary~\ref{Corollary: overt_combination} since the inclusion $D \hookrightarrow T$ and the indexing $I \to T \times \RR$ are overt, and $J$ is the preimage of the subset
					$$\st[1]{(\alpha, \beta, r) \in D \times T \times \RR}{\shift{4}(\alpha) \in \ball[C]{\beta}{r}}$$
					which is open in $D \times T \times \RR$.
				\item
					Assume $A$ is metrized by $d_C$. Then it is overt by Proposition~\ref{Proposition: metric_to_intrinsic_separability_in_weakly_metrized_spaces}, and $T$ is metrized (by $d_C$) by Proposition~\ref{Proposition: retracts_inherit_metrization}. Consequently $\mtr{X}$ is metrized by the previous item and Proposition~\ref{Proposition: transfer_of_metrization_via_quotients}.
			\end{enumerate}
		\end{proof}
		
		\begin{theorem}\label{Theorem: metrization_of_Baire/Cantor_implies_metrization_of_cmss_ctbs}
			Assume $\ACopn$.
			\begin{enumerate}
				\item
					If the Baire space is metrized, then all \cms[s] are.
				\item
					If the Cantor space is metrized, then all \ctb[s] are.
			\end{enumerate}
		\end{theorem}
		\begin{proof}
			Let $\mtr{X} = (X, d, s)$ be a complete metrically separable space. By Lemma~\ref{Lemma: metrization_of_inhabited_spaces_sufficient} we may without loss of generality assume that $X$ is inhabited, so let $s\colon \NN \to X$.
			\begin{enumerate}
				\item
					Use the preceding lemma for $a$ the constant sequence $\infty$ in which case $A = \NN^\NN$ and $T = T'$. We have to see that $q\colon T \to X$ is surjective. Take any $x \in X$, and use $\ACopn$ to obtain $\alpha \in \NN^\NN$ such that for all $n \in \NN$, $d_\mtr{X}(x, s_{\alpha(n)}) < 2^{-n}$. Then
					$$\delta(\alpha(k), \alpha(k+1), k) < d(s_{\alpha(k)}, s_{\alpha(k+1)}) + 2^{-k} \leq 2^{-k} + 2^{-k-1} + 2^{-k} < 2^{-k+2},$$ hence $\alpha \in T$ and $q(\alpha) = x$.
				\item
					Let $a\colon \NN \to \NN$ witness total boundedness of $\mtr{X}$. Define a map $bl\colon \NN \to \NN_{\geq 1}$,
					$$bl(n) \dfeq \inf\st{k \in \NN_{\geq 1}}{n \leq 2^k}$$
					(\ie the binary logarithm, rounded up to at least $1$) and sequences $v\colon \NN \to \NN_{\geq 1}$, $a'\colon \NN \to \NN_{\geq 2},$
					$$v(n) \dfeq bl\big(\sup\st{a(k)}{k \in \NN_{\leq n+4}}\big), \qquad a'(n) \dfeq 2^{v(n)}.$$
					Use the preceding lemma for $a'$. Then $A = \prod_{n \in \NN} \two^{v(n)}$ where $\two^{v(n)}$ is equipped with the discrete metric. Since the Cantor space is metrized, so is its variant $A$ by Corollary~\ref{Corollary: metrization_of_variants_of_Cantor_space_and_Hilbert_cube}. It remains to verify that $\rstr{q}_{T'}$ is surjective. For any $x \in X$ use $\ACopn$ and total boundedness of $\mtr{X}$ to obtain $\alpha \in \NN^\NN$ such that for all $n \in \NN$, $\alpha(n) < a(n)$ and $d_\mtr{X}(x, s_{\alpha(n)}) < 2^{-n}$. We infer $\alpha \in T$ the same way as in the previous item, and since we shifted the index by $4$ in the definition of $v$, we in fact have $\alpha \in T'$. Also, $q(\alpha) = x$.
			\end{enumerate}
		\end{proof}
		
		Recall that we already discussed metrization of the Cantor space at the end of the previous section.
		\begin{corollary}
			If $\ACopn$ and \wso hold, and $\two^\NN$ is compact, then all \ctb[s] are compact and metrized.
		\end{corollary}
		\begin{proof}
			Metrization of \ctb[s] follows from Theorems~\ref{Theorem: metrization_of_Cantor}~and~\ref{Theorem: metrization_of_Baire/Cantor_implies_metrization_of_cmss_ctbs}. As for compactness, any \ctb embeds as a decidable subset of an inhabited \ctb (by Lemma~\ref{Lemma: wlog_space_inhabited}) which is an image of $\two^\NN$, so compact, and then so is its decidable subset by Corollary~\ref{Corollary: decidable_subsets_of_overt/compact_sets}.
		\end{proof}

	\section{Transfer of Metrization via Embeddings}\label{Section: metrization_via_embeddings}
	
		Our next task is to examine how metrization of a space affects metrization of its metric subspaces. Clearly, metrization is not hereditary --- the real numbers might be metrized by the Euclidean metric (they are, for example, in models presented in Sections~\ref{Section: Brouwer's_intuitionism_model} and~\ref{Section: gros_topos_model}), but the (discrete) rationals cannot be (per Proposition~\ref{Proposition: overt_sets_with_decidable_equality_metrized} and the discussion below it). A metric subspace (at least when it contains a subovert metrically dense subset) inherits the metric topology\footnote{One direction is Proposition~\ref{Proposition: examples_of_metric_continuous_maps}(\ref{Proposition(examples_of_metric_continuous_maps)item: inclusion_of_metric_subspace_metrically_continuous}), and for the other, keep the same centers and radii in the union of balls when extending a metrically open subset.}. This suggests that in order to inherit metrization, it must inherit the intrinsic topology as well, \ie have the subspace topology. Here is the precise statement.
		\begin{theorem}\label{Theorem: metrization_and_subspaces}
			Let $A \subseteq X$, let $\mtr{X} = (X, d_\mtr{X})$ be a metric space and $\mtr{A} = (A, d_\mtr{A})$ its metric subspace.
			\begin{enumerate}
				\item If $\mtr{A}$ is metrized, then $A$ is a subspace of $X$.
				\item If $\mtr{X}$ is metrized, $A$ a subspace of $X$, and $D \subseteq A$ a subovert metrically dense subset of $A$, then $\mtr{A}$ is metrized.
			\end{enumerate}
		\end{theorem}
		\begin{proof}
			\begin{enumerate}
				\item
					Take any $U \in \tp(A)$. By assumption that $d_\mtr{A}$ metrizes $A$ the subset $U$ is an overtly indexed union $U = \bigcup_{i \in I} \ball[\mtr{A}]{a_i}{r_i}$. Then $V \dfeq \bigcup_{i \in I} \ball[\mtr{X}]{a_i}{r_i}$ is an open subset of $X$ such that $V \cap A = U$.
				\item
					Take any $U \in \tp(A)$. Since $A$ is a subspace in $X$, there is some $V \in \tp(X)$ such that $V \cap A = U$. Because $X$ is metrized, $V$ can be written as an overtly indexed union $V = \bigcup_{i \in I} \ball[\mtr{X}]{x_i}{r_i}$. We claim
					$$U = \bigcup_{(a, i) \in D \times I} \ball[\mtr{A}]{a}{r_i - d_\mtr{X}(x_i, a)}.$$
					
					Take any $x \in U$. Then $x \in V$, so there is some $i \in I$ such that $x \in \ball[\mtr{X}]{x_i}{r_i}$. Denote $\delta \dfeq \frac{r_i - d_\mtr{X}(x_i, x)}{2} > 0$, and let $a \in D$ be such that $d_\mtr{A}(x, a) < \delta$. Then
					$$2 \cdot d_\mtr{A}(a, x) < r_i - d_\mtr{X}(x_i, x) \leq r_i - d_\mtr{X}(x_i, a) + d_\mtr{A}(a, x),$$
					so $d_\mtr{A}(a, x) < r_i - d_\mtr{X}(x_i, a)$, proving that $x$ is in the right-hand side.
					
					Conversely, take some $x \in A$ such that $x \in \ball[\mtr{A}]{a}{r_i - d_\mtr{X}(x_i, a)}$ for some $a \in D$, $i \in I$. Then
					$$d_\mtr{X}(x_i, x) \leq d_\mtr{X}(x_i, a) + d_\mtr{X}(a, x) <  d_\mtr{X}(x_i, a) + r_i - d_\mtr{X}(x_i, a) = r_i,$$
					so $x \in V$, and therefore $x \in U$.
					
					Clearly the union on the right-hand side is overtly indexed, concluding the proof.
			\end{enumerate}
		\end{proof}
		
		\begin{corollary}\label{Corollary: metrization_of_subspaces}
			A separable metric subspace of a metrized space is metrized if any only if it has the subspace intrinsic topology.
		\end{corollary}
		\begin{proof}
			Immediate from Theorem~\ref{Theorem: metrization_and_subspaces}.
		\end{proof}
		
		Even though we have a characterization when metrization is inherited, it is in general difficult to use it since it is usually difficult to answer whether a subset is a subspace. There is one exception however --- recall that if $\opn$ is a dominance, the open subsets are subspaces, and if $\cld$ is a codominance, then closed subsets are subspaces. We are trying to characterize metrization of \cms[s] and \ctb[s] which classically are seldom open subsets in some larger metric space, but are necessarily closed. Going by this classical intuition, the idea is that, assuming $\cld$ is a codominance, complete metric subspaces inherit metrization.
		
		Unfortunately it is not so simple in the case of synthetic topology. Observe that $\one$, equipped with its only possible metric, is metrized, and all (not only closed) its metric subspaces are complete. We therefore need a stronger condition than completeness to infer closedness.
		\begin{definition}
			A subset $A \subseteq X$ is \df{strongly located} in a metric space $\mtr{X} = (X, d)$ when it is located, and $A = \st{x \in X}{d(A, x) = 0}$.
		\end{definition}
		
		If a set is located, its completeness is sufficient for strong locatedness.
		\begin{lemma}\label{Lemma: complete_located_is_strongly_located}
			Let $\mtr{X} = (X, d_\mtr{X})$ be a metric space, $A \subseteq X$ located in $\mtr{X}$, and $\mtr{A} = (A, d_\mtr{A})$ a metric subspace of $\mtr{X}$ with $D \subseteq A$ a subovert metrically dense subset. Then if $\mtr{A}$ is complete, it is strongly located.
		\end{lemma}
		\begin{proof}	
			Let $\lvl \subseteq \RR_{> 0}$ be a suitable domain for Cauchy families, and let $x \in X$ be arbitrary such that $d_\mtr{X}(A, x) = 0$. For $p \in \lvl$ define $S_p \dfeq \st{a \in D}{d_\mtr{X}(x, a) < p} = \ball[\mtr{X}]{x}{p} \cap D$. It is easily seen that this defines a Cauchy family $S$ in $\mtr{A}$, and since $\mtr{A}$ is complete, there is $b \in A$ which represents $[S]$. Since for any $p \in \lvl$ and any $a \in S_p$ we have $d_\mtr{X}(x, a) < p$ and $d_\mtr{A}(b, a) < p$, we conclude $d_\mtr{X}(x, b) = 0$, so $x = b \in A$.
		\end{proof}
		
		We will succeed in embedding relevant metric spaces as strongly located subsets, and will therefore state the assumption in the theorems ``assume strongly located subsets are subspaces''. Here is a sufficient condition for this assumption to hold.
		\begin{proposition}\label{Proposition: when_strongly_located_subsets_are_subspaces}
			If $\cld$ is a codominance (or at least $\nnst[\cld] \subseteq \overline{\cld}$, \ie stable closed subsets are subspaces) and $\RR$ is Hausdorff, then strongly located subsets are subspaces.
		\end{proposition}
		\begin{proof}
			If $A$ is strongly located in $\mtr{X} = (X, d)$, then $A = \big(d(A, \insarg)\big)^{-1}(\{0\})$. The singleton $\{0\}$ is (stable) closed in Hausdorff $\RR$, so its preimage $A$ is (stable) closed in $X$, therefore a subspace by assumption.
		\end{proof}
		
		These conditions reveal another reason why we wish to transfer metrization both via quotients and embeddings. Recall that two typical types of models of synthetic topology are sheaf and realizability topoi. The latter validate countable choice, so we can use the results of the previous section. The former generally don't, but (for the typical choice of \Sier object) stable closed subsets are subspaces in them (and $\RR$ is Hausdorff in both types of models by Corollary~\ref{Corollary: reals_are_Hausdorff}), so we can use the results below. See Chapter~\ref{Chapter: models} for examples.
		
		\intermission
		
		We first deal with \ctb[s]. Recall that classically every totally bounded metric space embeds into the Hilbert cube, itself a \ctb[]. This result works constructively as well. We recall (and adapt) the proof.
		\begin{theorem}
			Let
			\begin{itemize}
				\item
					$\mtr{X} = (X, d_\mtr{X}, s\colon \NN \to X, a\colon \NN \to \NN_{\geq 1})$ be an inhabited totally bounded metric space such that $a$ is a strictly increasing sequence;
				\item
					the sequences $b\colon \NN \to \NN$, $v\colon \NN \to \NN_{\geq 1}$, and the variant of the Hilbert cube $\Hilbert_a$ be as in Corollary~\ref{Corollary: metrization_of_variants_of_Cantor_space_and_Hilbert_cube};
				\item
					$C \dfeq \frac{1}{\sup\{\diam(\mtr{X}), 1\}} \hspace{1ex} \in \intoc{0}{1}$.
			\end{itemize}
			Then there exists (a canonical choice of) an injective Lipschitz map $e\colon \mtr{X} \to \Hilbert_a$ with a located image. Explicitly, let
			$$e(x)_n \dfeq C \cdot d_\mtr{X}(x, s_n);$$
			then the image of $e$ is indeed contained in $\II^\NN$, and the following holds.
			\begin{enumerate}
				\item
					The map $e$ is Lipschitz with coefficient $C$, \ie $d(e(x), e(y)) \leq C \cdot d_\mtr{X}(x, y)$ for all $x, y \in X$; since $C \leq 1$, it is in fact nonexpansive.
				\item
					For all $x, y \in X$,
					$$\frac{C \cdot d_\mtr{X}(x, y)^2}{32} \leq d(e(x), e(y)).$$
					In particular, if $d_\mtr{X}(x, y) > 0$, then $d(e(x), e(y)) > 0$, so $e$ is injective (in the strong sense).
				\item
					The image of $e$ is located in $\Hilbert_a$.
				\item
					If $\mtr{X}$ is complete, then the image of $e$ is strongly located in $\Hilbert_a$.
			\end{enumerate}
		\end{theorem}
		\begin{proof}
			\begin{enumerate}
				\item
					For $x, y \in X$ we have
					$$d(e(x), e(y)) = \sup\st{2^{-b(n)} \cdot |e(x) - e(y)|}{n \in \NN} =$$
					$$= \sup\st{\frac{C}{2^{b(n)}} \cdot |d_\mtr{X}(x, s_n) - d_\mtr{X}(y, s_n)|}{n \in \NN} \leq$$
					$$\leq \sup\st{\frac{C}{2^{b(n)}} \cdot d_\mtr{X}(x, y)}{n \in \NN} \leq C \cdot d_\mtr{X}(x, y),$$
					proving that $e$ is Lipschitz.
				\item
					Assume to the contrary that there are $x, y \in X$ such that $d(e(x), e(y)) <\frac{C \cdot d_\mtr{X}(x, y)^2}{32}$. Then in particular $0 < C \cdot d_\mtr{X}(x, y) \leq 1$, so there exists $n \in \NN$ with the property $2^{-n-1} < C \cdot d_\mtr{X}(x, y) < 2^{-n+1}$. Let $i \in \NN_{< a(n+3)}$ be such that  $d_\mtr{X}(s_i, x) < 2^{-n-3}$. Then $d_\mtr{X}(s_i, x) < \frac{2^{-n-1}}{4} < \frac{C \cdot d_\mtr{X}(x, y)}{4} \leq \frac{d_\mtr{X}(x, y)}{4}$, and so
					$$d(e(x), e(y)) \geq 2^{-b(i)} \cdot C \cdot |d_\mtr{X}(x, s_i) - d_\mtr{X}(y, s_i)| \geq 2^{-n-3} \cdot C \cdot (d_\mtr{X}(y, s_i) - d_\mtr{X}(x, s_i)) \geq$$
					$$\geq 2^{-n-3} \cdot C \cdot (d_\mtr{X}(x, y) - 2 d_\mtr{X}(x, s_i)) \geq 2^{-n-4} \cdot C \cdot d_\mtr{X}(x, y) \geq \frac{C \cdot d_\mtr{X}(x, y)^2}{32},$$
			a contradiction.
				\item
					We define the map $l\colon \II^\NN \to \RR$ as $l(\alpha) \dfeq (L, U)$ where
					$$L \dfeq \st{q \in \oirs}{\xsome{n}{\NN}\all{i}{\NN_{< a(b(n))}}{q + 2^{-b(n)} \cdot C < d\big(e(s_i), \alpha\big)}},$$
					$$U \dfeq \st{r \in \oirs}{\xsome{j}{\NN}{r > d\big(e(s_j), \alpha\big)}}.$$
					It is immediate that $L$ is a lower cut, $U$ is an upper cut, and that they are open and inhabited. Assume $q \in L \cap U$, \ie there are $n, j \in \NN$ such that $\xall{i}{\NN_{< a(b(n))}}{q + 2^{-b(n)} \cdot C < d(e(s_i), \alpha)}$ and $q > d(e(s_j), \alpha)$. There exists $k \in \NN_{< a(b(n))}$ such that $d_\mtr{X}(s_j, s_k) < 2^{-b(n)}$, and then $d(e(s_j), e(s_k)) < 2^{-b(n)} \cdot C$. This leads to the contradiction
					$$q > d(e(s_j), \alpha) \geq d(e(s_k), \alpha) - d(e(s_j), e(s_k)) > d(e(s_k), \alpha) - 2^{-b(n)} \cdot C > q.$$
					Let $q, r \in \oirs$, $q < r$, and let $n \in \NN$ be large enough so that $2^{-b(n)} \cdot C < r-q$. For each $i \in \NN_{< a(b(n))}$ make a choice of a true disjunct in
					$$q + 2^{-b(n)} \cdot C < d\big(e(s_i), \alpha\big) \quad \lor \quad d\big(e(s_i), \alpha\big) < r.$$
					If the first disjunct is always chosen, then $q \in L$. If the second is chosen at least once, then $r \in U$. We proved $(L, U)$ is an open Dedekind cut, so $l$ indeed maps into $\RR$, and it is clear that $l = d(e(X), \insarg)$, so $e(X)$ is located.
				\item
					Take any $\alpha \in \II^\NN$ such that $d(e(X), \alpha) = 0$. Let $\lvl = \st{2^{-n}}{n \in \NN}$, and define, for $n \in \NN$,
					$$I_{2^{-n}} \dfeq \st{i \in \NN}{\frac{32 \cdot d(e(s_i), \alpha)}{C} < 2^{-2n-1}}, \qquad S_{2^{-n}} \dfeq s(I_{2^{-n}}).$$
					Since $S_{2^{-n}}$ is the image of the composition of open inclusion $I_{2^{-n}} \hookrightarrow \NN$, the overt map $\id[\NN]$, and $s$, it is subovert in $X$. It is inhabited because
					$$0 = d(e(X), \alpha) = \inf\st{d(e(x), \alpha)}{x \in X} = \inf\st{d(e(s_i), \alpha)}{i \in \NN}.$$
					Now take any $a, b \in \NN$, $i \in I_{2^{-a}}$, $j \in I_{2^{-b}}$. Then
					$$d_\mtr{X}(s_i, s_j) \leq \sqrt{\frac{32 \cdot d(e(s_i), e(s_j))}{C}} \leq \sqrt{\frac{32 \cdot (d(e(s_i), \alpha) + d(e(s_j), \alpha))}{C}} \leq$$
					$$\leq \sqrt{2^{-2a-1} + 2^{-2b-1}} \leq \sqrt{2^{-2\inf\{a, b\}}} = 2^{-\inf\{a, b\}} \leq 2^{-a} + 2^{-b},$$
					proving that $S$ is a Cauchy family. By the assumption $\mtr{X}$ is complete, so there exists $y \in X$ which is represented by $[S]$, \ie for every $k \in \NN$ and $i \in I_{2^{-k}}$ we have $d_\mtr{X}(y, s_i) \leq 2^{-k}$.
					
					Take any $n, k \in \NN$; we wish to prove $|d_\mtr{X}(y, s_n) - \alpha_n| < 2^{-k}$. Let $m \in \NN$ strictly exceed $k$ and $\frac{b(n) + k - 5}{2}$, and take some $i \in I_{2^{-m}}$.
					$$|d_\mtr{X}(y, s_n) - \alpha_n| \leq d_\mtr{X}(y, s_i) + |d_\mtr{X}(s_i, s_n) - \alpha_n| \leq$$
					$$\leq 2^{-m} + \frac{2^{b(n)}}{C} 2^{-b(n)} C |d_\mtr{X}(s_i, s_n) - \alpha_n| \leq 2^{-m} + \frac{2^{b(n)}}{C} d(e(s_i), \alpha) <$$
					$$< 2^{-m} + 2^{b(n) - 2 m - 6} \leq 2^{-k-1} + 2^{-k-1} = 2^{-k}$$
					Since $k$ is arbitrary, $|d_\mtr{X}(y, s_n) - \alpha_n| = 0$ for all $n \in \NN$, and therefore $d(e(y), \alpha) = 0$, proving that $\alpha = e(y)$.
			\end{enumerate}
		\end{proof}
		
		\begin{theorem}\label{Theorem: metrization_of_Hilbert_cube_implies_metrization_of_ctbs}
			Assume strongly located subsets are subspaces.
			\begin{enumerate}
				\item
					If the Hilbert cube is metrized, then all \ctb[s] are.
				\item
					If $\RR$ is Hausdorff and the Hilbert cube is compact, then all \ctb[s] are compact.
			\end{enumerate}
		\end{theorem}
		\begin{proof}
			By Lemma~\ref{Lemma: metrization_of_inhabited_spaces_sufficient}, Lemma~\ref{Lemma: wlog_space_inhabited} and Corollary~\ref{Corollary: decidable_subsets_of_overt/compact_sets} we may restrict our attention to inhabited \ctb[s], so let $\mtr{X} = (X, d, s, a)$ be a \ctb with $s\colon \NN \to X$.
					
			Without loss of generality we may assume that $a$ is a strictly increasing sequence since we can always replace it with $n \mapsto \sup\st{a_k}{k \in \NN_{\leq n}} + n$ (and it has the image in $\NN_{\geq 1}$ because $X$ is inhabited). By Corollary~\ref{Corollary: metrization_of_variants_of_Cantor_space_and_Hilbert_cube} the variant of the Hilbert cube $\Hilbert_a$ is metrized, and the previous theorem supplies us with the map $e\colon X \to \Hilbert_a$ with a strongly located image $e(X)$ which is then a subspace by assumption, and metrized by Corollary~\ref{Corollary: metrization_of_subspaces}. Since the map $e$ is Lipschitz and injective (therefore bijective onto its image), the metric space $\mtr{X}$ is metrized by Lemma~\ref{Lemma: metrization_of_inhabited_spaces_sufficient}.
					
			Moreover, if $\RR$ is Hausdorff, strongly located subsets are closed (being preimages of $\{0\} \subseteq \RR$). Compactness of the Hilbert cube implies compactness of its variants (since there are bijections between them), and consequently compactness of $e(X)$ by Corollary~\ref{Corollary: strongly open/closed in overt/condensed/compact}. Since $e$ is bijective onto its image, we infer compactness of $X$.
		\end{proof}
		
		\intermission
		
		Having characterized metrization of \ctb[s] with metrization of a specific one, we wish to do the same with \cms[s]. We need a specific \cms into which we can embed any (inhabited) \cms as a strongly located subset. Here is how we could do it classically: it is known~\cite{Urysohn_PS_1927:_sur_un_espace_metrique_universel, Katetov_M_1988:_on_universal_metric_spaces, Holmes_MR_1992:_the_universal_separable_metric_space_of_urysohn_and_isometric_embeddings_thereof_in_banach_spaces} that every inhabited separable metric space $\mtr{X}$ isometrically embeds into the Urysohn space $\Ury$, itself a \cms[]. Classically, every inhabited subset is located, and if $\mtr{X}$ is complete, so is its isometric image, therefore it is closed in $\Ury$, and hence strongly located.
		
		This will serve us as the general idea, but our purpose is to provide a constructive solution. In particular, we need to define and construct the Urysohn space in constructive setting. We choose to diverge from the Urysohn's original (classical) construction~\cite{Urysohn_PS_1927:_sur_un_espace_metrique_universel}, and present our own. Actually we already did this in~\cite{Lesnik_D_2009:_constructive_urysohn_universal_metric_space} where the reader may find a lot of additional commentary whereas here we focus more on the proof of the existence of $\Ury$, as well as its usage in metrization theory. But first we discuss what the constructive definition of the Urysohn space should actually be.
		
		As usual, we call a map, defined on a subset of a set $X$, a \df{partial} map on $X$. In addition, call a partial map with finite domain a \df{finite partial} map. The classical definition of the Urysohn space $\Ury$ is that it is a \cms such that for every separable metric space $\mtr{X}$, every finite partial isometry $f\colon \mtr{X} \parto \Ury$ extends to the whole of $\mtr{X}$. Such Urysohn space is unique up to isometric isomorphism, but this does not concern us in this thesis --- any space satisfying these properties will serve us for the transfer of metrization. Our purpose is to actually construct one.
		
		We are interested in what the constructive content of extending finite partial isometries is. There are at least three issues here. First, a typical proof of the existence of Urysohn space starts with a separable metric space $(X, d_\mtr{X}, s)$ where $s$ is a map $s\colon \NN \to X$, \ie inhabitedness of $X$ is implicitly assumed. Classically this isn't a problem; the empty set is a trivial special case. Constructively however we must determine whether inhabitedness of $X$ is crucial. It turns out that it is not; one could use Lemma~\ref{Lemma: wlog_space_inhabited} to reduce the problem to just inhabited spaces, but as we shall see, we can use $s$ as a map $s\colon \NN \to \one + X$ directly. Not to mention that for our theory of metrization we could restrict to inhabited spaces anyway by Lemma~\ref{Lemma: metrization_of_inhabited_spaces_sufficient}.
		
		Second, we are using the Urysohn space in order to isometrically embed \cms[s] into it. The existence of an isometric embedding is immediate --- extend the finite partial isometry with the empty domain to the whole space. Classically we obtain the theorem that a metric space is metrically separable if and only if it is isometrically embeddable into the Urysohn space. The converse holds because second countability is a hereditary topological property, and for metric spaces second countability and separability are classically equivalent. Constructively this is not the case, and an analogous theorem would be that a metric space is isometrically embeddable into the Urysohn space if and only if it is isometrically embeddable into some metrically separable space. More generally, we could take for the definition of the Urysohn space that any finite partial isometry on a metric subspace of a metrically separable space extends to the whole space. To put it differently, perhaps constructively the right category in which to study the Urysohn space is not the category of metrically separable spaces, but rather the larger category of all their metric subspaces. We will not change the definition of the Urysohn space in light of this, however. Clearly, given a metric subspace of a metrically separable space, if we can solve the extension problem for the whole space, we can solve it for the subspace.
		
		The only issue which will actually be a (slight) step away from the classical treatment of the Urysohn space is the fact that the definition promises the existence of the extension, but does not say how constructive or canonical it is. Classically we could use the axiom of choice to provide a rule which takes a finite partial isometry from a metrically separable space to the Urysohn space, and returns its extension\footnote{This might seem too hasty a use of axiom of choice since metrically separable spaces form a proper class, as opposed to a set. However its skeleton is a set since the cardinalities of metrically separable spaces are bounded above by the cardinality of continuum.}. The question is, can we constructively obtain such a rule as well? The answer, as we will see, is yes. Thus we define: a \cms $\Ury$ is a Urysohn space when for every metrically separable metric space $\mtr{X}$ with a given enumeration of its countable metrically dense subset, its every finite subset $F$ with a given enumeration of its elements, and any partial isometry with domain $F$ into $\Ury$, there exists a canonical choice of its extension to the whole $\mtr{X}$.
		
		Let us be more precise what we mean by a `canonical choice'. The idea how to extend a partial isometry from a metrically separable space $\mtr{X}$ to $\Ury$ is as follows. Since the Urysohn space is by definition complete, it is sufficient to extend the isometry onto the countable metrically dense subset of $\mtr{X}$. This can be done inductively, one point at a time. This amounts to saying that if we choose finitely many points in $\Ury$ and the same number of nonnegative numbers representing distances, then there exists (a canonical choice of) a point in $\Ury$ which is at those distances from the chosen points (so long as the distances respect triangle inequalities). We state this as our chosen definition of the Urysohn space.
		
      \begin{definition}\label{Definition: Urysohn_space}
      	The tuple $(U, d_\mtr{U}, u, E)$ is a \df{Urysohn space} when
      	\begin{itemize}
      		\item
      			$(U, d_\mtr{U}, u)$ is a \cms with $u\colon \NN \to U$ a sequence with a metrically dense image,
      		\item
      			$E\colon P \to U$ is a map such that
      			$$d_\mtr{U}\Big(E\big(\Utuple[n]{i}{x}{\omega}\big), x_k\Big) = \omega_k$$
      			for all $\Utuple[n]{i}{x}{\omega} \in P$ and $k \in \NN_{< n}$ where
      		\item
      			$P \dfeq \st{\Utuple[n]{i}{x}{\omega} \in (U \times \RR_{\geq 0})^*}{\all{i, j}{\NN_{< n}}{\omega_i - \omega_j \leq d_\mtr{U}(x_i, x_j) \leq \omega_i + \omega_j}}$ is the set of all \df{permissible} tuples, \ie the ones representing points and distances from them which respect the triangle inequality.
      	\end{itemize}
      \end{definition}
		
		This gives us a starting point for the construction of the Urysohn space. The first approximation to it is a set $W$, inductively defined as follows: given any $n \in \NN$ elements $a_0, \ldots, a_{n-1} \in W$ and numbers $\alpha_0, \ldots, \alpha_{n-1} \in \lrs_{\geq 0}$, the tuple $\Utuple[n]{i}{a}{\alpha}$, representing the point which is at distance $\alpha_i$ from $a_i$ for all $i \in \NN_{< n}$, is also in $W$ (recall that $\lrs$ is used to denote an arbitrary lattice ring streak; we purposefully restrict the distances to $\lrs$, as it will be convenient for us to have the freedom of choosing $\lrs$ on the spot). Thus we start with the empty tuple $\eUt$, and then progressively construct new ones. We call $a_i$s the 
 \df{predecessors} of $a = \Utuple[n]{i}{a}{\alpha}$.
		
		We stratify $W$ by two criteria.
		\begin{itemize}
			\item
				The \df{length} of a tuple $a = \Utuple[n]{i}{a}{\alpha}$ is denoted by $\lUt(a) \dfeq n$ which defines a map $\lUt\colon W \to \NN$. The empty tuple is unique in having the length $0$, the others have length $\geq 1$.
			\item
				The \df{age} of a tuple $a$ is denoted by $\aUt(a)$, and is inductively defined as follows: if $\lUt(a) = 0$, \ie $a = \eUt$, then $\aUt(a) \dfeq 0$, and if $\lUt(a) \geq 1$ where $a = \Utuple{i}{a}{\alpha}$, then $\aUt(a) \dfeq \sup\st{\aUt(a_i)}{i \in \NN_{< \lUt(a)}} + 1$. This defines a map $\aUt\colon W \to \NN$, with the empty tuple unique in having the age $0$.
		\end{itemize}
		
		Next, we require the distance map $d\colon W \times W \to \lrs$. It turns out to be convenient to take as small distances as we can afford without violating the triangle inequality, so for $a = \Utuple{i}{a}{\alpha}, b = \Utuple{j}{b}{\beta} \in W$, we define inductively on $\aUt(a) + \aUt(b)$
		$$d(a, b) \dfeq \sup\Big(\st{|d(a_i, b) - \alpha_i|}{i \in \NN_{< \lUt(a)}} \cup \st{|d(a, b_j) - \beta_j|}{j \in \NN_{< \lUt(b)}}\Big).$$
		This, as it turns out, is only a protometric, but the tuples $\Utuple{i}{a}{\alpha}$ in which $\alpha_i - \alpha_j \leq d(a_i, a_j) \leq \alpha_i + \alpha_j$ for all $i, j \in \NN_{< \lUt(a)}$, \ie the ones where $\alpha_i$s actually represent allowable distances, are contained in the kernel of $(W, d)$, and thus form a pseudometric space. Its completion (equivalently, the completion of its Kolmogorov quotient) turns out to be the Urysohn space.
		
		We explicitly construct a model of such a set $W$. Let $\lrs_{\geq 0}^*$ denote the set of finite sequences of elements in $\lrs_{\geq 0}$. We will use these finite sequences to encode the tuples in $W$ by the following method: the first term of the sequence represents the age of the tuple, the second one represents the length, the next ones (as many as the length is) represent how long encodings of the predecessors are, and then the actual encodings of predecessors, together with the distances, follow.
		
		Here is the exact definition. Let $A^\lrs_n \subseteq \lrs_{\geq 0}^*$ be the set, inductively defined on $n \in \NN$ by
		\mlst{A^\lrs_n \dfeq}{\omega \in \lrs_{\geq 0}^*}
			{|\omega| \geq 2 \nc \omega_0 = n \nc \omega_1 \in \NN \nc (\omega_0 = 0 \iff \omega_1 = 0) \nl
			|\omega| \geq 2 + \omega_1 \nc \all{k}{\NN_{< \omega_1}}{\omega_{2+k} \in \NN} \nc |\omega| = a(\omega_1) \nl
			\all{k}{\NN_{< \omega_1}}{\omega_{a(k)} \in \NN_{< n} \land (\omega_{a(k)}, \ldots, \omega_{a(k+1)-2}) \in A^\lrs_{\omega_{a(k)}}} \nl
			\left(n \geq 1 \implies \omega_1 = \sup\st{\omega_{a(k)}}{k \in \NN_{< n}} + 1\right)\\
			&\text{where $a(k)$ is shorthand for $2 + \omega_1 + \sum_{i \in \NN_{< k}} (\omega_{2+i} + 1)$}}
		{;}
		in particular $A^\lrs_0 = \{(0, 0)\}$. For $t \in \opcN$ define $W^\lrs_t \dfeq \bigcup_{n \in \NN_{< \succm(t)}} A^\lrs_n$. Note that $A^\lrs_n$s are pairwise disjoint, so $W^\lrs_n$s are actually coproducts of $A^\lrs_n$s. Explicitly, we have
		\begin{align*}
			&W^\lrs_0 \ism A^\lrs_0 \quad \ism \one\\
			&\vsubset\\
			&W^\lrs_1 \ism A^\lrs_0 + A^\lrs_1\\
			&\vsubset\\
			&W^\lrs_2 \ism A^\lrs_0 + A^\lrs_1 + A^\lrs_2\\
			&\vsubset \\
			&\ \ \vdots\\
			&\vsubset\\
			&W^\lrs_\infty \ism \coprod_{n \in \NN} A^\lrs_n.
		\end{align*}
		However, we will not write the elements of $W^\lrs_\infty$ as finite sequences, but as tuples of the form $a = \Utuple{i}{a}{\alpha}$ which represents the finite sequence $([a]) \in \lrs_{\geq 0}^*$, inductively defined as
		$$([a]) = \Big(\aUt(a), \ \lUt(a), \quad |([a_0])|, \ \ldots, \ |([a_{\lUt(a)-1}])|, \quad [a_0], \ \alpha_0, \ \ldots, \ [a_{\lUt(a)-1}], \ \alpha_{\lUt(a)-1}\Big)$$
		(in particular, the empty tuple is represented by $(0, 0)$).\footnote{The definition of sets $A^\lrs_n$ can be slightly simplified if we do not go for a 1-1 correspondence between its elements and tuples of the form $\Utuple{i}{a}{\alpha}$, instead allowing that at any age we construct a tuple from elements of any lesser ages, not necessarily such that that their supremum is one less than the age. Thus for example the empty tuple $\eUt$ appears at every age. Explicitly, the alternative definition is
		\mlst{A^\lrs_n \dfeq}{\omega \in \lrs_{\geq 0}^*}
			{|\omega| \geq 2 \nc \omega_0 = n \nc \omega_1 \in \NN \nl
			|\omega| \geq 2 + \omega_1 \nc \all{k}{\NN_{< \omega_1}}{\omega_{2+k} \in \NN} \nc |\omega| = a(\omega_1) \nl
			\all{k}{\NN_{< \omega_1}}{\omega_{a(k)} \in \NN_{< n} \land (\omega_{a(k)}, \ldots, \omega_{a(k+1)-2}) \in A^\lrs_{\omega_{a(k)}}}\\
			&\text{where $a(k)$ is shorthand for $2 + \omega_1 + \sum_{i \in \NN_{< k}} (\omega_{2+i} + 1)$}}
		{.}
		This still works for the construction of the Urysohn space, as all relevant elements which represent the same tuple are at distance $0$.}
		
		As already announced, we define the map $d\colon W^\lrs_\infty \times W^\lrs_\infty \to \lrs_{\geq 0}$ inductively on $\aUt(a) + \aUt(b)$ by
		$$d(a, b) \dfeq \sup\Big(\st{|d(a_i, b) - \alpha_i|}{i \in \NN_{< \lUt(a)}} \cup \st{|d(a, b_j) - \beta_j|}{j \in \NN_{< \lUt(b)}}\Big)$$
		where $a = \Utuple{i}{a}{\alpha}, b = \Utuple{j}{b}{\beta} \in W^\lrs_\infty$. The supremum of an empty subset of $\lrs_{\geq 0}$ is understood to be the least element in $\lrs_{\geq 0}$, \ie $0$, so $d(\eUt, \eUt) = 0$.
		
		We claim that $d$ is a $\lrs$-protometric on $W^\lrs_\infty$ (and hence on $W^\lrs_t$ for all $t \in \opcN$). Clearly, $d$ indeed maps into $\lrs_{\geq 0}$ since $\lrs$ is a lattice ring streak and the absolute values, as well as their supremum, are nonnegative. Symmetry of $d$ is obvious. We prove the triangle inequality $d(a, b) + d(b, c) \geq d(a, c)$ for $a = \Utuple{i}{a}{\alpha}, b = \Utuple{j}{b}{\beta}, c = \Utuple{k}{c}{\gamma} \in W^\lrs_\infty$ inductively on $\aUt(a) + \aUt(b) + \aUt(c)$ by calculating the following for arbitrary $i \in \NN_{< \lUt(a)}$, $k \in \NN_{< \lUt(c)}$.
		$$d(a_i, c) - \alpha_i \leq d(a_i, b) - \alpha_i + d(b, c) \leq d(a, b) + d(b, c)$$
		$$\alpha_i - d(a_i, c) \leq \alpha_i - d(a_i, b) + d(b, c) \leq d(a, b) + d(b, c)$$
		$$d(a, c_k) - \gamma_k \leq d(a, b) + d(b, c_k) - \gamma_k \leq d(a, b) + d(b, c)$$
		$$\gamma_k - d(a, c_k) \leq d(a, b) + \gamma_k - d(b, c_k) \leq d(a, b) + d(b, c)$$
		However, $d$ is not a pseudometric, as evidenced for example by $a_{x, y} \dfeq (\eUt, x, \eUt, y)$ since
		$$d(a_{x, y}, a_{x, y}) = |x - y|.$$
		Intuitively, this is not surprising since $(\eUt, x, \eUt, y)$ should represent a point which is at distance both $x$ and $y$ from $\eUt$ which only makes sense when $x = y$. More generally, the distances $\alpha_i$s in $a = \Utuple{i}{a}{\alpha}$ should satisfy relevant triangle inequalities. We therefore define inductively on $n \in \NN$,
		\mlst{B^\lrs_n \dfeq}{a = \Utuple{i}{a}{\alpha} \in A^\lrs_n}
			{\big(\xall{i}{\NN_{< \lUt(a)}}{a_i \in B^\lrs_{\aUt(a_i)}}\big) \nl
			\all{i, j}{\NN_{< \lUt(a)}}{\alpha_i - \alpha_j \leq d(a_i, a_j) \leq \alpha_i + \alpha_j}}{;}
		in particular $B^\lrs_0 = A^\lrs_0 = \{\eUt\}$. Finally, for any $t \in \opcN$, define $V^\lrs_t \dfeq \bigcup_{n \in \NN_{< \succm(t)}} B^\lrs_n$. Similarly as before we have
		\begin{align*}
			W^\lrs_0 = \qquad &V^\lrs_0 \ism B^\lrs_0 \quad \ism \one\\
			&\vsubset\\
			W^\lrs_1 \supseteq \qquad &V^\lrs_1 \ism B^\lrs_0 + B^\lrs_1\\
			&\vsubset\\
			W^\lrs_2 \supseteq \qquad &V^\lrs_2 \ism B^\lrs_0 + B^\lrs_1 + B^\lrs_2\\
			&\vsubset \\
			&\ \ \vdots\\
			&\vsubset\\
			W^\lrs_\infty \supseteq \qquad &V^\lrs_\infty \ism \coprod_{n \in \NN} B^\lrs_n.
		\end{align*}
		We say that a tuple $a \in W^\lrs_\infty$ is \df{permissible} when $a \in V^\lrs_\infty$.
		
		\begin{theorem}\label{Theorem: distances_as_described_in_Urysohn_tuples}
      	For all $a = \Utuple{i}{a}{\alpha} \in V^\lrs_\infty$ we have $d(a, a) = 0$, or equivalently, $d(a, a_i) = \alpha_i$ for all $i \in \NN_{< \lUt(a)}$.
      \end{theorem}
      \begin{proof}
      	We prove this by double induction. We begin with the induction on $\aUt(a)$. Of course, in the base case $\aUt(a) = 0$, \ie $a = \eUt$, there is nothing to prove.
      	
      	For a general $a$ describe all predecessors of $a$ inductively as follows:
      	$$a_{j_0, j_1, \ldots, j_r} = (a_{j_0, j_1, \ldots, j_r, j_{r+1}}, \alpha_{j_0, j_1, \ldots, j_r, j_{r+1}})_{j_{r+1} \in \NN_{< \lUt(a_{j_0, j_1, \ldots, j_r})}}.$$
      	
      	The heart of the proof is in the following \textbf{Claim}:
      	\begin{itemize}
      		\item
      			{\it Let $l \in \NN$ and $j_i \in \NN_{< \lUt(a_{j_0, j_1, \ldots, j_{i-1}})}$ for all $i \in \NN_{\leq l}$. Assume that for all $j_{l+1} \in \NN_{< \lUt(a_{j_0, j_1, \ldots, j_l})}$ we have
      			$$d(a, a_{j_0, j_1, \ldots, j_{l+1}}) \leq \alpha_{j_0} +  \alpha_{j_0, j_1} + \alpha_{j_0, j_1, j_2} + \ldots + \alpha_{j_0, j_1, \ldots, j_l} + \alpha_{j_0, j_1, \ldots, j_{l+1}}.$$
      			Then $d(a, a_{j_0, j_1, \ldots, j_l}) \leq \alpha_{j_0} + \alpha_{j_0, j_1} + \alpha_{j_0, j_1, j_2} + \ldots + \alpha_{j_0, j_1, \ldots, j_l}$.}
	         	\bigskip
	         	\begin{proof}
	         		We have
	         		$$d(a, a_{j_0, j_1, \ldots, j_l}) = \sup\Big\{\st[1]{|d(a_i, a_{j_0, j_1, \ldots, j_l}) - \alpha_i|}{i \in \NN_{< \lUt(a)}} \cup$$
	         		$$\cup \st[1]{|d(a, a_{j_0, j_1, \ldots, j_{l+1}}) - \alpha_{j_0, j_1, \ldots, j_{l+1}}|}{j_{l+1} \in \NN_{< \lUt(a_{j_0, j_1, \ldots, j_l})}}\Big\}.$$
	         		For $i \in \NN_{< \lUt(a)}$ recall permissibility and the original induction hypothesis.
                  $$d(a_i, a_{j_0, j_1, \ldots, j_l}) - \alpha_i \leq$$
                  $$\leq d(a_i, a_{j_0}) - \alpha_i + d(a_{j_0}, a_{j_0, j_1}) + d(a_{j_0, j_1}, a_{j_0, j_1, j_2}) + \ldots + d(a_{j_0, j_1, \ldots, j_{l-1}}, a_{j_0, j_1, \ldots, j_l}) \leq$$
                  $$\leq \alpha_{j_0} + \alpha_{j_0, j_1} + \alpha_{j_0, j_1, j_2} + \ldots + \alpha_{j_0, j_1, \ldots, j_l}$$
                  
                  $$\alpha_i - d(a_i, a_{j_0, j_1, \ldots, j_l}) \leq$$
                  $$\leq \alpha_i - d(a_i, a_{j_0}) + d(a_{j_0}, a_{j_0, j_1}) + d(a_{j_0, j_1}, a_{j_0, j_1, j_2}) + \ldots + d(a_{j_0, j_1, \ldots, j_{l-1}}, a_{j_0, j_1, \ldots, j_l}) \leq$$
                  $$\leq \alpha_{j_0} + \alpha_{j_0, j_1} + \alpha_{j_0, j_1, j_2} + \ldots + \alpha_{j_0, j_1, \ldots, j_l}$$
                  
                  Now we deal with the second part. Take any $j_{l+1} \in \NN_{< \lUt(a_{j_0, j_1, \ldots, j_l})}$.
                  $$\alpha_{j_0, j_1, \ldots, j_{l+1}} - d(a, a_{j_0, j_1, \ldots, j_{l+1}}) \leq \alpha_{j_0, j_1, \ldots, j_{l+1}} - \big(d(a_{j_0}, a_{j_0, j_1, \ldots, j_{l+1}}) - \alpha_{j_0}\big) =$$
                  $$= d(a_{j_0, j_1, \ldots, j_l}, a_{j_0, j_1, \ldots, j_{l+1}}) - d(a_{j_0}, a_{j_0, j_1, \ldots, j_{l+1}}) + \alpha_{j_0} \leq d(a_{j_0}, a_{j_0, j_1, \ldots, j_l}) + \alpha_{j_0} \leq$$
                  $$\leq \alpha_{j_0} + d(a_{j_0}, a_{j_0, j_1}) + d(a_{j_0, j_1}, a_{j_0, j_1, j_2}) + \ldots + d(a_{j_0, j_1, \ldots, j_{l-1}}, a_{j_0, j_1, \ldots, j_l}) =$$
                  $$= \alpha_{j_0} + \alpha_{j_0, j_1} + \alpha_{j_0, j_1, j_2} + \ldots + \alpha_{j_0, j_1, \ldots, j_l}$$
                  The last inequality to prove,
                  $$d(a, a_{j_0, j_1, \ldots, j_{l+1}}) - \alpha_{j_0, j_1, \ldots, j_{l+1}} \leq \alpha_{j_0} + \alpha_{j_0, j_1} + \alpha_{j_0, j_1, j_2} + \ldots + \alpha_{j_0, j_1, \ldots, j_l},$$
                  holds by assumption.
	         	\end{proof}
      	\end{itemize}
         
         Notice that this {Claim} serves not only as the inductive step but also as the base of induction since when we reach the empty tuple (which is after $\aUt(a)$ steps), the condition is vacuous. In the end we obtain $d(a, a_{j_0}) \leq \alpha_{j_0}$.
         
			The reverse inequality is easier:
			$$d(a, a_{j_0}) \geq |d(a_{j_0}, a_{j_0}) - \alpha_{j_0}| = \alpha_{j_0}$$
			since $d(a_{j_0}, a_{j_0}) = 0$ by the induction hypothesis.
      \end{proof}
      
      This proves that $V^\lrs_\infty$ is a pseudometric space. It is not metric; consider for example the distance between any $a \in V^\lrs_\infty$ and $(a, 0)$. The Kolmogorov quotient of $V^\lrs_\infty$ --- which we denote by $\Ury^\lrs$ (and the induced metric on it by a slight abuse of notation again by $d$) --- has most properties of the Urysohn space, as we shall see.
      
      \begin{lemma}\label{Lemma: bijection_with_N_transfers_from_D_to_Urysohn}
      	Suppose $\lrs$ is countable, and has decidable equality. Then the sets $\lrs_{\geq 0}$, $\lrs_{\geq 0}^*$, $V^\lrs_\infty$, $\Ury^\lrs$ are in bijection with $\NN$.
      \end{lemma}
      \begin{proof}
      	\begin{itemize}
      		\item\proven{$\lrs_{\geq 0}$, $\lrs_{\geq 0}^*$, $V^\lrs_\infty$, $\Ury^\lrs$ have decidable equality}
      			Easy for $\lrs_{\geq 0}$, $\lrs_{\geq 0}^*$, $V^\lrs_\infty$. The set $\Ury^\lrs$ has decidable equality because it is a $\lrs$-metric space.
      		\item\proven{$\lrs_{\geq 0}$, $\lrs_{\geq 0}^*$, $V^\lrs_\infty$, $\Ury^\lrs$ are countable}
      			If $\lrs$ is countable, then so is $\lrs_{\geq 0}$ since it is its image by the absolute value. It is standard that then $\lrs_{\geq 0}^*$ is countable. By Lemma~\ref{Lemma: decidability_of_(in)equalities} if $=$ is a decidable relation on $\lrs$, then so is $\leq$, and therefore $V^\lrs_\infty$ is a decidable subset of $\lrs_{\geq 0}^*$, hence countable.
      		\item\proven{$\lrs_{\geq 0}$, $\lrs_{\geq 0}^*$, $V^\lrs_\infty$, $\Ury^\lrs$ are infinite}
      			The set $\lrs_{\geq 0}$ is infinite because of the inclusion $\NN \hookrightarrow \lrs_{\geq 0}$. Consequently, $V^\lrs_\infty$ is infinite because of the injective map $f\colon \lrs_{\geq 0} \to V^\lrs_\infty$, $f(x) \dfeq (\eUt, x)$. Since $V^\lrs_\infty \subseteq \lrs_{\geq 0}^*$, the latter is infinite as well. Finally, $\Ury^\lrs$ is infinite because $q \circ f$ is injective where $q$ is the Kolmogorov quotient map.
      	\end{itemize}
      	The result now follows from Lemma~\ref{Lemma: in_bijection_with_N}.
      \end{proof}
      
      \begin{proposition}\label{Proposition: similar_distances_and_points_yield_similar_extensions_in_Urysohn_space}
      	Let $\epsilon, \epsilon' \in \RR$, $n \in \NN$, and $a = \Utuple[n]{i}{a}{\alpha}, b = \Utuple[n]{i}{b}{\beta} \in V^\lrs_\infty$. If $d(a_i, b_i) \leq \epsilon$ and $|\alpha_i - \beta_i| \leq \epsilon'$ for all $i \in \NN_{< n}$, then $d(a, b) \leq \epsilon + \epsilon'$.
      \end{proposition}
      \begin{proof}
	      The following four calculations prove the statement.
	      $$d(a_i, b) - \alpha_i \leq d(a_i, b_i) + d(b_i, b) - \alpha_i = d(a_i, b_i) + \beta_i - \alpha_i \leq \epsilon + \epsilon'$$
	      $$\alpha_i - d(a_i, b) \leq \alpha_i - d(b_i, b) + d(a_i, b_i) = \alpha_i - \beta_i + d(a_i, b_i) \leq \epsilon' + \epsilon$$
	      $$d(a, b_i) - \beta_i \leq d(a, a_i) + d(a_i, b_i) - \beta_i = d(a_i, b_i) + \alpha_i - \beta_i \leq \epsilon + \epsilon'$$
	      $$\beta_i - d(a, b_i) \leq \beta_i - d(a, b_i) + d(a_i, b_i) = \beta_i - \alpha_i + d(a_i, b_i) \leq \epsilon' + \epsilon$$
      \end{proof}
      
      \begin{corollary}\label{Corollary: Urysohn_tuples_canonical_on_quotient}
      	Let $n \in \NN$, $\omega_0, \ldots, \omega_{n-1} \in \lrs_{\geq 0}$, and $a = \Utuple[n]{i}{a}{\omega}, b = \Utuple[n]{i}{b}{\omega} \in V^\lrs_\infty$ such that $d(a_i, b_i) = 0$ for all $i \in \NN_{< n}$. Then $d(a, b) = 0$.
      \end{corollary}
      \begin{proof}
      	Take $\epsilon = \epsilon' = 0$ in Proposition~\ref{Proposition: similar_distances_and_points_yield_similar_extensions_in_Urysohn_space}.
      \end{proof}
      
      \begin{lemma}\label{Lemma: extend_isometry_by_one_point_into_uncompleted_Urysohn}
      	Let
      	$$P^\lrs \dfeq \st{\Utuple[n]{i}{x}{\chi} \in (\Ury^\lrs \times \lrs_{\geq 0})^*}{\all{i, j}{\NN_{< n}}{\chi_i - \chi_j \leq d(x_i, x_j) \leq \chi_i + \chi_j}}.$$
      	There is a map $E^\lrs\colon P^\lrs \to \Ury^\lrs$ with the property $d(E(\Utuple[n]{i}{x}{\chi}), x_k) = \chi_k$ for all $k \in \NN_{< n}$.
      \end{lemma}
      \begin{proof}
      	The map
      	$$E^\lrs\Big(\big([a_i], \chi_i\big)_{i \in \NN_{< n}}\Big) \dfeq \big[\Utuple[n]{i}{a}{\chi}\big]$$
      	works by Theorem~\ref{Theorem: distances_as_described_in_Urysohn_tuples}. Note that it is well defined by Corollary~\ref{Corollary: Urysohn_tuples_canonical_on_quotient}.
      \end{proof}
      
      \begin{proposition}
      	Let
      	\begin{itemize}
      		\item
      			$\mtr{X} = (X, d_\mtr{X}, s)$ be a countable $\lrs$-metric space,
      		\item
      			$s\colon \NN \to \one + X$ the enumeration of elements of $X$,
      		\item
      			$F \subseteq X$ a finite subset with enumeration $F = \{y_0, \ldots, y_{k-1}\}$,
      		\item
      			$\mtr{F} = (F, d_\mtr{F})$ metric subspace of $\mtr{X}$, and
      		\item
      			$e\colon \mtr{F} \to \Ury^\lrs$ an isometry.
      	\end{itemize}
      	Then there exists a canonical choice of an isometry $f\colon \mtr{X} \to \Ury^\lrs$ such that $\rstr{f}_F = e$.
      \end{proposition}
      \begin{proof}
      	Lemma~\ref{Lemma: extend_isometry_by_one_point_into_uncompleted_Urysohn} allows us to extend a finite isometry by one point. The idea is to use it inductively, first for $F$, then adding more and more terms of the sequence $s$. Explicitly, if $s_n \in X$, define $f(s_n)$ inductively on $n \in \NN$ as
      	$$f(s_n) \dfeq E^\lrs\Big(\big(e(y_i), d_\mtr{X}(s_n, y_i)\big)_{i \in \NN_{< k}} \cnct \big(f(s_j), d_\mtr{X}(s_n, s_j)\big)_{j \in \NN_{< n} \cap s^{-1}(X)}\Big).$$
      	The map $f$ is well-defined --- if $s_n$ equals some $x \in F \cup \big(s(\NN_{< n}) \cap X\big)$, then $d\big(f(s_n), f(x)\big) = 0$, so $f(s_n) = f(x)$ because $\Ury^\lrs$ is a metric space. For the same reason $f$ is an extension of $e$.
      \end{proof}
      
      If $\Ury^\lrs$ was complete, we could extend the isometry from a countable to a metrically separable space. 
      We (preliminarily) define $\Ury$, equipped with the metric $d_\Ury$, to be the completion of $\Ury^\RR$, or equivalently, of $V^\RR_\infty$.
      
      Clearly if $\lrs'$, $\lrs''$ are two lattice ring streaks, and $\lrs' \subseteq \lrs''$, then $V^{\lrs'}_\infty \subseteq V^{\lrs''}_\infty$ and $\Ury^{\lrs'} \subseteq \Ury^{\lrs''}$. In particular $V^\lrs_\infty \subseteq V^\RR_\infty$, $\Ury^\lrs \subseteq \Ury^\RR$.
      
      \begin{lemma}\label{Lemma: approximation_of_Urysohn_tuples}
      	Suppose $\lrs$ is interpolating, and let
      	\begin{itemize}
      		\item
      			$X \subseteq \Ury$ such that $\xall{x}{X}\xall{r}{\RR_{> 0}}\xsome{a}{V^\lrs_\infty}{d_\Ury(x, [a]) \leq r}$,
      		\item
      			$n \in \NN$,
      		\item
      			$x_0, \ldots, x_{n-1} \in X$ and $\omega_0, \ldots, \omega_{n-1} \in \RR_{\geq 0}$ such that $\omega_i - \omega_j \leq d(x_i, x_j) \leq \omega_i + \omega_j$ for all $i, j \in \NN_{< n}$,
      		\item
      			$\epsilon \in \RR_{> 0}$.
      	\end{itemize}
      	Then there exists $\Utuple[n]{i}{a}{\alpha} \in V^\lrs_\infty$ such that $d_\Ury(x_i, [a_i]) \leq \epsilon$ and $|\omega_i - \alpha_i| \leq \epsilon$ for all $i \in \NN_{< n}$.
      \end{lemma}
      \begin{proof}
			If $n = 0$, then $\eUt$ works, and we are done. In the remainder of the proof assume $n \geq 1$.
			
			Let $\lambda \dfeq \frac{\epsilon}{4 n} > 0$. For each of finitely many $i \in \NN_{< n}$ choose $\alpha_i \in (\omega_i + \intoo{3 \lambda}{4 \lambda}) \cap \lrs$ and $a'_i \in V^\lrs_\infty$ such that $d_\Ury(x_i, [a'_i]) \leq \lambda$. Let
      	$$\delta_{i,j} \dfeq \begin{cases} 1, & \text{if } i = j,\\ 0, & \text{if } i \neq j, \end{cases}$$
      	denote the \df{Kronecker delta}, and $d_{i,j} \dfeq d(a'_i, a'_j) + 3 \lambda (1-\delta_{i,j})$. Define $a_0, \ldots, a_{n-1} \in V^\lrs_\infty$ inductively by
      	$$a_k \dfeq \big(a_i, d_{k,i}\big)_{i \in \NN_{< k}} \cnct \big(a'_k, \sup\st{|d_{k,j} - d(a'_k, a_j)|}{j \in \NN_{< k}}\big);$$
      	in particular $a_0 = (a'_0, 0)$. The calculations (by induction on $k > i, j$) below confirm these are indeed permissible tuples:
      	$$d_{k,i} - d_{k,j} = d(a'_k, a'_i) - d(a'_k, a'_j) \leq d(a'_i, a'_j) \leq d_{i,j} = d(a_i, a_j),$$
      	
      	$$d_{k,i} + d_{k,j} = d(a'_k, a'_i) + d(a'_k, a'_j) + 6 \lambda (1-\delta_{i,j}) \geq$$
      	$$\geq d(a'_i, a'_j) + 6 \lambda (1-\delta_{i,j}) = d_{i,j} + 3 \lambda (1-\delta_{i,j}) \geq d_{i,j} = d(a_i, a_j),$$
      	
      	$$d_{k,i} - \sup\st{|d_{k,j} - d(a'_k, a_j)|}{j \in \NN_{< k}} \leq d_{k,i} - (d_{k,i} - d(a'_k, a_i)) = d(a'_k, a_i),$$
      	
      	$$(d_{k,j} - d(a'_k, a_j)) - d_{k,i} = d(a'_k, a'_j) - d(a'_k, a'_i) - d(a'_k, a_j) \leq$$
      	$$\leq d(a'_i, a'_j) - d(a'_k, a_j) \leq d_{i,j} - d(a'_k, a_j) = d(a_i, a_j) - d(a'_k, a_j) \leq d(a'_k, a_i),$$
      	
      	$$(d(a'_k, a_j) - d_{k,j}) - d_{k,i} = d(a'_k, a_j) - d(a'_k, a'_j) - d(a'_k, a'_i) - 6 \lambda (1-\delta_{i,j}) \leq$$
      	$$\leq d(a'_k, a_j) - d(a'_i, a'_j) - 6 \lambda (1-\delta_{i,j}) = d(a'_k, a_j) - d_{i,j} - 3 \lambda (1-\delta_{i,j}) \leq$$
      	$$\leq d(a'_k, a_j) - d_{i,j} = d(a'_k, a_j) - d(a_i, a_j) \leq d(a'_k, a_i),$$
      	
      	$$\sup\st{|d_{k,j} - d(a'_k, a_j)|}{j \in \NN_{< k}} + d_{k,i} \geq (d(a'_k, a_i) - d_{k,i}) + d_{k,i} = d(a'_k, a_i).$$
      	\begin{itemize}
      		\item\proven{$\Utuple[n]{i}{a}{\alpha} \in V^\lrs_\infty$}
      			Take any $i, j \in \NN_{< n}$. The condition $\alpha_i - \alpha_j \leq d(a_i, a_j) \leq \alpha_i + \alpha_j$ clearly holds for $i = j$, so without loss of generality assume $i < j$.
      			
      			$$\alpha_i - \alpha_j \leq (\omega_i + 4 \lambda) - (\omega_j + 3 \lambda) = \omega_i - \omega_j + \lambda \leq d_\Ury(x_i, x_j) + \lambda \leq$$
      			$$\leq d_\Ury(x_i, [a'_i]) + d(a'_i, a'_j) + d_\Ury(x_j, [a'_j]) + \lambda \leq d(a'_i, a'_j) + 3 \lambda = d_{i,j} = d(a_i, a_j)$$
      			
      			$$\alpha_i + \alpha_j \geq (\omega_i + 3 \lambda) + (\omega_j + 3 \lambda) = \omega_i + \omega_j + 6 \lambda \geq d_\Ury(x_i, x_j) + 6 \lambda \geq$$
      			$$\geq d(a'_i, a'_j) - d_\Ury(x_i, [a'_i]) - d_\Ury(x_j, [a'_j]) + 6 \lambda \geq d(a'_i, a'_j) + 4 \lambda \geq d_{i,j} = d(a_i, a_j)$$
      		
      		\item\proven{$\xall{i}{\NN_{< n}}{d_\Ury(x_i, [a_i]) \leq \epsilon}$}
      			We claim that $d(a_k, a'_k) \leq 3 k \lambda$ for all $k \in \NN_{< n}$. This clearly holds for $k = 0$. By induction, for $k \geq 1$,
		      	$$d(a_k, a'_k) = \sup\st{|d_{k,j} - d(a'_k, a_j)|}{j \in \NN_{< k}} =$$
		      	$$= \sup\st{|d(a'_k, a'_j) - 3 \lambda - d(a'_k, a_j)|}{j \in \NN_{< k}} \leq$$
		      	$$\leq \sup\st{|d(a'_k, a'_j) - d(a'_k, a_j)| + 3 \lambda}{j \in \NN_{< k}} \leq$$
		      	$$\leq \sup\st{d(a_j, a'_j)}{j \in \NN_{< k}} + 3 \lambda \leq 3 (k-1) \lambda + 3 \lambda = 3 k \lambda.$$
		      	Therefore
		      	$$d_\Ury(x_i, [a_i]) \leq d_\Ury(x_i, [a'_i]) + d_\Ury([a'_i], [a_i]) = d_\Ury(x_i, [a'_i]) + d(a'_i, a_i) \leq \lambda + 3 i \lambda \leq 3 n \lambda \leq \epsilon.$$
      		
      		\item\proven{$\xall{i}{\NN_{< n}}{|\omega_i - \alpha_i| \leq \epsilon}$}
      			$$|\omega_i - \alpha_i| < 4 \lambda \leq 4 n \lambda \leq \epsilon$$
      	\end{itemize}
      \end{proof}
      
      \begin{proposition}
      	If $\lrs$ is interpolating, then $\Ury^\lrs$ is metrically dense in $\Ury^\RR$ (equivalently, $V^\lrs_\infty$ is metrically dense in $V^\RR_\infty$).
      \end{proposition}
      \begin{proof}
      	By induction on age $k \in \NN$. The proposition clearly holds for $k = 0$. Assume $k \geq 1$, and fix $r \in \RR_{> 0}$. Take $b = \Utuple[n]{i}{b}{\beta} \in V^\RR_k$, and suppose the proposition holds for ages less than $k$, in particular for predecessors of $b$. This means that we can use Lemma~\ref{Lemma: approximation_of_Urysohn_tuples} for $X = V^\RR_{k-1}$, $x_i = [b_i]$, $\omega_i = \beta_i$ and $\epsilon = \frac{r}{2}$ to obtain $a = \Utuple[n]{i}{a}{\alpha} \in V^\lrs_\infty$ so that $d(a_i, b_i) \leq \frac{r}{2}$ and $|\alpha_i - \beta_i| \leq \frac{r}{2}$ for all $i \in \NN_{< n}$. By Proposition~\ref{Proposition: similar_distances_and_points_yield_similar_extensions_in_Urysohn_space} $d(a, b) \leq r$.
      \end{proof}
      
      The corollary is that we can define $\Ury$ to be the completion of $V^\lrs_\infty$ and $\Ury^\lrs$ for \emph{any} interpolating lattice ring streak $\lrs$, and obtain the same result.
      
      \begin{theorem}
      	The metric space $\Ury$ satisfies the properties of the Urysohn space. Explicitly, the following holds.
      	\begin{enumerate}
	      	\item
	      		$\Ury$ is (an inhabited) \cms[].
	      	\item
		      	Let
		      	$$P \dfeq \st{\Utuple[n]{i}{x}{\omega} \in (\Ury \times \RR_{\geq 0})^*}{\all{i, j}{\NN_{< n}}{\omega_i - \omega_j \leq d_\mtr{U}(x_i, x_j) \leq \omega_i + \omega_j}}.$$
		      	There is a map $E\colon P \to \Ury$ with the property $d(E(\Utuple[n]{i}{x}{\omega}), x_k) = \omega_k$ for all $k \in \NN_{< n}$.
	      	\item
		      	Let
		      	\begin{itemize}
		      		\item
		      			$\mtr{X} = (X, d_\mtr{X}, s)$ be a metrically separable metric space,
		      		\item
		      			$F \subseteq X$ a finite subset with enumeration $F = \{y_0, \ldots, y_{k-1}\}$,
		      		\item
		      			$\mtr{F} = (F, d_\mtr{F})$ metric subspace of $\mtr{X}$, and
		      		\item
		      			$e\colon \mtr{F} \to \Ury$ an isometry.
		      	\end{itemize}
		      	Then there exists a canonical choice of an isometry $f\colon \mtr{X} \to \Ury$ such that $\rstr{f}_F = e$.
      	\end{enumerate}
      \end{theorem}
      \begin{proof}
      	For the purposes of this proof we take for $\lrs$ a countable interpolating lattice ring streak with decidable equality, \eg $\lrs = \QQ$. By Lemma~\ref{Lemma: bijection_with_N_transfers_from_D_to_Urysohn} $V^\lrs_\infty$ and $\Ury^\lrs$ are in bijection with $\NN$.
      	\begin{enumerate}
	      	\item
	      		$\Ury$ is the completion of an inhabited countable metric space $\Ury^\lrs$.
	      	\item
		      	Take any $\Utuple[n]{i}{x}{\omega} \in P$. Let $\lvl$ be any suitable set to index Cauchy families, and define for $p \in \lvl$
		      	$$S_p \dfeq \st{\Utuple[n]{i}{a}{\alpha} \in V^\lrs_\infty}{\all{i}{\NN_{< n}}{d_\Ury(x_i, [a_i]) < \tfrac{p}{2} \ \land \ |\alpha_i - \omega_i| < \tfrac{p}{2}}}$$
		      	The set $S_p$ is open in countable (hence overt) $V^\lrs_\infty$, so it is subovert. Use Lemma~\ref{Lemma: approximation_of_Urysohn_tuples} (for $X = \Ury$) to see that it is inhabited. Take any $p, q \in \lrs_{> 0}$ and any $a = \Utuple[n]{i}{a}{\alpha} \in S_p$, $b = \Utuple[n]{j}{b}{\beta} \in S_q$. Then for any $i \in \NN_{< n}$
      			$$d(a_i, b_i) \leq d_\Ury([a_i], x_i) + d_\Ury(x_i, [b_i]) \leq \tfrac{p+q}{2},$$
      			$$|\alpha_i - \beta_i| \leq |\alpha_i - \omega_i| + |\omega_i - \beta_i| \leq \tfrac{p+q}{2}$$
      			whence $d(a, b) \leq p+q$ by Proposition~\ref{Proposition: similar_distances_and_points_yield_similar_extensions_in_Urysohn_space}. We conclude that $S = \st{S_p}{p \in \lvl}$ is a Cauchy family. Since $\Ury$ is complete, we can consider $[S]$ to be in $\Ury$.
		      	
		      	We claim $d_\Ury([S], x_k) = \omega_k$ for all $k \in \NN_{< n}$. Denote explicitly the dense isometry $e\colon V^\lrs_\infty \to \Ury$. Recall from Proposition~\ref{Proposition: completion_by_Cauchy_families} that equivalence classes of Cauchy families have canonical representatives; in this case, $x_k = [T]$ where $T = \st{T_q}{q \in \lrs_{> 0}}$, $T_q \dfeq e^{-1}(\ball[\Ury]{x_k}{q})$ (which is an open subset in overt $V^\lrs_\infty$, so subovert). Then $d_\Ury([S], x_k) = (L, U)$ where
					$$L \dfeq \st{q \in \oirs}{\xsome{p, s}{\lrs_{> 0}}\xsome{a}{S_p}\some{b}{T_s}{q + p + s < d(a, b)}},$$
					$$U \dfeq \st{r \in \oirs}{\xsome{p, s}{\lrs_{> 0}}\xsome{a}{S_p}\some{b}{T_s}{r > p + s + d(a, b)}}.$$
					Take any $q \in \oirs_{< \omega_k}$, let $\epsilon \dfeq \frac{\omega_i - q}{4}$, and $p, s \in \intoo[\lrs]{0}{\epsilon}$. Choose any $a = \Utuple[n]{i}{a}{\alpha} \in S_p$, $b \in T_s$. Then
					$$d(a, b) \geq d(a, a_i) - d_\Ury([a_i], x_i) - d_\Ury(x_i, [b]) \geq$$
					$$\geq \alpha_i - \tfrac{p}{2} - s \geq \omega_i - p - s > \omega_i - 2 \epsilon = q + 2 \epsilon > q + p + s,$$
					so $q \in L$. In the same way we prove $\oirs_{> \omega_k} \subseteq U$, and therefore conclude $(L, U) = \omega_k$. We see that the definition $E\big(\Utuple[n]{i}{x}{\omega}\big) \dfeq [S]$ works.
				\item
					Define $f'\colon F \cup \big(s(\NN) \cap X\big) \to \Ury$ on $F$ by $f'(y_i) \dfeq e(y_i)$, and on $s(\NN) \cap X$ inductively on $n \in \NN$ as follows: if $s_n \in X$, then
					$$f'(s_n) \dfeq E\Big(\big(e(y_i), d_\mtr{X}(s_n, y_i)\big)_{i \in \NN_{< k}} \cnct \big(f(s_j), d_\mtr{X}(s_n, s_j)\big)_{j \in \NN_{< n} \cap s^{-1}(X)}\Big).$$
					The map $f'$ is well-defined --- if $s_n$ equals some $x \in F \cup \big(s(\NN_{< n}) \cap X\big)$, then $d\big(f(s_n), f(x)\big) = 0$, so $f(s_n) = f(x)$ because $\Ury$ is a metric space. For the same reason $f'$ is an extension of $e$. It follows from the definition of $E$ that $f'$ is an isometry, and since $F \cup \big(s(\NN) \cap X\big)$ is metrically dense in $\mtr{X}$, it extends to the isometry $f\colon X \to \Ury$ by Corollary~\ref{Corollary: extend_isometries_from_a_dense_subspace_into_a_complete_space}.
			\end{enumerate}
      \end{proof}
      
      The corollary is that any metrically separable metric space isometrically embeds into $\Ury$ --- just take $F = \emptyset$ in the preceding theorem. However, this is not enough; recall that we require the image of the embedding to be located. Classically, we would not require something as complicated as the Urysohn space. There are many other (metrically separable) metric spaces into which every metrically separable space can be isometrically embedded, such as $\ell^\infty$ and $\C(\II, \RR)$. Furthermore, if $\mtr{X}$ is inhabited, then its image is located, and if $\mtr{X}$ is also complete, it is strongly located. The completeness argument constructively also works, but of course we do not get locatedness for free. However, the Urysohn space allows us to find an isometric embedding with a located image constructively.\footnote{I found the statement of the following proposition in a circulated notes by Dag Normann, of which the majority was later published, but this part was cut, and as far as I know, Normann did not publish it anywhere else. He calls it the ``Swiss Cheese Principle'' (because it essentially says that if you ``poke holes in the Urysohn space'', \ie remove finitely many balls from it, you still obtain the Urysohn space). Normann did not provide a proof, but attributed it to Urysohn~\cite{Urysohn_PS_1927:_sur_un_espace_metrique_universel}.}
      
      \begin{proposition}\label{Proposition: isometric_embedding_into_Urysohn_space_with_located_image}
      	Let $\mtr{X} = (X, d_\mtr{X}, s\colon \NN \to X)$ be an inhabited metrically separable metric space, and $\mtr{U} = (U, d_\mtr{U}, u, E)$ (any) Urysohn space. There exists an isometric embedding $f\colon \mtr{X} \to \mtr{U}$ such that its image $f(X)$ is located in $\mtr{U}$.
      \end{proposition}
      \begin{proof}
      	Define the map $f'\colon \NN \to U$ inductively by
      	$$f'(n) \dfeq E\bigg(\Big(f'(i), d_\mtr{X}(s_n, s_i)\Big)_{i \in \NN_{< n}} \cnct \Big(u_i, \sup\big(\st{|d_\mtr{X}(s_n, s_k) - d_\mtr{U}(f'(k), u_i)|}{k \in \NN_{< n}} \cup \{b_i\}\big)\Big)_{i \in \NN_{< n}}\bigg)$$
      	where $b_i \dfeq \inf\st{d_\mtr{X}(s_n, s_k) + d_\mtr{U}(f'(k), u_i)}{k \in \NN_{\leq i}}$.
      	
      	We verify that $f'(n)$ is well defined, \ie that we used $E$ on an element of
      	$$P = \st{\Utuple[n]{i}{x}{\omega} \in (U \times \RR_{\geq 0})^*}{\all{i, j}{\NN_{< n}}{\omega_i - \omega_j \leq d_\mtr{U}(x_i, x_j) \leq \omega_i + \omega_j}}.$$
      	The induction hypothesis is that $f'$ is well defined on all natural numbers less that $n$.
      	\begin{itemize}
      		
      		\item\proven{$\all{i,j}{\NN_{< n}}{d_\mtr{X}(s_n, s_i) - d_\mtr{X}(s_n, s_j) \leq d_\mtr{U}(f'(i), f'(j)) \leq d_\mtr{X}(s_n, s_i) + d_\mtr{X}(s_n, s_j)}$}
      			Follows from triangle inequality once we notice that (by the induction hypothesis and properties of $E$) $d_\mtr{U}(f'(i), f'(j)) = d_\mtr{X}(s_i, s_j)$.
      		
      		\item\proven{$\all{i,j}{\NN_{< n}}{d_\mtr{X}(s_n, s_i) - \sup\big(\st{|d_\mtr{X}(s_n, s_k) - d_\mtr{U}(f'(k), u_j)|}{k \in \NN_{< n}} \cup \{b_j\}\big) \leq d_\mtr{U}(f'(i), u_j)}$}
      			$$d_\mtr{X}(s_n, s_i) - \sup\big(\st{|d(s_n, s_k) - d(f'(k), u_j)|}{k \in \NN_{< n}} \cup \{b_j\}\big) \leq$$
      			$$\leq d_\mtr{X}(s_n, s_i) - (d_\mtr{X}(s_n, s_i) - d_\mtr{U}(f'(i), u_j)) = d_\mtr{U}(f'(i), u_j)$$
      		
      		\item\proven{$\all{i,j}{\NN_{< n}}{\sup\st{|d_\mtr{X}(s_n, s_k) - d_\mtr{U}(f'(k), u_j)|}{k \in \NN_{< n}} - d_\mtr{X}(s_n, s_i) \leq d_\mtr{U}(f'(i), u_j)}$}
      			$$d_\mtr{X}(s_n, s_k) - d_\mtr{U}(f'(k), u_j) - d_\mtr{X}(s_n, s_i) \leq d_\mtr{X}(s_k, s_i) - d_\mtr{U}(f'(k), u_j) =$$
      			$$= d_\mtr{U}(f'(k), f'(i)) - d_\mtr{U}(f'(k), u_j) \leq d_\mtr{U}(f'(i), u_j)$$
      		
      		\item\proven{$\all{i,j}{\NN_{< n}}{b_j - d_\mtr{X}(s_n, s_i) \leq d_\mtr{U}(f'(i), u_j)}$}
      			If $i \leq j$, then
      			$$\inf\st{d_\mtr{U}(f'(k), u_j) + d_\mtr{X}(s_n, s_k)}{k \in \NN_{\leq j}} - d_\mtr{X}(s_n, s_i) \leq$$
      			$$\leq d_\mtr{U}(f'(i), u_j) + d_\mtr{X}(s_n, s_i) - d_\mtr{X}(s_n, s_i) = d_\mtr{U}(f'(i), u_j).$$
      			If $i > j$, then
      			$$d_\mtr{U}(f'(i), u_j) = \sup\big(\st{|d_\mtr{X}(s_i, s_k) - d_\mtr{U}(f'(k), u_j)|}{k \in \NN_{< i}} \cup \{b_j\}\big) \geq b_j \geq b_j - d_\mtr{X}(s_n, s_i).$$
      		
      		\item\proven{$\all{i,j}{\NN_{< n}}{d_\mtr{U}(f'(i), u_j) \leq d_\mtr{X}(s_n, s_i) + \sup\big(\st{|d_\mtr{X}(s_n, s_k) - d_\mtr{U}(f'(k), u_j)|}{k \in \NN_{< n}} \cup \{b_j\}\big)}$}
      			$$d_\mtr{X}(s_n, s_i) + \sup\big(\st{|d_\mtr{X}(s_n, s_k) - d_\mtr{U}(f'(k), u_j)|}{k \in \NN_{< n}} \cup \{b_j\}\big) \geq$$
      			$$\geq d_\mtr{X}(s_n, s_i) + d_\mtr{U}(f'(i), u_j) - d_\mtr{X}(s_n, s_i) = d_\mtr{U}(f'(i), u_j)$$
      		
      		\item\proven{$\xall{i,j}{\NN_{< n}}{}\big(\sup\st{|d_\mtr{X}(s_n, s_k) - d_\mtr{U}(f'(k), u_i)|}{k \in \NN_{< n}} -$}\proven{$- \sup\big(\st{|d_\mtr{X}(s_n, s_k) - d_\mtr{U}(f'(k), u_j)|}{k \in \NN_{< n}} \cup \{b_j\}\big) \leq d_\mtr{U}(u_i, u_j)\big)$}
      			$$|d_\mtr{X}(s_n, s_l) - d_\mtr{U}(f'(l), u_i)| - \sup\big(\st{|d_\mtr{X}(s_n, s_k) - d_\mtr{U}(f'(k), u_j)|}{k \in \NN_{< n}} \cup \{b_j\}\big) \leq$$
      			$$\leq |d_\mtr{X}(s_n, s_l) - d_\mtr{U}(f'(l), u_i)| - |d_\mtr{X}(s_n, s_l) - d_\mtr{U}(f'(l), u_j)| \leq$$
      			$$\leq |(d_\mtr{X}(s_n, s_l) - d_\mtr{U}(f'(l), u_i)) - (d_\mtr{X}(s_n, s_l) - d_\mtr{U}(f'(l), u_j))| =$$
      			$$= |d_\mtr{U}(f'(l), u_j) - d_\mtr{U}(f'(l), u_i)| \leq d_\mtr{U}(u_i, u_j)$$
      		
      		\item\proven{$\all{i,j}{\NN_{< n}}{b_i - \sup\big(\st{|d_\mtr{X}(s_n, s_k) - d_\mtr{U}(f'(k), u_j)|}{k \in \NN_{< n}} \cup \{b_j\}\big) \leq d_\mtr{U}(u_i, u_j)}$}
      			$$b_i - \sup\big(\st{|d_\mtr{X}(s_n, s_k) - d_\mtr{U}(f'(k), u_j)|}{k \in \NN_{< n}} \cup \{b_j\}\big) \leq$$
      			$$\leq d_\mtr{U}(f'(i), u_i) + d_\mtr{X}(s_n, s_i) - (d_\mtr{X}(s_n, s_i) - d_\mtr{U}(f'(i), u_j)) =$$
      			$$= d_\mtr{U}(f'(i), u_i) + d_\mtr{U}(f'(i), u_j) \leq d_\mtr{U}(u_i, u_j)$$
      		
      		\item\proven{$\xall{i,j}{\NN_{< n}}{}\big(d_\mtr{U}(u_i, u_j) \leq \sup\big(\st{|d_\mtr{X}(s_n, s_k) - d_\mtr{U}(f'(k), u_i)|}{k \in \NN_{< n}} \cup \{b_i\}\big) +$}\proven{$+ \sup\big(\st{|d_\mtr{X}(s_n, s_k) - d_\mtr{U}(f'(k), u_j)|}{k \in \NN_{< n}} \cup \{b_j\}\big)\big)$}
      			$$\sup\big(\st{|d_\mtr{X}(s_n, s_k) - d_\mtr{U}(f'(k), u_i)|}{k \in \NN_{< n}} \cup \{b_i\}\big) +$$
      			$$+ \sup\big(\st{|d_\mtr{X}(s_n, s_k) - d_\mtr{U}(f'(k), u_j)|}{k \in \NN_{< n}} \cup \{b_j\}\big) \geq$$
      			$$\geq \sup\st{|d_\mtr{X}(s_n, s_k) - d_\mtr{U}(f'(k), u_i)|}{k \in \NN_{< n}} + \inf\st{d_\mtr{U}(f'(l), u_j) + d_\mtr{X}(s_n, s_l)}{l \in \NN_{\leq j}} =$$
      			$$= \inf\st{\sup\st{|d_\mtr{X}(s_n, s_k) - d_\mtr{U}(f'(k), u_i)|}{k \in \NN_{< n}} + d_\mtr{U}(f'(l), u_j) + d_\mtr{X}(s_n, s_l)}{l \in \NN_{\leq j}} \geq$$
      			$$\geq \inf\st{d_\mtr{U}(f'(l), u_i) - d_\mtr{X}(s_n, s_k) + d_\mtr{U}(f'(l), u_j) + d_\mtr{X}(s_n, s_l)}{l \in \NN_{\leq j}} =$$
      			$$= \inf\st{d_\mtr{U}(f'(l), u_i) + d_\mtr{U}(f'(l), u_j)}{l \in \NN_{\leq j}} \geq d_\mtr{U}(u_i, u_j)$$
      	
      	\end{itemize}
      	
      	Observe that $d_\mtr{X}(s_i, s_j) = d_\mtr{U}(f'(i), f'(j))$ for all $i, j \in \NN$, so $f'$ induces the map $f''\colon s(\NN) \to U$ by $f' = f'' \circ s$. Moreover, $f''$ is an isometry, so it extends to an isometry $f\colon \mtr{X} \to \mtr{U}$ by Corollary~\ref{Corollary: extend_isometries_from_a_dense_subspace_into_a_complete_space}.
      	
      	For any $i \in \NN$ and $n \in \NN_{> i}$ we have
      	$$d_\mtr{U}(f(s_n), u_i) \geq \inf\st{d_\mtr{X}(s_n, s_k) + d_\mtr{U}(f(s_k), u_i)}{k \in \NN_{\leq i}} \geq \inf\st{d_\mtr{U}(f(s_k), u_i)}{k \in \NN_{\leq i}},$$
      	so for $l'\colon \NN \to \RR$, $l'(i) \dfeq \inf\st{d_\mtr{U}(f(s_k), u_i)}{k \in \NN_{\leq i}}$, we have
      	$$l'(i) = \inf\st{d_\mtr{U}(f(s_k), u_i)}{k \in \NN} = d_\mtr{U}(f(s(\NN)), u_i) = d_\mtr{U}(f(X), u_i),$$
      	the last equality holding because $s(\NN)$ is metrically dense in $\mtr{X}$ (and so $f(s(\NN))$ is metrically dense in $f(X)$ since $f$ is an isometry). Also, for $i, j \in \NN$,
      	$$l'(i) - l'(j) = l'(i) - \inf\st{d_\mtr{U}(f(s_k), u_j)}{k \in \NN} = \sup\st{l'(i) - d_\mtr{U}(f(s_k), u_j)}{k \in \NN} \leq$$
      	$$\leq \sup\st{d_\mtr{U}(f(s_k), u_i) - d_\mtr{U}(f(s_k), u_j)}{l \in \NN} \leq d_\mtr{U}(u_i, u_j),$$
      	so the map $l''\colon u(\NN) \to \RR$, $l' = l'' \circ u$, is well defined, and nonexpansive. It therefore extends to a nonexpansive map $l\colon U \to \RR$ by Theorem~\ref{Theorem: extend_Lipschitz_maps_from_a_dense_subspace_into_a_completion}, and we have $l = d_\mtr{U}(f(X), \insarg)$.
      \end{proof}
      
      \begin{theorem}\label{Theorem: metrization_of_Urysohn_space_implies_metrization_of_cmss}
      	Assume strongly located subsets are subspaces. If $\mtr{U}$ is (any) Urysohn space and it is metrized, then all \cms[s] are metrized.
      \end{theorem}
      \begin{proof}
      	By Lemma~\ref{Lemma: metrization_of_inhabited_spaces_sufficient} it is sufficient to prove metrization just for inhabited \cms[s]. For an inhabited \cms $\mtr{X} = (X, d_\mtr{X}, s)$ Proposition~\ref{Proposition: isometric_embedding_into_Urysohn_space_with_located_image} gives us an isometric embedding $f\colon \mtr{X} \to \mtr{U}$ such that $f(X)$ is located in $\mtr{U}$. Since $\mtr{X}$, and therefore also its isometric image, is complete, $f(X)$ is strongly located by Lemma~\ref{Lemma: complete_located_is_strongly_located}, thus metrized by Corollary~\ref{Corollary: metrization_of_subspaces}. Then $\mtr{X}$, being isometrically isomorphic to $f(\mtr{X})$, is metrized as well.
		\end{proof}

	\section{Consequences of Metrization}
		
		We conclude the chapter by summarizing the results of metrization, specifically of Theorems~\ref{Theorem: metrization_of_Baire/Cantor_implies_metrization_of_cmss_ctbs}, \ref{Theorem: metrization_of_Hilbert_cube_implies_metrization_of_ctbs} and~\ref{Theorem: metrization_of_Urysohn_space_implies_metrization_of_cmss} which imply metrization of \ctb[s]/\cms[s].
		
		\begin{lemma}\label{Lemma: smaller_basis_in_metrized_spaces}
			Let the metric space $\mtr{X} = (X, d)$ be metrized, $f\colon A \to X$ an overt map with a metrically dense image in $\mtr{X}$, and $\lvl \subseteq \RR_{> 0}$ a subset, subovert in $\RR$, with the strict infimum $0$. Then the map $A \times \lvl \to \tp(X)$, given by $(a, r) \mapsto \ball{f(a)}{r}$, is a basis for $X$.
		\end{lemma}
		\begin{proof}
			Let $U \subseteq X$ be intrinsically, and therefore metrically open, \ie an overtly indexed union $U = \bigcup_{i \in I} \ball{x_i}{r_i}$. Define $J \dfeq \st{(i, a, r) \in I \times A \times \lvl}{d(x_i, f(a)) + r < r_i}$. Then $U$ is an overtly indexed union
			$$U = \bigcup_{(i, a, r) \in J} \ball{f(a)}{r}$$
			by Corollary~\ref{Corollary: overt_combination} since the indexing $I \to X \times \RR$, the map $f\colon A \to X$, and the inclusion $\lvl \hookrightarrow \RR$ are overt, and $J$ is the preimage of the set $\st{(x, r, y, s) \in X \times \RR \times X \times \RR}{d(x, y) + s < r}$, open in $X \times \RR \times X \times \RR$.
		\end{proof}
		
		\begin{theorem}\label{Theorem: consequences_of_metrization}
			Suppose \ctb[s]/\cms[s] are metrized. Then the following holds.
			\begin{enumerate}
				\item
					The intrinsic topology of \ctb[s]/\cms[s] can be expressed in terms of their balls; specifically, a subset is open if and only if it is an overtly indexed union of balls.
				\item
					Up to topological equivalence, a set can be equipped with at most one metric which makes it a \ctb[]/\cms[].
				\item
					\ctb[s]/\cms[s] are countably based.
				\item
					\ctb[s]/\cms[s] are overt.
				\item
					The principle \wso[], along with all its consequences, holds.
				\item\label{Theorem(consequences_of_metrization)item: continuity_principle}
					For a \ctb[]/\cms[] $\mtr{X} = (X, d)$, every map $f\colon X \to Y$ is metrically continuous, and hence \ed-continuous, for any metric on $Y$.
			\end{enumerate}
		\end{theorem}
		\begin{proof}
			\begin{enumerate}
				\item
					By the definition of metrization.
				\item
					Immediate.
				\item
					By Lemma~\ref{Lemma: smaller_basis_in_metrized_spaces} (for $f$ enumerating the countable metrically dense subset, and $\lvl = \st{2^{-n}}{n \in \NN}$).
				\item
					By Proposition~\ref{Proposition: metric_to_intrinsic_separability_in_weakly_metrized_spaces}.
				\item
					By Lemma~\ref{Lemma: implies_wso} since $\opcN$ is a \ctb[].
				\item
					By Proposition~\ref{Proposition: examples_of_metric_continuous_maps}(\ref{Proposition(examples_of_metric_continuous_maps)item: maps_from_metrized_to_metric_space_are_metrically_continuous}).
			\end{enumerate}
		\end{proof}
   
   \chapter{Models of Synthetic Topology}\label{Chapter: models}

	In this chapter we discuss properties of synthetic topology and metric spaces in the following varieties of constructive mathematics, or topoi which model them.
	\begin{enumerate}
		\item
			Classical mathematics (\clsc).
		\item
			Russian constructivism (\russ), modeled by the effective topos~\cite{Hyland_JME_1982:_the_effective_topos} which is based on the original notion of number realizability by Kleene~\cite{Kleene_SC_1945:_on_the_interpretation_of_intuitionistic_number_theory}. In computable analysis this setting is known as \emph{Type I computability}.
		\item
			A slightly strengthened version of Brouwer's intuitionism (\bint) which is modeled by the realizability topos based on Kleene's function realizability~\cite{Kleene_SC_Vesli_R_1965:_the_foundations_of_intuitionistic_mathematics}. This is also known as \emph{Type II computability}~\cite{Weihrauch_K_2000:_computable_analysis}.
		\item
			The gros topos~\cite{Johnstone_PT_1979:_on_a_topological_topos} over (the skeleton of) separable metric spaces (\grtp).
	\end{enumerate}
	The following table summarizes validity of various statements in each of these models (for the standard choices of $\tst$, $\opn$, $\cld$). We see that \clsc[] serves as a trivialization of the theory, in \russ[] there is little connection between metric and synthetic topology while a good match between them is found in \bint[] and \grtp.
	
	\begin{figure}[!hb]
	  \newcommand{\yes}{$\bullet$} \newcommand{\no}{}
	  \begin{center}
	    \begin{tabular}{||c|c|c|c|c||}
	      \hline\hline
	      &\clsc	&\russ	&\bint	&\grtp	\\ \hline\hline
	      \vphantom{$\NN^{\NN^\NN}$}
	      $\NN$ is overt
	      &\yes		&\yes		&\yes		&\yes		\\ \hline
	      \vphantom{$\NN^{\NN^\NN}$}
	      $\NN$ is compact
	      &\yes		&\no		&\no		&\no		\\ \hline
	      \vphantom{$\NN^{\NN^\NN}$}
	      \wso holds
	      &\no		&\yes		&\yes		&\yes		\\ \hline\hline
	      \vphantom{$\NN^{\NN^\NN}$}
	      \cms[s] are metrized
	      &\no		&\no		&\yes		&\yes		\\ \hline
	      \vphantom{$\NN^{\NN^\NN}$}
	      \cms[s] are overt
	      &\yes		&\yes		&\yes		&\yes		\\ \hline
	      \vphantom{$\NN^{\NN^\NN}$}
	      \cms[s] are compact
	      &\yes		&\no		&\no		&\no		\\ \hline\hline
	      \vphantom{$\NN^{\NN^\NN}$}
	      \ctb[s] are metrized
	      &\no		&\no		&\yes		&\yes		\\ \hline
	      \vphantom{$\NN^{\NN^\NN}$}
	      \ctb[s] are overt
	      &\yes		&\yes		&\yes		&\yes		\\ \hline
	      \vphantom{$\NN^{\NN^\NN}$}
	      \ctb[s] are compact
	      &\yes		&\no		&\yes		&\yes		\\ \hline\hline
	    \end{tabular}
	  \end{center}\centering
	  
	  \label{Figure: synthetic_topological_properties_of_models}
	\end{figure}

	\section{Classical Mathematics}\label{Section: classical_mathematics_model}
		
		In classical set theory $\soc = \two$. We examine all four possibilities for $\tst \subseteq \soc$.
		\begin{itemize}
			\item\proven{$\tst = \emptyset$}
				Then $\opn = \cld = \two$, so we obtain the discrete topology. All sets are discrete, Hausdorff, overt, compact and condensed.
			\item\proven{$\tst = \{\bot\}$}
				Then $\opn = \two$, but $\cld = \{\top\}$, so all subsets are open, but only the whole sets are closed in themselves. All sets are discrete, overt and compact, only the empty set and singletons are Hausdorff, and inhabited sets are condensed.
			\item\proven{$\tst = \{\top\}$}
				Then $\opn = \{\tst\}$ and $\cld = \two$. Only the whole sets are open in themselves, but all subsets are closed. Only the empty set and singletons are discrete, inhabited sets are overt while all sets are Hausdorff, compact and condensed.
			\item\proven{$\tst = \two$}
				Then $\opn = \cld = \two$, so we again have the discrete topology. All sets are discrete, Hausdorff, overt, compact and condensed.
		\end{itemize}
		For our theory of metric spaces only the last (which is also the standard) choice $\tst = \opn = \cld = \two$ is interesting since we require countable sets to be overt, and $\tst = \opn$. In this case the conditions of Theorem~\ref{Theorem: standard_special_case_of_synthetic_topology} are satisfied. This is a trivial example of our theory since all sets are overt and have decidable equality, so they are all metrized by the discrete metric by Proposition~\ref{Proposition: overt_sets_with_decidable_equality_metrized}.

	\section{Number Realizability or Russian Constructivism}\label{Section: russian_constructivism_model}
		
		We work within the framework of Russian constructivism in the style of Richman's~\cite{Richman_F_1983:_churchs_thesis_without_tears, Bridges_DS_Richman_F_1987:_varieties_of_constructive_mathematics} and synthetic computability~\cite{Bauer_A_2006:_first_steps_in_synthetic_computability}. The following principles are valid in Russian constructivism.
		\begin{enumerate}
			\item
				Countable choice $\AC{\NN}$, in fact even dependent choice.
			\item
				\df{Markov Principle}: if not all terms of a binary sequence are one, then there exists a term which is zero.
			\item
				\df{Enumerability Axiom}: there are countably many countable subsets of $\NN$.
		\end{enumerate}
		
		At first we restrict attention to $\tst = \opn = \Ros$, the standard choice. Markov Principle says $\opn \subseteq \nnst$, so the conditions of Theorem~\ref{Theorem: standard_special_case_of_synthetic_topology} are satisfied. In particular, closed subsets are precisely those which have open complements, and $\RR$, as well as metric spaces in general, are Hausdorff.
		
		Countable choice (more precisely, its instance $\ACRos$) and $\opn = \Ros$ imply that the set of countable subsets of $\NN$ is just $\tp(\NN)$ (Proposition~\ref{Proposition: semidecidable_and_countable}). By the Enumerability Axiom there is an enumeration $W\colon \NN \to \tp(\NN)$. Furthermore, the Enumerability Axiom is equivalent to Richman's axiom \textbf{CFP} which states that there is an enumeration $\phi_0, \phi_1, \ldots$ of those partial maps $\NN \parto \NN$ that have countable graphs. A partial map $f\colon \NN \parto \NN$ has an enumerable graph if and only if ``$f(n)$ is defined'' is semidecidable for all $n \in \NN$.
		
		The classical nature of Markov Principle and the non-classical nature of Enumerability Axiom combine into a strange mix of consequences. We show that \wso is valid, but $\opcN$ is not metrized or compact.
		
		\begin{proposition}
			The principle \wso holds.
		\end{proposition}
		\begin{proof}
			A detailed proof can be found in~\cite[4.26]{Bauer_A_2006:_first_steps_in_synthetic_computability}, but we also present a proof here. By Lawvere's fixed point theorem~\cite{Lawvere_WF_1969:_diagonal_arguments_and_cartesian_closed_categories} every $f\colon \opn \to \opn$ has a fixed point, namely $W_n(n) = f(W_n(n))$ where $n \in \NN$ is such that $W_n(k) = f(W_k(k))$. To prove \wso[], take $U \in \tp(\opcN)$, and observe that by Markov Principle
			\begin{align*}
				&\infty \in U \implies \xsome{n}{\NN}{n \in U}\\
				\iff\quad &\lnot\lnot(\infty \in U) \implies \lnot\lnot(\xsome{n}{\NN}{n \in U})\\
				\iff\quad &\lnot(\xsome{n}{\NN}{n \in U}) \implies \lnot(\infty \in U)\\
				\iff\quad &(\xall{n}{\NN}{n \notin U}) \implies \infty \notin U.
			\end{align*}
			Suppose then $n \notin U$ for all $n \in \NN$. Then the characteristic map $\chi_U\colon \opcN \to \opn$ of the subset $U$ factors through the quotient map $q\colon \opcN \to \opn$, $q(t) \dfeq (t < \infty)$,
			$$\xymatrix@+1em{
				\opcN \ar[r]^{\chi_U} \ar[d]_q  &  \opn  \\
				\opn \ar@{-->}[ru]_f  &
			}$$
			to give a map $f\colon \opn \to \opn$ which has some fixed point $p \in \opn$. Because $f(\top) = \bot$ and $f(p) = p$, we see that $p \neq \top$, hence $p = \bot$. Thus we get $(\infty \in U) = \chi_U(\infty) = f(q(\infty)) = f(\bot) = \bot$, as required.
		\end{proof}
		
		\begin{proposition}\label{Proposition: No_not_weakly_metrized}
			There exists $V \in \tp(\opcN)$ such that $\infty \in V$, but no ball with the center in $\infty$ and positive radius is contained in $V$.
		\end{proposition}
		
		\begin{proof}
			Let $\Tot \subseteq \NN$ be the set of those $n \in \NN$ for which $\phi_n$ is a total function, and let $\psi\colon \Tot \to \opcN$ be the composition $\psi \dfeq r \circ \rstr{\phi}_\Tot^{\NN^\NN}$ where $r\colon \NN^\NN \to \opcN$ is the usual retraction,
			$$r(\alpha)(n) \dfeq \begin{cases} 0 & \text{if } \xsome{k}{\NN_{\leq n}}{\alpha(h) = 0},\\ 1 & \text{otherwise}. \end{cases}$$
			Since both $\rstr{\phi}_\Tot^{\NN^\NN}$ and $r$ are surjective, so is $\psi$.
			
			Furthermore, define the map $s\colon \opcN \to \opcN$ by $s(\alpha)(k) \dfeq \alpha\big(\frac{k (k+1)}{2}\big)$. We have $s(\infty) = \infty$ while for $n \in \NN$ $s(n)$ computes the $n$-th term of the sequence $(0, 1, 1, 2, 2, 2, 3, 3, 3, 3, 4\ldots)$ which repeats every natural number one more time than its value.
			
			Let
			$$U \dfeq \st[2]{n \in \Tot}{\some{k}{\NN_{< s(\psi_n)}}{\psi_k = \psi_n}}.$$
			
			\begin{itemize}
				\item\proven{$U$ is an open subset in $\Tot$}
					Rewrite $U$ as
					$$U = \st{n \in \Tot}{n < s(\psi_n) \lor \Big(\psi_n < \infty \land \some{k}{\NN_{< s(\psi_n)}}{\psi_k = \psi_n}\Big)}.$$
					The statement $n < s(\psi_n)$ is open. Assuming $\psi_n < \infty$, the condition $\psi_k = \psi_n$ is decidable, and since $\NN_{< s(\psi_n)}$ is overt, $\some{k}{\NN_{< s(\psi_n)}}{\psi_k = \psi_n}$ is open. By the dominance axiom, $\psi_n < \infty \land \some{k}{\NN_{< s(\psi_n)}}{\psi_k = \psi_n}$ is open as well.
			\end{itemize}
			
			Clearly $U = \psi^{-1}\big(\psi(U)\big)$, so $V \dfeq \psi(U)$ is an open subset of $\opcN$.
			
			\begin{itemize}
				\item\proven{the set $S_m \dfeq s^{-1}(m)$ is contained in $V$ for no $m \in \NN$}
					We have
					$$S_m \dfeq s^{-1}(m) = \left\{\tfrac{m (m+1)}{2}, \tfrac{m (m+1)}{2} + 1, \ldots, \tfrac{(m+1) (m+2)}{2} - 1\right\} = \tfrac{m (m+1)}{2} + \NN_{< m+1}.$$
					For every $i \in S_m$ there is $n_i \in \NN$ such that $\psi_{n_i} = i$. Each of the numbers $n_i$ is an element of $U$ when we can find a witness $k$ for the condition $\psi_k = \psi_{n_i}$ in the definition of $U$. Note that different $n_i$s cannot receive the same witness $k$, and as there are $m + 1$ numbers $n_i$, but only $m$ possible witnesses $k$, not all of $n_k$s are in $U$, therefore not all elements of $S_m$ are in $V$.
			\end{itemize}
			
			Obviously $\infty \in V$. For any $r \in \RR_{> 0}$ the ball $\ball{\infty}{r}$ is not contained in $V$ because that would also mean that $S_m$ is contained in $V$ for large enough $m$.\footnote{The complement of the set $V$ from the proof is neither finite nor infinite. In recursion theory such sets are called \df{immune}. In fact, for any $V$ satisfying the claim of the proposition in any model of synthetic topology in which \wso holds and $\opn \subseteq \nnst$, the complement of $V$ must be immune. Indeed, it cannot be bounded (much less finite) since otherwise we would have a ball with positive radius around $\infty$ in $V$. And if an injective sequence $\NN \to V^C$ existed we could replace it with a strictly increasing sequence which would induce a map $e\colon \opcN \to \opcN$. The set $e^{-1}(V) = \{\infty\}$ would then be open in $\opcN$, in violation of \wso[].}
		\end{proof}
		
		This proposition tells us that $\opcN$ is not (weakly) metrized (and by Theorem~\ref{Theorem: metrization_of_No}, because \wso holds, that it is not compact). The immediate consequence is that the Baire space $\Baire$ and the Cantor space $\Cantor$ are not (weakly) metrized\footnote{The fact that~$\Baire$ is not (weakly) metrized is essentially a result of Friedberg's who constructed an effective but not partial recursive operator~\cite{Friedberg_RM_1958:_un_contreexample_relatif_aux_fonctionnelles_recursives}.} or compact either since if they were, so would be their retract $\opcN$.
		
		With a little more effort we can extend this example considerably. Recall that the domain of a convergent sequence in a complete metric space can be extended to $\opcN$. In particular, if a complete metric space $\mtr{X}$ has an accumulation point, an application of dependent choice yields an injective map from $\opcN$ into $\mtr{X}$. Should the image of this map have the subspace intrinsic topology, $\mtr{X}$ cannot be metrized since otherwise $\opcN$ would be metrized as well by Corollary~\ref{Corollary: metrization_of_subspaces}. We show this idea works.
		
		\begin{lemma}\label{Lemma: stable_closure_of_subcompact_is_subcompact}
			Let $K \subseteq X$. Then $K$ is subcompact in $X$ if and only if its double complement ${K^C}^C = X \setminus (X \setminus K)$ is subcompact in $X$.
		\end{lemma}
		\begin{proof}
			For any open set $U \in \tp(X)$ we have
			$${K^C}^C \subseteq U \implies K \subseteq U \implies {K^C}^C \subseteq {U^C}^C \implies {K^C}^C \subseteq U$$
			(the last implication holding by Markov principle), so $K$ is subcompact if and only if ${K^C}^C$ is.
		\end{proof}
		
		\begin{lemma}\label{Lemma: element_together_with_neighbours_subcompact_in_No}
			Recall the successor map $\succm\colon \opcN \to \opcN$ and the predecessor map $\predm\colon \opcN \to \opcN$. For an arbitrary $u \in \opcN$ the following holds.
			\begin{enumerate}
				\item
					$$\intcc[\opcN]{\predm(u)}{\succm(u)} = \big(\opcN_{< \predm(u)} \cup \opcN_{> \succm(u)}\big)^C =$$
					$$= \big(\opcN_{< \predm(u)} \cup \intoo[\opcN]{\predm(u)}{u} \cup \intoo[\opcN]{u}{\succm(u)} \cup \opcN_{> \succm(u)}\big)^C = {\big\{\predm(u), u, \succm(u)\big\}^C}^C$$
				\item
					The set $\intcc[\opcN]{\predm(u)}{\succm(u)}$ is subcompact in $\opcN$.
			\end{enumerate}
		\end{lemma}
		\begin{proof}
			\begin{enumerate}
				\item
					The first equality holds by definition of $\leq$. It is easy to see 
					$$\intoo[\opcN]{\predm(u)}{u} = \intoo[\opcN]{u}{\succm(u)} = \emptyset$$
					which takes care of the second equality. For the final one, observe that for $t \in \opcN$
					$$t > a \land t < b \iff \lnot\lnot(t > a) \land \lnot\lnot(t < b) = \lnot(t \leq a \lor t \geq b)$$
					by Markov principle, so
					\begin{align*}
						&\lnot\Big(t < \predm(u) \lor \big(t > \predm(u) \land t < u\big) \lor \big(t > u \land t < \succm(u)\big) \lor t > \succm(u)\Big)\\
						\iff\quad &t \geq \predm(u) \land \lnot\lnot\big(t \leq \predm(u) \lor t \geq u\big) \land \lnot\lnot\big(t \leq u \lor t \geq \succm(u)\big) \land t \leq \succm(u)\\
						\iff\quad &\lnot\lnot\Big(t \geq \predm(u) \land \big(t \leq \predm(u) \lor t \geq u\big) \land \big(t \leq u \lor t \geq \succm(u)\big) \land t \leq \succm(u)\Big)\\
						\iff\quad &\lnot\lnot\Big(\big(t = \predm(u) \lor t \geq u\big) \land \big(t \leq u \lor t = \succm(u)\big)\Big)\\
						\iff\quad &\lnot\lnot\big(t = \predm(u) \lor t = u \lor t = \succm(u)\big).
					\end{align*}
				\item
					The set $\big\{\predm(u), u, \succm(u)\big\}$ is finite, hence compact, and since
					$$\intcc[\opcN]{\predm(u)}{\succm(u)} = {\big\{\predm(u), u, \succm(u)\big\}^C}^C$$
					by the previous item, the set $\intcc[\opcN]{\predm(u)}{\succm(u)}$ is subcompact in $\opcN$ by Lemma~\ref{Lemma: stable_closure_of_subcompact_is_subcompact}.
			\end{enumerate}
		\end{proof}
		
		\begin{theorem}\label{Theorem: accumulation_point_implies_embedding_of_No}
			Let $\mtr{X} = (X, d_\mtr{X})$ be a metric space with an accumulation point $p \in X$.
			\begin{enumerate}
				\item\label{Theorem(accumulation_point_implies_embedding_of_No)item: sequences}
					If $\mtr{X}$ is complete, there exist maps $a\colon \opcN \to \RR_{\geq 0}$, $b\colon \{-1\} \cup \opcN \to \RR_{\geq 0}$ and $y\colon \opcN \to X$ with the following properties (with $t \in \opcN$ when indexing $a$, $y$, and $t \in \{-1\} \cup \opcN$ when indexing $b$; also $\predm(0) \dfeq -1$ on $\{-1\} \cup \opcN$):
					\begin{itemize}
						\item
							$t < \infty \iff a_t > 0 \iff b_t > 0 \iff d_\mtr{X}(p, y_t) > 0$,
						\item
							the maps $a$, $b$ are (strictly) decreasing: for $t, u \in \opcN$
							$$t < u \implies a_t > a_u, \qquad\qquad t \leq  u \implies a_t \geq a_u,$$
							and for $t, u \in \{-1\} \cup \opcN$,
							$$t < u \implies b_t > b_u, \qquad\qquad t \leq  u \implies b_t \geq b_u,$$
						\item
							$d_\mtr{X}(p, y_t) \leq 2^{-\succm(t)}$,
						\item
							$a_t \geq b_{\predm(t)} \geq d_\mtr{X}(p, y_t) \geq a_{\succm(t)} \geq b_t$; if (and only if) $t \in \NN$, then the inequalities are strict, \ie $a_t > b_{t-1} > d_\mtr{X}(p, y_t) > a_{t+1} > b_t$.
					\end{itemize}
				\item\label{Theorem(accumulation_point_implies_embedding_of_No)item: embedding_of_No_has_strongly_closed_image}
					If the maps $a, b, y$ as above exist, then there is a (strongly) injective nonexpansive map $e\colon \opcN \to X$ such that $e(\infty) = p$, and the image of $e$ is a strongly closed subset of $X$.
			\end{enumerate}
		\end{theorem}
		\begin{proof}
			\begin{enumerate}
				\item
					We inductively define sequences $a\colon \NN \to \RR_{> 0}$, $b\colon \ZZ_{\geq -1} \to \RR_{> 0}$ and $y\colon \NN \to X$ as follows. Let $a_0 \dfeq 1$. Now take $n \in \NN$, and assume that we know $a_n$. Define $b_{n-1} \dfeq \frac{a_n}{2}$, and let $y_n \in X$ be some element for which $0 < d_\mtr{X}(p, y_n) < b_{n-1}$ (it exists because $p$ is an accumulation point). Also, let $a_{n+1} \dfeq \frac{d_\mtr{X}(p, y_n)}{2}$. Dependent choice gives us the actual sequences which have the following properties for all suitable $n$:
					$$b_{n+1} < \frac{b_n}{4}, \qquad\qquad a_{n+1} < \frac{b_{n-1}}{2} = \frac{a_n}{4}, \qquad\qquad b_n = \frac{a_{n+1}}{2} < a_{n+1},$$
					$$a_n \leq 2^{-2n}, \qquad\qquad b_n \leq 2^{-2n-3}, \qquad\qquad d_\mtr{X}(p, y_n) < b_{n-1} \leq 2^{-2n-1} \leq 2^{-n-1},$$
					$$a_n > d_\mtr{X}(p, y_n) > b_n, \qquad\qquad \bigcup_{n \in \NN} \intoo{a_n}{b_n} = \intoo{0}{1}.$$
					These sequences can be extended to $a, b\colon \opcN \to \RR_{\geq 0}$, $y\colon \opcN \to X$ (which we by a slight abuse of notation denote by the same letters) by Proposition~\ref{Proposition: convergent_sequences_and_maps_from_No}. Notice that $a_\infty = b_\infty = 0$ and $y_\infty = p$, and that $a$, $b$, $y$ satisfy all required properties.
					
				\item
					Define $e\colon \opcN \to X$ by $e(t) \dfeq y_t$.
					
					For any $i, j \in \NN$, $i \neq j$, we have
					$$d_\mtr{X}(y_i, y_j) \leq d_\mtr{X}(p, y_i) + d_\mtr{X}(p, y_j) \leq 2^{-i-1} + 2^{-j-1} \leq 2^{-\inf\{i,j\}} = d_C(i, j).$$
					By Theorem~\ref{Theorem: extend_Lipschitz_maps_from_a_dense_subspace_into_a_completion} the map $e$ is Lipschitz with coefficient $1$, \ie nonexpansive.
					
					To prove that $e$ is strongly injective, take $t, u \in \opcN$, and assume $t \apart u$, \ie $t < u \lor t > u$. Without loss of generality $t < u$ (in particular $t \in \NN$ and $t+1 \leq u$). Then
					$$d_\mtr{X}(y_t, y_u) \geq d_\mtr{X}(p, y_t) - d_\mtr{X}(p, y_u) > a_{t+1} - a_u \geq 0.$$
					
					It remains to show that the image of $e$ is strongly closed. We need to see that it is closed, and that it has the subspace intrinsic topology. Closedness is easy: $\opcN$ is an inhabited \ctb[], so its image is as well, and therefore strongly located in $X$. Since $\RR$ is Hausdorff, $e(\opcN)$ is closed.
					
					As for the subspace topology, take any $U \in \tp(e(\opcN))$. For $x \in X$ we define
					$$A_x \dfeq \st{t \in \opcN}{a_t \geq d_\mtr{X}(p, x) \geq b_t}.$$
					We claim that for all $x \in \ball[\mtr{X}]{p}{1}$ the set $A_x$ is inhabited, though not by much. More precisely, we claim that there exists $u \in \opcN$ such that $u \in A_x \subseteq \intcc[\opcN]{\predm(u)}{\succm(u)}$. Define $\alpha$ to be a sequence $\NN \to \two$ such that $\alpha_n = 1 \implies d_\mtr{X}(p, x) < a_{n+1}$ and $\alpha_n = 0 \implies d_\mtr{X}(p, x) > b_n$. Such $\alpha$ exists by countable choice because $b_n < a_{n+1}$. Let $u \dfeq r(\alpha)$ where $r\colon \two^\NN \to \opcN$ is the standard retraction. In particular this means that if $u \in \NN$, then $\alpha_u$ is the first zero in the sequence $\alpha$.
					\begin{itemize}
						\item\proven{$u \in A_x$}
							First we prove $a_u \geq d_\mtr{X}(p, x)$. If $u = 0$, we are done since $x \in \ball[\mtr{X}]{p}{1}$. Assume now $u \geq 1$, and $a_u < d_\mtr{X}(p, x)$. This means there exists $n \in \NN$ such that $2^{-n} < d_\mtr{X}(p, x)$. If $n < u$, then $\alpha_n = 1$, so $d_\mtr{X}(p, x) < a_{n+1} < d_\mtr{X}(p, y_n) \leq 2^{-n-1} < 2^{-n}$, a contradiction. If $n \geq u$, then $u \in \NN$, and $\alpha_u = 0$ while $\alpha_{u-1} = 1$ (recall $u \geq 1$), so $d_\mtr{X}(p, x) < a_u$ which is a contradiction also.
							
							Second, assume $d_\mtr{X}(p, x) < b_u$. Then $b_u > 0$, so $u \in \NN$. Consequently $\alpha_u = 0$, meaning $d_\mtr{X}(p, x) > b_u$, a contradiction.
						
						\item\proven{$A_x \subseteq \intcc[\opcN]{\predm(u)}{\succm(u)}$}
							Take any $t \in A_x$; then $a_t \geq d_\mtr{X}(p, x) \geq b_t$ holds.
							
							Assume $t < \predm(u)$. Then $t \in \NN$, and $t+1 < u$, so $\alpha_{t+1} = 1$ which means $d_\mtr{X}(p, x) < a_{t+2} < b_t$, a contradiction.
							
							Assume $t > \succm(u)$. Then $u \in \NN$, and $u+2 \leq t$. We have $\alpha_u = 0$ which means $d_\mtr{X}(p, x) > b_u > a_{u+2} \geq a_t$, a contradiction.
						
						\item\proven{$A_x$ is subcompact in $\opcN$}
							The set $A_x$ is closed in $\intcc[\opcN]{\predm(u)}{\succm(u)}$ (because $\geq$ is a closed relation since $<$ is open) which is in turn subcompact in $\opcN$ by Lemma~\ref{Lemma: element_together_with_neighbours_subcompact_in_No}.
					\end{itemize}
					Consequently $e(A_x)$ is subcompact in $e(\opcN)$, so the set
					$$V \dfeq \st{x \in \ball[\mtr{X}]{p}{1}}{e(A_x) \subseteq U}$$
					is open in $\ball[\mtr{X}]{p}{1}$, therefore also in $X$ (since $\opn$ is a dominance). We prove the theorem once we conclude $V \cap e(\opcN) = U$.
					\begin{itemize}
						\item\proven{$A_{y_u} = \{u\}$ for all $u \in \opcN$}
							Since $a_u \geq d_\mtr{X}(p, y_u) \geq b_u$, we have $u \in A_{y_u}$. Conversely, take any $t \in A_{y_u}$. If $t < u$, then $t \in \NN$, and we obtain a contradiction $d_\mtr{X}(p, y_u) > a_{t+1} \geq a_u \geq d_\mtr{X}(p, y_u)$. If $t > u$, we similarly obtain a contradiction $d_\mtr{X}(p, y_u) > a_{u+1} \geq a_t \geq d_\mtr{X}(p, y_u)$. Thus $t = u$.
						\item\proven{$V \cap e(\opcN) = U$}
							The claim follows from $e(A_{y_u}) \subseteq U \iff e(\{u\}) \subseteq U \iff y_u \in U$.
					\end{itemize}
			\end{enumerate}
		\end{proof}
		
		\begin{lemma}\label{Lemma: metrization_and_compactness_extend_from_potentially_whole_to_the_whole_space}
			Let $A \subseteq X$, $A^C = \emptyset$, and let $\mtr{X} = (X, d_\mtr{X})$ be a metric space, with $\mtr{A} = (A, d_\mtr{A})$ its metric subspace.
			\begin{enumerate}
				\item
					If $\mtr{A}$ is metrized, then so is $\mtr{X}$.
				\item
					If $A$ is compact, then so is $X$.
			\end{enumerate}
		\end{lemma}
		\begin{proof}
			\begin{enumerate}
				\item
					Take any $U \in \tp(X)$. Then $U \cap A$ is open in $A$, so under the assumption that $\mtr{A}$ is metrized, we can write it as an overtly indexed union
					$$U \cap A = \bigcup_{i \in I} \ball[\mtr{A}]{a_i}{r_i}$$
					where $a_i \in A$, $r_i \in \RR$.
					
					We claim that $U = \bigcup_{i \in I} \ball[\mtr{X}]{a_i}{r_i}$. Since the sets both on the left and the right are open, and therefore $\lnot\lnot$-stable, we can just as well prove that their complements match. But for any $V \in \tp(X)$ and $x \in X$ we have
					$$\lnot(x \in V \land x \in A) \iff \lnot(\lnot\lnot(x \in V) \land \lnot\lnot(x \in A)) \iff \lnot\lnot\lnot(x \in V) \iff \lnot(x \in V),$$
					so $(V \cap A)^C = V^C$. Since $U \cap A = \bigcup_{i \in I} \ball[\mtr{A}]{a_i}{r_i}$ and $\left(\bigcup_{i \in I} \ball[\mtr{X}]{a_i}{r_i}\right) \cap A = \bigcup_{i \in I} \ball[\mtr{A}]{a_i}{r_i}$, the result follows.
				
				\item
					If $A$ is compact, it is subcompact in $X$, so by Lemma~\ref{Lemma: stable_closure_of_subcompact_is_subcompact} ${A^C}^C = X$ is subcompact in $X$, \ie $X$ is compact.
			\end{enumerate}
		\end{proof}
		
		\begin{theorem}
			Let $\mtr{X} = (X, d)$ be a metric space with a metrically dense subovert subset (\eg a metrically separable space).
			\begin{enumerate}
				\item
					If $\mtr{X}$ has an accumulation point, then it is neither metrized nor compact.
				\item
					Say that $\mtr{X}$ is \df{locally semilocated} when for every $x \in X$ there exists $r \in \RR_{> 0}$ such that the distance $d\big(\st{y \in X}{0 < d(x, y) < r}, x\big)$ is an extended real number. The following are equivalent.
					\begin{itemize}
						\item
							$\mtr{X}$ is metrized and locally semilocated.
						\item
							The metric $d$ is equivalent to the discrete metric.
					\end{itemize}
					In particular, if either (and therefore both) of these conditions holds, then $X$ has decidable equality.
			\end{enumerate}
		\end{theorem}
		\begin{proof}
			\begin{enumerate}
				\item
					Let $p \in X$ be an accumulation point in $\mtr{X}$, and $\cmpl{\mtr{X}}$ the completion of $\mtr{X}$ (we identify $\mtr{X}$ with a metric subspace of $\cmpl{\mtr{X}}$). By Theorem~\ref{Theorem: accumulation_point_implies_embedding_of_No} there exist suitable sequences $a, b, y$ which induce an injective nonexpansive map $e\colon \opcN \to \cmpl{\mtr{X}}$ with a strongly closed image. Let $X' \dfeq X \cup e(\opcN)$, and $\mtr{X'}$ the metric subspace in $\cmpl{\mtr{X}}$ with the underlying set $X'$. By Theorem~\ref{Theorem: accumulation_point_implies_embedding_of_No}(\ref{Theorem(accumulation_point_implies_embedding_of_No)item: embedding_of_No_has_strongly_closed_image}) the image of $e$ is strongly closed also in $X'$. If $\mtr{X}$ is metrized/compact, then so is $\mtr{X'}$ by Lemma~\ref{Lemma: metrization_and_compactness_extend_from_potentially_whole_to_the_whole_space} which implies that $e(\opcN)$ is metrized/compact by Theorem~\ref{Theorem: metrization_and_subspaces}. Since $e$ is bijective onto its image, compactness of $e(\opcN)$ implies compactness of $\opcN$, and since $e$ is furthermore nonexpansive, metrization of $e(\opcN)$ implies metrization of $\opcN$ by Lemma~\ref{Lemma: Lipschitz_bijection_reflects_metrization}. Contradiction.
				\item
					Notice that $d$ is equivalent to the discrete metric if and only if for every $x \in X$ there exists $\epsilon \in \RR_{> 0}$ such that $\ball{x}{\epsilon} = \{x\}$, or equivalently, $\st{y \in X}{0 < d(x, y) < \epsilon} = \emptyset$.
					\begin{itemize}
						\item\proven{$(\Rightarrow)$}
							Take any $x \in X$, and let $r \in \RR_{> 0}$ witness local semilocatedness around $x$. Let $A \dfeq \st{y \in X}{0 < d(x, y) < r}$, and define $\epsilon \dfeq \inf\{d(A, x), r\}$; then $\epsilon$ is a nonnegative real number. Since $d(A, x)$ is the strict infimum, if $d(A, x) = 0$, then $x$ is an accumulation point, and $\mtr{X}$ could not be metrized. Thus (by Markov principle) $d(A, x) > 0$, so $\epsilon > 0$. If there is $y \in X$, $0 < d(x, y) < \epsilon$, then $y \in A$, therefore $d(x, y) \geq \epsilon$, a contradiction, so this $\epsilon$ works.
						\item\proven{$(\Leftarrow)$}
							Suppose for every $x \in X$ there is $\epsilon \in \RR_{> 0}$ such that $\ball{x}{\epsilon} = \{x\}$. Then $r \dfeq \epsilon$ witnesses local semilocatedness of $x$ (namely, $A \dfeq \st{y \in X}{0 < d(x, y) < r}$ is empty, so $d_\mtr{X}(A, x) = \infty$). Moreover, we see that the only metrically dense subset of $X$ is $X$ itself, so $X$ is overt, and for any $y \in X$,
							$$\top \iff 0 < d(x, y) \lor d(x, y) < \epsilon \implies x \neq y \lor x = y,$$
							so $X$ has decidable equality. Thus $X$ is metrized by the discrete metric (by Proposition~\ref{Proposition: overt_sets_with_decidable_equality_metrized}), therefore also by $d$ since it is equivalent to it.
					\end{itemize}
			\end{enumerate}
		\end{proof}
		
		The conclusion is that in Russian constructivism most interesting metric spaces are not metrized. Granted, we considered only the case $\opn = \Ros$ (the smallest $\opn$ for which $\NN$ is overt), but if a space is not metrized with respect to $\Ros$, then it is not metrized with respect to any larger $\opn$. This suggests that Type I computability is not an optimal choice for computation with metric spaces.

	\section{Function Realizability or Brouwer's Intuitionism}\label{Section: Brouwer's_intuitionism_model}
		
		In this section we consider a slightly strengthened version of Brouwer's intuitionism. We adopt the following principles:
		\begin{enumerate}
			\item
				\df{Function-Function Choice} $\AC{\NN^\NN, \NN^\NN}$,
			\item
				\df{Continuity Principle}: for every $f\colon \NN^\NN \to \NN$ and $\alpha \in \NN^\NN$ there exists $k \in \NN$ such that $\beta \in \ball[C]{\alpha}{2^{-k}}$ implies $f(\alpha) = f(\beta)$ (in short: every $f\colon \NN^\NN \to \NN$ is \ed-continuous),
			\item
				\df{Fan Principle}: every decidable bar is uniform, \cf Section~\ref{Section: wso}.
		\end{enumerate}
		In the usual setting INT~\cite[5.2]{Bridges_DS_Richman_F_1987:_varieties_of_constructive_mathematics} only the weaker \df{Function-Number Choice} $\AC{\NN^\NN, \NN}$ is used. Kleene's function realizability and the corresponding realizability topos validate not only Function-Function Choice but even \df{Function Choice} $\AC{\NN^\NN}$. In Type II effectivity Function Choice manifests itself as the fact that $\NN^\NN$ has an admissible injective representation.
		
		We show that $\opn = \Ros$, Function-Function Choice, and Continuity Principle together imply that the Baire space $\Baire$ is metrized. It then follows from Theorem~\ref{Theorem: metrization_of_Baire/Cantor_implies_metrization_of_cmss_ctbs} that all \cms[s] are metrized, and from the Fan Principle and Corollary~\ref{Corollary: Fan_Principle_and_compactness_of_Cantor} that the Cantor space $\Cantor$, and more generally any \ctb[], is compact.
		
		We start by proving that $\Baire$ is weakly metrized, \ie that every open subset $U \subseteq \NN^\NN$ is a union of balls. Because the map $q\colon \two^\NN \to \opn$, $q(\alpha) \dfeq (\xsome{n}{\NN}{\alpha_n = 0})$ is surjective, by Function-Function Choice there exists a map $f\colon \NN^\NN \to \two^\NN$ such that
		$$U = \st{\alpha \in \Baire}{\xsome{n}{\NN}{f(\alpha)(n) = 0}}.$$
		By Continuity Principle and Function-Number Choice there exists a modulus of continuity $\mu\colon \NN^\NN \times \NN \to \NN$ such that for every $\alpha \in \NN^\NN$, $\beta \in B(\alpha, 2^{-\mu(\alpha, n)})$ implies $f(\alpha)(n) = f(\beta)(n)$. Define
		$$I \dfeq \st{(\alpha, n) \in \NN^\NN \times \NN}{f(\alpha)(n) = 0},$$
		and observe that
		$$U = \bigcup_{(\alpha,n) \in I} \ball{\alpha}{2^{-\mu(\alpha, n)}}.$$
		We now know that open subsets are unions of open balls, therefore $\NN^\NN$ is overt by Proposition~\ref{Proposition: metric_to_intrinsic_separability_in_weakly_metrized_spaces}. Consequently, $I$ is overt as well. We proved that $\Baire$ is metrized.
		
		We see that metric spaces are far better behaved in Type II effectivity than Type I effectivity, suggesting that the former is better for modeling metric spaces than the latter.
		
		On a side note, one may wonder if Continuity Principle alone implies that \cms[s] are metrized for the case $\opn = \Ros$. This would make metrization of \cms[s] equivalent to the Continuity Principle since the former implies the latter by Theorem~\ref{Theorem: consequences_of_metrization}(\ref{Theorem(consequences_of_metrization)item: continuity_principle}). However, in Russian constructivism the Continuity Principle is validated by the Kreisel-Lacombe-Shoenfield-Ceitin theorem~\cite{Kreisel_G_Lacombe_D_Shoenfield_JR_1959:_partial_recursive_functionals_and_effective_operations, Ceitin_GS_1962:_algorithmic_operators_in_constructive_metric_spaces, Ceitin_GS_1959:_algorithmic_operators_in_constructive_complete_separable_metric_spaces}, but the \cms[s] are generally not metrized, as we showed in Section~\ref{Section: russian_constructivism_model}.

		\section{Gros Topos}\label{Section: gros_topos_model}
		
	   	In this section we review the notions of presheaves, sheaves and gros topos~\cite{Johnstone_PT_1979:_on_a_topological_topos, Mac_Lane_S_Moerdijk_I_1992:_sheaves_in_geometry_and_logic}, and consider synthetic topological properties in gros topos.
	   	
	   	Let $\site$ denote a small full subcategory of $\Top$ with the property that for any space $X$ in $\site$ all its open subsets (and consequently their open embeddings) are in $\site$. It is not strictly necessary to assume that $\site$ has finite products, but it is very convenient, and we shall do so. \df{Presheaves} on $\site$ are (contravariant) functors $\site\op \to \Set$; they form the functor category $\Set^{\site\op}$ with natural transformations as morphisms. The Yoneda functor $\y\colon \site \to \Set^{\site\op}$, given by $\y[X] = \C(\insarg, X)$ and $(\y[f])_Z = \C(Z, f) = f \circ \insarg$, is a categorical embedding (a full and faithful functor, injective on objects) of $\site$ into the category of presheaves.
	   	
	   	Presheaves form a topos. The limits are computed pointwise (for example $(F \times G)(X) = F(X) \times G(X)$, and the terminal object is the constant functor $\one$), and exponentials are given as $G^F(X) = \nat(\y[X] \times F, G)$ (where $\nat(F, G)$ denotes the set of natural transformations between functors $F$ and $G$) with the evaluation map $\ev\colon G^F \times F \to G$, $\ev_X(\eta, x) = \eta_X(\id[X], x)$. The Yoneda functor preserves limits and exponentials that exist in $\site$.
	   	
	   	In any category a morphism $f\colon X \to Y$ is a monomorphism if and only if the diagram
	   	$$\xymatrix@+1em{
	   		X \ar[r]^{\id[X]} \ar[d]_{\id[X]}  &  X \ar[d]^f  \\
	   		X \ar[r]_f  &  Y
	   	}$$
	   	is a pullback. The consequence for our case is that the Yoneda embedding $\y$ preserves monos, and that a natural transformation between presheaves is a mono in $\Set^{\site\op}$ if and only if its every component is an injective map (a mono in $\Set$).
	   	
	   	A \df{sieve} $\sieve \subseteq \ms{X}$ on a space $X$ in $\site$ is a set of maps in $\site$ with the codomain $X$ such that if $f\colon B \to X$ is in $\sieve$ and $g\colon A \to B$ is any map in $\site$, then $f \circ g\colon A \to X$ is also in the sieve $\sieve$. The subobject classifier in $\Set^{\site\op}$ is given as
	   	$$\soc(X) = \text{ the set of all sieves on $X$}$$
	   	while for $f\colon X \to Y$ in $\site$ and a sieve $\sieve \in \soc(Y)$
	   	$$\soc(f)(\sieve) = \st{g \in \ms{X}}{f \circ g \in \sieve}.$$
	   	The truth map $\top\colon \one \to \soc$ picks at each component $X$ the \df{maximal sieve} $\ms{X}$ on $X$, \ie the set of all maps in $\site$ with the codomain $X$. Given any monomorphism $\iota\colon F \to G$, the corresponding characteristic map $\chi^\iota\colon G \to \soc$ at component $X$ is
	   	$$\chi^\iota_X(x) = \st{f \in \ms{X}}{\some{a}{F(\dom(f))}{G(f)(x) = \iota_{\dom(f)}(a)}},$$
	   	\ie the sieve of all maps $f\colon A \to X$ in $\site$ for which $G(f)(x)$ is in the image of $\iota_A$.
	   	$$\xymatrix@+2em{
	   		a \in F(A) \ar[d]_{\iota_A}  &  F(X) \ar[l]_{F(f)} \ar[d]^{\iota_X} \phantom{\owns x}  \\
	   		{}\phantom{a \in} G(A)  &  G(X) \ar[l]^{G(f)} \owns x
	   	}$$
	   	
	   	A presheaf $F\colon \site\op \to \Set$ is a \df{sheaf} on $\site$ when for every $X$ in $\site$, every open cover $\{U_i\}_{i \in I}$ of $X$ and every family of elements $\{f_i \in F(U_i)\}_{i \in I}$ satisfying the \df{gluing condition}
	   	$$F(U_i \cap U_j \hookrightarrow U_i)(f_i) = F(U_i \cap U_j \hookrightarrow U_j)(f_j) \qquad \text{for all $i, j \in I$}$$
	   	there exists a unique element $f \in F(X)$ with the property that $F(U_i \hookrightarrow X)(f) = f_i$ for all $i \in I$.
	   	
	   	Let $\Sh$ denote the full subcategory of sheaves in $\Set^{\site\op}$. For any topological space $X$ the functor $\C(\insarg, X)$ is a sheaf; in particular, the Yoneda embedding restricts to $\y\colon \site \to \Sh$. Limits (but not colimits) and exponentials in $\Sh$ are computed the same as in $\Set^{\site\op}$, so $\y$ still preserves them.
	   	
	   	A \df{closed sieve} on $X$ is a sieve on $X$ with the additional property that for every space $A$ in $\site$, every open cover $\{U_i\}_{i \in I}$ of $A$, and every family of maps $\{f_i\colon U_i \to X\}_{i \in I}$ in the sieve which satisfy the gluing condition that $f_i$ and $f_j$ restrict to the same map on $U_i \cap U_j$ for all $i, j \in I$, the unique map $f\colon A \to X$ with the property that restricted to $U_i$ is $f_i$ for all $i \in I$, is also in the sieve\footnote{This has nothing to do with closedness in the sense of synthetic topology; it is the established terminology, in the sense that a sieve is closed for the operation of gluing maps together.}. The subobject classifier in $\Sh$ is given as $\soc(X) = \text{ the set of closed sieves on $X$}$. The morphism part of $\soc$, the truth map and the characteristic maps are the same as in $\Set^{\site\op}$. The category $\Sh$ is called \df{gros topos} (over $\site$).
	   	
	      While $\top$ is represented as the maximal sieve at each component, $\bot$ is represented as the minimal closed sieve; at component $X$, this is the sieve which contains only the map $\emptyset \hookrightarrow X$ (which every closed sieve must contain since $\emptyset$ is covered by an empty family of open subsets)\footnote{This differs from the situation in presheaves where the minimal sieve is the empty one.}. One may calculate that the negation $\lnot\colon \soc \to \soc$ maps a closed sieve $\sieve \in \soc(X)$ to the closed sieve of all maps in $\site$ with codomain $X$ which have the image disjoint from all maps in $\sieve$, \ie which map into the complement of the union of images of maps in $\sieve$. This means that sieves in the image of $\lnot$ are of the form ``all maps with the image in a certain subspace'', and so can be identified with the subspaces themselves, or alternatively, the subsets of the underlying set of a topological space. Thus, the $\lnot\lnot$-stable truth values can be represented as
	      $$\nnst(X) = \pst(X) \text{ = the power set of the underlying set of the space $X$},$$
	      $$\nnst(f) = f^{-1},$$
	      while the inclusion $\nnst \hookrightarrow \soc$ maps a subset to the closed sieve of all maps with the image in that subset.
	      
	      The natural choice of $\opn = \tst$ in the sheaf topos is to assign the set of all open subsets (the topology) to each space:
	      $$\opn(X) = \optp(X), \qquad \opn(f) = f^{-1},$$
	      and $\opn$ is embedded into $\soc$ the same way as $\nnst$ is, so it factors through the inclusion $\opn \hookrightarrow \nnst$ which just embeds the topologies into the power sets. Hence $\opn \subseteq \nnst$ in gros topos.
	      
	      To argue that this is a reasonable choice of topology we note that the Yoneda embedding $\y$ ``preserves topologies''. Recall that Yoneda lemma (in contravariant form) states (in our case) that for every topological space $X \in \obj{\site}$ and a (pre)sheaf $F$ we have $\nat(\y[X], F) \ism F(X)$. Thus the set of global elements of a (pre)sheaf $F$ is $F(\one)$ (since $F(\one) \ism \nat(\y[\one], F)$, and $\y[\one]$ is the terminal (pre)sheaf). Then by Lemma~\ref{Lemma: stable_subsets_of_a_representable_object} below the set of global elements of the intrinsic topology of $\y[X]$ is precisely the topology of $X$,
	      $$\opn^{(\y[X])}(\one) \ism \tp(\one \times X) \ism \tp(X).$$
	      
	      Consider now the Heyting lattice operations. The negation, restricted to $\nnst$, is just the complementation at each component. The conjunction on $\nnst$ is the intersection; the same on $\opn$ when finite. Because in $\Sh$ we restrict to closed sieves, the disjunction on $\opn$ (but not in general on $\nnst$) is the union.\footnote{That does not mean that $\opn$ has all (internal) disjunctions, though it has all external (\ie set-indexed) ones. However, in the case of finite disjunctions, internal and external view amount to the same thing.} In particular, $\opn$ is a bounded sublattice of $\soc$.
	      
	      Altogether, we may apply Theorem~\ref{Theorem: standard_special_case_of_synthetic_topology}. From it it is obvious that closed truth values are the subfunctor of $\soc$ as follows:
	      $$\cld(X) = \st[2]{\sieve \in \soc(X)}{\bigcup_{f \in \sieve} \im(f) \in \cltp(X)}$$
	      (in words, $\cld(X)$ is the set of all closed sieves $\sieve$ such that the union of images of maps in $\sieve$ is a closed subset of $X$) while the morphism part of $\cld$ is the same as for $\soc$. Also, the $\lnot\lnot$-stable closed truth values can be expressed as the subfunctor of $\nnst$,
	      $$\nnst[\cld](X) = \cltp(X) \text{ = the set of all closed subsets of $X$}$$
	      with the morphism part being the same as for $\nnst$. Finally, the decidable truth values are
	      $$\two(X) = \optp(X) \cap \cltp(X) = \text{the set of clopen subsets of $X$}.$$
	      
	      Note that functors $\tst = \opn$, $\nnst$, $\nnst[\cld]$ and $\two$ can be equivalently represented as continuous maps into familiar spaces --- recall Section~\ref{Section: to_synthetic_topology}. The functor $\nnst$ is isomorphic to $\C(\insarg, \trivtwo)$, the functor $\opn$ to $\C(\insarg, \sier)$, the functor $\nnst[\cld]$ to $\C(\insarg, \twsier)$, and $\two$ to $\C(\insarg, \two)$; in all cases, the isomorphism takes a subset to its characteristic map (mapping the elements of the subset to $1$ and the others to $0$). Thus if $\sier \in \obj{\site}$, then $\opn \ism \y[\sier]$.
	      
			We also remark that powers of these functors simplify when the exponent is representable.
	      \begin{lemma}\label{Lemma: stable_subsets_of_a_representable_object}
	      	For $X, Z \in \obj{\site}$ we have
	      	$$\nnst^{\y[X]}(Z) \ism \pst(Z \times X), \qquad \opn^{\y[X]}(Z) \ism \optp(Z \times X), \qquad \nnst[\cld]^{\y[X]}(Z) \ism \cltp(Z \times X).$$
	      \end{lemma}
	      \begin{proof}
	      	Let $F$ be one of $\nnst$, $\opn$ or $\nnst[\cld]$. The result follows from the calculation
	      	$$F^{(\y[X])}(Z) \ism \nat(\y[Z] \times \y[X], F) \ism \nat\big(\y[(Z \times X)], F\big) \ism F(Z \times X),$$
	      	using Yoneda lemma for the last isomorphism.
	      \end{proof}
	      
	      \intermission
	      
	      It is known that $\opn$ is a dominance. The proof is basically the same as the proof that stable closed subsets are subspaces which we present here, as we will need it later.
	      
	      \begin{proposition}\label{Proposition: stable_closedness_transitive_in_gros_topos}
	      	Stable closedness is a transitive relation in $\Sh$ (stable closed subobjects of stable closed subobjects are stable closed).
	      \end{proposition}
	      \begin{proof}
	      	Let $\iota\colon F \to G$ and $\kappa\colon G \to H$ be monomorphisms in $\Sh$, classified by $\nnst[\cld]$. We must prove that their composition is classified by $\nnst[\cld]$ as well.
	      	
	      	Take any $Z \in \obj{\site}$ and any $c \in H(Z)$. Since $\chi^\kappa$ maps into $\nnst[\cld]$, there exists a closed subset $Y \subseteq Z$ such that the sieve $\chi^\kappa_Z(c)$ consists of maps which factor through $Y \hookrightarrow Z$. In particular the inclusion $Y \hookrightarrow Z$ is in the sieve as well, so there exists $b \in G(Y)$ such that $\kappa_Y(b) = H(Y \hookrightarrow Z)(c)$.
	      	
	      	Again, since $\chi^\iota$ maps into $\nnst[\cld]$, there exists a closed subset $X \subseteq Y$ such that the sieve $\chi^\iota_Y(b)$ consists of maps which factor through $X \hookrightarrow Y$, so $(X \hookrightarrow Y) \in \chi^\iota_Y(b)$, and there exists $a \in F(X)$ such that $\iota_X(a) = G(X \hookrightarrow Y)(b)$.
	      	
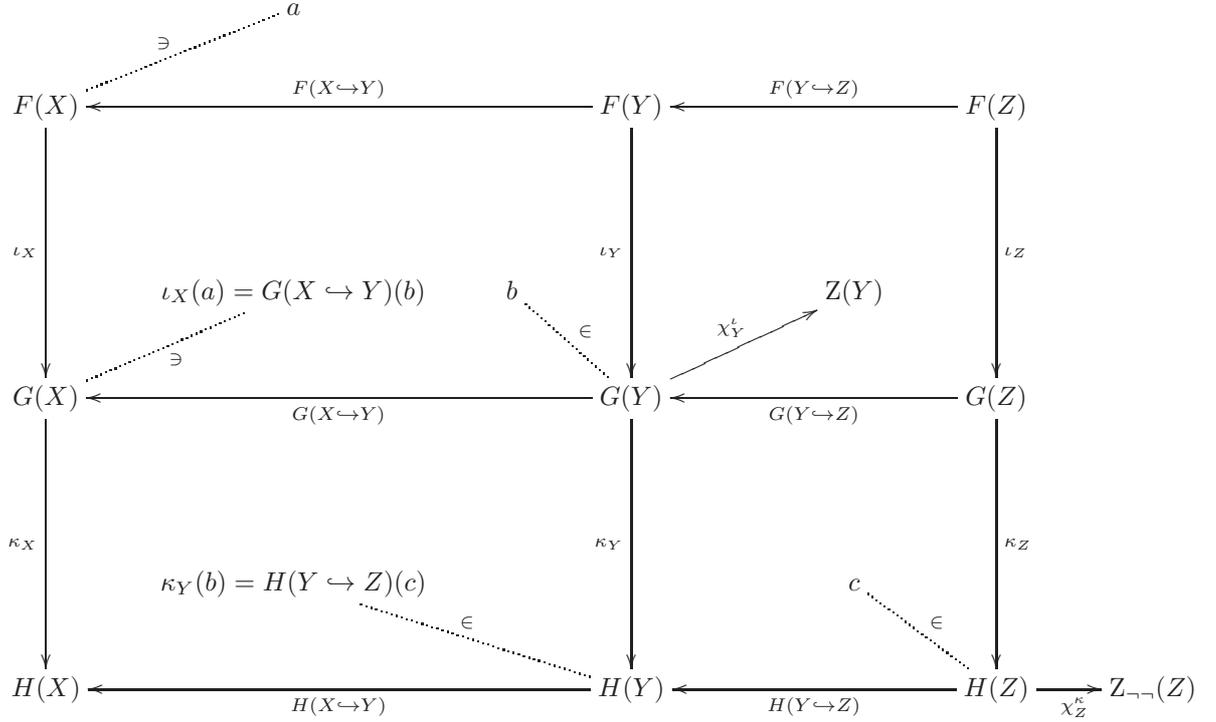
	      	$$\xymatrix@+0em{
	      		&  a \ar@{..}[ld]_\owns  &&&&&&\\
	      		F(X) \ar[ddd]_{\iota_X}  &&&  F(Y) \ar[lll]_{F(X \hookrightarrow Y)} \ar[ddd]_{\iota_Y}  &&&  F(Z) \ar[lll]_{F(Y \hookrightarrow Z)} \ar[ddd]^{\iota_Z}  &\\
	      		&&&&&&&\\
	      		&  \iota_X(a) = G(X \hookrightarrow Y)(b) \ar@{..}[ld]^\owns  &  b \ar@{..}[rd]^\in  &&&  \cld(Y)  &&\\
	      		G(X) \ar[ddd]_{\kappa_X}  &&&  G(Y) \ar[lll]^{G(X \hookrightarrow Y)} \ar[ddd]_{\kappa_Y} \ar[rru]^{\chi^\iota_Y}  &&&  G(Z) \ar[lll]^{G(Y \hookrightarrow Z)} \ar[ddd]^{\kappa_Z}  &\\
	      		&&&&&&&\\
	      		&  \kappa_Y(b) = H(Y \hookrightarrow Z)(c) \ar@{..}[rrd]^\in  &&&&  c \ar@{..}[rd]^\in  &&\\
	      		H(X)  &&&  H(Y) \ar[lll]^{H(X \hookrightarrow Y)}  &&&  H(Z) \ar[lll]^{H(Y \hookrightarrow Z)} \ar[r]_{\chi^\kappa_Z}  &  \nnst[\cld](Z)
	      	}$$
	      	
	      	We claim that $\chi^{\kappa \circ \iota}_Z(c)$ consists precisely of maps which factor through $X \hookrightarrow Z$. For one direction it is sufficient to verify that $(X \hookrightarrow Z) \in \chi^{\kappa \circ \iota}_Z(c)$. This holds because
	      	$$\kappa_X \circ \iota_X(a) = \kappa_X \circ G(X \hookrightarrow Y)(b) = H(X \hookrightarrow Y) \circ \kappa_Y(b) =$$
	      	$$= H(X \hookrightarrow Y) \circ H(Y \hookrightarrow Z)(c) = H\big((Y \hookrightarrow Z) \circ (X \hookrightarrow Y)\big) = H(X \hookrightarrow Z)(c).$$
	      	Conversely, take any map $f\colon A \to Z$ in $\chi^{\kappa \circ \iota}_Z(c)$, \ie there exists $o \in F(A)$ such that $H(f)(c) = \kappa_A \circ \iota_A(o)$. In particular $H(f)(c)$ is in the image of $\kappa_A$ which means that the restriction $\rstr{f}^Y\colon A \to Y$ exists (\ie the image of $f$ is contained in $Y$). We have
	      	$$\kappa_A \circ G(\rstr{f}^Y)(b) = H(\rstr{f}^Y) \circ \kappa_Y(b) =  H(\rstr{f}^Y) \circ H(Y \hookrightarrow Z)(c) =$$
	      	$$= H\big((Y \hookrightarrow Z) \circ \rstr{f}^Y\big)(c) = H(f)(c) = \kappa_A \circ \iota_A(o),$$
	      	and since $\kappa_A$ is injective, $G(\rstr{f}^Y)(b) = \iota_A(o)$. Thus $\rstr{f}^Y \in \chi^\iota_Y(b)$, so the image of $\rstr{f}^Y$, and therefore also the image of $f$, is contained in $X$.
	      	
	      	The statement of the proposition now follows from transitivity of closedness in classical topology: since $X$ is closed in $Y$ which is closed in $Z$, $X$ is also closed in $Z$.
	      \end{proof}
	      
	      \begin{corollary}
	      	Stable closed subobjects in $\Sh$ are subspaces.
	      \end{corollary}
	      \begin{proof}
	      	By Proposition~\ref{Proposition: stable_closedness_transitive_in_gros_topos} and Theorem~\ref{Theorem: standard_special_case_of_synthetic_topology}.
	      \end{proof}
	      
	      However, in general not all closed subobjects are subspaces in $\Sh$, \ie $\cld$ need not be a codominance. One can see this already on representable objects: the open subobjects of $\y[X]$ are given by open embeddings into $X$ while closed subobjects are given by continuous injective maps into $X$ with closed image which are not necessarily (closed) embeddings.
	      
	      \intermission
	      
	      Our next task is to show that Yoneda embedding preserves overtness and compactness.
	      \begin{lemma}\label{Lemma: quantifiers_on_the_topology_of_representable_objects}
	      	Let $X$ be a topological space from $\site$. The quantifiers $\exists^{\y[X]}, \forall^{\y[X]}\colon \soc^{\y[X]} \to \soc$ of the representable object $\y[X]$ restrict to $\exists^{\y[X]}, \forall^{\y[X]}\colon \opn^{\y[X]} \to \nnst$, given (in view of Lemma~\ref{Lemma: stable_subsets_of_a_representable_object}) as $\exists^{\y[X]}_Z, \forall^{\y[X]}_Z\colon \tp(Z \times X) \to \pst(Z)$,
	      	$$\exists^{\y[X]}_Z(U) = \st{z \in Z}{\xsome{x}{X}{(z, x) \in U}} = p(U)$$
	      	and
	      	$$\forall^{\y[X]}_Z(U) = \st{z \in Z}{\xall{x}{X}{(z, x) \in U}} = p(U^C)^C$$
	      	where $U \in \tp(Z \times X)$ and $p\colon Z \times X \to Z$ is the projection.
	      \end{lemma}
	      \begin{proof}
				Recall the construction of the quantifiers in a topos from Section~\ref{Section: topoi}.
				
				We start with the existential quantifier. An exercise in using Yoneda lemma shows that the evaluation map $\epsilon\colon \nnst^{\y[X]} \times \y[X] \to \nnst$ is given by
	      	$$\epsilon_Z\colon \pst(Z \times X) \times \C(Z, X) \to \pst(Z),$$
	      	$$\epsilon_Z(A, f) = \st{z \in Z}{(z, f(z)) \in A} = (\id[Z], f)^{-1}(A).$$
	      	Thus the pullback of the truth map along $\epsilon$ at component $Z$ equals
	      	$$\st{(A, f) \in \pst(Z \times X) \times \C(Z, X)}{(\id[Z], f)^{-1}(A) = Z}.$$
	      	We skip the calculation what the image of its projection onto $\nnst^{\y[X]}$ is by recalling how internal statements in $\Sh$ are interpreted externally~\cite{Mac_Lane_S_Moerdijk_I_1992:_sheaves_in_geometry_and_logic}, obtaining
	      	$$\st{A \in \pst(Z \times X)}{\exists (U_i)_{i \in I} \text{ an open cover of } Z.\xall{i}{I}\xsome{f_i}{\C(U_i, X)}{\im{(\id[U_i], f_i)} \subseteq A}}.$$
	      	The (restriction of) the existential quantifier is the characteristic map $\chi\colon \nnst^{\y[X]} \to \soc$ of the inclusion of this set,
	      	$$\chi_Z\colon \pst(Z \times X) \to \soc(Z),$$
	      	\begin{multline*}
	      		\chi_Z(A) = \big\{f \in \ms{Z}\,\big|\,\exists (U_i)_{i \in I} \text{ an open cover of } \dom(f).\\
	      		\xall{i}{I}\xsome{g_i}{\C\big(U_i, X\big)}{\im{(\id[U_i], g_i)} \subseteq (f \times \id[X])^{-1}(U)}\big\} =
	      	\end{multline*}
	      	$$= \st{f \in \ms{Z}}{\exists (U_i)_{i \in I} \text{ an open cover of } \dom(f).\xall{i}{I}\xsome{g_i}{\C\big(U_i, X\big)}{\im{(\rstr{f}_{U_i}, g_i)} \subseteq U}}.$$
	      	We prove that $\chi$ factors through $\nnst$, and is given by the map $\pst(Z \times X) \to \pst(Z)$, $U \mapsto p(U)$. Clearly if $f$ is in the sieve above, its image is contained in $p(U)$. To prove the converse, it is sufficient to verify that the inclusion $p(U) \hookrightarrow Z$ is in the sieve. For any $x \in X$ let $e_x\colon p(U) \to Z \times X$ be the map $e_x(z) \dfeq (z, x)$. Then $\st{e_x^{-1}(U)}{x \in X}$ is an open cover of $p(U)$, and for every $x \in X$ the constant map $g_x(a) \dfeq x$ is such that $\im{(e_x^{-1}(U) \hookrightarrow Z, g_x)} \subseteq U$.
	      	
	      	The universal quantifier is simpler. The transpose of the map $\y[\one] \times \y[X] \to \y[\one] \stackrel{\top}{\longrightarrow} \opn$ is the map $\y[\one] \to \opn^{\y[X]}$, given at component $Z$ as
	      	$$\one \to \tp(Z \times X), \qquad * \mapsto Z \times X.$$
	      	Thus its characteristic map $\chi\colon \opn^{\y[X]} \to \soc$ at $Z$ is
	      	$$\chi_Z\colon \optp(Z \times X) \to \soc(Z),$$
	      	$$\chi_Z(U) = \st{f \in \ms{Z}}{(f \times \id[X])^{-1}(U) = \cod(f) \times X} =$$
	      	$$= \st{f \in \ms{Z}}{\im(f \times \id[X]) \subseteq U} = \st{f \in \ms{Z}}{\im(f) \times X \subseteq U}.$$
	      	From here the result easily follows.
	      \end{proof}
	      
	      \begin{proposition}\label{Proposition: overtness_and_compactness_of_representable_sheaves_in_gros_topos}
	      	\
	      	\begin{enumerate}
	      		\item
	      			All representable objects are overt in $\Sh$.
	      		\item
	      			If $X \in \obj{\site}$ is a compact topological space, then $\y[X]$ is a compact object in $\Sh$.
	      	\end{enumerate}
	      \end{proposition}
	      \begin{proof}
	      	\begin{enumerate}
	      		\item
				      By Lemma~\ref{Lemma: quantifiers_on_the_topology_of_representable_objects} the existential quantifier at a component $Z$ restricts to $\tp(Z \times X) \to \pst(Z)$, and is given by $U \mapsto p(U)$. We need to see that it further restricts to $\tp(Z \times X) \to \tp(Z)$, but this is true since classically all projections are open maps.
	      		\item
	      			By Lemma~\ref{Lemma: quantifiers_on_the_topology_of_representable_objects} the universal quantifier at a component $Z$ restricts to $\tp(Z \times X) \to \pst(Z)$, given by $U \mapsto p(U^C)^C$. If $X$ is compact, then projections along it are closed maps by Theorem~\ref{Theorem: characterization_of_classical_compactness}, so $p(U^C)^C$ is open.
	      	\end{enumerate}
	      \end{proof}
	      
	      We note however, that representable objects can be compact even if they are not represented by a compact topological space. This is because we ``test compactness just on spaces in $\site$''. For example, if $\site$ contains only discrete spaces, then $\opn = \cld = \nnst = \soc = \two$, so all objects in $\Sh$ are compact.
	      
	      \intermission
	      
	      We have enough preparation to consider metrization in gros topos. Which metric spaces are metrized depends very much on what $\site$ we take. For example, if $\site$ contains only $\emptyset$ and a singleton, then $\Sh$ is equivalent to $\Set$, and we are in the situation of Section~\ref{Section: classical_mathematics_model}. Topological properties of gros topos are nicer when $\site$ contains more diverse spaces. Here we consider the case when $\site$ consists of separable metric spaces\footnote{Separable metric spaces form a proper class rather than a set, so we can't literally take them all. However, we can take its skeleton, \ie one out of every homeomorphism class, which amounts to a set, as the cardinalities of such spaces are bounded above by the cardinality of continuum. When we discuss that a space is in this $\site$, we actually refer to its representative in its homeomorphism class, but we don't bother keeping mentioning this.}. This $\site$ is closed for countable limits and countable coproducts (all of which are calculated the same as in $\Top$).
	      
	      It is known~\cite{Mac_Lane_S_Moerdijk_I_1992:_sheaves_in_geometry_and_logic} that for such $\site$ the objects of natural numbers and real numbers in $\Sh$ are simply $\y[\NN]$ and $\y[\RR]$, respectively. Actually, $\y[\RR]$ is the object of Dedekind reals, but it is easy to see that its strict order is open, therefore in $\Sh$ Dedekind reals match what we defined as real numbers. Similarly, certain subspaces such as internal $\NN_{< n}$ and $\RR_{\geq 0}$ are representable as well.
	      
	      Note that we can study metrization in this $\Sh$ since $\y[\NN]$ and $\y[\emptyset]$ are overt by Proposition~\ref{Proposition: overtness_and_compactness_of_representable_sheaves_in_gros_topos}. Also, observe that the real numbers (and hence all metric spaces) are Hausdorff in $\Sh$ by Corollary~\ref{Corollary: reals_are_Hausdorff}.
			
			\begin{theorem}\label{Theorem: representable_metric_spaces_are_metrized}
				Let $(X, d)$ be a metric space in $\site$. Then $(\y[X], \y[d])$ is a metrized metric space in $\Sh$.
			\end{theorem}
			\begin{proof}
				Since the object of real numbers in $\Sh$ is $\y[\RR]$, and because the conditions for the metric space are expressed in the negative fragment of logic (which is preserved by $\y$ since limits are), $\y[d]\colon \y[(X \times X)] \to \y[\RR]$, or equivalently $\y[d]\colon \y[X] \times \y[X] \to \y[\RR]$, is a metric on $\y[X]$.
				
				We phrase the internal metrization of $(\y[X], \y[d])$ in such a way as to be suitable to interpret it externally. It is sufficient to prove
				\begin{align*}
					\xall{U}{\tp(\y[X])}\xsome{V}{\tp(\y[X] \times \y[\RR])}{}&&&&&&&&&&&&
				\end{align*}
				\vspace{-2em}
				\begin{align*}
					\Big(&\all{(x, r, y)}{\y[X] \times \y[\RR] \times \y[X]}{(x, r) \in V \land \y[d](x, y) < r \implies y \in U} \land\\
					&\all{y}{\y[X]}{y \in U \implies \some{(x, r)}{\y[X] \times \y[\RR]}{(x, r) \in V \land \y[d](x, y) < r}}\Big)
				\end{align*}
				since by Proposition~\ref{Proposition: overtness_and_compactness_of_representable_sheaves_in_gros_topos} $\y[X] \times \y[\RR]$ is overt, so an open subset in it is subovert (in fact even overt by Corollary~\ref{Corollary: if_dominance_then_open_subset_of_overt_is_overt_(original)}).
				
				For the proof we have to take an arbitrary $A \in \obj{\site}$, and any $U \in \tp(A \times X)$. Next we need a suitable open cover of $A$, but we claim that the cover with the single member $A$ will do. We pick the following $V \in \tp(A \times X \times \RR)$:
				$$V \dfeq \interior\st{(a, x, r) \in A \times X \times \RR}{\{a\} \times \ball{x}{r} \subseteq U}.$$
				We prove the two conjuncts individually.
				\begin{itemize}
					\item
						Take any $f\colon B \to A$ from $\site$, any $(x, r, y) \in \C(B, X) \times \C(B, \RR) \times \C(B, X)$, and any $g\colon C \to B$ from $\site$. Suppose that
						$$\xall{c}{C}{\big(f(g(c)), x(g(c)), r(g(c))\big) \in V} \quad \text{and} \quad \xall{c}{C}{d\big(x(g(c)), y(g(c))\big) < r(g(c))}$$
						hold; by the definition of $V$ this clearly implies $\xall{c}{C}{\big(f(g(c)), y(g(c))\big) \in U}$.
					\item
						Take any $f\colon B \to A$ from $\site$, any $y \in \C(B, X)$, and any $g\colon C \to B$ from $\site$. Suppose $\xall{c}{C}{\big(f(g(c)), y(g(c))\big) \in U}$ holds. We require a suitable open cover of $C$, and again we take just the whole $C$ while picking $(x, r) \in \C(C, X) \times \C(C, \RR)$ in the following way:
						$$x(c) \dfeq y(g(c)), \qquad r(c) \dfeq \tfrac{1}{4} \sup\st{s \in \intcc{0}{1}}{\ball[A]{f(g(c))}{s} \times \ball[X]{x(c)}{s} \subseteq U}.$$
						The map $r$ is indeed continuous since it is the quarter of the infimum of $1$ and the product distance of the point $\big(f(g(c)), x(c)\big)$ from the complement of $U$ (we bound $s$ by $1$ in case the complement of $U$ is empty); also note that it is always positive (by the supposition above since $U$ is open).
						
						The statement $\xall{c}{C}{\big(f(g(c)), x(c), r(c)\big) \in V}$ holds because we can wiggle $f(g(c))$, $x(c)$ and $r(c)$ by the distance $r(c)$, and still satisfy the condition $\{f(g(c))\} \times \ball{x(c)}{r(c)} \subseteq U$. The statement $\xall{c}{C}{d\big(x(c), y(g(c))\big) < r(c)}$ obviously holds since the distance is zero.
				\end{itemize}
			\end{proof}

	      Let $(U, d_\mtr{U}, u, E)$ be the classical Urysohn space where the map $E\colon P \to \Ury$ witnesses Urysohn properties; here $P = \st{\Utuple[n]{i}{x}{\omega} \in (U \times \RR_{\geq 0})^*}{\all{i, j}{\NN_{< n}}{\omega_i - \omega_j \leq d_\mtr{U}(x_i, x_j) \leq \omega_i + \omega_j}}$ (recall Definition~\ref{Definition: Urysohn_space}). We can write
	      $$(U \times \RR_{\geq 0})^* \ism \coprod_{n \in \NN} (U \times \RR_{\geq 0})^n.$$
	      The Yoneda embedding preserves coproducts in $\site$ which are calculated the same as in $\Top$, as is the case here. Since it preserves the negative fragment of logic as well, in which the other conditions for the Urysohn space can be expressed, we conclude that $(\y[U], \y[d_\mtr{U}])$ is the Urysohn space in $\Sh$, witnessed by $\y[E]$ (after the tedious verification that it is a \cms in $\Sh$). In short, the Urysohn space in $\Sh$ is representable by the Urysohn space in $\site$.
	      
	      \begin{theorem}
	      	\
	      	\begin{enumerate}
	      		\item
	      			All \cms[s] in $\Sh$ are metrized.
	      		\item
	      			All \ctb[s] in $\Sh$ are metrized and compact.
	      	\end{enumerate}
	      \end{theorem}
	      \begin{proof}
	      	\begin{enumerate}
	      		\item
	      			By Theorem~\ref{Theorem: metrization_of_Urysohn_space_implies_metrization_of_cmss}, using Proposition~\ref{Proposition: when_strongly_located_subsets_are_subspaces} and the fact that the Urysohn space is metrized by Theorem~\ref{Theorem: representable_metric_spaces_are_metrized}.
	      		\item
	      			By the previous item $\ctb[s]$ are metrized. In particular, the Hilbert cube is metrized, and since it is representable (because $(\y[\II])^{\y[\NN]} \ism \y[(\II^\NN)]$), it is compact by Proposition~\ref{Proposition: overtness_and_compactness_of_representable_sheaves_in_gros_topos}, as the classical Hilbert cube $\II^\NN$ is compact. We know that $\y[\RR]$ is Hausdorff. We conclude that $\ctb[s]$ are compact by Theorem~\ref{Theorem: metrization_of_Hilbert_cube_implies_metrization_of_ctbs} (and Proposition~\ref{Proposition: when_strongly_located_subsets_are_subspaces}).
	      	\end{enumerate}
	      \end{proof}


   \phantom{\cite{*}}

   \bibliographystyle{plain}
   \addcontentsline{toc}{chapter}{\numberline{}Bibliography}
   \markboth{}{Bibliography}
   
   {
   \raggedright
   \renewcommand{\markboth}[2]{}
   \bibliography{PhD_Davorin_Bibliography}
   }
   
   
	\chapter*{Povzetek v slovenskem jeziku}\addcontentsline{toc}{chapter}{\numberline{}Povzetek v slovenskem jeziku}

	Klasični matematični pristop k obravnavi matematičnih pojmov, kot so grupe, vektorski prostori, topološki prostori, gladke mnogoterosti itd., je obravnavati jih kot \emph{množice} z ustrezno \emph{dodatno strukturo}, morfizmi med njimi pa so preslikave, ki to strukturo ohranjajo. Tak pristop se je izkazal za plodovitega, saj imajo različne matematične panoge skupno osnovo --- \emph{teorijo množic}.
	
	Neprijetna posledica takega pristopa je, da je struktura, ki nas zanima, pravzaprav dodatek in ne intrinzična lastnost množic, ki jih obravnavamo, kar prinese zapletenejše definicije in dokaze. Na primer, v analizi na mnogoterostih je tangentni sveženj matematično zelo kompleksen objekt, pa tudi predstavljati si smerne vektorje kot ekvivalenčne razrede poti ni blizu običajni intuiciji. Splošneje, pri konstrukciji kateregakoli svežnja je potrebno opredeliti gladko strukturo na njegovi domeni in kodomeni ter preveriti, da je projekcija gladka, kar je postopek, ki je pogosto samo računsko zahteven, nima pa resne matematične vsebine, kljub vsemu pa ga je treba opraviti, četudi na gladkost sklepamo že ``po občutku'', kadar uporabljamo standardne metode za konstrukcije.
	
	Filozofija \df{sintetičnega} pristopa k matematiki je obratna klasični: ne izberemo eno osnovo za matematiko, pač pa za posamezno vejo zanjo specializirano v smislu, da so objekti, preslikave, aksiomi, logika in definicije prilagojeni tako, da zajamejo tisto vejo, izbrana struktura postane intrinzična lastnost objektov in da se dokazi poenostavijo. Sintetična diferencialna geometrija (s katero se je sintetični pristop začel) reši zgoraj opisane težave.
   
   V tej disertaciji predstavimo sintetično verzijo \emph{topologije}. Za model sintetične topologije vzamemo kategorijo z dovolj logične strukture (topos, tj.~kategorični model konstruktivne intuicionistične logike višjega reda), v katerem so vsi objekti avtomatsko opremljeni z lastno (intrinzično) topologijo in vsi morfizmi so glede na to topologijo zvezni. To dosežemo z izbiro podobjekta $\opn$ klasifikatorja podobjektov $\soc$ v toposu (tj.~podmnožice v množici resničnostnih vrednosti). Elementi $\opn$ so \df{odprte} resničnostne vrednosti, podmnožice, klasificirane z njimi, pa \df{odprte} podmnožice. Pogosto zahtevamo še dodatne lastnosti za $\opn$, kot naprimer, da je dominanca (kar pomeni, da so končni preseki odprtih množic odprti in da je odprtost tranzitivna lastnost --- podmnožica, odprta v odprti podmnožici, je odprta).
   
   V drugem delu disertacije se osredotočimo na metrične prostore v modelih sintetične topologije. Posebej proučimo, kdaj se metrična in intrinzična topologija ujemata.

	\section*{Sintetična topologija}
	
		Definicijo sintetične topologije posplošimo tako, da začnemo s testnimi resničnostnimi vrednostmi, iz katerih izpeljemo tako odprtost kot zaprtost. Naj bo $\tst \subseteq \soc$ poljubna podmnožica množice resničnostnih vrednosti $\soc$. Definiramo
      $$\opn \dfeq \st{u \in \soc}{(u \land t) \in \tst \text{ za vse } t \in \tst},$$
      $$\cld \dfeq \st{f \in \soc}{(f \impl t) \in \tst \text{ za vse } t \in \tst}.$$
      Elementi množic $\tst$, $\opn$ in $\cld$ se imenujejo \df{testne}, \df{odprte} in \df{zaprte} resničnostne vrednosti. Podobno poimenujemo podmnožice, klasificirane z njimi. Za množico $X$ vpeljemo oznake $\tetp(X)$, $\optp(X)$, $\cltp(X)$ za njene testne, odprte in zaprte podmnožice.
      
      $\opn$ in $\cld$ sta spodnji omejeni polmreži.
      
      \begin{definicija}
      	Množica $X$ je
      	\begin{itemize}
      		\item
      			\df{odkrita}, kadar za vse $U \in \optp(X)$ velja, da je resničnostna vrednost $\xsome{x}{X}{x \in U}$ odprta,
      		\item
      			\df{kompaktna}, kadar za vse $U \in \optp(X)$ velja, da je resničnostna vrednost $\xall{x}{X}{x \in U}$ odprta,
      		\item
      			\df{kondenzirana}, kadar za vse $F \in \cltp(X)$ velja, da je resničnostna vrednost $\xsome{x}{X}{x \in F}$ zaprta,
      		\item
      			\df{Hausdorffova}, kadar je diagonala v $X \times X$ zaprta podmnožica.
      	\end{itemize}
      \end{definicija}
      
      Odkritost, kompaktnost in kondenziranost se posplošijo na podmnožice.
      \begin{definicija}
      	Podmnožica $A \subseteq X$ je
      	\begin{itemize}
      		\item
      			\df{pododkrita} v $X$, kadar za vse $U \in \optp(X)$ velja, da je resničnostna vrednost $\xsome{x}{A}{x \in U}$ odprta,
      		\item
      			\df{podkompaktna} v $X$, kadar za vse $U \in \optp(X)$ velja, da je resničnostna vrednost $\xall{x}{A}{x \in U}$ odprta,
      		\item
      			\df{podkondenzirana} v $X$, kadar za vse $F \in \cltp(X)$ velja, da je resničnostna vrednost $\xsome{x}{A}{x \in F}$ zaprta.
      	\end{itemize}
      \end{definicija}
      
      \begin{trditev}	
      	Množica $X$ je kondenzirana natanko tedaj, ko je za vsako množico $Y$ projekcija $p\colon X \times Y \to Y$ zaprta preslikava.
      \end{trditev}
      
      \begin{trditev}
      	Če je $C \subseteq X$ podkondenzirana in $A \subseteq X$ zaprta, tedaj je $C \cap A$ podkondenzirana.
      \end{trditev}
      
      \begin{trditev}
      	Če je $X$ Hausdorffov in $C \subseteq X$ podkondenzirana, tedaj je $C$ zaprta.
      \end{trditev}
      
      Odkritost, kompaktnost in kondenziranost posplošimo še na preslikave.
      \begin{definicija}
      	Preslikava $f\colon X \to Y$ je
      	\begin{itemize}
      		\item
      			\df{odkrita}, kadar za vse $U \in \optp(Y)$ velja, da je resničnostna vrednost $\xsome{x}{X}{f(x) \in U}$ odprta,
      		\item
      			\df{kompaktna}, kadar za vse $U \in \optp(Y)$ velja, da je resničnostna vrednost $\xall{x}{X}{f(x) \in U}$ odprta,
      		\item
      			\df{kondenzirana}, kadar za vse $F \in \cltp(Y)$ velja, da je resničnostna vrednost $\xsome{x}{X}{f(x) \in F}$ zaprta.
      	\end{itemize}
      \end{definicija}
      
      \begin{trditev}
      	\
      	\begin{enumerate}
      		\item
      			Kompozicija odkrite/kompaktne/kondenzirane preslikave s poljubno preslikavo na levi je odkrita/kompaktna/kondenzirana.
      		\item
      			Končni produkt odkritih/kompaktnih/kondenziranih preslikav je odkrita/kompaktna/kon\-den\-zi\-ra\-na.
      		\item
      			Diagonalna preslikava v povleku preslikave z odprto/zaprto sliko vzdolž odkrite/kondenzirane preslikave je odkrita/kondenzirana. Če $\tst = \opn$, analogno velja za kompaktne preslikave in zaprte slike.
      	\end{enumerate}
      \end{trditev}
      
      Vemo, da so odprte podmnožice podprostori, kadar je $\opn$ dominanca. Za zaprte množice obstaja analogen pojem kodominance.
      
      \begin{definicija}
      	\
      	\begin{itemize}
      		\item
      			Podmnožica $A \subseteq X$ je \df{podprostor} v $X$, kadar, za vsako testno množico $B \in \tetp(A)$ obstaja taka testna množica $C \in \tetp(X)$, da je $B = A \cap C$.
      		\item
      			Odprta podmnožica, ki je obenem podprostor, se imenuje \df{krepko odprta}.
      		\item
      			Zaprta podmnožica, ki je obenem podprostor, se imenuje \df{krepko zaprta}.
      	\end{itemize}
      \end{definicija}
      
      Krepko odprte in krepko zaprte podmnožice so podprostori na kanoničen način.
      \begin{lema}
      	\
      	\begin{enumerate}
      		\item
      			Naj bo $A \subseteq U \subseteq X$. Če je $U$ krepko odprta v $X$ in $A$ testna v $U$, tedaj je $A$ ($= U \cap A$) testna v $X$.
      		\item
      			Naj bo $A \subseteq F \subseteq X$. Če je $F$ krepko zaprta v $X$ in $A$ testna v $F$, tedaj je $F \setimpl[X] A$ testna v $X$.
      	\end{enumerate}
      \end{lema}
      
      \begin{trditev}
      	Podmnožica, odprta/zaprta v krepko odprti/zaprti podmnožici, je odprta/zaprta.
      \end{trditev}
      
      \begin{trditev}
      	\
      	\begin{enumerate}
      		\item
      			Resničnostna vrednost $u \in \soc$ je krepko odprta natanko tedaj, ko zanjo velja pogoj
      			$$\all{p}{\soc}{(u \impl (p \in \tst)) \implies (u \land p) \in \tst}.$$
      		\item
      			Resničnostna vrednost $f \in \soc$ je krepko zaprta natanko tedaj, ko zanjo velja pogoj
      			$$\all{p}{\soc}{(f \impl (p \in \tst)) \implies (f \impl p) \in \tst}.$$
      	\end{enumerate}
      \end{trditev}
      
      \begin{definicija}
      	\
      	\begin{itemize}
      		\item
      			$\opn$ je \df{dominanca} za $\tst$, kadar so vse odprte resničnostne vrednosti krepko odprte.
      		\item
      			$\cld$ je \df{kodominanca} za $\tst$, kadar so vse zaprte resničnostne vrednosti krepko zaprte.
      	\end{itemize}
      \end{definicija}
      
      \begin{lema}
      	Naj bo $U \subseteq X$. Če za vsak $x \in X$ velja, da je resničnostna vrednost $x \in U$ krepko odprta, tedaj je $U$ krepko odprta podmnožica v $X$. Analogno za zaprtost.
      \end{lema}
      
      \begin{trditev}
      	Naslednje trditve so ekvivalentne.
      	\begin{enumerate}
      		\item
      			$\opn$ je dominanca.
      		\item
      			Krepko odprte podmnožice so klasificirane s krepko odprtimi resničnostnimi vrednostmi.
      		\item
      			Krepko odprte podmnožice so klasificirane z neko podmnožico v $\soc$.
      		\item
      			Preslikave so zvezne glede na krepko odprte podmnožice.
      	\end{enumerate}
      	Analogno za zaprtost.
      \end{trditev}
      
      \intermission
      
      V tipičnih modelih sintetične topologije ima običajno $\opn$ dodatne lastnosti.
      \begin{izrek}
      	Predpostavimo, da velja $\tst = \opn \subseteq \nnst$ in da je $\opn$ zaprta za $\lor$. Tedaj velja:
      	\begin{enumerate}
      		\item
      			$\cld = \st{f \in \soc}{\lnot{f} \in \opn}$,
      		\item
      			$\opn = \lnot \cld$,
      		\item
      			$\lnot\lnot \cld = \lnot \opn = \nnst[\cld]$,
      		\item
      			kondenziranost in kompaktnost se ujemata; poleg tega je kondenziranost dovolj preveriti le za stabilne zaprte podmnožice,
      		\item
      			$\cld$ je omejena podmreža v $\soc$,
      		\item
      			$\opn \cap \cld = \two = \opn \cap \nnst[\cld]$,
      		\item
      			$\opn$ je dominanca natanko tedaj, ko je odprtost tranzitivna,
      		\item
      			$\cld$ je kodominanca natanko tedaj, ko je zaprtost tranzitivna, in prav tako so stabilne zaprte podmnožice podprostori natanko tedaj, ko je stabilna zaprtost tranzitivna.
      	\end{enumerate}
      \end{izrek}
      
      \intermission
      
      Topologijo lahko opišemo z bazo.
      \begin{definicija}
      	Preslikava $b\colon \basis \to \tp(X)$ je \df{baza} za $X$, kadar za vsak $U \in \tp(X)$ obstaja množica $I$ in odkrita (indeksirajoča) preslikava $a\colon I \to \basis$, da velja $U = \bigcup\st{b(a(i))}{i \in I}$.
      \end{definicija}
      
      Primer uporabe baze je definicija produktne topologije.
      \begin{definicija}
      	Množica $X \times Y$ ima \df{produktno topologijo}, kadar je preslikava $p\colon \tp(X) \times \tp(Y) \to \tp(X \times Y)$, dana kot $P(U, V) \dfeq U \times V$, baza za $X \times Y$.
      \end{definicija}

	\section*{Realna števila in metrični prostori}
	
		Naš namen je obravnavati metrične prostore v kontekstu sintetične topologije. Posebej želimo uvideti povezavo med metrično in intrinzično topologijo metričnih prostorov. Osnovni pogoj za to obravnavo je, da so števne množice odkrite (v tem in naslednjem razdelku to predpostavimo) in da je relacija stroge urejenosti $<$ na realnih številih odprta (ekvivalentno, da so metrične odprte krogle intrinzično odprte). V ta namen rekonstruiramo realna števila.
		
		\begin{definicija}
			Struktura $(X, <, +, 0, \cdot, 1)$ je \df{proga}, kadar je
			\begin{itemize}
				\item
					$<$ stroga linearna urejenost,
				\item
					$(X, +, 0)$ komutativni monoid,
				\item
					$(X_{> 0}, \cdot, 1)$ komutativni monoid na pozitivnih elementih iz $X$, pri čemer je množenje di\-stri\-bu\-tiv\-no glede na seštevanje,
				\item
					$<$ arhimedsko ureja $X$,
				\item
					za vsak $x \in X$ obstaja $a \in X_{> 0}$, da je $x + a > 0$,
				\item
					$<$ je intrinzično odprta relacija na $X$.
			\end{itemize}
			Preslikava med progama je \df{morfizem prog}, kadar ohranja vso strukturo proge.
		\end{definicija}
		
		Izkaže se:
		\begin{itemize}
			\item
				vsi morfizmi prog so injektivni in kategorija prog je šibka urejenost,
			\item
				množica naravnih števil $\NN$ je začetna proga,
			\item
				množica celih števil $\ZZ$ je začetna med progami, ki so kolobarji,
			\item
				množica racionalnih števil $\QQ$ je začetna med progami, ki so obsegi.
		\end{itemize}
		
		Intuitivno so proge aditivni podmonoidi realnih števil, ki vsebujejo $1$ in so zaprti za množenje pozivnih elementov. Posledično je smiselna sledeča definicija.
		
		\begin{definicija}
			Realna števila $\RR$ so končna proga.
		\end{definicija}
		
		\begin{izrek}
			V toposu z naravnimi števili za vsak $\opn$, glede na katerega so števne množice odkrite, realna števila obstajajo. Primer modela so odprti (dvostranski) Dedekindovi rezi.
		\end{izrek}
		
		Naša konstrukcija realnih števil je torej običajna, le da se omejimo na odprte Dedekindove reze in s tem dosežemo odprtost relacije $<$. V primeru $\opn = \soc$ dobimo standardna Dedekindova realna števila (v standardnih modelih sintetične topologije so sicer vsi Dedekindovi rezi odprti in se torej naša realna števila ujemajo z običajnimi).
		
		Definicija metričnih prostorov in njihovih lastnosti je večidel standardna. Izjemi sta (zaradi odsotnosti aksioma izbire) popolna omejenost in polnost.
		
		\begin{definicija}
			Metrični prostor $(X, d, s, a)$, skupaj z zaporedjema $s\colon \NN \to \one + X$, $a\colon \NN \to \NN$, je \df{popolnoma omejen}, kadar velja
			$$\xall{n}{\NN}\xall{x}{X}\xsome{k}{(s^{-1}(X) \cap \NN_{< a_n})}{d(x, s_k) < 2^{-n}}.$$
		\end{definicija}
		
		\begin{definicija}
			\df{Napolnitev metričnega prostora} $(X, d)$ je končna med gostimi izometrijami z domeno $X$. Prostor je \df{poln}, kadar je identiteta na njem njegova napolnitev.
		\end{definicija}
		
		\begin{izrek}
			Če metrični prostor vsebuje gosto pododkrito podmnožico, tedaj je njegova vložitev v prostor ekvivalenčnih razredov Cauchyjevih družin oziroma v prostor lokacij njegova napolnitev. Če velja $\AC[\opn]{\NN}$, tedaj je vložitev v prostor ekvivalenčnih razredov Cauchyjevih zaporedij njegova napolnitev.
		\end{izrek}

	\section*{Metrizacijski izreki}
	
		Priprava iz prejšnjega razdelka implicira, da so v metričnih prostorih krogle intrinzično odprte in se torej lahko vprašamo, ali tvorijo bazo prostora.
		
		\begin{definicija}
			Za metrični prostor $(X, d)$ rečemo, da je \df{metriziran} (oziroma da je $X$ \df{metriziran} z metriko $d$), kadar je preslikava $B\colon X \times \RR \to \tp(X)$, definirana kot
			$$\ball{a}{r} \dfeq \st{x \in X}{d(a, x) < r},$$
			baza za $X$.
		\end{definicija}
		
		Primer metriziranih prostorov so odkrite množice z odločljivo enakostjo, opremljene z diskretno metriko. Po drugi strani je jasno, da vsi metrični prostori ne morejo biti metrizirani, saj na množici v splošnem obstajajo medsebojno neekvivalentne metrike. V tem razdelku podamo nekaj zadostnih (in potrebnih) pogojev za to, kdaj so določene skupine metričnih prostorov metrizirane. To dosežemo s preslikavami med metričnimi prostori, ki ohranjajo metriziranost. Zanimivi prostori, ki imajo obilo takih preslikav in se jih poslužimo v disertaciji, so sledeči:
		\begin{itemize}
			\item
				poln metrično separabilen ultrametričen \df{Bairov prostor} $\NN^\NN$, tj.~števni produkt prostorov $\NN$, opremljenih z diskretno metriko;
			\item
				Cantorjev prostor $\two^\NN$, ki je poln popolnoma omejen ultrametričen podprostor (še več, retrakt) Bairovega prostora;
			\item
				podprostor (in retrakt) Cantorjevega prostora, definiran kot
				$$\opcN \dfeq \st{\alpha \in \two^\NN}{\all{n}{\NN}{\alpha_n = 0 \implies \alpha_{n+1} = 0}}$$
				(intuitivno je $\opcN$ ``kompaktifikacija prostora $\NN$ z eno točko'', v smislu, da zaporedje $n$ enic, ki se nadaljuje z ničlami, predstavlja $n \in \NN$, medtem ko zaporedje samih enic predstavlja $\infty$);
			\item
				prostora $\RR$ in $\II$, opremljena z evklidsko topologijo;
			\item
				\df{Hilbertova kocka} $\II^\NN$;
			\item
				\df{Urisonov prostor}, tj.~poln metrično separabilen prostor, univerzalen za razširjanje izometrij vanj s končne na metrično separabilno domeno (ob tem pokažemo, da tak prostor obstaja tudi konstruktivno).
		\end{itemize}
      
      Kdaj je diskreten metrični prostor metriziran, vemo. Za obravnavo prostorov s stekališči je uporaben princip \wso[:]
      $$\all{U}{\tp(\opcN)}{\infty \in U \implies U \between \NN}.$$
      Gre za to, da je $\infty$ stekališče $\NN$ v metričnem smislu, \wso pa zagotavlja, da je tudi v intrinzičnem. Okrajšava \wso pomeni (v angleščini) `weakly sequentially open'.
      
      Princip \wso je tesno povezan s sintetično kompaktnostjo in metriziranostjo (predvsem metrično separabilnih ter popolnoma omejenih) metričnih prostorov.
      
      \begin{trditev}
         Princip \wso je neklasičen (tj.~implicira negacijo izključene tretje možnosti).
      \end{trditev}
      
      \begin{definicija}
         Metrični prostor je \df{Lebesgueov}, kadar za vsako njegovo odkrito indeksirano pokritje s kroglami velja, da obstaja $\epsilon \in \RR_{> 0}$, za katerega lahko zmanjšamo radije, pa krogle še vedno pokrivajo prostor.
      \end{definicija}
      
      \begin{izrek}
         Za metrični prostor $\mtr{X} = (X, d)$ veljajo sledeče trditve.
         \begin{enumerate}
            \item
               Če je $X$ odkrit, \wso velja in so zaprte krogle v $\mtr{X}$ s polmeri oblike $2^{-t}$, $t \in \opcN$, podkompaktne, tedaj je $\mtr{X}$ metriziran.
            \item
               Če je $X$ kompakten in \wso velja, je $\mtr{X}$ Lebesgueov.
            \item
               Naj velja $\ACopn$ in \wso[]. Če je $X$ kompakten in $\mtr{X}$ metrično separabilen, tedaj je $\mtr{X}$ popolnoma omejen.
            \item
               Če je $\mtr{X}$ popolnoma omejen, Lebesgueov in metriziran, tedaj je odkrit in kompakten.
         \end{enumerate}
      \end{izrek}
      
      \begin{posledica}
         Naj velja \wso in naj bo $\RR$ Hausdorffov. Sledeči trditvi sta ekvivalentni za popolnoma omejen metrični prostor $\mtr{X} = (X, d, s, a)$.
         \begin{enumerate}
            \item
               $X$ je odkrit in kompakten.
            \item
               $\mtr{X}$ je Lebesgueov in metriziran.
         \end{enumerate}
         Posledično je odkrita kompaktna množica metrizirana s katerokoli popolnoma omejeno metriko na njej (posebej to pomeni, da so vse te metrike medsebojno ekvivalentne).
      \end{posledica}
      
      \begin{izrek}
         Predpostavimo \wso[].
         \begin{enumerate}
            \item
               Vsaka kompaktna preslikava, ki slika v $\RR$, je omejena.
            \item
               Naj velja $\opn \subseteq \nnst$. Če je $X$ posedujoč in $f\colon X \to \RR$ odkrita kompaktna preslikava, tedaj ima $f$ supremum in infimum.
         \end{enumerate}
      \end{izrek}
      
      \begin{posledica}
         Če \wso velja, tedaj so (pod)kompaktni metrični (pod)prostori omejeni.
      \end{posledica}
      
      \begin{trditev}
      	\
      	\begin{enumerate}
      		\item
      			$\opcN$ je Lebesgueov.
      		\item
      			$\two^\NN$ je Lebesgueov natanko tedaj, ko je vsaka pododkrita prečka enakomerna (različica Brou\-wer\-je\-ve\-ga načela za pahljače).
      		\item
      			$\NN^\NN$ ni Lebesgueov.
      	\end{enumerate}
      \end{trditev}
      
      \begin{izrek}
      	Izmed sledečih trditev sta prvi dve ekvivalentni, iz njiju sledi tretja in če je $\opn$ dominanca, so vse tri ekvivalentne.
      	\begin{enumerate}
      		\item
      			$\opcN$ je kompakten in \wso velja.
      		\item
      			$\opcN$ je metriziran.
      		\item
      			Za vsak $U \in \tp(\opcN)$ velja, da če vsebuje $\infty$, tedaj obstaja $\epsilon \in \RR_{> 0}$, da $\ball{\infty}{r} \subseteq U$.
      	\end{enumerate}
      \end{izrek}
      
      \begin{izrek}
      	Spodnji trditvi sta ekvivalentni.
      	\begin{enumerate}
      		\item
      			$\two^\NN$ je kompakten in \wso velja.
      		\item
      			$\two^\NN$ je metriziran in vsaka pododkrita prečka je enakomerna.
      	\end{enumerate}
      \end{izrek}
      
      \begin{izrek}
      	Spodnji trditvi sta ekvivalentni tako za Bairov kot za Cantorjev prostor.
      	\begin{enumerate}
      		\item
      			Bairov/Cantorjev prostor je metriziran.
      		\item
      			Bairov/Cantorjev prostor ima produktno intrinzično topologijo in \wso velja.
      	\end{enumerate}
      \end{izrek}
      
      \intermission
      
      Naš cilj je sedaj karakterizirati metrizacijo polnih metrično separabilnih in polnih popolnoma omejenih metričnih prostorov z metrizacijo enega prostora iz vsake skupine. Spomnimo se dveh skupin preslikav iz klasične topologije, preko katerih je topologija enega prostora porojena s topologijo drugega: kvocientne preslikave in topološke vložitve. V splošnem v sintetični topologiji surjekcije ohranjajo intrinzično topologijo, ne pa nujno metrične, medtem ko za injekcije velja ravno obratno.
      
      \begin{definicija}
      	Preslikava med metričnimi prostori $f\colon \mtr{X} \to \mtr{Y}$ je \df{metrična kvocientna}, kadar je surjektivna in ohranja metrično topologijo v smislu, da za vsak $U \subseteq \mtr{Y}$ velja
      	$$U \text{ metrično odprta v } \mtr{Y} \iff f^{-1}(U) \text{ metrično odprta v } \mtr{X}.$$
      \end{definicija}
      
      \begin{trditev}
      	Če je $f\colon \mtr{X} \to \mtr{Y}$ metrična kvocientna in je $\mtr{X}$ metriziran, potem je metriziran tudi $\mtr{Y}$.
      \end{trditev}
      
      \begin{izrek}
      	Predpostavimo $\ACopn$.
      	\begin{enumerate}
      		\item
      			Za vsak poln metrično separabilen prostor obstaja kvocientna preslikava z retrakta Bairovega prostora nanj.
      		\item
      			Za vsak poln popolnoma omejen metrični prostor obstaja kvocientna preslikava z retrakta Cantorjevega prostora nanj.
      	\end{enumerate}
      \end{izrek}
      
      \begin{izrek}
      	Naj velja $\ACopn$.
      	\begin{enumerate}
      		\item
      			Če je Bairov prostor metriziran, so metrizirani vsi polni metrično separabilni prostori.
      		\item
      			Če je Cantorjev prostor metriziran, so metrizirani vsi polni popolnoma omejeni metrični prostori.
      		\item
      			Če je Cantorjev prostor kompakten, so kompaktni vsi polni popolnoma omejeni metrični prostori.
      	\end{enumerate}
      \end{izrek}
      
      \intermission
      
      Zdaj se posvetimo še vložitvam.
      
      \begin{izrek}
      	Naj bo $A \subseteq X$, $\mtr{X} = (X, d_{\mtr{X}})$ metrični prostor in $\mtr{A} = (A, d_{\mtr{A}})$ njegov metrični podprostor.
      	\begin{enumerate}
      		\item
      			Če je $\mtr{A}$ metriziran, je $A$ podprostor (v intrinzičnem smislu) v $X$.
      		\item
      			Če je $\mtr{X}$ metriziran, $A$ podprostor v $X$ in $A$ vsebuje pododkrito metrično gosto podmnožico, tedaj je $\mtr{A}$ metriziran.
      	\end{enumerate}
      \end{izrek}
      
      \begin{posledica}
      	Metrično separabilen metrični podprostor v metriziranem prostoru je metriziran natanko tedaj, ko je podprostor v intrinzičnem smislu.
      \end{posledica}
      
      \begin{definicija}
      	Podmnožica $A \subseteq X$ je \df{krepko umeščena} v metričnem prostoru $\mtr{X} = (X, d)$, kadar je umeščena v $\mtr{X}$ in velja $A = \st{x \in X}{d(A, x) = 0}$.
      \end{definicija}
      
      \begin{lemma}
      	Naj bo $\mtr{X} = (X, d_{\mtr{X}})$ metrični prostor, $A \subseteq X$ umeščena v $\mtr{X}$ in $\mtr{A} = (A, d_{\mtr{A}})$ poln metrični podprostor v $\mtr{X}$, ki vsebuje pododkrito metrično gosto podmnožico. Tedaj je $A$ krepko umeščena.
      \end{lemma}
      
      \begin{trditev}
      	Če je $\RR$ Hausdorffov in so stabilne zaprte množice podprostori, potem so krepko umeščene podmnožice podprostori.
      \end{trditev}
      
      \begin{izrek}
      	\
      	\begin{enumerate}
      		\item
      			Vsak metrično separabilen metrični prostor se vloži v Urisonov prostor kot umeščena podmnožica. Če je poln, je vložitev krepko umeščena.
      		\item
      			Vsak popolnoma omejen metrični prostor se vloži v Hilbertovo kocko kot umeščena podmnožica. Če je poln, je vložitev krepko umeščena.
      	\end{enumerate}
      \end{izrek}
      
      \begin{izrek}
      	Naj velja, da so krepko umeščene podmnožice podprostori.
      	\begin{enumerate}
      		\item
      			Če je Urisonov prostor metriziran, so metrizirani vsi polni metrično separabilni prostori.
      		\item
      			Če je Hilbertova kocka metrizirana, so metrizirani vsi polni popolnoma omejeni metrični prostori.
      		\item
      			Če je Hilbertova kocka kompaktna in $\RR$ Hausdorffov, so vsi polni popolnoma omejeni metrični prostori kompaktni.
      	\end{enumerate}
      \end{izrek}
      
      \intermission
      
      Sledeči izrek povzame teorijo metrizacije polnih metrično separabilnih oziroma popolnoma omejenih metričnih prostorov.
      
      \begin{izrek}
      	Predpostavimo, da so polni metrično separabilni/popolnoma omejeni prostori metrizirani. Tedaj veljajo naslednje trditve.
      	\begin{enumerate}
				\item
					Intrinzična topologija polnih metrično separabilnih/popolnoma omejenih prostorov se izraža z metričnimi kroglami. Natančneje, podmnožica je odprta natanko tedaj, ko je odkrito indeksirana unija krogel.
				\item
					Do topološke ekvivalence natančno lahko množico opremimo z največ eno metriko, glede na katero postane poln metrično separabilen/popolnoma omejen prostor.
				\item
					Polni metrično separabilni/popolnoma omejeni prostori imajo števno bazo.
				\item
					Polni metrično separabilni/popolnoma omejeni prostori so odkriti.
				\item
					Princip \wso[], skupaj s svojimi posledicami, velja.
				\item
					Za poljuben poln metrično separabilen/popolnoma omejen prostor $\mtr{X} = (X, d)$ je vsaka pre\-sli\-ka\-va $f\colon X \to Y$ metrično zvezna in posledično $\epsilon\text{-}\delta$-zvezna, ne glede na izbiro metrike na $Y$.
      	\end{enumerate}
      \end{izrek}

	\section*{Modeli sintetične topologije}
	
		V tem razdelku uporabimo splošno teorijo za štiri konkretne modele sintetične topologije:
		\begin{itemize}
			\item
				klasično matematiko,
			\item
				izračunljivost prvega tipa oziroma ruski konstruktivizem,
			\item
				izračunljivost drugega tipa oziroma močnejša verzija Brouwerjevega intuicionizma,
			\item
				gros topos nad separabilnimi metričnimi prostori.
		\end{itemize}
		
		\intermission
		
		V klasični teoriji množic velja $\soc = \two$. Imamo torej štiri možnosti za $\tst \subseteq \soc$.
		\begin{itemize}
			\item\proven{$\tst = \emptyset$}
				Tedaj $\opn = \cld = \two$. Vse množice so diskretne, Hausdorffove, odkrite, kompaktne in kondenzirane.
			\item\proven{$\tst = \{\bot\}$}
				Tedaj $\opn = \two$, ampak $\cld = \{\top\}$, torej so vse podmnožice odprte, a le celotne množice so zaprte v samih sebi. Vse množice so diskretne, odkrite in kompaktne, samo prazna in enojci so Hausdorffovi in posedujoče množice so kondenzirane.
			\item\proven{$\tst = \{\top\}$}
				Tedaj $\opn = \{\tst\}$ in $\cld = \two$. Samo celotne množice so odprte v sebi, so pa vse množice zaprte. Le prazna množica in enojci so diskretni, posedujoče množice so odkrite, medtem ko so vse množice Hausdorffove, kompaktne in kondenzirane.
			\item\proven{$\tst = \two$}
				Tedaj spet $\opn = \cld = \two$. Vse množice so diskretne, Hausdorffove, odkrite, kompaktne in kondenzirane.
		\end{itemize}
		Za našo teorijo metričnih prostorov pride v poštev le slednja (ki je tudi standardna) izbira $\tst = \opn = \cld = \two$, saj želimo, da so števne množice odkrite in $\tst = \opn$. To je trivialni model naše teorije, saj so vse množice odkrite in imajo odločljivo enakost ter so posledično metrizirane z diskretno metriko.
		
		\intermission
		
		Za izračunljivost prvega tipa veljajo sledeči aksiomi.
		\begin{itemize}
			\item
				Števna izbira $\AC{\NN}$ in splošneje odvisna izbira.
			\item
				\df{Princip Markova}: če v dvojiškem zaporedju niso vsi členi enaki $1$, obstaja člen, ki je enak $0$.
			\item
				\df{Aksiom naštevnosti}: števnih podmnožic $\NN$ je števno mnogo.
		\end{itemize}
		
		V tem modelu vzamemo $\opn = \Ros$. Čeprav \wso velja, nediskretni metrični prostori nasplošno niso metrizirani.
		
		\begin{izrek}
			\wso velja, obenem pa obstaja taka $V \in \tp(\opcN)$, da $\infty \in V$, vendar $\ball{\infty}{r}$ ni vsebovana v $V$ za noben $r \in \RR_{> 0}$.
		\end{izrek}
		
		\begin{izrek}
			Naj metrični prostor $\mtr{X} = (X, d)$ vsebuje pododkrito metrično gosto podmnožico.
			\begin{enumerate}
				\item
					Če $\mtr{X}$ vsebuje stekališče, tedaj ni niti metriziran niti kompakten.
				\item
					Definirajmo, da je $\mtr{X}$ lokalno polumeščen, kadar za vsak $x \in X$ obstaja tak $r \in \RR_{> 0}$, da je razdalja od $x$ do prebodene krogle okrog $x$ s polmerom $r$, tj.
					$$d\big(\st{y \in X}{0 < d(x, y) < r}, x \big),$$
					razširjeno realno število. Tedaj sta sledeči trditvi ekvivalentni.
					\begin{itemize}
						\item
							$\mtr{X}$ je metriziran in lokalno polumeščen.
						\item
							Metrika $d$ je ekvivalentna diskretni metriki.
					\end{itemize}
					Posebej, če drži ena od teh trditev (in s tem obe), ima $X$ odločljivo enakost.
			\end{enumerate}
		\end{izrek}
		Ta izrek velja tudi, če povečamo $\opn$.
		
		\intermission
		
		Za izračunljivost drugega tipa veljajo sledeči aksiomi.
		\begin{itemize}
			\item
				\df{Funkcijsko funkcijska izbira} $\AC{\NN^\NN, \NN^\NN}$.
			\item
				\df{Princip zveznosti}: za vsako $f\colon \NN^\NN \to \NN$ in $\alpha \in \NN^\NN$ obstaja $k \in \NN$, da $\beta \in \ball{\alpha}{2^{-k}}$ implicira $f(\alpha) = f(\beta)$ (na kratko: $f$ je $\epsilon\text{-}\delta$-zvezna).
			\item
				\df{Princip pahljače}: vsaka odločljiva prečka je enakomerna.
		\end{itemize}
		
		Spet vzamemo $\opn = \Ros$.
		
		\begin{izrek}
			Bairov prostor $\NN^\NN$ je metriziran. Posledično, vsi polni metrično separabilni prostori so metrizirani.
		\end{izrek}
		
		\begin{izrek}
			Cantorjev prostor je kompakten. Posledično so vsi polni popolnoma omejeni metrični prostori kompaktni.
		\end{izrek}
		
		\intermission
		
		Gros topos $\Sh$ je topos snopov nad majhno polno podkategorijo $\site$ kategorije topoloških prostorov, zaprto za odprte vložitve in končne produkte. \df{Yonedova vložitev} $\y\colon \site \to \Sh$ je poln, zvest in na objektih injektiven funktor; objekti in morfizmi so \df{predstavljivi}, kadar so v njeni sliki.
		
		Smiselna izbira za $\opn$ je $\C(\insarg, \sier)$, kjer je $\sier$ prostor Sierpińskega.
		
		\begin{lema}
			Stabilna zaprtost je tranzitivna v $\Sh$. Posledično so stabilne zaprte podmnožice podprostori.
		\end{lema}
		
		\begin{izrek}
			\
			\begin{enumerate}
				\item
					Vsi predstavljivi objekti v $\Sh$ so odkriti.
				\item
					Objekti v $\Sh$, predstavljivi s kompaktnim topološkim prostorom iz $\site$, so kompaktni v $\Sh$.
			\end{enumerate}
		\end{izrek}
		
		\begin{izrek}
			Naj bo $\site$ skelet separabilnih metričnih prostorov.
			\begin{enumerate}
				\item
					Za vsak metrični prostor $(X, d)$ iz $\site$ je $(\y[X], \y[d])$ metriziran metrični prostor v $\Sh$.
				\item
					Urisonov prostor v $\Sh$ je predstavljiv (s klasičnim Urisonovim prostorom), torej metriziran. Posledično so vsi polni metrično separabilni prostori v $\Sh$ metrizirani.
				\item
					Hilbertova kocka v $\Sh$ je predstavljiva s klasično Hilbertovo kocko in torej kompaktna. Ker so obenem realna števila Hausdorffova, so vsi popolnoma omejeni metrični prostori v $\Sh$ kompaktni.
			\end{enumerate}
		\end{izrek}   
   
   \chapter*{Izjava o avtorstvu}
   Spodaj podpisani \textsc{Davorin Lešnik} izjavljam, da je disertacija plod lastnega študija in raziskav.

\end{document}